 \documentclass [final]{amsart}
\usepackage{mathabx}
\usepackage[utf8]{inputenc} 
\usepackage[english]{babel}
\usepackage{amstext}
\usepackage{euscript}
\usepackage{bbm}
\usepackage{mathrsfs}
\usepackage{enumerate}
\usepackage{graphicx}
\usepackage{caption}
\usepackage{amscd}
\usepackage{verbatim}
\usepackage{amssymb}
\usepackage{amsthm}
\usepackage{amsmath}
\usepackage{amstext}
\usepackage{amsfonts}
\usepackage{latexsym}
\usepackage{mathtools}
\usepackage{enumitem} 
\usepackage{accents}
\usepackage{stmaryrd}
\usepackage{float}

\usepackage[foot]{amsaddr} %package that puts the authors addresses on the title page as a footnote (when using AMSART document class only).
\usepackage{textcomp} % this package adds the section symbol with the \textsection command
\usepackage{cite} % this package makes a list of citations appear in a numerically ordered list and ...

\usepackage[usenames,dvipsnames]{xcolor} 
\definecolor{dkblue}{RGB}{1,31,91} % This is a dark Blue     
\definecolor{oxfordblue}{RGB}{4,30,66}

\definecolor{mycolor}{rgb}{0.97,0.98,0.99}
\definecolor{mysecondcolor}{rgb}{0.2,0.2,0.8}

\definecolor{mydarkgreen}{RGB}{42,55,46}
\definecolor{mygreenone}{RGB}{76,131,122}
\definecolor{mytan}{RGB}{225,221,191}
\definecolor{mytant}{RGB}{247.5,243.1,210.1}
\definecolor{mydarkblue}{RGB}{4,37,58}
\definecolor{mydarkbluet}{RGB}{6,55.5,87}
\definecolor{mydarkbluett}{RGB}{12,111,174}

\usepackage[colorlinks=true, pdfstartview=FitV, linkcolor=dkblue, citecolor=dkblue, urlcolor=dkblue]{hyperref}

\usepackage{todonotes}
\usepackage[color]{showkeys}

\definecolor{myred}{rgb}{0.7,0.1,0.1}
\definecolor{mygreen}{rgb}{0.1,0.7,0.1}
\definecolor{myblue}{rgb}{0.2,0.2,0.5}

\theoremstyle{plain}
\newtheorem{thm}{Theorem}
\newtheorem{remark}[thm]{Remark}
\newtheorem{prop}[thm]{Proposition}

\newtheorem{lemma}[thm]{Lemma}
\newtheorem{defn}[thm]{Definition}

\numberwithin{equation}{section}
\numberwithin{thm}{section}

\newcommand{\pv}{\text{p.v.}}

\def\R {\mathbb R}

\newcommand{\minspace}{\hspace{0.05cm}}
\newcommand{\etab}{{\boldsymbol{\eta}}}
\newcommand{\thetab}{{\boldsymbol{\theta}}}
\newcommand{\xib}{\bm{\xi}}
\newcommand{\alphab}{{\bm{\alpha}}}
\newcommand{\betab}{{\bm{\beta}}}
\newcommand{\gammab}{{\bm{\gamma}}}
\newcommand{\deltab}{{\bm{\delta}}}
\newcommand{\xb}{{\bm{x}}}
\newcommand{\yb}{{\bm{y}}}
\newcommand{\hb}{{\boldsymbol{h}}}
\newcommand{\hx}{\bm{\widehat{x}}}
\newcommand{\hxh}{\bm{\widehat{x}_h}}
\newcommand{\hX}{\bm{\widehat{X}}}
\newcommand{\hy}{\bm{\widehat{y}}}
\newcommand{\hz}{\bm{\widehat{z}}}

\newcommand{\hxi}{\bm{\widehat{\xi}}}

\newcommand{\combin}[2]{\begin{pmatrix}
                    #1\\
                    #2
                 \end{pmatrix}}

\newcommand{\domainA}{\mc{DA}_{\sigma_1,\sigma_2}}

\newcommand{\paren}[1]{\left(#1\right)}
\newcommand{\jump}[1]{\llbracket#1\rrbracket}
\newcommand{\D}[2]{\frac{d#1}{d#2}}
\newcommand{\PD}[2]{\frac{\partial#1}{\partial#2}}

\newcommand{\PDD}[3]{\frac{\partial^{#1}{#2}}{\partial{#3}^{#1}}}

\newcommand{\at}[2]{\left. #1 \right|_{#2}}

\newcommand{\mc}[1]{\mathcal{#1}}
\newcommand{\wh}[1]{\widehat{#1}}

\newcommand{\bm}[1]{\boldsymbol{#1}}
\newcommand{\abs}[1]{\left\lvert #1 \right\rvert}
\newcommand{\norm}[1]{\left\lVert #1 \right\rVert}

\newcommand{\mbs}{{\mathbb{S}^2}}
\newcommand{\mcu}{{\mc{U}}}

\newcommand{\mcbr}[1]{{\mc{B}\paren{#1}}}
\newcommand{\mbr}{{\mathbb{R}^2}}
\newcommand{\mbrn}[1]{{\mathbb{R}^{#1}}}

\newcommand{\colofmat}[2]{\bm{#1}_{\bullet,#2}}
\newcommand{\rowofmat}[2]{\bm{#1}_{#2,\bullet}}

\newcommand{\starnorm}[1]{\left| #1 \right|_*}
\newcommand{\circlenorm}[1]{\left| #1 \right|_\circ}

\newcommand{\vph}{\varphi}

\begin{document}

\date{\today}

\title[Well-Posedness of the 3D Peskin Problem]{Well-Posedness of the 3D Peskin Problem}

\author[E. Garc\'ia-Ju\'arez]{Eduardo Garc\'ia-Ju\'arez$^{\ast}$}
\thanks{$^{\ast}$supported by the European Union’s Horizon 2020 research and innovation programme under the Marie Skłodowska-Curie grant agreement CAMINFLOW No 101031111, and the AEI project PID2021-125021NAI00 (Spain).
}
\address{$^{\ast}$Departamento de An\'alisis Matem\'atico, Universidad de Sevilla, C/Tarfia s/n, Campus Reina Mercedes, 41012, Sevilla, Spain. \href{mailto:egarciajuarez@ub.edu}{egarciajuarez@ub.edu}}

\author[P.-C. Kuo]{Po-Chun Kuo$^{\dagger,\ddagger}$}
\thanks{$^{\ddagger}$partially supported by NSF grant DMS-2042144 (USA) awarded to YM}

\author[Y. Mori]{Yoichiro Mori$^{\dagger,\mathsection}$}
\thanks{$^{\mathsection}$partially supported by the NSF grant DMS-1907583, 2042144 (USA) and the Math+X award from the Simons Foundation.}

\author[R. M. Strain]{Robert M. Strain$^{\dagger,\mathparagraph}$}
\address{$^\dagger$Department of Mathematics, University of Pennsylvania, David Rittenhouse Lab., 209 South 33rd St., Philadelphia, PA 19104, USA. 
$^{\ddagger}$\href{mailto:kuopo@sas.upenn.edu}{kuopo@sas.upenn.edu}
$^{\mathsection}$\href{mailto:y1mori@math.upenn.edu}{y1mori@math.upenn.edu}
$^{\mathparagraph}$\href{mailto:strain@math.upenn.edu}{strain@math.upenn.edu}}
\thanks{$^{\mathparagraph}$partially supported by the NSF grants DMS-1764177 and DMS-2055271 (USA)}

\begin{abstract}
This paper introduces the {\em{3D Peskin problem}}: a two-dimensional elastic membrane immersed in a three-dimensional steady Stokes flow. We obtain the equations that model this free boundary problem and show that they admit a boundary integral reduction, providing an evolution equation for the elastic interface. We consider general nonlinear elastic laws, i.e., the fully nonlinear Peskin problem, and prove that the problem is well-posed in low-regularity H\"older spaces. Moreover, we prove that the elastic membrane becomes smooth instantly in time.
\end{abstract}

\maketitle

\tableofcontents

\thispagestyle{empty}

\section{Introduction}

The immersed boundary method, introduced by Peskin \cite{Peskin77, Peskin02} to study the blood flow around heart valves, has been widely applied to numerically study fluid-structure interaction (FSI) problems. These FSI problems, in which a fluid interacts with elastic structures, appear naturally in many engineering and biophysics applications \cite{Pozrikidis03, richter17}. Despite their importance, both the computational methods and the FSI problems themselves are poorly understood from an analytical standpoint. A major impediment has been the lack of analytical understanding of the underlying PDEs, which are typically nonlinear and nonlocal. Results are particularly scarce in the more realistic three-dimensional settings, where the coupling of nonlocal effects with non-trivial geometry substantially increases the complexity of the problem.

Since the recent breakthrough works \cite{LinTong19} and \cite{MoriRodenbergSpirn19}, which provided the strong solution theory for the problem of an immersed elastic string in a two-dimensional fluid, the so-called \textit{2D Peskin problem} has attracted a lot of attention \cite{Tong21, Li21, GJMoriStrain20, GancedoGraneroScrobogna21, KeNguyen21, CameronStrain21, Tong22}. 
In this paper, we initiate the study of its three-dimensional counterpart. We introduce the formulation and develop the well-posedness theory for the \textit{three-dimensional (fully nonlinear) Peskin problem} of an elastic membrane immersed in a fluid. 

\subsection{Description of the problem}

We consider the following problem in which a three-dimensional incompressible Stokes fluid interacts with an elastic membrane in $\mathbb{R}^3$.
A closed elastic interface $\Gamma$ encloses a simply connected bounded domain $\Omega\subset\mathbb{R}^3$
filled with a Stokes fluid with viscosity $\mu$. The outside region $\mathbb{R}^3\backslash \Omega$ is filled with 
a Stokes fluid of viscosity $1$. The equations satisfied are:
\begin{align}
\label{uin}
\mu \Delta \bm{u}-\nabla p&=0 \text{ in } \Omega, \\
\label{uout}
\Delta \bm{u}-\nabla p &=0 \text{ in } \mathbb{R}^3\backslash \Omega,\\
\label{incomp}
\nabla \cdot \bm{u}&=0 \text{ in } \mathbb{R}^3\backslash \Gamma.
\end{align}
Here $\bm{u}$ is the velocity field and $p$ is the pressure. We impose the following condition in the far field:
\begin{equation}\label{uinf}
\bm{u}\to 0 \text{ as } |\bm{x}|\to \infty.
\end{equation}
We  supplement the above with interface conditions on the time-evolving surface $\Gamma$.
For any quantity $w$ defined on 
$\Omega$ and $\mathbb{R}^3\backslash \Omega$, we set:
\begin{equation*}
\jump{w}=\at{w}{\Gamma_{\rm i}}-\at{w}{\Gamma_{\rm e}}
\end{equation*}
where $\at{w}{\Gamma_{\rm i,e}}$ are the trace values of $w$ at $\Gamma$ evaluated 
from the $\Omega$ (interior) and $\mathbb{R}^3\backslash\Omega$ (exterior) sides of $\Gamma$.
Let $\bm{n}$ be the outward pointing unit normal vector on $\Gamma$.
The interface conditions are:
\begin{align}
\label{ujump}
\jump{\bm{u}}&=0,\\
\label{stressjump}
\jump{\Sigma\bm{n}}&=\bm{F}_{\rm el},\; \Sigma=\begin{cases}
\mu \paren{\nabla \bm{u}+(\nabla \bm{u})^{\rm T}}-pI &\text{ in } \Omega\\
\nabla \bm{u}+(\nabla \bm{u})^{\rm T}-pI &\text{ in } \mathbb{R}^3\backslash \Omega
\end{cases},\\\label{convect}
\PD{\bm{X}}{t}&=\bm{u}(\bm{X},t),
\end{align}
where $I$ is the $3\times 3$ identity matrix and $\bm{X}:\mathbb{S}^2\mapsto \Gamma(t)$ the map that describes the evolving membrane. This map gives the deformation of the reference configuration $\mathbb{S}^2$, the standard embedding of the sphere of radius $1$ in $\mathbb{R}^3$.
The first condition is the no-slip boundary condition and the second is the stress balance condition where $\Sigma$
is the fluid stress and $\bm{F}_{\rm el}$
is the elastic force exerted by the interface $\Gamma$. The last condition states that the membrane evolves with the fluid flow. Note that the elastic surface $\Gamma=\Gamma(t)$ and hence $\Omega=\Omega(t)$ changes with time. Once given the constitutive equation for the elastic force $\bm{F}_{\text{el}}$, equations \eqref{uin}-\eqref{convect} form the so-called \textit{jump} formulation of the 3D Peskin problem.
Let $\wh{g}$ and $g$ denote the metric tensors on $\mathbb{S}^2$ and $\Gamma$ respectively. A natural choice for the elastic stretching force is given by \cite{evans1980mechanics,griffith2020immersed}
\begin{equation}\label{Felast}
    \begin{aligned}
    \bm{F}_{\text{el}}=\sqrt{\text{det}(\widehat{g}^{-1}g)}\nabla_{\mathbb{S}^2}\cdot\bm{T}(\nabla_{\mathbb{S}^2}\bm{X}),
    \end{aligned}
\end{equation}
where 
\begin{equation*}
    \bm{T}(\nabla_{\mathbb{S}^2}\bm{X}):=\frac{\mathcal{T}(|\nabla_{\mathbb{S}^2}\bm{X}|)}{|\nabla_{\mathbb{S}^2}\bm{X}|}\nabla_{\mathbb{S}^2}\bm{X}=:T(|\nabla_{\mathbb{S}^2}\bm{X}|)\nabla_{\mathbb{S}^2}\bm{X},
\end{equation*}
$\nabla_{\mathbb{S}^2}$ denotes the surface gradient on $\mathbb{S}^2$, $|A|$ denotes the Frobenius norm of matrix $A$, and $\mathcal{T}$ has to satisfy $\mathcal{T}>0$,  $d\mathcal{T}/d\lambda\geq0$ (see Section \ref{sec:Notation} for further notation). In Section \ref{sec:formulation} more details about the derivation of the elastic force are given. 
For a Hookean material, $\mathcal{T}$ is linear and hence the elastic force is linear in $\bm{X}$. We will consider general $\mathcal{T}$, i.e., the \textit{fully nonlinear} Peskin problem.

Compared to fluid interface problems, such as a drop of liquid surrounded by another fluid or vacuum \cite{Solonnikov77, Pozrikidis92, Denisova94, Shimizu12, PrussSimonett16}, where only the shape of the interface matters, here it is not expected that Eulerian methods on their own should suffice. Due to the elastic nature of the membrane, the stretching, given by the parametrization, has a strong influence on the evolution. Thus, one needs to keep track of the membrane configuration. Lagrangian methods are needed, making it harder to work in higher dimensions. In particular, one cannot freely reparametrize the surface, an idea frequently used to obtain extra cancellations in the study of fluid interfaces \cite{HouLowengrubShelley94, CordobaCordobaGancedo11, GancedoGarciaJuarezPatelStrain21}.

An important feature of the Peskin problem is that it admits a Boundary Integral formulation, whose derivation is given in Section \ref{sec:formulation}. When $\mu=1$, the problem \eqref{uin}-\eqref{Felast} is equivalent to the following evolution equation for $\bm{X}$:
\begin{equation}\label{Peskin_3D_BIF}
\begin{aligned}
\PD{\bm{X}}{t}(\bm{\hx})&=\int_{\mathbb{S}^2}\!\!\! G(\bm{X}(\hx)-\bm{X}(\hy))\nabla_{\mathbb{S}^2}\!\cdot\! \Big(\mc{T}(|\nabla_{\mathbb{S}^2}\bm{X}(\hy)|)\frac{\nabla_{\mathbb{S}^2}\bm{X}(\bm{\hy})}{|\nabla_{\mathbb{S}^2}\bm{X}(\hy)|}\Big)d\mu_{\mathbb{S}^2}(\bm{\hy}),\\
\bm{X}(\hx)|_{t=0}&=\bm{X}_0(\hx),
\end{aligned}
\end{equation}
where $G(\bm{x})$ is the Stokeslet tensor in $\R^3$:
\begin{equation}\label{G_def}
    G(\bm{x})=\frac{1}{8\pi}\Big(\frac{1}{|\bm{x}|}I_3+\frac{\bm{x}\otimes\bm{x}}{|\bm{x}|^3}\Big).
\end{equation}
We have suppressed the dependence of $\bm{X}$ on $t$ to avoid cluttered notation. Henceforth, we will assume $\mu=1$.
It will be sometimes convenient in the analysis to work with coordinates. Let $\bm{\theta}=(\theta_1,\theta_2)$ be a (local)
coordinate system on $\mathbb{S}^2$ and let $\wh{\bm{x}}=\wh{\bm{X}}(\bm{\theta})\in \mathbb{S}^2\subset\mathbb{R}^3$ be the point on $\mathbb{S}^2$ corresponding to $\bm{\theta}$. 
Let $\bm{X}(\bm{\theta})=\bm{X}(\hX(\thetab))\in \Gamma\subset \mathbb{R}^3$ be the position on $\Gamma$
corresponding to the coordinate point $\bm{\theta}$ (see Figure \ref{fig1}). 
If $\wh{\bm{x}}=\wh{\bm{X}}(\bm{\theta})$, we will  write $\bm{X}(\wh{\bm{x}})$ and $\bm{X}(\bm{\theta})$ in an abuse of notation.
Then, after integration by parts and choosing an isothermal coordinate system, equation \eqref{Peskin_3D_BIF} becomes
\begin{equation}\label{Peskin_3D_coor}
\begin{split}
\PD{\bm{X}}{t}(\thetab)&=-\text{p.v.}\int_{\mathbb{R}^2} \PD{}{\eta_i}G(\bm{X}(\thetab)\!-\!\bm{X}(\bm{\eta}))\tilde{\bm{F}}_{\text{el},i}(\bm{X})(\bm{\eta})d\eta_1d\eta_2,
\end{split}
\end{equation}
where we denote 
\begin{equation*}
    \begin{aligned}
    \tilde{\bm{F}}_{\text{el},i}(\bm{X})(\etab)&=\frac{\mc{T}(\lambda(\etab))}{\lambda(\etab)}\PD{}{\eta_i}\bm{X}(\etab),\quad \lambda(\etab)=\sqrt{\text{tr}(\wh{g}^{-1}(\etab)g(\etab))} .
    \end{aligned}
\end{equation*}
Above we use the explicit definitions of $\wh{g}$ and $g$ given in \eqref{metricTensors}.
\begin{figure}[h]
\includegraphics[scale=0.85]{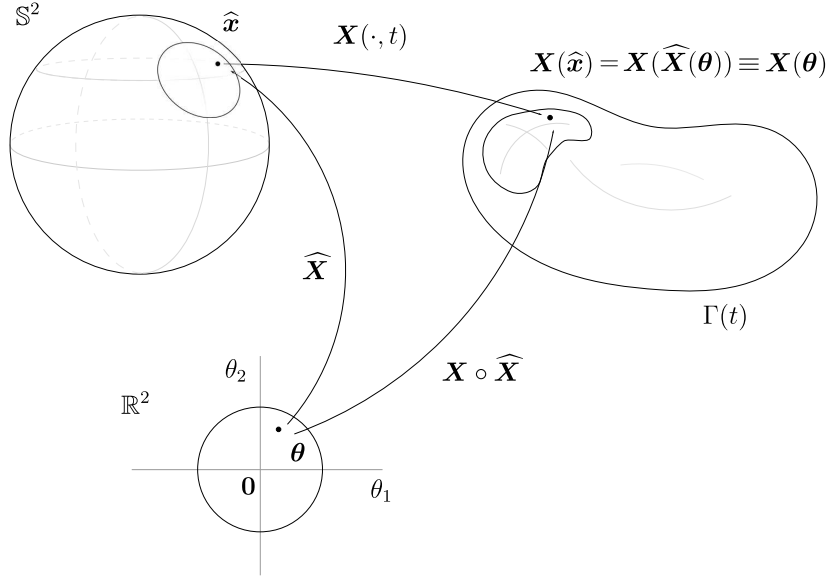}
\caption{Deformation map $\bm{X}(\cdot,t):\mathbb{S}^2\rightarrow \Gamma(t)$.}\label{fig1}
\end{figure}

Some important properties of the solutions to the Peskin problem \eqref{Peskin_3D_BIF} are easier to deduce from the jump formulation \eqref{uin}-\eqref{convect}. The incompressibility condition \eqref{incomp}, together with \eqref{convect}, implies the conservation of the volume of the enclosed region $\Omega$:
\begin{equation*}
    \begin{aligned}
    \frac{d}{dt}|\Omega(t)|&=0,\quad |\Omega(t)|=\frac13\int_{\mathbb{S}^2}\bm{X}(\hx)\cdot\bm{n}(\hx)d\mu_{\Gamma(t)}(\hx).
    \end{aligned}
\end{equation*}
Moreover, the elastic energy defined as follows
\begin{equation*}
    \begin{aligned}
    \mathcal{E}(\bm{X})=\int_{\mathbb{S}^2}A_E(|\nabla_{\mathbb{S}^2}\bm{X}(\hx)|)d\mu_{\mathbb{S}^2}(\hx),\quad A_E'(\lambda)= \mathcal{T}(\lambda),
    \end{aligned}
\end{equation*}
satisfies the balance
\begin{equation*}
    \begin{aligned}
    \frac{d}{dt}\mathcal{E}(\bm{X})&=-\int_{\mathbb{R}^3}|\nabla\bm{u}|^2d\bm{x},
    \end{aligned}
\end{equation*}
which shows that the elastic energy is dissipated due to the viscosity of the fluid.
This relation follows from \eqref{convect}, integration by parts, and using conditions \eqref{stressjump}, \eqref{incomp}, \eqref{uin}-\eqref{uout}, and \eqref{ujump}, consecutively.
For a linear elasticity law, the elastic energy is the $\dot{H}^1(\mathbb{S}^2)$ norm of the interface,
 \begin{equation*}
     \begin{aligned}
     \mathcal{E}(\bm{X})=\frac12\int_{\mathbb{S}^2}|\nabla_{\mathbb{S}^2}\bm{X}(\hx)|^2d\mu_{\mathbb{S}^2}(\hx).
     \end{aligned}
 \end{equation*}
A third important property of the Peskin problem is that it satisfies a scaling invariance.
We must first mention that the definition of solutions requires that the interface is non-degenerate and does not self-intersect. This is typically enforced through the \textit{arc-chord} condition:
\begin{equation*}
    \begin{aligned}
      \starnorm{\bm{X}}
    &:= \inf_{\substack{\widehat{\bm{x}}\neq\widehat{\bm{y}}\\\hx,\hy\in\mathbb{S}^2}} \frac{\abs{\bm{X}\paren{\widehat{\bm{x}}}-\bm{X}\paren{\widehat{\bm{y}}}}}{\abs{\widehat{\bm{x}}-\widehat{\bm{y}}}}>0.
    \end{aligned}
\end{equation*}
If $\bm{X}(\hx,t)$ solves \eqref{Peskin_3D_BIF}, then, for any $\lambda>0$, $\bm{X}_{\lambda}(\hx,t):=\lambda^{-1}\bm{X}(\lambda\hx,\lambda t)$ also solves the equation, and $|\bm{X}_{\lambda}|_*=|\bm{X}|_*$. 
Hence, $\dot{C}^1(\mathbb{S}^2)$ and spaces with the same scaling, such as $\dot{H}^2(\mathbb{S}^2)$, are \textit{critical} spaces for 3D Peskin problem. Notice that the energy balance above only gives control of the $\dot{H}^1(\mathbb{S}^2)$ norm, hence the Peskin problem is \textit{supercritical}.

\subsection{Main results} The formulation of the problem, both in jump and Boundary Integral forms, is derived in Section \ref{sec:formulation}. Once the formulation is provided, the main objective of the paper is to show that the problem is well-posed. More specifically, we will first show the existence and uniqueness of strong solutions with initial data in little H\"older spaces, $h^{1,\gamma}(\mathbb{S}^2)$, $\gamma\in(0,1)$, defined as the completion of the set of smooth functions in $C^{1,\gamma}(\mathbb{S}^2)$. 
\begin{defn}[Strong solution] Let $\bm{X}\in C([0,T];C^{1,\gamma}(\mathbb{S}^2))\cap C^1([0,T];C^\gamma(\mathbb{S}^2))$, $\gamma\in(0,1)$, and $|\bm{X}(t)|_*>0$ for $t\in[0,T]$. Then, $\bm{X}$ is a strong solution to the 3D Peskin problem with initial data $\bm{X}(0)=\bm{X}_0$ if it satisfies equation \eqref{Peskin_3D_BIF} for $t\in(0,T]$ and $\bm{X}(t)\rightarrow\bm{X}_0$ in $C^{1,\gamma}(\mathbb{S}^2)$ as $t\to0$. 
\end{defn}
The choice of little H\"older spaces will be needed to obtain the convergence to the initial data.
In Section \ref{sec: local well-posed} we will prove the following:
\begin{thm}\label{MainTheorem1}
Consider the 3D Peskin problem \eqref{Peskin_3D_BIF} with initial data  satisfying $\bm{X}_0\in h^{1,\gamma}(\mathbb{S}^2)$, $|\bm{X}_0|_*>0$, and $\mc{T}\in C^3$ such that $\mc{T}>0$, $d\mc{T}/d\lambda\geq0$. Then, there exists some time $T>0$ such that \eqref{Peskin_3D_BIF} has a unique strong solution $\bm{X}$,
\begin{equation*}
    \bm{X}\in C([0,T];h^{1,\gamma}(\mathbb{S}^2))\cap C^1([0,T];h^\gamma(\mathbb{S}^2)).
\end{equation*}
\end{thm}
It is instructive to briefly recall the idea of the proof for the 2D linear Peskin problem \cite{MoriRodenbergSpirn19}. For $\bm{X}$ a non-degenerate, closed simple plane curve, the boundary integral formulation in 2D is given by
\begin{equation*}
    \begin{aligned}
    \partial_t\bm{X}(\theta,t)=-\int_{\mathbb{S}^1}\partial_{\eta}G(\bm{X}(\theta)-\bm{X}(\eta))\partial_\eta\bm{X}(\eta)d\eta,
    \end{aligned}
\end{equation*}
where $G$ is the Stokeslet in $\mathbb{R}^2$. It turns out that one can perform a \textit{small-scale decomposition} \cite{HouLowengrubShelley94, MoriRodenbergSpirn19} to write it as follows
\begin{equation*}
    \begin{aligned}
    \partial_t\bm{X}=\frac14\Lambda \bm{X}+\mathcal{R}(\bm{X}),\quad \Lambda\bm{X}=\mathcal{H}\partial_\theta\bm{X},
    \end{aligned}
\end{equation*}
with $\mathcal{R}(\bm{X})$ a lower order operator compared to $\Lambda$.
Then, it is natural to construct the solution as a fixed point of the equation written in Duhamel form:
\begin{equation*}
    \begin{aligned}
    \bm{X}(t)=e^{t\Lambda}\bm{X}_0+\int_0^te^{(t-\tau)\Lambda}\mathcal{R}(\bm{X}(\tau))d\tau
    \end{aligned}
\end{equation*}
We notice two important facts: the semigroup is explicit, both in space and Fourier variables, and the equation is semilinear. Even for nonlinear elastic law, the leading term has a kernel not depending on the curve itself, $-\frac14\mathcal{H}(T(|\nabla_{\mathbb{S}^2}\bm{X}|)\nabla_{\mathbb{S}^2}\bm{X})$, making it possible to use the $\Lambda$-like structure via energy methods \cite{CameronStrain21}.

Here, we first consider the strategy adapted to nonlinear equations in \cite{Lunardi:analytic-semigroups-optimal-parabolic} (see also \cite{Rodenberg:peskin-thesis} for 2D Peskin). Let us write equation \eqref{Peskin_3D_BIF} as follows
\begin{equation*}
\begin{aligned}
\PD{\bm{X}}{t}=F(\bm{X}),\quad \bm{X}|_{t=0}=\bm{X}_0.
\end{aligned}
\end{equation*}
Then, at least formally, linearization around the initial data would give
\begin{equation}\label{Duhamel_abs}
    \begin{aligned}
    \bm{X}(t)=e^{\mathcal{L}(\bm{X}_0)t}\bm{X}_0+\int_0^t e^{\mathcal{L}(\bm{X}_0)(t-\tau)}E(\bm{X}(\tau))d\tau,
    \end{aligned}
\end{equation}
with $\mathcal{L}(\bm{X}_0)=\partial_{\bm{X}}F(\bm{X}_0)\bm{X}$ the Gateaux derivative of $F$ at $\bm{X}_0$ and $E(\bm{X})=F(\bm{X})-\mathcal{L}(\bm{X}_0)$. Hence, while $E(\bm{X})$ is not expected to be smoother than $\mathcal{L}(\bm{X}_0)$, it should be small for short time. However, one first need to make sense of the above expression \eqref{Duhamel_abs} by showing that  $\mathcal{L}(\bm{X}_0)$ generates an analytic semigroup, which amounts to proving that the operator is sectorial in adequate spaces. This is the core of the abstract Theorem \ref{Lunardi_thm}, whose proof encompasses a fixed point argument. The application of this theorem to our problem soon becomes highly involved. This is done in Propositions \ref{Festimate}-\ref{Fsemigroup}.
Since the equation is not semilinear, the process will require to further decompose the operator $\mathcal{L}(\bm{X}_0)$ and then freeze the coefficients at a given point. 
The decomposition must be done maintaining a derivative structure for the kernel that allows extra cancellations, required to control the singular integral operators that appear, and so that we can invert the frozen-coefficient operator (the study of this part is done separately in Section \ref{sec:frozen}). 
Schematically, we decompose the kernel in \eqref{Peskin_3D_coor} as follows
\begin{equation}\label{apprx}
\begin{split}
\PD{}{\eta_i}\big(G(\bm{X}(\thetab)\!-\!\bm{X}(\bm{\eta}))\big)&\approx -\PD{}{x_j}G\paren{\nabla\bm{X}(\etab)(\thetab-\etab)}\PD{X_{j}}{\eta_i}(\etab)+R(\bm{X})(\thetab)\\
&\approx \frac{1}{8\pi}\frac{\PD{\bm{X}(\etab)}{\eta_i}\cdot(\nabla \bm{X}(\etab)(\thetab\!-\!\etab))}{|\nabla \bm{X}(\etab)(\thetab\!-\!\etab)|^3}+\dots+R(\bm{X})(\thetab),
\end{split}
\end{equation}
where one expects $R(\bm{X})$ to be lower order and the dots represent additional terms of high order coming from the second term in $G(\hx)$ \eqref{G_def} (we note that in 2D these additional high-order terms cancel each other). These leading kernels are not of convolution type and cannot be written as a derivative. For this purpose, one could be tempted to use $\nabla\bm{X}(\thetab)$ in the approximation instead of $\nabla{\bm{X}}(\etab)$. However, higher derivatives of $\bm{X}$ would appear later in the proof and the argument would not close. Thus, to take advantage of the derivative structure, we will be forced to estimate together the leading and remainder terms. In a second step, we approximate $\nabla\bm{X}(\etab)$ in the leading kernels above by its value at a given point (see Lemma \ref{lem_dif} for more details), which requires the introduction of a partition of unity for the sphere. 
Due to the geometry of the problem, we need to work with charts, and due to the nonlocal character of the equation a second localization procedure will be needed. A fine implementation of these localization procedures will be crucial to avoid transition maps that would otherwise overcomplicate the proof.

For the fully nonlinear Peskin problem, we must linearize and freeze the coefficient of the elastic force as well. In Section \ref{sec:symbol}, we show that the frozen-coefficient linear operator in the general force case is given by
\begin{equation*}
(\mc{L}_A\bm{Y})_k(\thetab)=-\int_{\mathbb{R}^2} \PD{}{\eta_i}(G_{k,l}(A\paren{\bm{\theta}-\bm{\eta}}))(T_F(A)\nabla\bm{Y})_{l,i}(\etab)d\eta_1d\eta_2,
\end{equation*}
where $A$ is a constant matrix and $T_F(A)$ a tensor \eqref{tensionD}.
Thus, in the general case, the multiplier for the frozen-coefficient linear operator becomes:
\begin{equation*}
\begin{split}
L_A&(\bm{\xi})=\frac{I+\bm{v}(\bm{\xi})\otimes \bm{v}(\bm{\xi})}{4{\rm det}(B)\abs{B^{-1}\bm{\xi}}}\paren{
\frac{\mc{T}(\norm{A}_F)}{\norm{A}_F}\paren{\abs{\bm{\xi}}^2 I\!-\!\frac{A\bm{\xi}\otimes A\bm{\xi}}{\norm{A}_F}}\!+\!\D{\mc{T}}{\lambda}(\norm{A}_F)\frac{A\bm{\xi}\otimes A\bm{\xi}}{\norm{A}_F}},
\end{split}
\end{equation*}
where $||\cdot||_{F}$ denotes the Frobenius norm.
It is not difficult to see that, if $\mc{T}>0$ and $d\mc{T}/d\lambda\geq 0$, then the above is coercive in $\abs{\bm{\xi}}^2$. 
Moreover, in contrast to the 2D case, $d\mc{T}/d\lambda=0$ is allowed. In fact, if $\mc{T}$ satisfies $\mc{T}>0$ and  $(d\mc{T}/d\lambda)/\mc{T}>-1$, then the problem is expected to be locally well-posed if the initial condition is sufficiently close to the uniform sphere ($\wh{g}^{-1}g$ is close to a multiple of the identity matrix). This is an interesting difference between 2D and 3D Peskin. We will use this operator (in conjunction with the localization procedures) to show that the full operator $\mathcal{L}(\bm{X}_0)=\partial_{\bm{X}}F(\bm{X}_0)\bm{X}$ is sectorial.
The approximation in \eqref{apprx} is done on the equation written in coordinates partly to obtain a linear leading operator given by a Fourier multiplier.

Next, we notice that the regularity obtained in Theorem \ref{MainTheorem1} for the strong solutions is not enough to satisfy equation \eqref{Peskin_3D_BIF} in a classical sense. 
Obtaining higher regularity for the solutions is also important since this further regularity is needed for the equivalence between different formulations to hold.
The abstract theory for nonlinear equations in \cite{Lunardi:analytic-semigroups-optimal-parabolic} does not yield gain of smoothness for the solution, and in fact this important point is left open in the 2D results in \cite{Rodenberg:peskin-thesis}. Nevertheless, we are able to prove that initial data in little H\"older spaces become smooth for positive times.
\begin{thm}\label{MainTheorem2}
Let $\bm{X}$ be the solution to the Peskin problem with initial data $\bm{X}_0\in h^{1,\gamma}(\mathbb{S}^2)$ constructed in Theorem \ref{MainTheorem1}. Then,  for any $\alpha\in(0,1)$, it holds that $\bm{X}\in C^1((0,T]; C^{3,\alpha}(\mathbb{S}^2))$. Moreover, for any $3\leq n\in\mathbb{N}$ and $\alpha\in(0,1)$, assuming that $\mc{T}\in C^{n,\alpha}$, it holds that $\bm{X}\in C^1((0,T]; C^{n+1,\beta}(\mathbb{S}^2))$, for any $\beta<\alpha$.
\end{thm}
We use the solutions constructed in the previous theorem and Duhamel formula \eqref{Duhamel_abs} to perform a bootstrapping argument. We build on the properties of the semigroup $e^{-t\mc{L}_A}$ (see Section \ref{sec:frozen} and Appendix \ref{sec:appB}) to first gain regularity in mixed-type spaces $L^p(0,T; C^{n,\alpha}(\mathbb{S}^2))$ and then transfer this higher regularity in space to show regularity in time as well. A key point is that, while the kernels are not of convolution type, we find that it is possible to move derivatives in $\thetab$ to derivatives in $\etab$ at the expense of new terms of the same order, but not higher (see \eqref{move_der}). As explained above, we must work with the equation localized around a given point and later deal with the corresponding commutators and combine the estimates (see Section \ref{sec:highreg} for more details). However, the bootstrapping argument cannot be done on \eqref{Duhamel_abs} directly, because the right-hand side contains terms of highest regularity. We combine this process with a regularization argument (see \eqref{convergence_split}), where the use of little H\"older spaces becomes crucial.

\subsection{Related results}

The first analytical results for the 2D Peskin problem appeared recently in \cite{LinTong19, MoriRodenbergSpirn19}. In \cite{LinTong19}, energy arguments are used to prove local well-posedness for $H^{\frac{5}{2}}$ initial data and also  exponential convergence to steady states for sufficiently close to equilibrium initial data is shown. The authors in \cite{MoriRodenbergSpirn19} lowered the required initial regularity to barely subcritical spaces, $h^{1,\gamma}$, $\gamma\in(0,1)$, showed instant smoothing, and provided a blow up criterion.

After these works, many improvements for the 2D Peskin problem have appeared. The work \cite{GJMoriStrain20} deals with the setting in which the enclosed fluid is different to the exterior one, and shows asymptotic stability for small data in Wiener algebra critical spaces. The result \cite{CameronStrain21} shows the local well-posedness and smoothing for general data in the critical Besov space $B_{2,1}^{\frac32}$, including  the case of nonlinear elastic law. The sharpest result in terms of regularity appeared in
\cite{KeNguyen21}, where the semilinear 2D Peskin problem is shown to be well-posed in $B_{\infty,\infty}^1$, and thus with possibly non-Lipschitz curves.

In relation to the Peskin problem, the article \cite{Tong21}  introduces a regularization of the problem inspired by the immersed boundary method and studies its convergence.  Filaments that resist both bending and stretching are considered in \cite{Li21}.
Finally, we mention two works that introduce simplified models of the 2D Peskin problem. The work \cite{GancedoGraneroScrobogna21} considers a model for the normal component and shows the existence of global solutions for Lipschitz data near the equilibrium. Very recently, \cite{Tong22} derives a PDE to model the tangential effects of the Peskin problem in the case of an infinitely long and straight string and obtains global solutions with initial data in the energy class. Moreover, the author presents many connections of the model with well-known one-dimensional PDEs.

From a mathematical point of view, there are remarkable similarities between the 2D Peskin problem and the so-called Muskat problem.
In particular, both problems have the same leading linear operator, they can be written in Boundary Integral form \cite{HouLowengrubShelley94, CordobaCordobaGancedo11}, they have the same scaling and satisfy an energy balance \cite{ConstantinCordobaGancedoStrain2013, HaziotPausader2022}.
The Muskat problem, which describes the movement of the interface between incompressible fluids in a porous medium, has been intensively studied in the last two decades \cite{Ambrose04, CordobaCordobaGancedo11, EscherMatioc11, ConstantinCordobaGancedoStrain2013, Cameron19, Matioc19, AlazardLazar20, NguyenPausader20, AlazardNguyen22}, and some of the techniques developed there have been successfully extended in the last  years to lower the required regularity for the well-posedness of the 2D Peskin problem \cite{GJMoriStrain20, CameronStrain21, GancedoGraneroScrobogna21, KeNguyen21}. However, while there are also results for the 3D Muskat problem \cite{Ambrose07, CordobaCordobaGancedo13, AlazardNguyen22_2, ChenNguyenXu21, GancedoLazar22}, in all these results the interface is a surface given by graph, hence the geometry does not play a major role. Even in 2D, in the recent non-graph setting \cite{GancedoGarciaJuarezPatelStrain21} that considers a bubble of fluid surrounded by another in a porous medium, a change of parametrization becomes crucial, which is not allowed in the Peskin problem. 

We finally mention some results with more complex elastic interactions
\cite{CoutandShkoller06, ChengShkoller10, WangZhangZhang12, CanicGaliMuha2020, LengelerRuzicka14, MuhaCanic13, PlotnikovToland11, AmbroseSiegel17}, mostly dedicated to more qualitative results and weak solutions. Part of the interest generated by the Peskin problem is due to its relative simplicity, which makes it possible to initiate the analytical study of the rich variety of behaviors in FSI problems, including longtime dynamics.

\subsection{Outline} The rest of the paper is structured as follows. In Section \ref{sec:formulation}, we obtain the expression for the elastic law and show the Boundary Integral formulation for the 3D Peskin problem. Section \ref{sec:Notation} contains the notation used along the paper as well as some definitions and standard results concerning the stereographic projection. Next, in Section \ref{sec:leading}, we introduce the operators that will be used later in the paper, we decompose the equation and compute the multiplier of the leading term. 
Section \ref{calculus_est} is dedicated to study the operators previously defined, to show the needed commutators estimates, and to prove Lemma \ref{lem_dif}. These lemmas will be repeatedly used in the proof of the main theorems. In Section \ref{sec:frozen}, we show that the frozen-coefficient operator generates an analytic semigroup (for which we need the multiplier results contained in Appendix \ref{sec:appA}, with further properties studied in Appendix \ref{sec:appB}). Finally, Sections \ref{sec: local well-posed} and Section \ref{sec:highreg} contain the proofs of the main results: Theorems \ref{MainTheorem1} and \ref{MainTheorem2}.

\section{Formulation and Boundary Integral Reduction}\label{sec:formulation}

The formulation of the problem \eqref{uin}-\eqref{stressjump} is closed once that the expression for $\bm{F}_{\rm el}$ is given.
To specify the elastic force $\bm{F}_{\rm el}$ in \eqref{stressjump}, we consider the elastic energy 
of the interface $\Gamma$.
We consider an elastic energy $\mc{E}(\bm{X})$ of the form:
\begin{equation*}
\mc{E}(\bm{X})=\int_{\mathbb{S}^2} E\paren{\PD{\bm{X}}{\bm{\theta}},\bm{\theta}} d\mu_{\mathbb{S}^2},
\end{equation*}
where $\mu_{\mathbb{S}^2}$ is the standard measure on the unit sphere. 
From this, we may compute the elastic force by taking the variational derivative as follows.
Let $\bm{X}=(X_1,X_2,X_3)^{\rm T}$.
Define the following metric tensors $\wh{g}$ and $g$ on $\mathbb{S}^2$ and $\Gamma$ respectively, whose $i,j$ components are given by:
\begin{equation}\label{metricTensors}
\wh{g}_{ij}=\PD{\wh{\bm{X}}}{\theta_i}\cdot\PD{\wh{\bm{X}}}{\theta_j},\quad g_{ij}=\PD{\bm{X}}{\theta_i}\cdot\PD{\bm{X}}{\theta_j}.
\end{equation}
We write the energy density as follows:
\begin{equation}\label{AE}
E=A_{E}(s_{ij}, \bm{\theta}), \; s_{ij}=\PD{X_i}{\theta_j},\; i=1,\cdots 3, \; j=1,2.
\end{equation}
Let $\bm{Y}=(Y_1,Y_2,Y_3)^{\rm T}$ be a perturbation of the configuration that is compactly supported on the open set $\mc{U}$ on 
which the coordinate system $\bm{\theta}$ is defined. 
We have: 
\begin{equation}\label{Evar}
\begin{split}
\at{\D{}{\tau}\mc{E}(\bm{X}+\tau\bm{Y})}{\tau=0}&=\int_{\mc{U}} \PD{A_E}{s_{ij}}\PD{Y_i}{\theta_j}\sqrt{{\rm det}\wh{g}}d\theta_1d\theta_2\\
&=-\int_{\mc{U}} \PD{}{\theta_j}\paren{\PD{A_E}{s_{ij}}\sqrt{{\rm det}\wh{g}}}Y_i d\theta_1d\theta_2,
\end{split}
\end{equation}
where the summation convention is in effect. We set:
\begin{equation*}
F_{{\rm el},i}=\frac{1}{\sqrt{{\rm det}g}} \PD{}{\theta_j}\paren{\PD{A_E}{s_{ij}}\sqrt{{\rm det}\wh{g}}},
\end{equation*}
where $F_{{\rm el},i}$ are the components of the elastic force $\bm{F}_{\rm el}$ of equation \eqref{stressjump}. 
With this prescription of the elastic force, the solutions satisfy the following energy relation:
\begin{equation}\label{dEdt}
\D{\mc{E}}{t}=-\paren{\int_\Omega 2\mu\abs{\nabla_{\rm S}\bm{u}}^2d\bm{x}+\int_{\mathbb{R}^3\backslash\Omega}2\abs{\nabla_{\rm S}\bm{u}}^2}d\bm{x},
\; \nabla_{\rm S}\bm{u}=\frac{1}{2}\paren{\nabla \bm{u}+(\nabla \bm{u})^{\rm T}}.
\end{equation}
We will now impose symmetry conditions to determine the explicit form of $A_E$ and hence $E$. 
Let $\bm{\theta}$ be a (local) orthogonal coordinate system on $\mathbb{S}^2$
so that the two coordinate tangent vectors are orthogonal:
\begin{equation*}
\PD{\wh{\bm{X}}}{\theta_1}\cdot \PD{\wh{\bm{X}}}{\theta_2}=0.
\end{equation*}
We thus have an orthonormal frame on (a neighborhood of) $\mathbb{S}^2$ given by the two vectors:
\begin{equation*}
\wh{\bm{e}}_i=\PD{\wh{\bm{X}}}{\theta_i}\abs{\PD{\wh{\bm{X}}}{\theta_i}}^{-1},\; i=1,2.
\end{equation*}
The deformation map $\bm{X}$ maps the above unit orthogonal vectors to the following two vectors:
\begin{equation*}
\bm{e}_i=\PD{\bm{X}}{\theta_i}\abs{\PD{\wh{\bm{X}}}{\theta_i}}^{-1},\; i=1,2.
\end{equation*}
Consider the matrix $3\times 2$ matrix $B=(\bm{e}_1,\bm{e}_2)$ whose column vectors are given by $\bm{e}_i$.
We may say that the energy density $E$ is a function of $B$ and $\bm{\theta}$:
\begin{equation*}
E=A_E(B,\bm{\theta}),
\end{equation*}
where we have continued to use the notation $A_E$ as in \eqref{AE}.
By homogeneity of the unit sphere, we impose that $A_E$ does not have an explicit dependence on $\bm{\theta}$.
Furthermore, the value of $A_E$ should not depend on the choice of orthonormal frame $\wh{\bm{e}}_i$ or the coordinate 
system in which $\bm{X}$ resides. This implies the following. 
\begin{equation}\label{invarianceSO23}
A_E(B)=A_E(R_3BR_2) \text{ for all } R_3\in {\rm SO}(3) \text{ and } R_2\in {\rm SO}(2) 
\end{equation}
where ${\rm SO}(2)$ and ${\rm SO}(3)$ are the group on rotation matrices in $2$ and $3$ dimensions respectively.
Let:
\begin{equation*}
H=\begin{pmatrix} 
\bm{e}_1\cdot\bm{e}_1 & \bm{e}_1\cdot \bm{e}_2 \\ 
\bm{e}_1\cdot\bm{e}_2 & \bm{e}_2\cdot \bm{e}_2
\end{pmatrix}.
\end{equation*}
The invariance condition \eqref{invarianceSO23} implies that $A_E$ can only be a function of the trace and determinants of $H$,
\begin{equation}\label{EAE}
E=A_E(\lambda,\gamma), \; \lambda=\sqrt{{\rm tr}(H)}, \; \gamma=\sqrt{{\rm det}(H)}.
\end{equation}
In terms of the metric tensors $g$ and $\wh{g}$, we can write $\lambda$ and $\gamma$ as:
\begin{equation}\label{lamg}
\lambda^2={\rm tr}(\wh{g}^{-1}g), \; \gamma^2={\rm det}(\wh{g}^{-1}g).
\end{equation}
The above expressions for $\lambda$ and $\gamma$ are valid even when $\bm{\theta}$ is not an orthogonal coordinate system.
We may substitute \eqref{EAE} into \eqref{Evar} to obtain:
\begin{equation}
\begin{split}\label{Felkgeneral}
{F}_{{\rm el},k}&=F_{\lambda,k}+F_{\gamma,k},\\
F_{\lambda,k}&=\frac{1}{\sqrt{{\rm det}{g}}}\PD{}{\theta_i}\paren{\frac{1}{\lambda}\PD{A_E}{\lambda}\sqrt{{\rm det}\wh{g}}\minspace\wh{g}^{ij}\PD{X_k}{\theta_j}},\\
F_{\gamma,k}&=\frac{1}{\sqrt{{\rm det}{g}}}\PD{}{\theta_i}\paren{\PD{A_E}{\gamma}\sqrt{{\rm det}g}\minspace g^{ij}\PD{X_k}{\theta_j}},
\end{split}
\end{equation}
where we use the standard notation $a^{ij}$ to denote the inverse tensor $(a^{-1})_{ij}$.
Note that the expressions $F_{\lambda,k}$ and $F_{\gamma,k}$ are similar but differ crucially in whether $\wh{g}$ or $g$ features inside the force expressions.
This is most clearly seen in the following simple cases. If we let $A_E=\lambda^2/2$, we have:
\begin{equation}\label{Fel}
F_{{\rm el},k}=F_{\lambda,k}=\gamma \Delta_{\mathbb{S}^2} X_k, \;  \Delta_{\mathbb{S}^2} X_k=\frac{1}{\sqrt{{\rm det}\wh{g}}}\PD{}{\theta_i}\paren{\sqrt{{\rm det}\wh{g}}\minspace \wh{g}^{ij}\PD{X_k}{\theta_j}},
\end{equation}
where $\Delta_{\mathbb{S}^2}$ is the Laplace-Beltrami operator on the unit sphere. If we let $A_E=\gamma$, we have:
\begin{equation*}
F_{{\rm el},k}=F_{\gamma,k}=\Delta_{\Gamma}X_k=\frac{1}{\sqrt{{\rm det}{g}}}\PD{}{\theta_i}\paren{\sqrt{{\rm det}g}\minspace g^{ij}\PD{X_k}{\theta_j}}=-2\kappa_\Gamma n_k,
\end{equation*}
where $\Delta_\Gamma$ is the Laplace-Beltrami operator of the closed elastic surface $\Gamma$, $\kappa_\Gamma$ is the mean curvature of $\Gamma$ 
and $n_k$ is the $k$-th component of the outward normal vector $\bm{n}$ of $\Gamma$. 
This is just the well-known statement on the variation of surface area. 
We see from the above expressions that the $F_{\lambda,k}$ expresses an elastic force that depends strongly on the stretching of the spherical reference configuration whereas $F_{\gamma,k}$ is a surface tension force.

The prescription of interfacial elastic energy density as in \eqref{AE} or \eqref{EAE} has its origins the classical work of \cite{evans1980mechanics}, and may be called the membrane neo-Hookean model. Specific forms for this energy have been used extensively in the modeling and simulation of fluid-structure interaction problems \cite{griffith2020immersed, fai2013immersed, ko2016parametric}. 

We now rewrite our evolution equation in a form suitable for our analysis. 
Henceforth we focus on the case when $A_E$ is only a function of $\lambda$, and the viscosity $\mu$ of the interior fluid is equal to $1$. 

Let us rewrite the equations of motion. 
Let $G$ be the Stokeslet tensor in $\mathbb{R}^3$:
\begin{equation}\label{Stokeslet}
G_{i,j}(\bm{x})=\frac{1}{8\pi}\paren{\frac{\delta_{i,j}}{\abs{\bm{x}}}+\frac{x_ix_j}{\abs{\bm{x}}^3}}, \; \bm{x}=(x_1,x_2,x_3).
\end{equation}
Let
\begin{equation*}
\begin{split}
\wh{F}_{\rm{el},k}&=\frac{1}{\gamma}F_{\rm{el},k}=
\frac{1}{\sqrt{{\rm det}{\wh{g}}}}\PD{}{\theta_i}\paren{\lambda^{-1}\mc{T}(\lambda)\sqrt{{\rm det}\wh{g}}\minspace \wh{g}^{ij}\PD{X_k}{\theta_j}}\\
&=\nabla_{\mathbb{S}^2}\cdot\paren{\lambda^{-1}\mc{T}(\lambda)\nabla_{\mathbb{S}^2}X_k},
\end{split}
\end{equation*}
where $\mc{T}(\lambda)=\partial A_E/\partial \lambda$ (see \eqref{Felkgeneral}).
Let $\wh{\bm{F}}_{\rm{el}}=(F_{\rm{el},1},F_{\rm{el},2},F_{\rm{el},3})^{\rm T}$. 
Notice that we can write $\lambda$ in terms of $\bm{X}$,
\begin{equation}\label{Tlaw}
    \lambda(\hx)^2=|\nabla_{\mathbb{S}^2}\bm{X}(\hx)|^2,
\end{equation}
which can be seen from their definitions
\begin{equation*}
    \begin{aligned}
    \lambda^2&=\text{tr}(\wh{g}^{-1}g)=\wh{g}^{ij}g_{ji},\\
    |\nabla_{\mathbb{S}^2}\bm{X}|^2&=\nabla_{\mathbb{S}^2}X_k\cdot\nabla_{\mathbb{S}^2}X_k=\PD{X_k}{x_i}\PD{X_k}{x_l}\wh{g}^{ij}\wh{g}^{lm}\wh{g}_{jm}=g_{il}\wh{g}^{il}.
    \end{aligned}
\end{equation*}
When $\mu=1$, we may write the evolution of $\bm{X}$ as
\begin{equation*}
\begin{aligned}
\PD{\bm{X}}{t}(\bm{\hx})&=\int_{\mathbb{S}^2} G(\bm{X}(\hx)-\bm{X}(\hy))\wh{\bm{F}}_{{\rm el}}(\bm{\hy})d\mu_{\mathbb{S}^2}(\bm{\hy})\\
&=\int_{\mathbb{S}^2}\!\!\! G(\bm{X}(\hx)-\bm{X}(\hy))\nabla_{\mathbb{S}^2}\!\cdot\! \Big(\mc{T}(|\nabla_{\mathbb{S}^2}\bm{X}(\hy)|)\frac{\nabla_{\mathbb{S}^2}\bm{X}(\bm{\hy})}{|\nabla_{\mathbb{S}^2}\bm{X}(\hy)|}\Big)d\mu_{\mathbb{S}^2}(\bm{\hy}),
\end{aligned}
\end{equation*}
and integrating by parts, we obtain
\begin{equation}\label{Xieqn}
\begin{aligned}
\PD{\bm{X}}{t}(\bm{\hx})&\!=\!-\pv\!\int_{\mathbb{S}^2}\!\!\! \nabla_{\mathbb{S}^2}G(\bm{X}(\hx)\!-\!\bm{X}(\hy))\!\cdot\! \mc{T}(|\nabla_{\mathbb{S}^2}\bm{X}(\hy)|)\frac{\nabla_{\mathbb{S}^2}\bm{X}(\bm{\hy})}{|\nabla_{\mathbb{S}^2}\bm{X}(\hy)|}d\mu_{\mathbb{S}^2}(\bm{\hy}).
\end{aligned}
\end{equation}
In the following, we will suppress the principal value notation.
Introducing a smooth partition of unity $\{\rho_n\}$, subordinate to a finite atlas of the sphere $\{\mc{U}_n\}$, we may write our problem as follows
\begin{equation*}
\begin{aligned}
\PD{\bm{X}}{t}(\bm{\hx})&\!=\!-\!\sum_n\!\int_{\mathbb{S}^2}\!\!\! \nabla_{\mathbb{S}^2}G(\bm{X}(\hx)\!-\!\bm{X}(\hy))\!\cdot\!\frac{\mc{T}(|\nabla_{\mathbb{S}^2}\bm{X}(\hy)|)}{|\nabla_{\mathbb{S}^2}\bm{X}(\hy)|}\nabla_{\mathbb{S}^2}\big(\rho_n(\hy)\bm{X}(\bm{\hy})\big)d\mu_{\mathbb{S}^2}(\bm{\hy}).
\end{aligned}
\end{equation*}
This can be rewritten using the local charts:
\begin{equation*}
\begin{split}
\PD{\bm{X}}{t}(\hx)&=-\!\sum_n\!\int_{\mc{U}_n}\! \!\wh{g}^{ij}(\etab)\PD{}{\eta_i}G(\bm{X}(\hx)\!-\!\bm{X}_n(\bm{\eta}))\\
&\hspace{0.cm}\times\wh{g}_{jm}(\etab)\frac{\mc{T}(\sqrt{\wh{g}^{qr}g_{rq}(\etab)})}{\sqrt{\wh{g}^{qr}g_{rq}(\etab)}}\wh{g}^{pm}(\etab)\PD{}{\eta_p}\big(\rho_n(\bm{\eta})\bm{X}_{n}(\bm{\eta})\big)\sqrt{{\rm det}\wh{g}(\bm{\eta})}d\eta_1d\eta_2,
\end{split}
\end{equation*}
where $\bm{X}_n(\eta)$ is the coordinate map on the $n$-th coordinate chart and $\rho_n(\etab)=\rho_n(\hX(\etab))$ (see Section \ref{sec:Notation} for details of notation).
We may take an isothermal coordinate system (the stereographic projection gives such a system, for example) on each chart $\mc{U}_n$, which yields:
\begin{equation}\label{3DPeskin_charts}
\begin{split}
\PD{\bm{X}}{t}(\hx)&=-\sum_n\int_{\mc{U}_n} \PD{}{\eta_i}G(\bm{X}(\hx)\!-\!\bm{X}_n(\bm{\eta}))\frac{\mc{T}(\lambda_n(\etab))}{\lambda_n(\etab)}\PD{}{\eta_i}\big(\rho_n(\bm{\eta})\bm{X}_n(\bm{\eta})\big)d\eta_1d\eta_2,
\end{split}
\end{equation}
where we denote 
\begin{equation}\label{lambda_n}
    \begin{aligned}
    \lambda_n(\etab)&=\sqrt{\text{tr}(\wh{g}^{-1}(\etab)g(\etab))}=\sqrt{2}\|\nabla\bm{X}_n(\etab)\|_F\|\nabla \hX_n(\etab)\|_F^{-1},
    \end{aligned}
\end{equation}
and $\|A\|_F:=\sqrt{\text{tr}(A^TA)}$ is the Frobenius norm.

\section{Preliminaries}

In this section we introduce the notations that will be used in the rest of the paper and summarize some standard results about stereographic projection charts for the sphere.

\subsection{Notations}\label{sec:Notation}
Einstein notation over repeated indices will be of constant use.
Given vectors $\bm{v},\bm{w}$ and matrices $A,B,C$ with the same size, we denote
\begin{align*}
    \abs{\bm{v}}:=&\norm{\bm{v}}=\sqrt{v_iv_i},\\
    \norm{A}:=&\sup_{\abs{\bm{v}}> 0}\frac{\abs{A\bm{v}}}{\abs{\bm{v}}}=\sup_{\abs{\bm{v}}= 1}\abs{A\bm{v}},\\
    A : B :=&\rm{tr}\paren{A^T B}=A_{ij}B_{ij},\\
    |A|:=&\norm{A}_F:=\sqrt{A:A},\\
    \bm{v}\otimes\bm{w}:=&\bm{v}\bm{w}^T,\\
    \paren{A\otimes B}_{ijkl}:=&A_{ij}B_{kl},\\
    \paren{\paren{A\otimes B}C}_{ij}:=& A_{ij}B_{kl}C_{kl}=\paren{B:C}A_{ij}.
\end{align*}
We will denote $\colofmat{C}{j}$ to the vector given by the $jth$ column of $C$, and $\rowofmat{C}{j}$ to the one given by the $jth$ row.

We will denote $\mu_{\mathbb{S}^2}$ the standard measure on the unit sphere, and for simplicity we will write $d\widehat{y}$ instead of $d\mu_{\mathbb{S}^2}(\widehat{y})$.

We will write high partial derivatives in $\mathbb{R}^k$ by multi-index $\bm{\alpha}$, where multi-index $\bm{\alpha}$ is a sequence of $k$ nonegative integers.
i.e. $\bm{\alpha}=\paren{\alpha_1, \alpha_2, \cdots, \alpha_k}\in \mathbb{N}_0^k$, where $\mathbb{N}_0=\{0\}\bigcap \mathbb{N}$.
\begin{defn}[Multi-index]
Given $\bm{\alpha}, \bm{\beta}\in \mathbb{N}_0^k$, we have the following arithmetic about the multi-index.

(i)
\begin{align*}
    \abs{\bm{\alpha}}=&\alpha_1+\alpha_2+\dots+\alpha_k\\
    \bm{\alpha}!=&\alpha_1!\alpha_2!\dots\alpha_k!\\
    \bm{\alpha}+\bm{\beta}=&\paren{\alpha_1+\beta_1, \alpha_2+\beta_2, \cdots, \alpha_k+\beta_k}
\end{align*}

(ii) 
We set $\bm{\alpha}\leq \bm{\beta}$, which is $\alpha_i\leq\beta_i$ for all $i=1,2,\cdots,k$.
Then, we have
\begin{align*}
    \bm{\alpha}-\bm{\beta}=&\paren{\alpha_1-\beta_1, \alpha_2-\beta_2, \cdots, \alpha_k-\beta_k}\\
    \begin{pmatrix}
    \bm{\alpha}\\
    \bm{\beta}
    \end{pmatrix}
    =&\frac{\bm{\alpha}!}{\paren{\bm{\alpha}-\bm{\beta}}!\bm{\beta}!}=\frac{\alpha_1 !}{\paren{\alpha_1-\beta_1} ! \beta_1 !}\cdots\frac{\alpha_k !}{\paren{\alpha_k-\beta_k} ! \beta_k !}
\end{align*}

\end{defn}
High partial derivatives can be written as

\begin{align}
    \partial_{\xb}^{\bm{\alpha}} f\paren{\xb}:= \PD{^{\alpha_1}}{x_1^{\alpha_1}}\PD{^{\alpha_2}}{x_2^{\alpha_2}}\cdots\PD{^{\alpha_k}}{x_k^{\alpha_k}}f\paren{\xb}\label{high_pd}
\end{align}
where $\bm{\alpha}:=(\alpha_1, \alpha_2, \cdots, \alpha_k)$ and $\abs{\bm{\alpha}}=\alpha_1+\dots+\alpha_k$ is the total number of derivatives.

We will use the following set of non-singular matrices 
\begin{equation}\label{s1s2}
\mc{DA}_{\sigma_1,\sigma_2}:=\left\{ A: \forall\xib\neq\bm{0}, \sigma_2\abs{\bm{\xi}}\leq \abs{A\bm{\xi}}\leq \sigma_1\abs{\bm{\xi}}\right\}
\end{equation}
where $\sigma_1\geq \sigma_2>0$.

Euclidean balls of radius $R$ centered at $\hx\in \mathbb{R}^n$ will be denoted by $B_{\hx,R}$, and for balls centered at the origin we will also denote $\mc{B}(R):=B_{\bm{0},R}$.

We will denote $\bm{X}(\hx;t):\mathbb{S}^2\mapsto \Gamma(t)$ the deformation map that describes the evolving membrane, and we will omit the dependence on time for simplicity of notation, $\bm{X}(\hx)$. We will consider a finite atlas $\{\mc{U}_n, \wh{\bm{X}}_n\}$ of the sphere with $\bm{0}\in\mc{U}_n$, such that the coordinate functions $\wh{\bm{X}}_n(\bm{\theta}):\mc{U}_n\subset\mathbb{R}^2\bigcup \{\infty\}\mapsto \mathbb{S}^2$ satisfy
\begin{equation}\label{isothermal}
    \PD{\wh{\bm{X}}_n(\bm{\theta})}{\theta_1}\cdot \PD{\wh{\bm{X}}_n(\bm{\theta})}{\theta_2}=0,\qquad \norm{\PD{\wh{\bm{X}}_n(\bm{\theta})}{\theta_1}}=\norm{\PD{\wh{\bm{X}}_n(\bm{\theta})}{\theta_2}}.
\end{equation}
In particular, we will choose the standard stereographic coordinates. We set $\{\rho_n\}$ to be a smooth partition of unity subordinate to the coordinate patches $\{\mcu_n\}$.
For convenience in the definition of H\"older continuity, we take our system $\{\mathcal{U}_n, \widehat{\bm{X}}_n, \rho_n\}$ satisfying the following properties with some $R, \delta>0$.
\begin{defn}[System $\{\mathcal{U}_n, \widehat{\bm{X}}_n, \rho_n\}$]\label{def_charts}
Given $R>2\delta>0$, we set our isothermal coordinate charts $\{\mathcal{U}_n\}$ with the coordinate functions $\{\widehat{\bm{X}}_n({\thetab})\}$ and the partition $\{\rho_n\}$ to have the following properties:
\begin{enumerate}
    \item[i)] Set $\hx_n=\hX_n\paren{0}$, then $\mbs \subset \bigcup_n B_{\widehat{\bm{x}}_n,R}$, and there exists $0<R_n<\infty$ s.t
    \begin{align*}
        B_{\widehat{\bm{x}}_n,4R}\cap\mbs\subset\hX_n\paren{ B_{\bm{0},R_n}}\subset\hX_n(\mathcal{U}_n)
    \end{align*}
    \item[ii)] $\forall \widehat{\bm{x}}\in \mathbb{S}^2$, $\exists n$ s.t. 
\begin{align*}
    \widehat{\bm{X}}_n\paren{\thetab}= \widehat{\bm{x}}\quad \mbox{ for some } \thetab\in B_{\bm{0},R_n}, \mbox{ and } ( B_{\widehat{\bm{x}},2\delta}\cap\mbs)\subset \widehat{\bm{X}}_n\paren{B_{\bm{0},R_n}}.
\end{align*}
\item[iii)]$
    0\leq\rho_n\leq 1, \quad \rho_n\in C^\infty(\mathbb{S}^2), \quad \overline{supp\paren{\rho_n}} \subset B_{\widehat{\bm{x}}_n,2R}\cap\mbs.
$
\item[iv)]$\forall \widehat{\bm{x}}\in \mbs$,
$    \sum_{n}\rho_n\paren{\hx}=1.$
\end{enumerate}
\end{defn}
\begin{remark}
    If $|\widehat{\bm{X}}_n\paren{\thetab}-\widehat{\bm{X}}_n\paren{\etab}|\geq C \abs{\thetab-\etab}$ on $\mcu_n$, then $\mcu_n$ is totally bounded.
\end{remark}
Given $f\paren{\widehat{\bm{x}}}: \mbs \rightarrow \mathbb{R}$, we will denote $f_n\paren{\thetab}: \mcu_n\subset{\mathbb{R}^2}\rightarrow \mathbb{R}$ with $f_n({\thetab}):=f({\widehat{\bm{X}}_n({\thetab})})$. Analogously, we will denote $\bm{X}_n(\thetab)=\bm{X}(\hX_n(\thetab))$.
\begin{remark}
    If $\widehat{\bm{X}}_n\paren{\thetab}=\widehat{\bm{X}}_m\paren{\etab}$, then $\bm{X}_n\paren{\thetab}=\bm{X}_m\paren{\etab}$.
\end{remark}
\begin{defn}[H\"older semi-norm]
\begin{align*}
\begin{split}
    \jump{f\paren{\widehat{\bm{x}}}}_{C_\delta^\gamma\paren{\mbs}}
    :=&\!\!\!\!\sup_{0<\abs{\widehat{\bm{x}}-\widehat{\bm{y}}}<\delta}\!\!\!\!\frac{\abs{f\paren{\widehat{\bm{x}}}\!-\!f\paren{\widehat{\bm{y}}}}}{\abs{\widehat{\bm{x}}-\widehat{\bm{y}}}^\gamma}
    =\sup_n\sup_{0<\abs{\widehat{\bm{X}}_n\paren{\thetab}-\widehat{\bm{X}}_n\paren{\etab}}<\delta}\frac{\abs{f_n\paren{\thetab}-f_n\paren{\etab}}}{|\widehat{\bm{X}}_n\paren{\thetab}\!-\!\widehat{\bm{X}}_n\paren{\etab}|^\gamma},    
\end{split}\\
\begin{split}
    \jump{f\paren{\widehat{\bm{x}}}}_{C^\gamma\paren{\mbs}}
    :=&\sup_{0<\abs{\widehat{\bm{x}}-\widehat{\bm{y}}}}\frac{\abs{f\paren{\widehat{\bm{x}}}-f\paren{\widehat{\bm{y}}}}}{\abs{\widehat{\bm{x}}-\widehat{\bm{y}}}^\gamma}
    =\sup_{\widehat{\bm{X}}_n\paren{\thetab}\neq\widehat{\bm{X}}_m\paren{\etab}}\frac{\abs{f_n\paren{\thetab}-f_m\paren{\etab}}}{|\widehat{\bm{X}}_n\paren{\thetab}-\widehat{\bm{X}}_m\paren{\etab}|^\gamma}.
\end{split}
\end{align*}
\end{defn}
\begin{defn}[Arc-chord condition]
\begin{align*}
    \starnorm{f}
    &:= \inf_{\widehat{\bm{x}}\neq\widehat{\bm{y}}} \frac{\abs{f\paren{\widehat{\bm{x}}}-f\paren{\widehat{\bm{y}}}}}{\abs{\widehat{\bm{x}}-\widehat{\bm{y}}}}
    =\inf_{\widehat{\bm{X}}_n\paren{\thetab}\neq\widehat{\bm{X}}_m\paren{\etab}}\frac{\abs{f_n\paren{\thetab}-f_m\paren{\etab}}}{|\widehat{\bm{X}}_n\paren{\thetab}-\widehat{\bm{X}}_m\paren{\etab}|}.
\end{align*}
\end{defn}
\begin{defn}[Locally Arc-chord condition in the Charts]
Given $V_n\subset \mcu_n$,
\begin{align*}
    \abs{f}_{\circ, n}:=&\inf_{\thetab\neq\etab, \thetab,\etab \in V_n} \frac{|f_n(\thetab)-f_n(\etab)|}{|\thetab-\etab|},\\
    \circlenorm{f}:=&\inf_n \abs{f}_{\circ, n}.
\end{align*}
\end{defn}
\begin{defn}[$L^p$ norms]
\begin{equation*}
\begin{aligned}
    \norm{f\paren{\widehat{\bm{x}}}}_{L^p\paren{\mbs}}^p&:=\sum_n \int_{\mcu_n} \rho_n \paren{\thetab}\abs{f_n\paren{\thetab}}^p \sqrt{\det \widehat{g}_n} d\thetab\\
    &= \sum_n \norm{\paren{\rho_n}^{\frac{1}{p}} f_n}_{L^p\paren{\mcu_n}}^p
    \leq \sum_n \norm{f_n\paren{\thetab}}_{L^p\paren{\mcu_n}}^p.
\end{aligned}
\end{equation*}
\end{defn}
\subsection{Standard Stereographic Projection}
We will see the properties of the standard stereographic projection (i.e. the projection point is $(0,0,1)$).
For the other projection points, because $\mbs$ is centrosymmetric, we just need to rotate the coordinates of $\mbs$.
Hence, most properties among the projection charts are the same.

\begin{defn}[Standard Stereographic Projection]
We set $\widehat{\bm{X}}: \mbr \bigcup \{\infty\}\rightarrow \mbs$ with
\begin{align*}
    \begin{split}
        \widehat{\bm{X}}\paren{\thetab}=&\paren{\frac{2\theta_1}{1+\abs{\thetab}^2},\frac{2\theta_2}{1+\abs{\thetab}^2},\frac{-1+\abs{\thetab}^2}{1+\abs{\thetab}^2}},\\
        \widehat{\bm{X}}\paren{\infty}:=&\lim_{\abs{\thetab}\rightarrow \infty}\widehat{\bm{X}}\paren{\thetab}=\paren{0,0,1}.
    \end{split}
\end{align*}
Then,
\begin{align*}
    \thetab\paren{\widehat{\bm{x}}}=\paren{\frac{\widehat{x}_1}{1-\widehat{x}_3},\frac{\widehat{x}_2}{1-\widehat{x}_3}}, \thetab\paren{0,0,1}=\infty.
\end{align*}
We call the parameterization $\widehat{\bm{X}}$ the standard stereographic projection.
\end{defn}
We will denote $V_R$ the coordinate balls in $\mathbb{R}^2$, 
\begin{equation}\label{def_VR}
    \begin{aligned}
    V_R&:=\mc{B}\Big(\frac{R}{\sqrt{4-R^2}}\Big)\subset\mathbb{R}^2.
    \end{aligned}
\end{equation}
\begin{prop}
The standard stereographic projection has the following properties:

\begin{enumerate}
    \item[i)] 
    \begin{align*}
            \PD{\wh{\bm{X}}(\bm{\theta})}{\theta_1}\cdot \PD{\wh{\bm{X}}(\bm{\theta})}{\theta_2}=0,\qquad \norm{\PD{\wh{\bm{X}}(\bm{\theta})}{\theta_1}}=\norm{\PD{\wh{\bm{X}}(\bm{\theta})}{\theta_2}}=\frac{2}{1+\abs{\thetab}^2}.
    \end{align*}
    \item[ii)] For all $R>0$ and all $V_R$ \eqref{def_VR},
    $\widehat{\bm{X}}(V_R)= B_{(0,0,-1), R}\cap \mbs.$
    \item[iii)] For $\thetab,\etab\in\mathbb{R}^2$,
    \begin{align}
        |\widehat{\bm{X}}\paren{\thetab}-\widehat{\bm{X}}\paren{\etab}|\leq 2\abs{\thetab-\etab},\label{right_metric_ineq}
    \end{align}
    and if $\thetab,\etab\in V_R$ with $R\leq \sqrt{2}$,
    \begin{align}
        |\widehat{\bm{X}}\paren{\thetab}-\widehat{\bm{X}}\paren{\etab}|\geq \frac{2}{\pi}\abs{\thetab-\etab}.\label{left_metric_ineq}
    \end{align}
 \end{enumerate}
\end{prop}
\begin{proof}
For iii), set $\hxi=\frac{\thetab-\etab}{\abs{\thetab-\etab}}$, then
\begin{align*}
\begin{split}
    |\widehat{\bm{X}}\paren{\thetab}\!-\!\widehat{\bm{X}}\paren{\etab}|
    &=  \Big|\int_0^{\abs{\thetab-\etab}}\!\!\PD{}{s} \widehat{\bm{X}}(\etab\!+\!s\hxi)ds\Big| \leq\int_0^{\abs{\thetab-\etab}}|\nabla\widehat{\bm{X}}(\etab\!+\!s\hxi)\cdot \hxi|ds\\
    &=  \int_0^{\abs{\thetab-\etab}}\frac{2}{1\!+\!|\etab+s\hxi|^2}ds \leq2\abs{\thetab-\etab}.
\end{split}
\end{align*}
Set $\text{dist} (\hx,\hy; \mbs)$ as the length of shortest curve connecting $\hx$ and $\hy$ on $\mbs$.
When $\thetab,\etab\in V_{R}$ with $R\leq \sqrt{2}$, the shortest curve $\ell$ for $\text{dist}(\widehat{\bm{X}}\paren{\thetab},\widehat{\bm{X}}\paren{\etab};\mbs)$ is on $B_{(0,0,-1), R}\cap \mbs$ and
\begin{align*}
    \frac{2}{\pi}\leq \frac{R\sqrt{4-R^2}}{2\cos^{-1}\frac{2-R^2}{2}}\leq \frac{|\widehat{\bm{X}}\paren{\thetab}-\widehat{\bm{X}}\paren{\etab}|}{\text{dist}(\widehat{\bm{X}}\paren{\thetab},\widehat{\bm{X}}\paren{\etab};\mbs)}\leq 1,
\end{align*}
where the above function of $R$ is a decreasing function.
Then, since $\hX$ is isothermal and $\hX^{-1}\paren{\ell}$ is in $V_R$,
\begin{align*}
\begin{split}
    \text{dist}(\widehat{\bm{X}}\paren{\thetab},\widehat{\bm{X}}\paren{\etab};\mbs)
    =&\int_\ell dl\paren{\hx}
    =\int_{\hX^{-1}\paren{\ell}}\frac{2}{1+\abs{\thetab}^2}d l\paren{\thetab}\\
    \geq&\frac{4-R^2}{2}\int_{\hX^{-1}\paren{\ell}}d l\paren{\thetab}
    \geq\frac{4-R^2}{2}\abs{\thetab-\etab}.
\end{split}
\end{align*}
Therefore
\begin{align*}
\begin{split}
        |\widehat{\bm{X}}\paren{\thetab}-\widehat{\bm{X}}\paren{\etab}|
    \geq&\frac{2}{\pi}\text{dist}\paren{\widehat{\bm{X}}\paren{\thetab},\widehat{\bm{X}}\paren{\etab};\mbs}\geq\frac{2}{\pi}\abs{\thetab-\etab}.
\end{split}
\end{align*}
\end{proof}
\subsection{The Quantitative Relationships between $\mbs$ and $\mbr$ on the Standard Stereographic Projection Chart}
Given a function $f\paren{\hx}$ on $\mbs$, we may define  $f\paren{\thetab}:=f(\hX \paren{\thetab})$ on $\mbr$.
$\hX \paren{\thetab}$ is isothermal, but the chart is neither isometric nor area-preserving.
Therefore, some quantities of $f$ between $\mbs$ and $\mbr$ are different. 
We have to check their quantitative relationships.

First, let see the H\"older continuous seminorm $C^\gamma$ and the arc-chord condition.
\begin{prop}
Given $f$ on $\mbs$, it holds that $\jump{f}_{C^\gamma\paren{\mbr}}\leq 2^\gamma\jump{f}_{C^\gamma\paren{\mbs}}$.
\end{prop}
\begin{proof}
\begin{align*}
\begin{split}
    \jump{f}_{C^\gamma\paren{\mbr}}
    &=\sup_{\thetab\neq \etab}\Big(\frac{|\hX\paren{\thetab}-\hX\paren{\etab}|}{\abs{\thetab-\etab}}\Big)^\gamma\frac{|f(\hX\paren{\thetab})-f(\hX\paren{\etab})|}{|\hX\paren{\thetab}-\hX\paren{\etab}|^\gamma}\leq 2^\gamma\jump{f}_{C^\gamma\paren{\mbs}}.
\end{split}
\end{align*}
\end{proof}
Notice that $\circlenorm{f}$ for $\mbr$ is zero, and the other sided inequality between $C^\gamma\paren{\mcu}$ and $C^\gamma\paren{\mcu}$ cannot hold with some constant $C>0$, thus we only can find local inequalities. 
Set $B_R= B_{(0,0,-1), R}\cap \mbs$, $V_R$ as in \eqref{def_VR}, and $\rho$ smooth on $\mbs$ and supported in $B_R$.
\begin{prop}
Given $R\leq \sqrt{2}$, it holds that
\begin{align*}
    \jump{\rho f}_{C^\gamma\paren{\mbs}}\leq&\big(\frac{\pi}{2}\big)^\gamma\jump{\rho f}_{C^\gamma\paren{\mbr}}, \\
    \starnorm{f}\leq&\frac{\pi}{2}\circlenorm{f}.
\end{align*}
\end{prop}

\begin{proof}
For $\jump{\rho f}_{C^\gamma\paren{\mbs}}$, when $\hx=\hX\paren{\thetab},\hy=\hX\paren{\etab}\in B_R$, $\thetab, \etab\in V_R$, so
\begin{align*}
\begin{split}
        \frac{\abs{\rho f\paren{\hx}-\rho f\paren{\hy}}}{\abs{\hx-\hy}^\gamma}
    =   &\Big(\frac{\abs{\thetab-\etab}}{|\hX\paren{\thetab}-\hX\paren{\etab}|}\Big)^\gamma\frac{\abs{\rho f\paren{\thetab}-\rho f\paren{\etab}}}{\abs{\thetab-\etab}^\gamma}\leq \big(\frac{\pi}{2}\big)^\gamma\jump{\rho f}_{C^\gamma\paren{\mbr}}.
\end{split}
\end{align*}
Next, when $\hx=\hX\paren{\thetab},\hy=\hX\paren{\etab}\in B_R^c$, $\frac{\abs{\rho f\paren{\hx}-\rho f\paren{\hy}}}{\abs{\hx-\hy}^\gamma}=0$.
Finally, when $\hx=\hX\paren{\thetab}\in B_R,\hy=\hX\paren{\etab}\in B_R^c$, $\thetab\in V_R, \etab\in V_R^c$, set $\hz=\hX\paren{\xib}\in \partial B_R$ s.t. $\abs{\hx-\hz}=\text{dist}\paren{\hx, \partial B_R}$.
Since $\rho$ is smooth on $\mbs$ and supported in $B_R$, $\rho\paren{\hz}=0=\rho\paren{\hy}$.
Then,
\begin{align*}
    \abs{\thetab-\xib}\leq \frac{\pi}{2}\abs{\hx-\hz}\leq \frac{\pi}{2}\abs{\hx-\hy},
\end{align*}
so
\begin{align*}
\begin{split}
        \frac{\abs{\rho f\paren{\hx}-\rho f\paren{\hy}}}{\abs{\hx-\hy}^\gamma}
    &=   \frac{\abs{\rho f\paren{\hx}-\rho f\paren{\hz}}}{\abs{\hx-\hy}^\gamma}
    =   \Big(\frac{\abs{\thetab-\xib}}{|\hX\paren{\thetab}-\hX\paren{\etab}|}\Big)^\gamma\frac{\abs{\rho f\paren{\thetab}-\rho f\paren{\xib}}}{\abs{\thetab-\xib}^\gamma}\\
    &\leq\big(\frac{\pi}{2}\big)^\gamma\jump{\rho f}_{C^\gamma\paren{\mbr}}.
\end{split}
\end{align*}
Now, for $\circlenorm{f}$,
\begin{align*}
\begin{split}
    \circlenorm{f}
    &=  \inf_{\thetab\neq\etab, \thetab,\etab \in V} \frac{|\hX\paren{\thetab}-\hX\paren{\etab}|}{\abs{\thetab-\etab}}\frac{|f(\hX\paren{\thetab})-f(\hX\paren{\etab})|}{|\hX\paren{\thetab}-\hX\paren{\etab}|}\geq \frac{2}{\pi}\starnorm{f}.
\end{split}
\end{align*}

\end{proof}

Next, let us discuss the relationship between $C^{1,\gamma}\paren{\mbs}$ and $C^{1,\gamma}\paren{\mbr}$ on the standard stereographic projection chart.
In the standard stereographic projection chart $\hx=\hX\paren{\thetab}$, the surface gradient of $f$, $\nabla_\mbs f\paren{\hx}$, is
\begin{align*}
    \nabla_\mbs f\paren{\hx}
    =\sum_{i,j}\wh{g}^{i,j}\PD{}{\theta_i}f\paren{\thetab}\PD{}{\theta_j}\hX\paren{\thetab}
    =\Big(\frac{1+\abs{\thetab}^2}{2}\Big)^2\sum_i \PD{}{\theta_i}f\paren{\thetab}\PD{}{\theta_i}\hX\paren{\thetab},
\end{align*}
where $\wh{g}^{ij}$ denotes the inverse tensor of $\wh{g}$.
Hence,
\begin{align*}
        \frac{2}{1+\abs{\thetab}^2}\abs{\nabla_\mbs f(\widehat{\bm{X}}(\thetab))}=\abs{\nabla f\paren{\thetab}}.
\end{align*}
We may use the above expressions to obtain the following proposition.
\begin{prop}\label{prop: equi_norm}
There exist $R_0>0$ and $C\geq2^{1+\gamma}$ s.t. for all $R<R_0$
\begin{align*}
    \norm{f}_{C^{1,\gamma}\paren{\mbr}}\leq& C\norm{f}_{C^{1,\gamma}\paren{\mbs}},\\
    \norm{\rho f}_{C^{1,\gamma}\paren{\mbs}}\leq&\norm{\rho f}_{C^{1,\gamma}\paren{\mbr}},\\
    \starnorm{f}\leq&\circlenorm{f},
\end{align*}
where $\rho$ is smooth on $\mbs$ and supported in $B_R$.
\end{prop}

\subsection{Stereographic Projection Charts $\widehat{\bm{X}}_n$ Covering $\mbs$}\label{covering_ste}

We consider a finite cover of the sphere  consisting of balls of radius $R<R_0$ and center $\hx_n$, $\{B_{\hx_n,R}\cap\mathbb{S}^2\}$, and a smooth partition of unity subordinated to it, $\{\rho_n\}$, such that $\rho_n(\hx)=0$ if $|\hx-\hx_n|\geq 2R$. For convenience, we set $\widehat{\bm{x}}_0= (0,0,-1)$. Then we take stereographic projection charts $\widehat{\bm{X}}_n$ covering $\mbs$ (see Figure \ref{fig2}). 
\begin{figure}[H]
\includegraphics[scale=0.75]{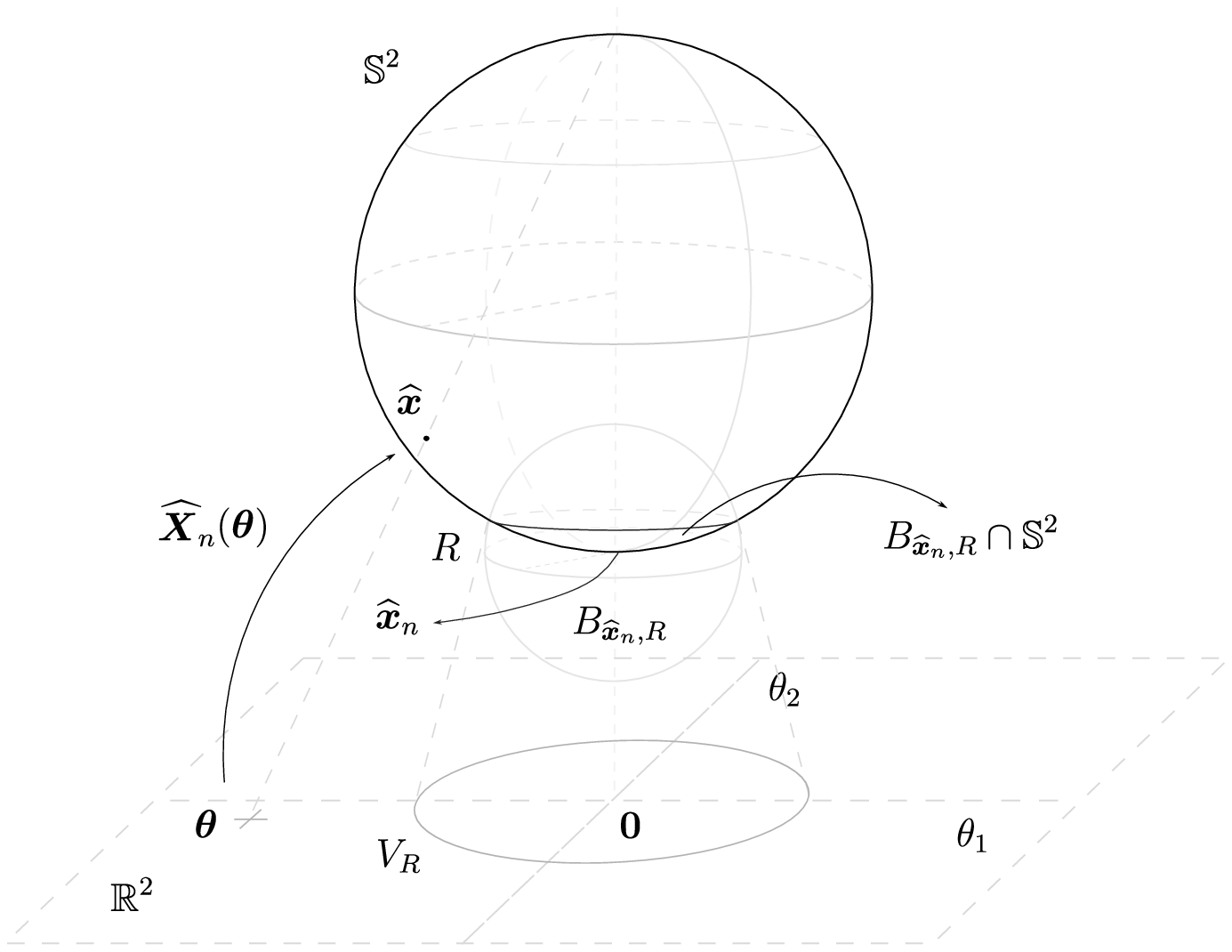}
\caption{Stereographic projection charts, $\hX_n(\thetab)$.}\label{fig2}
\end{figure}
\begin{defn}[Stereographic Projection Charts covering $\mbs$]
$\widehat{\bm{X}}_n$ are standard stereographic projection charts with 
\begin{equation*}
\left.\begin{aligned}
    \widehat{\bm{X}}_n&: \mbr\rightarrow  \mbs\\
    \widehat{\bm{X}}_n(\thetab)&=\Theta_n \widehat{\bm{X}}(\thetab) 
\end{aligned}\right\},
\end{equation*}
where $\Theta_n$ is the rotation matrix with $\widehat{\bm{x}}_n=\Theta_n\widehat{\bm{x}}_0$.
\end{defn}
\begin{prop}
In each chart $\widehat{\bm{X}}_n$, given $R<\frac{R_0}{4}$ and $R_0,C$ from Proposition \ref{prop: equi_norm}, since $\rho$ are supported in $B_n$, 
\begin{align*}
    \norm{f}_{C^{1,\gamma}\paren{\mbr}}&\leq C\norm{f}_{C^{1,\gamma}\paren{\mbs}},\\
    \norm{\rho_n f}_{C^{1,\gamma}\paren{\mbs}}&\leq\norm{\rho_n f}_{C^{1,\gamma}\paren{\mbr}},\\
    \starnorm{f}&\leq\abs{f}_{\circ,n}.
\end{align*}
As a consequence, $\starnorm{f}\leq\circlenorm{f}$.
\end{prop}

\section{Nonlinear decomposition}\label{sec:leading}
In this section we extract the leading structure of the equation and compute its symbol, introducing the notation for the operators that will appear along the paper.

\subsection{Nonlinear decomposition}\label{sec:leading:N}
We will usually consider separately the two terms involved in the Stokeslet kernel \eqref{Stokeslet},
 \begin{equation}\label{GG1G2}
\begin{aligned}
    G_{k,l}(\bm{x})&=G^1_{k,l}(\bm{x})+G^2_{k,l}(\bm{x}),\\
    G^1_{k,l}(\bm{x})&=\frac{1}{8\pi}\frac{\delta_{k,l}}{\abs{\bm{x}}},\quad G^2_{k,l}(\bm{x})=\frac{1}{8\pi}\frac{x_k x_l}{\abs{\bm{x}}^3},
\end{aligned}
\end{equation}
 and thus we will write our equation \eqref{Xieqn} as follows:
\begin{equation}\label{Xeq_F}
    \PD{\bm{X}}{t}(\hx)=F(\bm{X})(\hx)=F^1(\bm{X})(\hx)+F^2(\bm{X})(\hx),
\end{equation}
where
\begin{equation}\label{F_expr}
    \begin{aligned}
    F(\bm{X})(\hx)&=-\int_{\mathbb{S}^2} \nabla_{\mathbb{S}^2}G(\bm{X}(\hx)-\bm{X}(\hy))\cdot T(|\nabla_{\mathbb{S}^2}\bm{X}(\hy)|)\nabla_{\mathbb{S}^2}\bm{X}(\hy)d\hy,\\
    F^j(\bm{X})(\hx)&=-\int_{\mathbb{S}^2} \nabla_{\mathbb{S}^2}G^j(\bm{X}(\hx)-\bm{X}(\hy))\cdot T(|\nabla_{\mathbb{S}^2}\bm{X}(\hy)|)\nabla_{\mathbb{S}^2}\bm{X}(\hy)d\hy.    \end{aligned}
\end{equation}
and we introduced the notation
\begin{equation}\label{Taulaw}
   T(|\nabla_{\mathbb{S}^2}\bm{X}|)=\frac{\mc{T}(|\nabla_{\mathbb{S}^2}\bm{X}|)}{|\nabla_{\mathbb{S}^2}\bm{X}|}.
\end{equation}
Above, we use the shorter notation $d\hy=d\mu_{\mathbb{S}^2}(\hy)$.
We define the following associate linear operators,
\begin{equation}\label{Ndef}
    \begin{aligned}
    (\mc{N}(\bm{X})Z)_k(\hx)&=-\int_{\mathbb{S}^2} \nabla_{\mathbb{S}^2}G_{kl}(\bm{X}(\hx)-\bm{X}(\hy))\cdot \rowofmat{Z}{l}(\hy)d\hy,\\
      (\mc{N}^j(\bm{X})Z)_k(\hx)&=-\int_{\mathbb{S}^2} \nabla_{\mathbb{S}^2}G_{kl}^j(\bm{X}(\hx)-\bm{X}(\hy))\cdot \rowofmat{Z}{l}(\hy)d\hy.
    \end{aligned}
\end{equation}
Then, we compute the kernels:
\begin{equation}\label{Gderx}
\begin{aligned}
\PD{}{x_i} G^1_{k,l}(\bm{x})&=\frac{-1}{8\pi}\frac{x_i}{|\bm{x}|^3}\delta_{k,l},\\
\PD{}{x_i} G^2_{k,l}(\bm{x})&=\frac{1}{8\pi}\frac{\delta_{k,i}x_l+x_k\delta_{i,l}}{|\bm{x}|^3}-\frac{3}{8\pi}\frac{x_k x_l x_i}{|\bm{x}|^5},
\end{aligned}    
\end{equation}
and by the chain rule,
\begin{equation}\label{q_kernels}
    \begin{aligned}
    q_{k,l}^j(\hx,\hy):&=\nabla_{\mathbb{S}^2}G_{k,l}^j(\bm{X}(\hx)-\bm{X}(\hy))\\
    &=-\PD{}{x_i}G_{k,l}^j(\bm{X}(\hx)-\bm{X}(\hy))\nabla_{\mathbb{S}^2}X_i(\hy).
    \end{aligned}
\end{equation}
so that we write
\begin{equation}\label{Ndef_q}
    (\mc{N}^j(\bm{X})Z)_k(\hx)=-\int_{\mathbb{S}^2}q^j_{k,l}(\hx,\hy)\cdot\rowofmat{Z}{l}(\hy)d\hy.
\end{equation}
The explicit expression for $q^j_{k,l}$ is given by
\begin{equation*}
    \begin{aligned}
     q^1_{k,l}(\hx,\hy)&=\frac{1}{8\pi} \frac{\Delta_{\hy}X_j(\hx)\nabla_{\mathbb{S}^2}X_j(\hy)}{|\Delta_{\hy}\bm{X}(\hx)|^3}\frac{\delta_{k,l}}{|\hx-\hy|^2},
    \end{aligned}
\end{equation*}
and
\begin{equation*}
    \begin{aligned}
     q^2_{k,l}(\hx,\hy)&=-\frac{1}{8\pi} \frac{\Delta_{\hy}X_l(\hx)\nabla_{\mathbb{S}^2}X_k(\hy)+\Delta_{\hy}X_k(\hx)\nabla_{\mathbb{S}^2}X_l(\hy)}{|\Delta_{\hy}\bm{X}(\hx)|^3}\frac{1}{|\hx-\hy|^2}\\
     &\quad+\frac{3}{8\pi} \frac{\Delta_{\hy}X_k(\hx)\Delta_{\hy}X_l(\hx)\Delta_{\hy}X_j(\hx)\nabla_{\mathbb{S}^2}X_j(\hy)}{|\Delta_{\hy}\bm{X}(\hx)|^5}\frac{1}{|\hx-\hy|^2}.
    \end{aligned}
\end{equation*}
Using the standard stereographic projection (see Section \ref{sec:Notation}) and the notation $\bm{X}(\thetab)=\bm{X}(\hX(\thetab))$, the equation for each component of $\bm{X}(\thetab)$ becomes
\begin{equation*}
\begin{aligned}
\PD{X_k}{t}(\thetab)&=(F(\bm{X}))_k(\thetab)\\
&=-\int_{\mbr} \PD{}{\eta_i}G_{k,l}(\bm{X}(\thetab)\!-\!\bm{X}(\bm{\eta}))T(\lambda(\etab))\PD{X_l}{\eta_i}(\bm{\eta})d\eta_1d\eta_2,
\end{aligned}
\end{equation*}
where $\lambda(\etab)$ is given in \eqref{lambda_n}
and we denote accordingly $F^j(\bm{X})(\thetab)$, $\mc{N}^j(\bm{X})Z(\thetab)$.
If we use the stereographic projection centered at $\hx_n$, then we denote $\mc{N}^j(\bm{X})Z_n(\thetab)$.

Then, we take the derivative in $G_{k,l}$ \eqref{GG1G2} to obtain
\begin{equation*}
    \begin{aligned}
    \PD{}{\eta_i}G_{k,l}(\bm{X}(\bm{\theta})\!-\!\bm{X}(\bm{\eta}))&=q_{i,k,l}(\thetab,\etab)\\
    &=q_{i,k,l}^1(\thetab,\etab)+q_{i,k,l}^2(\thetab,\etab),
    \end{aligned}
\end{equation*}
where
\begin{equation}\label{Gder}
    \begin{aligned}
    q^1_{i,k,l}(\thetab,\etab)&=\PD{}{\eta_i}G_{k,l}^1(\bm{X}(\bm{\theta})\!-\!\bm{X}(\bm{\eta}))=\frac{1}{8\pi}\delta_{k,l} \frac{\delta_{\etab}\bm{X}(\thetab)\cdot \PD{\bm{X}}{\eta_i}(\etab)}{|\delta_{\etab}\bm{X}(\thetab)|^3},
    \end{aligned}
\end{equation}
\begin{equation*}
    \begin{aligned}
     q^2_{i,k,l}(\thetab,\etab)&=\PD{}{\eta_i}G_{k,l}^2(\bm{X}(\bm{\theta})\!-\!\bm{X}(\bm{\eta}))\\
    &=-\frac{1}{8\pi}\frac{\PD{X_{k}}{\eta_i}(\etab)\delta_{\etab}X_{l}(\thetab)+\delta_{\etab}X_{k}(\thetab)\PD{X_{l}}{\eta_i}}{|\delta_{\etab} \bm{X}(\thetab)|^3}\\
    &\quad+\frac{3}{8\pi}\frac{\delta_{\etab}X_{k}(\thetab)\delta_{\etab}X_{l}(\thetab)}{|\delta_{\etab} \bm{X}(\thetab)|^5}\delta_{\etab}\bm{X}(\thetab)\cdot\PD{\bm{X}(\etab)}{\eta_i},
    \end{aligned}
\end{equation*}
so that we can write
\begin{equation}\label{Ntheta}
    (\mc{N}^j(\bm{X})Z)_k(\thetab)=-\int_{\mathbb{R}^2}q^j_{m,k,l}(\thetab,\etab)\PD{\wh{X}_i}{\eta_m}(\etab)Z_{l,i}(\etab)d\eta_1d\eta_2.
\end{equation}
We notice that the kernels in \eqref{Gder} are given by
\begin{equation}\label{kernels_q}
    q_{i,k,l}^j(\thetab,\etab)=-\PD{}{x_m}G_{k,l}^j(\bm{X}(\thetab)-\bm{X}(\etab))\PD{X_{m}}{\eta_i}(\etab).
\end{equation}
We introduce the following notation for finite differences,
\begin{equation}\label{delta_eta}
    \delta_{\bm{\eta}} g(\bm{\theta})=g(\bm{\theta})-g(\bm{\eta}),\qquad \Delta_{\bm{\eta}} g(\bm{\theta})=\frac{\delta_{\bm{\eta}} g(\bm{\theta})}{|\thetab-\etab|},
\end{equation}
and we extract the expected leading terms by replacing 
\begin{equation*}
    \delta_{\etab}\bm{X}(\thetab)\approx \nabla \bm{X}(\etab)(\thetab-\etab),\qquad (\nabla \bm{X})_{p,q}(\etab)=\frac{\partial X_{n,p}}{\partial \eta_q}(\etab).
\end{equation*}
Hence, we define the associate kernels
\begin{equation}\label{m_kernels}
    \begin{aligned}
    m_{i,k,l}(\thetab,\etab)&=-\PD{}{x_j}G_{k,l}\paren{\nabla\bm{X}(\etab)(\thetab-\etab)}\PD{X_{j}}{\eta_i}(\etab)\\
    &=m_{i,k,l}^1(\thetab,\etab)+m_{i,k,l}^2(\thetab,\etab),
    \end{aligned}
\end{equation}
and define the linear operators $\mc{M}(\nabla \bm{X})z$ as follows:
\begin{equation}\label{defM}
    \begin{aligned}
    (\mc{M}(\nabla\bm{X})Z)_k(\thetab)&=-\int_{\mbr} m_{m,k,l}(\thetab,\etab)
    \PD{\wh{X}_i}{\eta_m}(\etab)Z_{l,i}(\etab)d\eta_1d\eta_2\\
    &=(\mc{M}^1(\nabla\bm{X})Z)_k(\thetab)+(\mc{M}^2(\nabla\bm{X})Z)_k(\thetab),
    \end{aligned}
\end{equation}
with
\begin{equation*}
    \begin{aligned}
    (\mc{M}^j(\nabla\bm{X})Z)_k(\thetab)&=-\int_{\mbr} m_{m,k,l}^j(\thetab,\etab)
    \PD{\wh{X}_i}{\eta_m}(\etab)Z_{l,i}(\etab)d\eta_1d\eta_2.
    \end{aligned}
\end{equation*}
We compute the explicit expression of these kernels $m_{i,k,l}^j$ \eqref{m_kernels},
\begin{equation*}
    \begin{aligned}
    m_{i,k,l}^1(\thetab,\etab)&=\frac{1}{8\pi}\frac{\PD{\bm{X}(\etab)}{\eta_i}\cdot(\nabla \bm{X}(\etab)(\thetab\!-\!\etab))}{|\nabla \bm{X}(\etab)(\thetab\!-\!\etab)|^3},
    \end{aligned}
\end{equation*}
\begin{equation*}
    \begin{aligned}
    m_{i,k,l}^2(\thetab,\etab)&=-\frac{1}{8\pi}\frac{\PD{X_{k}(\etab)}{\eta_i}(\nabla \bm{X}(\etab)(\thetab\!-\!\etab))_l+\PD{X_{l}(\etab)}{\eta_i}(\nabla \bm{X}(\etab)(\thetab\!-\!\etab))_k}{|\nabla \bm{X}(\etab)(\thetab\!-\!\etab)|^3}\\
    &\hspace{-1.5cm}+\frac{3}{8\pi}\frac{(\nabla \bm{X}(\etab)(\thetab\!-\!\etab))_k(\nabla \bm{X}(\etab)(\thetab\!-\!\etab))_l}{|\nabla \bm{X}(\etab)(\thetab\!-\!\etab)|^5}(\nabla \bm{X}(\etab)(\thetab\!-\!\etab))\cdot\PD{\bm{X}(\etab)}{\eta_i}.
    \end{aligned}
\end{equation*}
We will use the notation
\begin{equation*}
    \nabla \bm{X}(\etab)(\thetab-\etab)=(\thetab-\etab)\cdot \nabla \bm{X}(\etab),\qquad \widehat{\bm{z}}=\frac{\bm{z}}{|\bm{z}|},
\end{equation*}
and we define
\begin{equation}\label{EetaX}
    E^{\etab}\bm{X}(\thetab):=(\widehat{\thetab-\etab})\cdot\nabla\bm{X}(\etab)-\Delta_{\etab} \bm{X}(\thetab),
\end{equation}
for which we have that,
\begin{equation}\label{Eeta_bound}
    \frac{|E^{\etab}(\bm{X}(\thetab))|}{|\thetab-\etab|^\gamma}\leq \jump{\nabla \bm{X}}_{C^{\gamma}(\mathbb{R}^2)}.
\end{equation}
Thus, we can write
\begin{equation}\label{nonlinear_split}
    \begin{aligned}
        \mc{N}(\bm{X})Z(\thetab)&=\mc{M}(\nabla \bm{X})Z(\thetab)+\mc{R}(\bm{X})Z(\thetab),
    \end{aligned}
\end{equation}
where the remainder term
\begin{equation*}
    \mc{R}(\bm{X})Z(\thetab)=\sum_{j=1}^2\mc{R}^j(\bm{X})Z(\thetab)
\end{equation*}
is given by
\begin{equation}\label{defR}
\begin{aligned}
    (\mc{R}^j(\bm{X})Z))_k(\thetab)&=(\mc{N}^j(\bm{X})Z)_k(\thetab)-(\mc{M}^j(\nabla\bm{X})(Z))_k(\thetab)\\
    &=\int_{\mbr}\mathcal{K}^j_{m,k,l}(\thetab,\etab)
    \PD{\wh{X}_i}{\eta_m}(\etab)Z_{l,i}(\etab)d\eta_1d\eta_2,
\end{aligned}
\end{equation}
with kernels
\begin{equation*}
    \begin{aligned}
    \mathcal{K}^1_{i,k,l}(\thetab,\etab)
    &=-q_{i,k,l}^1\paren{\thetab,\etab}+m_{i,k,l}^1\paren{\thetab,\etab}\\
    &=\frac{1}{8\pi}\frac{\delta_{k,l}}{|\thetab-\etab|^2}\PD{\bm{X}(\etab)}{\eta_i}\cdot\bigg(\frac{E^{\etab}\bm{X}(\thetab)}{|\Delta_{\etab}\bm{X}(\thetab)|^3}\\
    &\qquad-((\widehat{\thetab-\etab})\cdot\nabla\bm{X}(\etab))\Big(\frac{1}{|\Delta_{\etab}\bm{X}(\thetab)|^3}-\frac{1}{|(\widehat{\thetab-\etab})\cdot\nabla \bm{X}(\etab)|^3}\Big)\bigg),
    \end{aligned}
\end{equation*}
and
\begin{equation*}
    \mathcal{K}^2_{i,k,l}(\thetab,\etab)=-q_{i,k,l}^2\paren{\thetab,\etab}+m_{i,k,l}^2\paren{\thetab,\etab}=\mathcal{K}^{2,1}_{i,k,l}(\thetab,\etab)+\mathcal{K}^{2,2}_{i,k,l}(\thetab,\etab),
\end{equation*}
\begin{equation*}
    \begin{aligned}
    \mathcal{K}^{2,1}_{i,k,l}(\thetab,\etab)&=\frac{1}{8\pi}\frac{1}{|\thetab-\etab|^2}\bigg(\frac{-\PD{\bm{X}_{k}(\etab)}{\eta_i}}{|\Delta_{\etab}\bm{X}(\thetab)|^3} E^{\etab}X_{l}(\thetab)\\
    &\hspace{-0.5cm}+\PD{\bm{X}_{k}(\etab)}{\eta_i}((\widehat{\thetab-\etab})\cdot\nabla\bm{X}(\etab))_l\Big(\frac{1}{|\Delta_{\etab}\bm{X}(\thetab)|^3}-\frac{1}{|(\widehat{\thetab-\etab})\cdot\nabla \bm{X}(\etab)|^3}\Big)\\
    &\hspace{-0.5cm}-\frac{\PD{\bm{X}_{l}(\etab)}{\eta_i}}{|\Delta_{\etab}\bm{X}(\thetab)|^3} E^{\etab}X_{k}(\thetab)\\
    &\hspace{-0.5cm}+\PD{\bm{X}_{l}(\etab)}{\eta_i}((\widehat{\thetab-\etab})\cdot\nabla\bm{X}(\etab))_k\Big(\frac{1}{|\Delta_{\etab}\bm{X}(\thetab)|^3}-\frac{1}{|(\widehat{\thetab-\etab})\cdot\nabla \bm{X}(\etab)|^3}\Big)\bigg),
    \end{aligned}
\end{equation*}
\begin{equation*}
    \begin{aligned}
    \mathcal{K}^{2,2}_{i,k,l}&(\thetab,\etab)=\frac{1}{8\pi}\frac{3}{|\thetab-\etab|^2}\bigg(\frac{\Delta_{\etab}X_{k}(\thetab)\Delta_{\etab}X_{l}(\thetab)}{|\Delta_{\etab}\bm{X}(\thetab)|^5}\PD{\bm{X}(\etab)}{\eta_i}\cdot E^{\etab}\bm{X}(\thetab)\\
    &\quad+\frac{\PD{\bm{X}(\etab)}{\eta_i}\cdot((\widehat{\thetab-\etab})\cdot\nabla \bm{X}(\etab))}{|\Delta_{\etab}\bm{X}(\thetab)|^5}\Delta_{\etab}X_{k}(\thetab)E^{\etab}X_{l}(\thetab)\\
    &\quad+\frac{\PD{\bm{X}(\etab)}{\eta_i}\cdot((\widehat{\thetab-\etab})\cdot\nabla \bm{X}(\etab))}{|\Delta_{\etab}\bm{X}(\thetab)|^5}((\widehat{\thetab-\etab})\cdot\nabla\bm{X}(\etab))_l E^{\etab}X_{k}(\thetab)\\
    &\quad-\frac{\PD{\bm{X}(\etab)}{\eta_i}\cdot((\widehat{\thetab-\etab})\cdot\nabla \bm{X}(\etab))}{|\Delta_{\etab}\bm{X}(\thetab)|^5}((\widehat{\thetab-\etab})\cdot\nabla\bm{X}(\etab))_l((\widehat{\thetab-\etab})\cdot\nabla\bm{X}(\etab))_k\\
    &\hspace{4cm}\times\Big(\frac{1}{|\Delta_{\etab}\bm{X}(\thetab)|^5}-\frac{1}{|(\widehat{\thetab-\etab})\cdot\nabla\bm{X}(\etab)|^5}\Big)\bigg).
    \end{aligned}
\end{equation*}
\begin{remark}
Note that for all  positive odd integers $k$, it holds that
\begin{align}\label{odd_fractions}
\begin{split}
    \frac{1}{|\bm{u}|^{k}}&-\frac{1}{|\bm{v}|^{k}}=\frac{(\bm{v}-\bm{u})\cdot(\bm{u}+\bm{v})}{|\bm{u}|^k+|\bm{v}|^k}\frac{\sum_{i=1}^k |\bm{u}|^{2(i-1)}|\bm{v}|^{2(k-i)}}{|\bm{u}|^k|\bm{v}|^k}
\end{split}
\end{align}
and
\begin{align}\label{odd_fractions02}
    \abs{\frac{1}{|\bm{u}|^{k}}-\frac{1}{|\bm{v}|^{k}}}=\frac{\abs{\abs{\bm{v}}-\abs{\bm{u}}}\sum_{i=1}^k \abs{\bm{u}}^{i-1}\abs{\bm{v}}^{k-i}}{\abs{\bm{u}}^k \abs{\bm{v}}^k}\leq \frac{\abs{\bm{v}-\bm{u}}\sum_{i=1}^k \abs{\bm{u}}^{i-1}\abs{\bm{v}}^{k-i}}{\abs{\bm{u}}^k \abs{\bm{v}}^k} 
\end{align}
\end{remark}
In particular, formula \eqref{odd_fractions} with $\bm{u}=\Delta_{\etab}\bm{X}(\thetab)$ and $\bm{v}=(\widehat{\thetab-\etab})\cdot\nabla\bm{X}(\etab)$, together with \eqref{EetaX}-\eqref{Eeta_bound}, makes it clear that there is an extra cancellation in the kernels of \eqref{defR},

In summary, our equation \eqref{Xeq_F} is given by
\begin{equation*}
    \begin{aligned}
     \PD{\bm{X}}{t}(\thetab)&=F(\bm{X})(\thetab)\\
     &=\mc{N}(\bm{X})(T(|\nabla_{\mathbb{S}^2}\bm{X}|)\nabla_{\mathbb{S}^2}\bm{X})(\thetab)\\\
     &=\mc{M}(\nabla\bm{X})(T(|\nabla_{\mathbb{S}^2}\bm{X}|)\nabla_{\mathbb{S}^2}\bm{X})(\thetab)+\mc{R}(\bm{X})(T(|\nabla_{\mathbb{S}^2}\bm{X}|)\nabla_{\mathbb{S}^2}\bm{X})(\thetab).
    \end{aligned}
\end{equation*}

\subsection{Symbol of the leading term}\label{sec:symbol}

As a preliminary step towards studying the leading term $\mathcal{M}$ \eqref{defM}, let us consider its frozen-coefficient counterpart, i.e., replacing  $\nabla \bm{X}(\etab)$ by a constant matrix $A$ and letting $\wh{g}=I_2$. We start with the case $\mc{T}=\text{Id}$, that is, $T\equiv 1$:
\begin{equation}\label{defnL1}
(\mc{L}_A^L\bm{Y})_k(\thetab)=(\tilde{\mc{M}}(A)\nabla\bm{Y})_k(\thetab),
\end{equation}
where we define $\tilde{\mc{M}}(A)=\tilde{\mc{M}}^1(A)+\tilde{\mc{M}}^2(A)$ and
\begin{equation}\label{Mtilde}
    \begin{aligned}
    (\tilde{\mc{M}}^j(A)Z)_k(\thetab)=-\int_{\mathbb{R}^2} \PD{}{\eta_i}(G^j_{k,l}(A\paren{\bm{\theta}-\bm{\eta}}))Z_{l,i}(\etab)d\eta_1d\eta_2.
    \end{aligned}
\end{equation}
 Let $\mc{F}_{\bm{\theta}}$ be the 2D Fourier transform in $\bm{\theta}$ and $\bm{\xi}=(\xi_1,\xi_2)^{\rm T}$:
\begin{equation*}
\mc{F}_{\bm{\theta}}w(\bm{\xi})=\int_{\mathbb{R}^2}w(\bm{\theta})\exp(-i\bm{\theta}\cdot \bm{\xi})d\bm{\theta}.
\end{equation*}
We now compute the Fourier transform of the function $G_A$:
\begin{equation*}
G_A(\bm{\theta})=G(A\bm{\theta})=\frac{1}{8\pi}\paren{\frac{I}{\abs{A\bm{\theta}}}+\frac{A\bm{\theta}\otimes A\bm{\theta}}{\abs{A\bm{\theta}}^3}},
\end{equation*}
where $I_3$ is the $3\times 3$ identity matrix. Given that $\bm{\theta}\in \mathbb{R}^2$ and $A$ is a $3\times 2$ matrix, it is convenient 
to rewrite $G_A$ as follows. First, note that:
\begin{equation*}
\abs{A\bm{\theta}}^2=A\bm{\theta}\cdot A\bm{\theta}= \bm{\theta}\cdot \paren{A^{\rm T}A\bm{\theta}}=\abs{B\bm{\theta}}^2, \; B=\sqrt{A^{\rm T}A}.
\end{equation*}
Notice that $B$ is a $2\times 2$ symmetric positive definite matrix. Using this $B$, we have:
\begin{equation}\label{GAexp}
G_A(\bm{\theta})=\frac{1}{8\pi}\paren{\frac{I}{\abs{B\bm{\theta}}}+Q\paren{\frac{B\bm{\theta}\otimes B\bm{\theta}}{\abs{B\bm{\theta}}^3}} Q^{\rm T}},\; Q=AB^{-1}.
\end{equation}
We note that $Q$ is an isometry in the sense that $Q^{\rm T}Q=I_2$ where $I_2$ is the $2\times 2$ identity matrix.
We are now ready to compute the Fourier transform of $G_A$. First, note that:
\begin{equation}\label{Fthetainv}
\mc{F}_{\bm{\theta}}\paren{\frac{1}{\abs{\bm{\theta}}}}=\frac{2\pi}{\abs{\bm{\xi}}}.
\end{equation}
Thus, a simple change of variable yields:
\begin{equation}\label{FBthetainv}
\mc{F}_{\bm{\theta}}\paren{\frac{1}{\abs{B\bm{\theta}}}}=\frac{2\pi}{{\rm det}(B)\abs{B^{-1}\bm{\xi}}}.
\end{equation}
Next, note that:
\begin{equation*}
\begin{split}
\mc{F}_{\bm{\theta}}\paren{\frac{\theta_i\theta_j}{\abs{\bm{\theta}}^3}}&=\mc{F}_\theta\paren{\theta_i\theta_j\Delta_{\bm{\theta}}\paren{\frac{1}{\abs{\bm{\theta}}}}}=
\frac{\partial^2}{\partial\xi_i\partial\xi_j}\paren{\abs{\bm{\xi}}^2\mc{F}_\theta\paren{\frac{1}{\abs{\bm{\theta}}}}}\\
&=2\pi\frac{\partial^2}{\partial\xi_i\partial\xi_j}\abs{\bm{\xi}}=2\pi\paren{\frac{\delta_{ij}}{\abs{\bm{\xi}}}-\frac{\xi_i\xi_j}{\abs{\bm{\xi}}^3}},
\end{split}
\end{equation*}
where $\Delta_{\bm{\theta}}$ is the Laplacian in $\mathbb{R}^2$, $\delta_{ij}$ is the Kronecker delta and we used \eqref{Fthetainv} in the third equality. In matrix notation, the above can be written as:
\begin{equation*}
\mc{F}_{\bm{\theta}}\paren{\frac{\bm{\theta}\otimes\bm{\theta}}{\abs{\bm{\theta}}^3}}=2\pi \paren{\frac{I_2}{\abs{\bm{\xi}}}-\frac{\bm{\xi}\otimes \bm{\xi}}{\abs{\bm{\xi}}^3}}.
\end{equation*}
Again, by changing variables, we see that:
\begin{equation}\label{FBthetaBtheta}
\mc{F}_{\bm{\theta}}\paren{\frac{B\bm{\theta}\otimes B\bm{\theta}}{\abs{B\bm{\theta}}^3}}=\frac{2\pi}{{\rm det}(B)} \paren{\frac{I_2}{\abs{B^{-1}\bm{\xi}}}-\frac{B^{-1}\bm{\xi}\otimes B^{-1}\bm{\xi}}{\abs{B^{-1}\bm{\xi}}^3}}.
\end{equation}
Using \eqref{FBthetaBtheta}, \eqref{Fthetainv} and \eqref{GAexp}, we obtain:
\begin{equation*}
\begin{split}
&\mc{F}_{\bm{\theta}} G_A=\frac{1}{4{\rm det}(B)}\paren{\frac{I+QQ^{\rm T}}{\abs{B^{-1}\bm{\xi}}}-\frac{QB^{-1}\bm{\xi}\otimes QB^{-1}\bm{\xi}}{\abs{B^{-1}\bm{\xi}}^3}}\\
=&\frac{1}{4\sqrt{{\rm det}(A^{\rm T}A)}}\paren{\frac{I+A(A^{\rm T}A)^{-1}A^{\rm T}}{{\paren{\bm{\xi}\cdot (A^{\rm T}A)^{-1}\bm{\xi}}^{1/2}}}-
\frac{A(A^{\rm T}A)^{-1}\bm{\xi}\otimes A(A^{\rm T}A)^{-1}\bm{\xi}}{\paren{\bm{\xi}\cdot (A^{\rm T}A)^{-1}\bm{\xi}}^{3/2}}}
\end{split}
\end{equation*}
This implies that the Fourier symbol of $\mc{L}_A^L$ is given by:
\begin{equation*}
\mc{L}_A^L\bm{Y}=-\mc{F}^{-1}_{\bm \xi}L_A^L(\bm{\xi})\mc{F}_{\bm{\theta}}\bm{Y},\; L_A^L(\xi)=\abs{\bm{\xi}}^2\paren{\mc{F}_{\bm{\theta}}G_A}(\bm{\xi}).
\end{equation*}
To better understand the properties of Fourier multiplier $L_A^L(\bm{\xi})$, we first note that:
\begin{equation}
\begin{split}
QQ^{\rm T}-\frac{QB^{-1}\bm{\xi}\otimes QB^{-1}\bm{\xi}}{\abs{B^{-1}\bm{\xi}}^2}&=Q\paren{I_2-\frac{B^{-1}\bm{\xi}\otimes B^{-1}\bm{\xi}}{\abs{B^{-1}\bm{\xi}}^2}}Q^{\rm T}=\bm{v}(\bm{\xi})\otimes \bm{v}(\bm{\xi}),\\
\bm{v}(\bm{\xi})&=Q\mc{R}_{\pi/2}\frac{B^{-1}\bm{\xi}}{\abs{B^{-1}{\bm{\xi}}}}, \; \mc{R}_{\pi/2}=\begin{pmatrix} 0 & -1 \\ 1 & 0\end{pmatrix}.
\end{split}\label{Ltensorterm}
\end{equation}
Note that $\bm{v}(\bm{\xi})\in \mathbb{R}^3$ is a unit vector, and hence, the above matrix $3\times 3$ matrix is an orthogonal projection on to the subspace spanned by $\bm{v}(\bm{\xi})$.
We see that:
\begin{equation}\label{Lxisimple}
L_A^L(\bm{\xi})=\frac{\abs{\bm{\xi}}^2}{4{\rm det}(B)\abs{B^{-1}\bm{\xi}}}\paren{I+\bm{v}(\bm{\xi})\otimes \bm{v}(\bm{\xi})}.
\end{equation}
It is now immediate that $L_A^L(\bm{\xi})$ is a symmetric positive definite matrix for each $\bm{\xi}\neq 0$ with eigenvalues:
\begin{equation}\label{Lspec}
\lambda=\frac{\mu}{4}\abs{\bm{\xi}}^2 \text{ and } \frac{\mu}{2}\abs{\bm{\xi}}^2, \quad \mu=\frac{1}{{\rm det}(B)\abs{B^{-1}\bm{\xi}}},
\end{equation}
where the eigenspace for $\mu/2$ is spanned by $\bm{v}(\bm{\xi})$ and the two-dimensional eigenspace of $\mu/4$ is spanned by the orthogonal complement of $\bm{v}(\bm{\xi})$.
We also have:
\begin{equation}\label{Lmubnd}
\frac{\mu}{4}\abs{\bm{\xi}}^2\abs{\bm{w}}^2\leq \bm{w}\cdot L(\bm{\xi})\bm{w}\leq \frac{\mu}{2}\abs{\bm{\xi}}^2\abs{\bm{w}}^2
\end{equation}
for any $\bm{w}\in \mathbb{R}^3$.

In the case of general $\mc{T}$, the frozen coefficient linear operator can be obtained by a further linearization of the force function. 
Consider the expression:
\begin{equation*}
\frac{\mc{T}(\lambda_\tau)}{\lambda_\tau}\PD{(X_l+\tau Y_l)}{\theta_i}, \; \lambda_\tau=\lambda(\bm{X}+\tau\bm{Y})
\end{equation*}
where $\lambda$ is here viewed as a function of $\bm{X}$ through its dependence on $g$ (see \eqref{lamg}). Now, 
\begin{equation*}
\begin{split}
&\at{\D{}{\tau}\paren{\frac{\mc{T}(\lambda_\tau)}{\lambda_\tau}\PD{(X_l+\tau Y_l)}{\theta_i}}}{\tau=0}
=\paren{\frac{1}{\lambda}\D{\mc{T}}{\lambda}-\frac{\mc{T}}{\lambda^2}}\at{\D{\lambda_\tau}{\tau}}{\tau=0}\PD{X_l}{\theta_i}+\frac{\mc{T}}{\lambda}\PD{Y_l}{\theta_i}\\
=&\paren{\frac{1}{\lambda}\D{\mc{T}}{\lambda}-\frac{\mc{T}}{\lambda^2}}\frac{1}{\lambda}(\wh{g}^{-1})_{m,n}\PD{X_q}{\theta_m}\PD{Y_q}{\theta_n}\PD{X_l}{\theta_i}+\frac{\mc{T}}{\lambda}\PD{Y_l}{\theta_i}.
\end{split}
\end{equation*}
Now, the frozen coefficient approximation amounts to taking $\wh{g}=I_2$, $\partial X_l/\partial \theta_i=A_{l,i}$ and $\lambda=\norm{A}_F$. Thus, 
\begin{equation*}
\begin{split}
\at{\D{}{\tau}\paren{\frac{\mc{T}(\lambda_\tau)}{\lambda_\tau}\PD{(X_l+\tau Y_l)}{\theta_i}}}{\tau=0}&\approx (T_F(A))_{i,l,m,q}\PD{Y_q}{\theta_m},
\end{split}
\end{equation*}
with
\begin{equation}\label{tensionD}
T_F(A)=\frac{\mc{T}(\norm{A}_F)}{\norm{A}_F}I_2\otimes I_2-\paren{\frac{\mc{T}(\norm{A}_F)}{\norm{A}_F}-\D{\mc{T}}{\lambda}(\norm{A}_F)}\frac{A\otimes A}{\norm{A}_F^2}.    
\end{equation}
Thus, the frozen-coefficient linear operator in the general force case is given by
\begin{equation}\label{defnL2}
(\mc{L}_A\bm{Y})_k(\thetab)=-\int_{\mathbb{R}^2} \PD{}{\eta_i}(G_{k,l}(A\paren{\bm{\theta}-\bm{\eta}}))(T_F(A)\nabla\bm{Y})_{l,i}(\etab)d\eta_1d\eta_2.
\end{equation}
Let us now take the Fourier transform of the divergence of the above:
\begin{equation*}
\begin{split}
\mc{F}(\nabla&\cdot(T_F(A)\nabla\bm{Y}))(\xib)=-M_A(\xib)\mc{F}\bm{Y}(\xib),
\end{split}
\end{equation*}
where
\begin{equation}
    M_A(\xib)=\frac{\mc{T}(\norm{A}_F)}{\norm{A}_F}\paren{\abs{\bm{\xi}}^2 I-\frac{A\bm{\xi}\otimes A\bm{\xi}}{\norm{A}_F^2}}+\D{\mc{T}}{\lambda}(\norm{A}_F)\frac{A\bm{\xi}\otimes A\bm{\xi}}{\norm{A}_F^2}.
\label{tesionD_fm}
\end{equation}
Note that, if we set $\mc{T}=\text{Id}$, 
then $T_F=\text{Id}$ and the above reduces to $M(\xib)=|\xib|^2$. 
Thus, in the general case, the multiplier in $L_A(\xib)$ of \eqref{defnL2} becomes:
\begin{equation}\label{multiplierNon}
\begin{split}
L_A(\bm{\xi})&=\paren{\mc{F}_\theta G_A}(\bm{\xi})M_A(\bm{\xi})\\
&=\frac{I+\bm{v}(\bm{\xi})\otimes \bm{v}(\bm{\xi})}{4{\rm det}(B)\abs{B^{-1}\bm{\xi}}}\paren{
\frac{\mc{T}(\norm{A}_F}{\norm{A}_F}\paren{\abs{\bm{\xi}}^2 I-\frac{A\bm{\xi}\otimes A\bm{\xi}}{\norm{A}_F}}+\D{\mc{T}}{\lambda}(\norm{A}_F)\frac{A\bm{\xi}\otimes A\bm{\xi}}{\norm{A}_F}}.
\end{split}
\end{equation}
It is not difficult to see that, if $\mc{T}>0$ and $d\mc{T}/d\lambda\geq 0$, then the above is coercive in $\abs{\bm{\xi}}^2$.

\section{Calculus estimates}\label{calculus_est}

In this section we include some estimates of the operators that will be frequently used in later sections. 

\begin{lemma} Let $\bm{X}\in C^1(\mathbb{S}^2)$, such that $|\bm{X}|_*>0$. Then, the kernels $G_{k,l}^j(\bm{x})$ \eqref{GG1G2} and $q_{k,l}^j(\hx,\hy)$ \eqref{q_kernels} satisfy the following bounds
\begin{equation}\label{qkernel_bounds}
    \begin{aligned}
    \abs{\partial_{\xb}^{\bm{\alpha}}G_{k,l}^j(\bm{x})}&\leq \frac{C}{\abs{\xb}^{1+|\bm{\alpha}|}},\\
    \abs{\Delta_{\bm{y}}\paren{\partial_{\xb}^{\bm{\alpha}}G_{k,l}^j(\bm{x})}}&\leq C\frac{M^{1+|\bm{\alpha}|}}{m^{3+2|\bm{\alpha}|}},\\
    |q_{k,l}^j(\hx,\hy)|&\leq C\frac{|\nabla_{\mathbb{S}^2}\bm{X}(\hy)|}{|\Delta_{\hy}\bm{X}(\hx)|^2}\frac{1}{|\hx-\hy|^2}\leq C\frac{\|\nabla_{\mathbb{S}^2}\bm{X}\|_{C^0(\mathbb{S}^2)}}{|\bm{X}|_*^2}\frac{1}{|\hx-\hy|^2},
    \end{aligned}
\end{equation}
where $\abs{\bm{\alpha}}$ defined in \eqref{high_pd},  $M=\max\paren{\abs{\xb},\abs{\yb}}$, and $m=\min\paren{\abs{\xb},\abs{\yb}}$.
\end{lemma}
For the sake of completeness, we include a version of the divergence theorem that will be used. Notice that, following standard convention, we will not explicitly write the principal values elsewhere.
\begin{lemma}\label{Div_thm}
Given a matrix $A$ and a compact set $\mc{D}\subset\mathbb{R}^2$ containing $\bm{0}$, then 
\begin{align*}
    \pv \int_{\mc{D}^c} \nabla G(A\bm{\eta})d\etab:&=\lim_{L\rightarrow\infty} \int_{\mc{D}^c\cap \mcbr{L}} \nabla G(A\bm{\eta})d\etab\\
    &=-\int_{\partial\mc{D}} G(A\bm{\eta})\bm{n}\paren{\etab} dl\paren{\etab},
\end{align*}
where $\mcbr{L}\subset\mathbb{R}^2$ is the ball centered at $\bm{0}$ of radius $L$.
In particular,
\begin{align*}
    \pv \int_\mbr \nabla G(A\bm{\eta})d\etab:=\lim_{L\rightarrow\infty,\varepsilon\rightarrow 0} \int_{\mcbr{L}\setminus\mcbr{\varepsilon} } \nabla G(A\bm{\eta})d\etab=0.
\end{align*}
\end{lemma}

\begin{proof}
Since $\mc{D}$ is compact and contains $\bm{0}$, $\mc{D}\subset \mcbr{L}$ when $L$ is large enough. Then, by integration by parts
\begin{align*}
\begin{split}
     \int_{\mc{D}^c\cap \mcbr{L}}\!\!\!\!\!\!\! \nabla G(A\bm{\eta})d\etab    =-\!\int_{\partial\mc{D}} \!\!\!\!G(A\bm{\eta})\bm{n}\paren{\etab} dl\paren{\etab}+\!\int_{\partial\mcbr{L}}\!\!\!\!\!\! G(A\bm{\eta})\bm{n}\paren{\etab} dl\paren{\etab}.
\end{split}
\end{align*}
Since $G(A\bm{\eta})$ is even, the boundary term vanishes.
Therefore,
\begin{align*}
    \int_{\mc{D}^c} \nabla G(A\bm{\eta})d\etab=\lim_{L\rightarrow\infty} \int_{\mc{D}^c\cap \mcbr{L}} \nabla G(A\bm{\eta})d\etab=-\int_{\partial\mc{D}} G(A\bm{\eta})\bm{n}\paren{\etab} dl\paren{\etab}.
\end{align*}
Next, set $\mc{D}= \mcbr{\varepsilon}$,
\begin{align*}
    \int_\mbr \nabla G(A\bm{\eta})d\etab=\lim_{L\rightarrow\infty,\varepsilon\rightarrow 0} \int_{\mcbr{L}\cap \mcbr{\varepsilon}^c} \nabla G(A\bm{\eta})d\etab=0.
\end{align*}

\end{proof}

\begin{lemma}\label{Mtilde_bound}
Let $A$ be a matrix in the set $\mc{DA}_{\sigma_1,\sigma_2}$. Then, the linear operator $\tilde{\mc{M}}(A)$ \eqref{Mtilde} maps $C^\gamma(\mathbb{R}^2)\cap L^2(\mathbb{R}^2)$ to $C^\gamma(\mathbb{R}^2)\cap L^2(\mathbb{R}^2)$ for any any $\gamma\in(0,1)$. Moreover, 
\begin{equation*}
    \begin{aligned}
    \|\tilde{\mc{M}}^j(A)Z\|_{C^\gamma(\mathbb{R}^2)}&\leq \frac{C}{\sigma_2}\Big(1+\Big(\frac{\sigma_1}{\sigma_2}\Big)^2\Big)\|Z\|_{C^\gamma(\mathbb{R}^2)\cap L^2(\mathbb{R}^2)}.
    \end{aligned}
\end{equation*}
And given $A_1, A_2 \in \mc{DA}_{\sigma_1,\sigma_2}$
\begin{align*}
    \|\tilde{\mc{M}}^j(A_1)Z-\tilde{\mc{M}}^j(A_2)Z\|_{C^\gamma(\mathbb{R}^2)}&\leq \frac{C}{\sigma_2^2}\Big(1+\Big(\frac{\sigma_1}{\sigma_2}\Big)^5\Big)\|Z\|_{C^\gamma(\mathbb{R}^2)\cap L^2(\mathbb{R}^2)}\norm{A_1-A_2}.
\end{align*}
\end{lemma}

\begin{proof}
Taking into account Lemma \ref{Div_thm}, we have
\begin{equation*}
    \begin{aligned}
    (\tilde{\mc{M}}^j(A) Z)_k(\thetab)=-\int_{\mathbb{R}^2}\PD{}{\eta_m}G^j_{k,l}(A(\thetab-\etab))\big(Z_{l,m}(\etab)-\mc{C}_{l,m}\big)d\eta_1d\eta_2,
    \end{aligned}
\end{equation*}
where $C_{l,m}$ is an arbitrary constant, that we will take to be zero or $Z_{l,m}(\thetab)$.
Then, the estimate for $|(\tilde{\mc{M}}^j(A) Z)_k(\thetab)|$ follows by splitting the integral in two terms,
\begin{equation*}
    \begin{aligned}
    (\tilde{\mc{M}}^j(A) Z)_k(\thetab)=I_1(\thetab)+I_2(\thetab),
    \end{aligned}
\end{equation*}
with
\begin{equation*}
    \begin{aligned}
    I_1(\thetab)&=-\int_{|\thetab-\etab|\leq 1}\PD{}{\eta_m}G^j_{k,l}(A(\thetab-\etab))\big(Z_{l,m}(\etab)-Z_{l,m}(\thetab)\big)d\eta_1d\eta_2,\\
    I_2(\thetab)&=-\int_{|\thetab-\etab|\geq 1}\PD{}{\eta_m}G^j_{k,l}(A(\thetab-\etab))Z_{l,m}(\etab)d\eta_1d\eta_2.
    \end{aligned}
\end{equation*}
Since
\begin{align}
    \PD{}{\eta_m}G^j_{k,l}(A(\thetab-\etab))=-\PD{G^j_{k,l}}{x_i}(A(\thetab-\etab)) A_{i,m},\label{DGA_01}
\end{align}
the kernel bounds \eqref{qkernel_bounds} and the fact that $A\in\mc{D}\mc{A}_{\sigma_1,\sigma_2}$ provide that
\begin{equation*}
    \begin{aligned}
    |I_1(\thetab)|&\leq C\frac{\sigma_1}{\sigma_2^2}\jump{Z}_{C^\gamma(\mathbb{R}^2)},\\
    |I_2(\thetab)|&\leq C\frac{\sigma_1}{\sigma_2^2}\|Z\|_{L^2(\mathbb{R}^2)},
    \end{aligned}
\end{equation*}
hence
\begin{equation*}
    \begin{aligned}
    | (\tilde{\mc{M}}^j(A) Z)_k(\thetab)|\leq C\frac{\sigma_1}{\sigma_2^2}\|Z\|_{C^\gamma(\mathbb{R}^2)\cap L^2(\mathbb{R}^2)}.
    \end{aligned}
\end{equation*}
We proceed with the seminorm. Let $\hb\in\mathbb{R}^2$, $|\hb|\leq 1$, and perform the following splitting
\begin{equation*}
    \begin{aligned}
    (\tilde{\mc{M}}^j(A) Z)_k(\thetab)-(\tilde{\mc{M}}^j(A) Z)_k(\thetab+\hb)=J_1+J_2+J_3+J_4,
    \end{aligned}
\end{equation*}
where
\begin{align*}
    J_1&=-\int_{\abs{\thetab-\etab}\leq 2\abs{\hb}} \PD{}{\eta_m}G^j_{k,l}(A(\thetab-\etab))\delta_{\etab}Z_{l,m}(\thetab) d\etab,\\
    J_2&=\int_{\abs{\thetab-\etab}\leq 2\abs{\hb}} \PD{}{\eta_m}G^j_{k,l}(A(\thetab+\hb-\etab))\delta_{\etab}Z_{l,m}(\thetab+\hb) d\etab,\\
    J_3&=\delta_{\thetab}Z_{l,m}(\thetab+\hb)\int_{\abs{\thetab-\etab}> 2\abs{\hb}} \PD{}{\eta_m}G^j_{k,l}(\thetab-\etab)d\etab,\\
    J_4&=\int_{\abs{\thetab-\etab}> 2\abs{\hb}}\!\!\Big(\PD{}{\eta_m}G^j_{k,l}\paren{A(\thetab\!+\!\hb\!-\!\etab)}\!-\! \PD{}{\eta_m}G^j_{k,l}(A(\thetab\!-\!\etab))\Big)\delta_{\etab}Z_{l,m}(\thetab\!+\!\hb)d\etab.
\end{align*}
Absolutely, 
$$\abs{J_1}+\abs{J_2}\leq C\frac{\sigma_1}{\sigma_2^2}\jump{Z}_{C^{\gamma}\paren{\mbr}}\abs{\hb}^\gamma.$$
Then, by Lemma \ref{Div_thm},
\begin{align*}
    J_3&=(Z_{l,m}(\thetab+\hb)-Z_{l,m}(\thetab))\int_{\abs{\thetab-\etab}= 2\abs{\hb}} G^j_{k,l}\paren{A(\thetab-\etab)}n_m(\etab)dl(\etab),
\end{align*}
and so
\begin{align*}
    \abs{J_3}&\leq C\jump{Z}_{C^{\gamma}\paren{\mbr}}\abs{\hb}^\gamma\frac{1}{\sigma_2}\int_{\abs{\thetab-\etab}= 2\abs{\hb}} \frac{1}{\abs{\thetab-\etab}}d l\paren{\etab}\\
    &\leq C\frac{1}{\sigma_2}\jump{Z}_{C^{\gamma}\paren{\mbr}}\abs{\hb}^\gamma.
\end{align*}
Finally, since
\begin{align}
    \PD{}{\theta_p}\PD{}{\eta_m}G^j_{k,l}(A(\thetab-\etab))=-\PD{}{x_q}\PD{}{x_i}G^j_{k,l}(A(\thetab-\etab)) A_{i,m}A_{q,p},\label{DGA_02}
\end{align}
it follows that
\begin{align*}
\begin{split}
        \abs{J_4}
    &=\Big|\int_{\abs{\thetab-\etab}> 2\abs{\hb}}\int_0^1 h_p\PD{}{\theta_p}\PD{}{\eta_m}G^j_{k,l}\paren{A(\thetab+s\hb-\etab)}\delta_{\etab}(Z_{l,m}(\thetab+\hb)dsd\etab\Big|\\
    &\leq C\frac{\sigma_1^2}{\sigma_2^3}\jump{Z}_{C^{\gamma}\paren{\mbr}}\abs{\hb}\int_{\abs{\thetab-\etab}> 2\abs{\hb}}\int_0^1 \frac{\abs{\thetab+\hb-\etab}^\gamma}{\abs{\thetab+s\hb-\etab}^{3}} dsd\etab.
\end{split}
\end{align*}
In the domain where $\abs{\thetab-\etab}> 2\abs{\hb}$, it holds that for $s\in [0,1]$,
\begin{align*}
    \frac{1}{2}\abs{\thetab+\hb-\etab}\leq\abs{\thetab+s\hb-\etab}\leq \frac{3}{2}\abs{\thetab+\hb-\etab}.
\end{align*}
Hence,
\begin{align*}
    \abs{J_4}\leq C\frac{\sigma_1^2}{\sigma_2^3}\jump{Z}_{C^{\gamma}\paren{\mbr}}\abs{\hb}^\gamma.
\end{align*}
Therefore, we obtain 
\begin{align*}
    \norm{\tilde{\mc{M}}^j(A) Z}_{C^\gamma\paren{\mathbb{R}}}& \leq\frac{C}{\sigma_2}\Big(1+\frac{\sigma_1^2}{\sigma_2^2}\Big)\|Z\|_{C^\gamma(\mathbb{R}^2)\cap L^2(\mathbb{R}^2)}.
\end{align*}
For the $L^2$ norm, since $\xi_m\mc{F}_\thetab \left[G^j_{k,l}\paren{A\thetab}\right]\paren{\xib}$ is bounded by $\sigma_1$ and $\sigma_2$, we have
\begin{align}
\begin{split}
        \norm{(\tilde{\mc{M}}^j(A) Z)_k}_{L^2\paren{\mathbb{R}}}
    &=\norm{\xi_m\mc{F}_\thetab \left[G^j_{k,l}\paren{A\thetab}\right]\paren{\xib} \mc{F}[Z_{l,m}]\paren{\xib}}_{L^2\paren{\mathbb{R}}}\\
    &\leq C\paren{\sigma_1,\sigma_2}\norm{ \mc{F}[Z]}_{L^2\paren{\mathbb{R}}}
    =C\paren{\sigma_1,\sigma_2}\norm{ Z}_{L^2\paren{\mathbb{R}}}
\end{split}
\end{align}
Next, through \eqref{qkernel_bounds}, \eqref{DGA_01}, and \eqref{DGA_02}, we have
\begin{align*}
\begin{split}
        \abs{G^j_{k,l}(A_1(\thetab-\etab))\!-\!G^j_{k,l}(A_2(\thetab-\etab))}
        &\leq C\frac{\sigma_1 \abs{\paren{A_1\!-\!A_2}(\thetab-\etab)}}{\sigma_2^3 \abs{\thetab-\etab}^3}\\
        &\leq C\frac{\sigma_1 }{\sigma_2^3 }\frac{\norm{\paren{A_1\!-\!A_2}}}{\abs{\thetab-\etab}},
\end{split}
\end{align*}
\begin{align*}
\begin{split}
        \Big|\PD{}{\eta_m}G^j_{k,l}(A_1(\thetab\!-\!\etab))&\!-\!\PD{}{\eta_m}G^j_{k,l}(A_2(\thetab\!-\!\etab))\Big|\\
    &\leq\Big|\PD{G^j_{k,l}}{x_i}(A_1(\thetab\!-\!\etab))\!-\!\PD{G^j_{k,l}}{x_i}(A_2(\thetab\!-\!\etab))\Big| |A_{1,i,m}|\\
        &\quad+\Big|\PD{G^j_{k,l}}{x_i}(A_2(\thetab\!-\!\etab))\paren{ A_{1,i,m}\!-\!A_{2,i,m}}\Big|\\
    &\leq C\Big(\frac{1}{\sigma_2^2}\!+\!\frac{\sigma_1^3}{\sigma_2^5 }\Big)\frac{\norm{\paren{A_1-A_2}}}{\abs{\thetab-\etab}^2},
\end{split}
\end{align*}
and
\begin{align*}
\begin{split}
      \Big|\PD{}{\theta_p}&\PD{}{\eta_m}G^j_{k,l}(A_1(\thetab\!-\!\etab))\!-\!\PD{}{\theta_p}\PD{}{\eta_m}G^j_{k,l}(A_2(\thetab\!-\!\etab))\Big|\\
    &\leq\Big|\PD{}{x_q}\PD{}{x_i}G^j_{k,l}(A_1(\thetab\!-\!\etab))\!-\!\PD{}{x_q}\PD{}{x_i}G^j_{k,l}(A_2(\thetab\!-\!\etab))\Big||A_{1,i,m}A_{1,q,p}|\\
        &\quad+\Big|\PD{}{x_q}\PD{}{x_i}G^j_{k,l}(A_2(\thetab\!-\!\etab))\paren{ A_{1,i,m}A_{1,q,p}\!-\!A_{2,i,m}A_{2,q,p}}\Big|\\
    &\leq C\Big(\frac{\sigma_1}{\sigma_2^3}\!+\!\frac{\sigma_1^5 }{\sigma_2^7 }\Big)\frac{\norm{\paren{A_1\!-\!A_2}}}{\abs{\thetab-\etab}^3}.
\end{split}
\end{align*}
Hence, we obtain
\begin{align*}
    \|\tilde{\mc{M}}^j(A_1)Z\!-\!\tilde{\mc{M}}^j(A_2)Z\|_{C^\gamma(\mathbb{R}^2)}&\leq \frac{C}{\sigma_2^2}\Big(1\!+\!\Big(\frac{\sigma_1}{\sigma_2}\Big)^5\Big)\|Z\|_{C^\gamma(\mathbb{R}^2)\cap L^2(\mathbb{R}^2)}\norm{A_1\!-\!A_2}.
\end{align*}

\end{proof}
As an immediate consequence of the previous lemma with
$\tilde{Z}_{l,m}=\PD{\wh{X}_i}{\eta_m}(\etab)Z_{l,i}(\etab),$
we obtain the following lemma for $\mc{M}(A)$:

\begin{lemma}\label{MA_bound}
Let $A$ be a matrix in the set $\mc{DA}_{\sigma_1,\sigma_2}$. Then, the linear operators $\mc{M}^j(A)$ \eqref{defM} map $C^\gamma(\mathbb{S}^2)$ to $C^\gamma(\mathbb{R}^2)$ for any any $\gamma\in(0,1).$ Moreover, 
\begin{equation*}
    \begin{aligned}
    \|\mc{M}^j(A)Z\|_{C^\gamma(\mathbb{R}^2)}&\leq \frac{C}{\sigma_2}\Big(1+\Big(\frac{\sigma_1}{\sigma_2}\Big)^2\Big)\sup_{l,m}\|\PD{\hX}{\eta_m}(\etab)\cdot \rowofmat{Z}{l}(\etab)\|_{C^\gamma(\mathbb{R}^2)\cap L^1(\mathbb{R}^2)}\\
    &\leq C(\sigma_1,\sigma_2)\|Z\|_{C^\gamma(\mathbb{S}^2)}.
    \end{aligned}
\end{equation*}
\end{lemma}

\begin{remark} In most cases, Proposition \ref{MA_bound} will be used with $Z$ compactly supported and given by a multiple of a gradient. Notice that in that case, for $Z=\lambda\nabla_{\mathbb{S}^2}\bm{X}$, $\lambda:\mathbb{R}^2\mapsto\mathbb{R}$, 
\begin{equation*}
    \begin{aligned}
    \PD{\hX}{\eta_m}(\etab)\cdot \rowofmat{Z}{l}(\etab)=\lambda(\etab)\PD{X_l}{\eta_m}(\etab),
    \end{aligned}
\end{equation*}
and therefore
\begin{equation*}
    \begin{aligned}
    \|\mc{M}^j(A)(\lambda\nabla_{\mathbb{S}^2}\bm{X})\|_{C^\gamma(\mathbb{R}^2)}&\leq C(\sigma_1,\sigma_2)\|\lambda\nabla\bm{X}\|_{C^\gamma(\mathbb{R}^2)}.
    \end{aligned}
\end{equation*}
\end{remark}

\begin{lemma}\label{Nbound_lem} Let $\bm{X}\in C^1(\mathbb{S}^2)$ such that $|\bm{X}|_*>0$. Then, the linear operators $\mc{N}^j(\bm{X})$ \eqref{Ndef} map $C^{\gamma}(\mathbb{S}^2)$ to $C^\gamma(\mathbb{S}^2)$ for any $\gamma\in(0,1)$. Moreover, 
\begin{equation}\label{Nbound}
    \begin{aligned}
    \|\mc{N}^j(\bm{X})Z\|_{C^\gamma(\mathbb{S}^2)}&\leq \frac{C}{|\bm{X}|_*}\Big(1+\Big(\frac{\|\nabla_{\mathbb{S}^2}\bm{X}\|_{C^0(\mathbb{S}^2)}}{|\bm{X}_*|}\Big)^2\Big)\|Z\|_{C^\gamma(\mathbb{S}^2)}.
    \end{aligned}
\end{equation}
\end{lemma}

\begin{proof}
We first notice that we can introduce an arbitrary constant (matrix),
\begin{equation}\label{N_C}
\begin{aligned}
\mc{N}(\bm{X})Z(\hx)&=-\int_{\mathbb{S}^2} \nabla_{\mathbb{S}^2}G(\bm{X}(\hx)-\bm{X}(\hy))\cdot(Z(\hy)-\mc{C})d\hy.
\end{aligned}
\end{equation}
We will usually take $\mc{C}=0$ or $\mc{C}=Z(\hx)$. Recalling the kernels \eqref{q_kernels} and \eqref{Ndef_q}, we write the equation for each component
\begin{equation*}
\begin{aligned}
(\mc{N}^j(\bm{X})Z)_k(\hx)&=-\int_{\mathbb{S}^2} q_{k,l}^j(\hx,\hy)\cdot(\rowofmat{Z}{l}(\hy)-\rowofmat{\mc{C}}{l})d\hy,
\end{aligned}
\end{equation*}
where $\rowofmat{\mc{C}}{l}=\bm{0}$ or $\rowofmat{\mc{C}}{l}=\rowofmat{Z}{l}(\hx)$. 
We first perform the estimate for $|\mc{N}(\bm{X})Z(\hx)|$, $\hx\in\mathbb{S}^2$,
\begin{equation}\label{NC0bound}
\begin{aligned}
    |\mc{N}(\bm{X})Z(\hx)|&\leq \sum_{j=1}^2|\mc{N}^j(\bm{X})Z(\hx)|.
\end{aligned}
\end{equation}
Using the bound \eqref{qkernel_bounds} for the kernel, we have
\begin{equation}\label{NLinf_bound}
    \begin{aligned}
    |\mc{N}(\bm{X})Z(\hx)|&\leq C\frac{\|\nabla_{\mathbb{S}^2}\bm{X}\|_{C^0(\mathbb{S}^2)}}{|\bm{X}|_*^2}\int_{\mathbb{S}^2}\frac{|\rowofmat{Z}{l}(\hx)-\rowofmat{Z}{l}(\hy)|}{|\hx-\hy|^2}d\hy\\
    &\leq C\frac{\|\nabla_{\mathbb{S}^2}\bm{X}\|_{C^0(\mathbb{S}^2)}}{|\bm{X}|_*^2}\|Z\|_{C^\gamma(\mathbb{S}^2)}.
    \end{aligned}
\end{equation}
We proceed to estimate the H\"older seminorm. Let $\hx, \hxh\in\mathbb{S}^2$, and denote $h=|\hx-\hxh|$. We write
\begin{equation*}
\begin{aligned}
 (\mc{N}^j(\bm{X})Z)_k(\hx)\!-\!(\mc{N}^j(\bm{X})Z)_k(\hxh)&=\int_{\mathbb{S}^2}\! q_{k,l}^j(\hx, \hy)\cdot(\rowofmat{Z}{l}(\hx)\!-\!\rowofmat{Z}{l}(\hy))d\hy\\
 &\quad-\!\int_{\mathbb{S}^2} \!\!q_{k,l}^j(\hxh, \hy)\!\cdot\!(\rowofmat{Z}{l}(\hxh)\!-\!\rowofmat{Z}{l}(\hy))d\hy,
\end{aligned}
\end{equation*}
and perform the following splitting
\begin{equation}\label{Nholder_split}
    \begin{aligned}
    (\mc{N}^j(\bm{X})Z)_k(\hx)-(\mc{N}^j(\bm{X})Z)_k(\hxh)=I^1+I^2+I^3+I^4,
    \end{aligned}
\end{equation}
where
\begin{equation*}
    I^1=\int_{\{|\hx-\hy|\leq 2h\}\cap \mathbb{S}^2} q_{k,l}^j(\hx, \hy)\cdot(\rowofmat{Z}{l}(\hx)-\rowofmat{Z}{l}(\hy))d\hy,
\end{equation*}
\begin{equation*}
I^2=-\int_{\{\hx-\hy|\leq 2h\}\cap \mathbb{S}^2} q_{k,l}^j(\hxh, \hy)\cdot(\rowofmat{Z}{l}(\hxh)-\rowofmat{Z}{l}(\hy))d\hy,
\end{equation*}
\begin{equation*}
    I^3=(\rowofmat{Z}{l}(\hx)-\rowofmat{Z}{l}(\hxh))\cdot\int_{\{|\hx-\hy|\geq 2h\}\cap \mathbb{S}^2} q_{k,l}^j(\hx, \hy)d\hy,
\end{equation*}
\begin{equation*}
\begin{aligned}
    I^4=\int_{\{|\hx-\hy|\geq 2h\}\cap \mathbb{S}^2} &(q_{k,l}^j(\hx, \hy)-q_{k,l}^j(\hxh, \hy))\cdot(\rowofmat{Z}{l}(\hxh)-\rowofmat{Z}{l}(\hy))d\hy.
\end{aligned}
\end{equation*}
The first two terms are estimated directly
\begin{equation}\label{I1I2bound}
    \begin{aligned}
    |I^1|+|I^2|&\leq C\frac{\|\nabla_{\mathbb{S}^2}\bm{X}\|_{C^0(\mathbb{S}^2)}}{|\bm{X}|_*^2}\|Z\|_{C^\gamma(\mathbb{S}^2)}h^\gamma.
    \end{aligned}
\end{equation}
For the third, we use that the kernel is a derivative to integrate by parts and obtain that
\begin{equation}\label{I3bound}
    \begin{aligned}
    |I^3|&=|(\rowofmat{Z}{l}(\hx)-\rowofmat{Z}{l}(\hxh))\cdot\int_{\{|\hx-\hy|=2h\}\cap \mathbb{S}^2}G^j(\bm{X}(\hx)-\bm{X}(\hy))\bm{n}(\hy) dl_{\mathbb{S}^2}(\hy)|\\
    &\leq C\frac{\|Z\|_{C^\gamma(\mathbb{S}^2)}}{|\bm{X}|_*}|h|^\gamma.
    \end{aligned}
\end{equation}
Finally, we use the mean-value theorem on the kernel to estimate $I^4$. Set $\ell(s)$ the shortest path function from $\hxh$ to $\hx$ respect to arc-length $s$ variable, and $L=\text{dist}\paren{\hxh,\hx; \mbs}$.
Then, we have
\begin{equation*}
    \begin{aligned}
        G^j_{k,l}(\bm{X}(\hx)\!-\!\bm{X}(\bm{\hy}))- G^j_{k,l}&(\bm{X}(\hxh)\!-\!\bm{X}(\bm{\hy}))=\int_0^L \PD{}{s}G^j_{k,l}(\bm{X}(\ell\paren{s})\!-\!\bm{X}(\bm{\hy}))ds\\
    &=\int_0^L \nabla_\mbs G^j_{k,l}(\bm{X}(\ell\paren{s})\!-\!\bm{X}(\bm{\hy}))\cdot \PD{}{s}\ell\paren{s} ds.        
    \end{aligned}
\end{equation*}
Hence, for $q^j_{k,l}$ \eqref{q_kernels},
\begin{equation*}
\begin{aligned}
        |&q^j_{k,l}(\hx,\hy)-q^j_{k,l}(\hxh,\hy)|=   \Big|\nabla_\mbs\int_0^L \nabla_\mbs G_{k,l}(\bm{X}(\ell\paren{s})\!-\!\bm{X}(\bm{\hy}))\cdot \PD{}{s}\ell\paren{s} ds\Big|\\
    =   &\Big|\nabla_\mbs\int_0^L \PD{}{x_i} G_{k,l}(\bm{X}(\ell\paren{s})\!-\!\bm{X}(\bm{\hy}))\paren{\nabla_\mbs X_i\paren{\ell\paren{s}}\cdot \PD{}{s}\ell\paren{s}} ds\Big|\\
    =   &\Big|\int_0^L \!\!\!\PD{}{x_j} \PD{}{x_i} G_{k,l}(\bm{X}(\ell\paren{s})\!-\!\bm{X}(\bm{\hy}))\paren{\nabla_\mbs X_i\paren{\ell\paren{s}}\!\cdot\! \PD{}{s}\ell\paren{s}\!} \nabla_\mbs X_j\paren{\hy}ds\Big|,
    \end{aligned}
\end{equation*}
and recalling \eqref{qkernel_bounds} we obtain the bound
\begin{equation*}
        \begin{aligned}
     |q^j_{k,l}(\hx,\hy)-q^j_{k,l}(\hxh,\hy)| \leq\minspace & C\frac{\norm{ \nabla_\mbs \bm{X}}_{C^0\paren{\mbs}}^2}{\starnorm{\bm{X}}^3}\int_0^L \frac{1}{\abs{\ell\paren{s}-\hy}^3}ds.        
        \end{aligned}
    \end{equation*}
Then, we have that
\begin{equation*}
    \begin{aligned}
        |I_4|&\leq C\frac{\|\nabla_{\mathbb{S}^2}\bm{X}\|_{C^0(\mathbb{S}^2)}^2\|Z\|_{C^\gamma(\mathbb{S}^2)}}{\starnorm{\bm{X}}^3}\int_{\{\abs{\hx-\hy}\geq 2h\}\cap\mbs}\!\!\!\!\!\!\!\!\!\!\!\!\!\!\!\!|\hxh\!-\!\hy|^\gamma\int_0^L\!\! \frac{ds}{\abs{\ell\paren{s}\!-\!\hy}^3}d\hy.    
    \end{aligned}
\end{equation*}
We notice that since $\ell(s)$ is the shortest path function from $\hxh$ to $\hx$ on $\mbs$, it holds that $\abs{\ell\paren{s}-\hx}\leq \abs{\hx-\hxh}$.
Thus,
\begin{align*}
\begin{split}
    \abs{\ell\paren{s}-\hy}
    \geq\abs{\hx-\hy} -\abs{\ell\paren{s}-\hx}
    \geq\abs{\hx-\hy} -\abs{\hxh-\hx}
    \geq\frac{1}{2}\abs{\hx-\hy}.
\end{split}
\end{align*}
In addition, 
\begin{align*}
\begin{split}
    \abs{\hxh-\hy}\leq \abs{\hx-\hy}+\abs{\hxh-\hx}\leq\frac{3}{2}\abs{\hx-\hy}.
\end{split}
\end{align*}
Finally, since $L\leq C h$, we conclude that
\begin{equation}\label{I4bound}
    \begin{aligned}
        |I_4|&\leq C\frac{\|\nabla_{\mathbb{S}^2}\bm{X}\|_{C^0(\mathbb{S}^2)}^2\|Z\|_{C^\gamma(\mathbb{S}^2)}}{\starnorm{\bm{X}}^3}h\int_{\{\hx-\hy|\geq2h\}\cap \mathbb{S}^2}\frac{d\hy}{|\hx-\hy|^{3-\gamma}}\\
        &\leq C\frac{\|\nabla_{\mathbb{S}^2}\bm{X}\|_{C^0(\mathbb{S}^2)}^2\|Z\|_{C^\gamma(\mathbb{S}^2)}}{\starnorm{\bm{X}}^3}h^\gamma.
    \end{aligned}
\end{equation}
Joining the bounds \eqref{I1I2bound}, \eqref{I3bound}, and \eqref{I4bound} back in \eqref{Nholder_split}, we conclude that
\begin{equation}\label{NHolderbound}
    \begin{aligned}
    [\mc{N}(\bm{X})Z]_{C^\gamma(\mathbb{S}^2)}&\leq \frac{C}{|\bm{X}|_*}\Big(1+\Big(\frac{\|\nabla_{\mathbb{S}^2}\bm{X}\|_{C^0(\mathbb{S}^2)}}{|\bm{X}_*|}\Big)^2\Big)\|Z\|_{C^\gamma(\mathbb{S}^2)},
    \end{aligned}
\end{equation}
and, with \eqref{NC0bound}, the same bound holds for the H\"older norm.

\end{proof}

We will need to localize the operators above. For that purpose, let us define the cutoff function $\wh{\rho}_n$, 
\begin{equation}\label{cutoff}
    \wh{\rho}_n(\hx)=\left\{\begin{aligned}
    &1 \quad\text{ if } \quad|\hx-\hx_n|\leq 3R,\\
    &0 \quad\text{ if }\quad |\hx-\hx_n|\geq 4R,
    \end{aligned}\right. 
\end{equation} 
and recall the partition of unity $\{\rho_n\}$ based on the points $\hx_n$ (see Subsection \ref{covering_ste}).

\begin{lemma}\label{Nout_bound}
Let $\bm{X}\in C^1(\mathbb{S}^2)$ such that $|\bm{X}|_*>0$, $\mc{N}(\bm{X})$ the linear operator defined by \eqref{Ndef}, and $\wh{\rho}_n$ the cutoff function \eqref{cutoff}. Then, for $Z\in C^0(\mathbb{S}^2)$ compactly supported on $B_{\hx_n,2R}\cap\mathbb{S}^2$, it holds that,
\begin{equation*}
    \begin{aligned}
    \|(1-\wh{\rho}_n)\mc{N}^j(\bm{X})Z\|_{C^1(\mathbb{S}^2)}\leq C(R,|\bm{X}|_*,\|\nabla_{\mathbb{S}^2}\bm{X}\|_{C^0(\mathbb{S}^2)})\|Z\|_{C^0(\mathbb{S}^2)}.
    \end{aligned}
\end{equation*}
\end{lemma}

\begin{proof}
Let
\begin{equation*}
    I(\hx)=(1-\wh{\rho}_n(\hx))\mc{N}^j(\bm{X})Z(\hx).
\end{equation*}
Since $1-\wh{\rho}(\hx)=0$ when $\hx\in B_{\hx_n,3R}$, let $\hx\in\mathbb{S}^2\setminus B_{\hx_n,3R}$. Then, recalling the condition on the support of $Z$, 
\begin{equation*}
    \begin{aligned}
    I(\hx)&\!=\!(\wh{\rho}_n(\hx)\!-\!1)\!\int_{B_{\hx_n,2R}\cap \mathbb{S}^2}\!\!\!\!\!\!\!\!\!\!\!\!\!\!\!\!\!\! \nabla_{\mathbb{S}^2}G(\bm{X}(\hx)\!-\!\bm{X}(\hy))\!\cdot\!Z(\hy)d\hy,
    \end{aligned}
\end{equation*}
and using the bound \eqref{qkernel_bounds} for the kernel,
\begin{equation*}
    \begin{aligned}
    |I(\hx)|&\leq C\frac{\|\nabla_{\mathbb{S}^2}\bm{X}\|_{C^0(\mathbb{S}^2)}}{|\bm{X}|_*^2}\|Z\|_{C^0(\mathbb{S}^2)}\!\int_{B_{\hx_n,2R}\cap \mathbb{S}^2}\!\!\!\!\!\!\!\!\!\!\!\!\!\!\!\!\!\!\!|\hx\!-\!\hy|^{-2}d\hy.
    \end{aligned}
\end{equation*}
Since  we have that $|\hx-\hy|\geq R$, we obtain
\begin{equation}\label{In12C0bound}
    \begin{aligned}
    |I(\hx)|&\leq C(|\bm{X}|_*,\|\nabla_{\mathbb{S}^2}\bm{X}\|_{C^0(\mathbb{S}^2)})\|Z\|_{C^0(\mathbb{S}^2)}.
    \end{aligned}
\end{equation}
To estimate the H\"older seminorm, consider two points $\hx, \hxh\in\mathbb{S}^2$, $h=|\hx-\hxh|$. Due to the cut-off function $\wh{\rho}_n$,  the only non-trivial case is $\hx, \hxh\in \mathbb{S}^2\setminus B_{\hx_n,3R}$:
\begin{equation*}
    \begin{aligned}
    |I(\hxh)-I(\hx)|&=(\wh{\rho}_n(\hx)-\wh{\rho}_n(\hxh))\int_{B_{\hx_n,2R}\cap \mathbb{S}^2} \!\!\!\!\!q_{k,l}(\hx,\hy)\cdot\rowofmat{Z}{l}(\bm{\hy})d\hy\\
    &\quad+(1\!-\!\wh{\rho}_n(\hxh))\!\int_{B_{\hx_n,2R}\cap \mathbb{S}^2}\!\!\!\!\!\!\!\!\!\!\!\!\!\big(q_{k,l}(\hx,\hy)\!-\!q_{k,l}(\hxh,\hy)\big)\cdot\rowofmat{Z}{l}(\hy)d\hy\\
    &=J_1+J_2.
    \end{aligned}
\end{equation*}
The first term is bounded as \eqref{In12C0bound},
\begin{equation*}
    \begin{aligned}
    |J_1|&\leq C\|\wh{\rho}_n\|_{C^1(\mathbb{S}^2)}\frac{\|\nabla_{\mathbb{S}^2}\bm{X}\|_{C^0(\mathbb{S}^2)}}{|\bm{X}|_*^2}\|Z\|_{C^0(\mathbb{S}^2)}h.
    \end{aligned}
\end{equation*}
Recalling the expression of $q_{k,l}$ \eqref{q_kernels}, we can check that, since $|\hx-\hy|\geq R$, $|\hxh-\hy|\geq R$, 
\begin{equation*}
    \begin{aligned}
    |q_{k,l}(\hx,\hy)-q_{k,l}(\hxh,\hy)|&\leq C\frac{\|\nabla_{\mathbb{S}^2}\bm{X}\|_{C^0(\mathbb{S}^2)}^2}{|\bm{X}|_*^3R^3}h,
    \end{aligned}
\end{equation*}
hence
\begin{equation*}
    \begin{aligned}
    |J_2|&\leq \frac{C}{R}\frac{\|\nabla_{\mathbb{S}^2}\bm{X}\|_{C^0(\mathbb{S}^2)}^2}{|\bm{X}|_*^3}\|Z\|_{C^0(\mathbb{S}^2)}h.
    \end{aligned}
\end{equation*}
Therefore, 
\begin{equation}\label{In21bound}
    \begin{aligned}
    \|I\|_{C^1(\mathbb{S}^2)}&\leq C(R,|\bm{X}|_*,\|\nabla_{\mathbb{S}^2}\bm{X}\|_{C^0(\mathbb{S}^2)}\|_{C^1})\|Z\|_{C^0(\mathbb{S}^2)}.
    \end{aligned}
\end{equation}

\end{proof}

The previous lemma  holds analogously for the operators $\tilde{M}(A)$:
\begin{lemma}\label{MAout_bound}
Let $A\in\mc{D}\mc{A}_{\sigma_1,\sigma_2}$, $\tilde{\mc{M}}^j(A)$ the linear operator defined by \eqref{Mtilde}, and $\wh{\rho}_n$ the cutoff function \eqref{cutoff}. Then, for $Z\in C^0(\mathbb{R}^2)$ compactly supported on $V_{2R}$, it holds that,
\begin{equation*}
    \begin{aligned}
    \|(1-\wh{\rho}_n)\tilde{\mc{M}}^j(A)Z\|_{C^1(\mathbb{R}^2)}\leq C(R,\sigma_1,\sigma_2)\|Z\|_{C^0(\mathbb{R}^2)}.
    \end{aligned}
\end{equation*}
\end{lemma}

\begin{lemma}\label{lem_commutator} Let $\hx_n\in\mathbb{S}^2$, $\bm{X}\in C^1(\mathbb{S}^2)$, $|\bm{X}|_*>0$, and $\rho_n\in C^\infty(B_{\hx_n,2R}\cap\mathbb{S}^2)$, $R<1/10$. Then, the commutator
\begin{equation*}
    \begin{aligned}
    [\rho_n,&\mc{N}(\bm{X})]Z(\hx)=-\!\int_{\mbs}\!\! \nabla_\mbs G(\bm{X}(\hx)\!-\!\bm{X}(\bm{\hy}))\!\cdot Z(\hy)(\rho_n(\hy)\!-\!\rho_n(\hx))d\hy,
    \end{aligned}
\end{equation*}
 satisfies that, for any $\gamma\in(0,1)$, 
\begin{equation}\label{In2bound}
    \begin{aligned}
    \norm{[\rho_n,\mc{N}(\bm{X})]Z}_{C^\gamma\paren{\mathbb{S}^2}}\leq C(|\bm{X}_*,\|\nabla_\mbs \bm{X}\|_{C^0(\mbs)})\|\nabla_\mbs \rho_n\|_{C^0(\mbs)}\|Z\|_{C^0\paren{\mbs}}.
    \end{aligned}
\end{equation}
\end{lemma}

\begin{proof}
Recalling the kernel bound \eqref{qkernel_bounds},
\begin{align*}
\begin{split}
       |[\rho_n,\mc{N}(\bm{X})]Z(\hx)|&\leq C\frac{\|\nabla_\mbs \rho_n\|_{C^0(\mbs)}\|Z\|_{C^0(\mbs)}\| \nabla_\mbs \bm{X}\|_{C^0(\mbs)}}{|\bm{X}|_*^2}\int_{\mbs}\frac{1}{|\hx-\hy|} d\hy\\
    &\leq C(|\bm{X}_*,\| \nabla_\mbs \bm{X}|_{C^0(\mbs)})\|\nabla_\mbs \rho_n\|_{C^0(\mbs)}\|Z\|_{C^0(\mbs)}.
\end{split}
\end{align*}
Next, we study the H\"older seminorm. We take two points $\hx$, $\hxh$, denote $h=|\hx-\hxh|$, and perform a splitting analogous to \eqref{Nholder_split}. Using the kernel notation \eqref{q_kernels},
\begin{align*}
\begin{split}
       [\rho_n,\mc{N}(\bm{X})]Z(\hx)\!-\![\rho_n,\mc{N}(\bm{X})]Z(\hxh)&\!=\! \int_{\mbs}q_{k,l}(\hx,\hy)\paren{\rho_n(\hy)\!-\!\rho_n(\hx)}\!\cdot\! \rowofmat{Z}{l}(\hy)d\hy\\
        &\quad-\!\int_{\mbs}\!\!q_{k,l}(\hxh,\hy)\paren{\rho_n(\hy)\!-\!\rho_n(\hxh)}\!\cdot\! \rowofmat{Z}{l}(\hy)d\hy\\
  &  = I_n^{1}+I_n^{2}+I_n^{3}+I_n^{4},
\end{split}
\end{align*}
where
\begin{align*}
    I_n^{1}=&\int_{\{\abs{\hx-\hy}\leq 2h\}\cap\mbs}q_{k,l}(\hx,\hy)\paren{\rho_n(\hy)-\rho_n(\hx)}\cdot \rowofmat{Z}{l}(\hy)d\hy\\
    I_n^{2}=&-\int_{\{\abs{\hx-\hy}\leq 2h\}\cap\mbs}q_{k,l}(\hxh,\hy)\paren{\rho_n(\hy)-\rho_n(\hxh)}\cdot \rowofmat{Z}{l}(\hy)d\hy\\
    I_n^{3}=&\int_{\{\abs{\hx-\hy}\geq 2h\}\cap\mbs}q_{k,l}(\hx,\hy)\paren{\rho_n(\hxh)-\rho_n(\hx)}\cdot \rowofmat{Z}{l}(\hy)d\hy
\end{align*}
and
\begin{align*}
\begin{split}
        I_n^{4}=\int_{\{\abs{\hx-\hy}\geq 2h\}\cap\mbs}\hspace{-1.8cm}(q_{k,l}(\hx,\hy)\!-\!q_{k,l}(\hxh,\hy))\paren{\rho_n(\hy)\!-\!\rho_n(\hxh)}\cdot \rowofmat{Z}{l}(\hy)d\hy.
\end{split}
\end{align*}
Then, for $I_n^{1}$ and $I_n^{2}$,
\begin{align*}
\begin{split}
    |I_n^{1}|+|I_n^{2}|
    &\leq  C\frac{\norm{ \nabla_\mbs \rho_n}_{C^0\paren{\mbs}}\norm{ \nabla_\mbs \bm{X}}_{C^0\paren{\mbs}}}{\starnorm{\bm{X}}^2}\|Z\|_{C^0\paren{\mbs}}h.
\end{split}
\end{align*}
Next, for $I_n^{3}$, integration by parts gives that
\begin{align*}
\begin{split}
    \abs{I_n^{3}}
    &\leq C\abs{\rho_n(\hxh)\!-\!\rho_n(\hx)}\int_{\{\abs{\hx-\hy}\geq 2h\}\cap\mbs}\!\!\!\!\!\!\!\!\!\!\frac{\|Z\|_{C^0\paren{\mbs}}\norm{ \nabla_\mbs \bm{X}}_{C^0\paren{\mbs}}}{\starnorm{\bm{X}}^2\abs{\hx-\hy}^2}d\hy\\
    &\leq C(|\bm{X}_*,\|\nabla_\mbs \bm{X}\|_{C^0(\mbs)})\|\nabla_\mbs \rho_n\|_{C^0(\mbs)}\|Z\|_{C^0\paren{\mbs}}h\log h^{-1}.
\end{split}
\end{align*}
Finally, in $I_n^{4}$, the use of the mean-value theorem provides that
\begin{align*}
\begin{split}
        \abs{I_n^{4}}&
    \leq C\frac{\norm{ \nabla_\mbs \rho_n}_{C^0\paren{\mbs}}\|Z\|_{C^0\paren{\mbs}}\norm{ \nabla_\mbs \bm{X}}_{C^0\paren{\mbs}}^2}{\starnorm{\bm{X}}^3}\\
    &\hspace{2cm}\times\int_{\{\abs{\hx-\hy}\leq 2h\}\cap\mbs}\int_0^L \frac{\abs{\hxh-\hy}}{\abs{\ell\paren{s}-\hy}^3}ds d\hy\\
    &\leq C(|\bm{X}_*,\norm{ \nabla_\mbs \bm{X}}_{C^0\paren{\mbs}})\norm{ \nabla_\mbs \rho_n}_{C^0\paren{\mbs}}\norm{Z}_{C^0\paren{\mbs}}h.
\end{split}
\end{align*}
Thus,
\begin{align*}
    \norm{[\rho_n,\mc{N}(\bm{X})]Z}_{C^\gamma\paren{\mathbb{S}^2}}\leq C(|\bm{X}_*,\|\nabla_\mbs \bm{X}\|_{C^0(\mbs)})\|\nabla_\mbs \rho_n\|_{C^0(\mbs)}\|Z\|_{C^0\paren{\mbs}}.
\end{align*}

\end{proof}

\begin{lemma}\label{lem_dif} Let $\hx_n\in\mathbb{S}^2$ and $\widehat{\bm{X}}_n:\mathbb{R}^2\cup \{\infty\}\to\mathbb{S}^2$ the stereographic projection centered at $\hx_n$.
Let $\bm{X}\in C^{1,\gamma}(\mathbb{S}^2)$,  $|\bm{X}|_*>0$, and $\wh{\rho}_n$ given by \eqref{cutoff}. Let the linear operators $\mc{N}(\bm{X})$ and $\mc{M}(A)$ be defined by in \eqref{Ndef} and \eqref{defM}, with $A=\nabla\bm{X}_n(\bm{0})$, where we are using the notation $\bm{X}_n(\thetab)=\bm{X}(\widehat{\bm{X}}_n(\thetab))$.
Denote
\begin{equation*}
    \begin{aligned}
    I_n(\thetab)&=\wh{\rho}_n(\thetab)[\mc{M}(A)-\mc{N}(\bm{X})]Z_n(\thetab),
    \end{aligned}
\end{equation*}
Then, for $Z\in C^\gamma(\mathbb{S}^2)$ compactly supported on $B_{\hx_n,2R}\cap\mathbb{S}^2$, the following estimates hold:
\begin{equation}\label{Jn8bound}
    \begin{aligned}
    \|I_n\|_{C^\gamma(\mathbb{R}^2)}&\leq C(|\bm{X}|_*,\|\nabla_{\mathbb{S}^2}\bm{X}\|_{C^0(\mathbb{S}^2)})\big((1+\|\nabla_{\mathbb{S}^2}\bm{X}\|_{C^\gamma(B_{\hx_n,5R}\cap\mathbb{S}^2)})\|Z\|_{C^0(\mathbb{S}^2)}\\
    &\quad+\varepsilon(R)\|Z\|_{C^\gamma(\mathbb{S}^2)}\big),
    \end{aligned}
\end{equation}
and
\begin{equation}\label{Jn8bound2}
    \begin{aligned}
    \|I_n\|_{C^\gamma(\mathbb{R}^2)}&\leq C(|\bm{X}|_*,\|\nabla_{\mathbb{S}^2}\bm{X}\|_{C^0(\mathbb{S}^2)})\big(\|Z\|_{C^0(\mathbb{S}^2)}\!+\!\|\nabla_{\mathbb{S}^2}\bm{X}\|_{C^{\frac{\gamma}{2}}(B_{\hx_n,5R}\cap\mathbb{S}^2)}\|Z\|_{C^{\frac{\gamma}{2}}(\mathbb{S}^2)}\\
    &\quad+\varepsilon(R)\|Z\|_{C^\gamma(\mathbb{S}^2)}\big),
    \end{aligned}
\end{equation}
with $\varepsilon(R)\to 0$ as $R\to 0$ given by the modulus of continuity of $\nabla_{\mathbb{S}^2}\bm{X}$.
\end{lemma}

\begin{proof}
First, we write
\begin{equation}\label{Jn8split}
    \begin{aligned}
    I_n&=I_n^1+I_n^2\\
    &=\sum_{j=1}^2\wh{\rho}_n(\thetab)[\mc{M}^j(A)-\mc{N}^j(\bm{X})]Z_n(\thetab)\\
    &=\sum_{j=1}^2\wh{\rho}_n(\thetab)[\mc{M}^j(A)-\mc{M}^j(\nabla\bm{X}_n)]Z_n(\thetab)-\mc{R}^j(\bm{X}_n)Z_n(\thetab),
    \end{aligned}
\end{equation}
and focus on the term $I_n^1$, as the other $I_n^2$ will follow similarly. 
We then write
\begin{equation*}
    \begin{aligned}
    I_n^1&=-\wh{\rho}_n(\thetab)\int_{\mathbb{R}^2} P_{m,k,l}^1(\thetab, \etab)\PD{\wh{X}_i}{\eta_m}(\etab) Z_{n,li}(\etab)d\eta_1d\eta_2,
    \end{aligned}
\end{equation*}
where we used the notation \eqref{def_VR} and
\begin{equation}\label{mcP}
    \begin{aligned}
    P_{m,k,l}^1(\thetab,\etab)
    =&\frac{\delta_{k,l}}{8\pi}\PD{}{\eta_m}\Big(\frac{1}{|A(\thetab-\etab)|}-\frac{1}{|\bm{X}_n(\thetab)-\bm{X}_n(\etab)|}\Big)\\
    =&-\frac{1}{8\pi}\frac{(B(\thetab\!-\!\etab))_m}{|A(\thetab\!-\!\etab)|^3}\delta_{k,l}+q_{m,k,l}^1(\thetab,\etab)\\
    =&-\frac{1}{8\pi}\frac{(B(\thetab\!-\!\etab))_m}{|A(\thetab\!-\!\etab)|^3}\delta_{k,l}+m_{m,k,l}^1(\thetab,\etab)-\mc{K}_{m,k,l}^1(\thetab,\etab)\\
    =&-\frac{1}{8\pi}\paren{\frac{(B(\thetab\!-\!\etab))_m}{|A(\thetab\!-\!\etab)|^3}-\frac{(B_n(\etab)(\thetab\!-\!\etab))_m}{|A_n(\etab)(\thetab\!-\!\etab)|^3}}\delta_{k,l}-\mc{K}_{m,k,l}^1(\thetab,\etab)\\
    :=& \mc{P}^1_{m,k,l}(\thetab,\etab)-\mc{K}_{m,k,l}^1(\thetab, \etab).
    \end{aligned}
\end{equation}
Above, we are denoting $B=A^TA$, $A=\nabla\bm{X}_n(\bm{0})$,  $A_n(\etab)=\nabla\bm{X}_n(\etab)$, and $\mc{K}_{m,k,l}^1$ was given in \eqref{defR}. 
We denote
\begin{equation*}
    \tilde{Z}_{l,m}(\etab)=\PD{\wh{X}_i}{\eta_m}(\etab)Z_{n,li}(\etab),
\end{equation*}
and note that
\begin{equation*}
    \begin{aligned}
    \|\tilde{Z}\|_{C^\gamma(\mathbb{R}^2)}\leq C\|Z\|_{C^\gamma(\mathbb{S}^2)}.
    \end{aligned}
\end{equation*}
We start with a bound for $|I_n^1|$. 
We split it as follows
\begin{equation}\label{Jn81}
    \begin{aligned}
    I_n^1&=O^1+O^2,
    \end{aligned}
\end{equation}
with
\begin{equation*}
    \begin{aligned}
    O^1&=\wh{\rho}_n(\thetab)\!\!\int_{V_{5R}} \!\!\!P_{mkl}^1(\thetab, \etab) \big(\tilde{Z}_{l,m}(\etab)-\tilde{Z}_{l,m}(\thetab)\big)d\eta_1d\eta_2\\
    O^2&=\wh{\rho}_n(\thetab)\tilde{Z}_{l,m}(\thetab)\int_{V_{5R}} P_{m,k,l}^1(\thetab, \etab)d\eta_1d\eta_2.
    \end{aligned}
\end{equation*}
Then, we split $O^1$ further
\begin{equation*}
    |O^1|\leq O^{1,1}+O^{1,2},
\end{equation*}
where
\begin{equation*}
    \begin{aligned}
    O^{1,1}&=\|\tilde{Z}\|_{C^\gamma(V_{5R})}\int_{V_{5R}}|\mc{P}^1_{m,k,l}(\thetab,\etab)||\thetab-\etab|^\gamma  d\eta_1d\eta_2,
    \end{aligned}
\end{equation*}    
\begin{equation*}
    \begin{aligned}
    O^{1,2}&=\wh{\rho}_n(\thetab)\int_{V_{5R}}|\mc{K}_{m,k,l}^1(\thetab,\etab)||\delta_{\etab}\tilde{Z}_{l,m}(\thetab)| d\eta_1d\eta_2.
    \end{aligned}
\end{equation*}    
For $\thetab\in V_{4R}$, $\etab\in V_{5R}$, we have the following bound for $\mc{P}^1_{m,k,l}$ \eqref{mcP},
\begin{equation}\label{mcPbound}
    \begin{aligned}
    |\mc{P}^1_{m,k,l}|&\leq \frac{1}{8\pi}|\frac{(B_n(\etab)-B)(\thetab-\etab)}{|A(\thetab-\etab)|^3}|\\
    &\quad+\frac{1}{8\pi}|B_n(\etab)(\thetab-\etab)||\frac{1}{|A(\thetab-\etab)|^3}-\frac{1}{|A_n(\etab)(\thetab-\etab)|^3}|\\
    &\leq C\frac{\|B_n(\cdot)-B\|_{C^0(V_{5R})}}{|\bm{X}|_{\circ,n}^3|\thetab-\etab|^2}+C\frac{\|A_n(\cdot)-A\|_{C^0(V_{5R})}\|\nabla\bm{X}_n\|_{C^0(V_{5R})}^7}{|\bm{X}|_{\circ,n}^9|\thetab-\etab|^2},
    \end{aligned}
\end{equation}
where \eqref{odd_fractions} has been used for the last term.
Then, the first term $O^{1,1}$ is high-order but with small coefficients,
\begin{equation*}
    \begin{aligned}
    O^{1,1}&\leq C\Big(\frac{\|B_n(\cdot)-B\|_{C^0(V_{5R})}}{|\bm{X}|_{\circ,n}^3}+\frac{\|A_n(\cdot)-A\|_{C^0(V_{5R})}\|\nabla\bm{X}_n\|_{C^0(V_{5R})}^7}{|\bm{X}|_{\circ,n}^9}\Big)R^\gamma\\
    &\hspace{1cm}\times\|\tilde{Z}\|_{C^{\gamma}(V_{5R})},
    \end{aligned}
\end{equation*}
since we have that
\begin{equation*}
    \|A_n(\cdot)-A\|_{C^0(V_{5R})}\leq \varepsilon(R)C(\|\nabla \bm{X}\|_{C^0(V_{5R})}),
\end{equation*}
with $\varepsilon(R)\to 0$ as $R\to 0$.
The kernel bound, with $\thetab\in V_{4R}$, $\etab\in V_{5R}$,
\begin{align}\label{kRbounds1}
    |\mathcal{K}^1_{m,k,l}(\thetab,\etab)|&\leq C\frac{\|\nabla \bm{X}_n\|_{C^0(V_{5R})}}{|\bm{X}|_{\circ,n}^3}[\nabla \bm{X}_n]_{C^\gamma(V_{5R})}\frac{1}{|\thetab-\etab|^{2-\gamma}},
\end{align}
gives that
\begin{equation*}
    \begin{aligned}
    O^{1,2}&\leq C\frac{\|\nabla\bm{X}_n\|_{C^0(V_{5R})}\|\nabla\bm{X}_n\|_{C^\gamma(V_{5R})}}{|\bm{X}|_{\circ,n}^3}R^\gamma\|\tilde{Z}\|_{C^0(V_{5R})}.
    \end{aligned}
\end{equation*}
But one also has the bound
\begin{align*}
    |\mathcal{K}^1_{m,k,l}(\thetab,\etab)|&\leq C\frac{\|\nabla \bm{X}_n\|_{C^0(V_{5R})}\|\nabla \bm{X}_n\|_{C^{\frac{\gamma}{2}}(V_{5R})}}{|\bm{X}|_{\circ,n}^3}\frac{1}{|\thetab-\etab|^{2-\frac{\gamma}2}},
\end{align*}
which gives that
\begin{equation*}
    \begin{aligned}
    O^{1,2}&\leq C\frac{\|\nabla\bm{X}_n\|_{C^0(V_{5R})} \|\nabla \bm{X}_n\|_{C^{\frac{\gamma}{2}}}}{|\bm{X}|_{\circ,n}^3}R^\gamma\|\tilde{Z}\|_{C^{\frac{\gamma}{2}}(V_{5R})}.
    \end{aligned}
\end{equation*}
Joining the two bounds, we obtain
\begin{equation}\label{O1bound}
    \begin{aligned}
    |O^1|&\leq C(|\bm{X}|_*,\|\nabla_{\mathbb{S}^2}\bm{X}\|_{C^0(\mathbb{S}^2)})R^\gamma\big(\|\tilde{Z}\|_{C^0(V_{5R})}\|\nabla_{\mathbb{S}^2}\bm{X}\|_{C^\gamma(B_{\hx_n,5R}\cap\mathbb{S}^2)}\\
    &\quad+\varepsilon(R)\|\tilde{Z}\|_{C^\gamma(V_{5R})}\big),
    \end{aligned}
\end{equation}
and 
\begin{equation*}
    \begin{aligned}
    |O^1|&\leq C(|\bm{X}|_*,\|\nabla_{\mathbb{S}^2}\bm{X}\|_{C^0(\mathbb{S}^2)})R^\gamma\big(\|\nabla_{\mathbb{S}^2}\bm{X}\|_{C^{\frac{\gamma}2}(B_{\hx_n,5R}\cap\mathbb{S}^2)}\|\tilde{Z}\|_{C^{\frac{\gamma}2}(V_{5R})}\\
    &\quad+\varepsilon(R)\|\tilde{Z}\|_{C^\gamma(V_{5R})}\big).
    \end{aligned}
\end{equation*}
To estimate $O^2$, we use \eqref{mcP} to integrate by parts,
\begin{equation*}
    \begin{aligned}
        |O^2|&\leq C|\wh{\rho}_n(\thetab)||\tilde{Z}_{l,m}(\thetab)|\int_{\partial V_{5R}} \Big|\frac{1}{|A(\thetab-\etab)|}-\frac{1}{|\bm{X}_n(\thetab)-\bm{X}_n(\etab)|}\Big|dl(\etab).
    \end{aligned}
\end{equation*}
Next, we note that since $\thetab\in V_{4R}$, we have that for $\etab\in\partial V_{5R}$, $|\thetab-\etab|\geq \frac{R}{2}$.
We compute the difference
\begin{equation}\label{boundary_terms}
\begin{aligned}
    \frac{1}{|A(\thetab-\etab)|}&-\frac{1}{|\bm{X}_n(\thetab)-\bm{X}_n(\etab)|}\\
    &=\frac{1}{|\thetab-\etab|}\Big(\frac{1}{|A(\widehat{\thetab-\etab})|}-\frac{1}{|\Delta_{\etab}\bm{X}_n(\thetab)|}\Big)\\
    &=\frac{1}{|\thetab\!-\!\etab|}\frac{\big(\Delta_{\etab}\bm{X}_n(\thetab)\!-\!A(\widehat{\thetab-\etab})\big)\cdot\big(\Delta_{\etab}\bm{X}_n(\thetab)\!+\!A(\widehat{\thetab\!-\!\etab})\big)}{|A(\widehat{\thetab\!-\!\etab})||\Delta_{\etab}\bm{X}_n(\thetab)|\big(|A(\widehat{\thetab\!-\!\etab})|\!+\!|\Delta_{\etab}\bm{X}_n(\thetab)|\big)},
\end{aligned}
\end{equation}
and 
\begin{equation*}
    \begin{aligned}
    \Delta_{\etab}\bm{X}_n(\thetab)\!-\!A(\widehat{\thetab-\etab})&=\Delta_{\etab}\bm{X}_n(\thetab)\!-\!\nabla\bm{X}_n(\etab)(\widehat{\thetab-\etab})\!+\!(\nabla\bm{X}_n(\etab)-A)(\widehat{\thetab-\etab})\\
    &=-E^{\etab}\bm{X}_n(\thetab)+(A_n(\etab)-A)(\widehat{\thetab-\etab}),
    \end{aligned}
\end{equation*}
where $E^{\etab}\bm{X}_n(\thetab)$ is given in \eqref{EetaX}.
Therefore, 
\begin{equation*}
    \begin{aligned}
        |O^2|&\leq C\|\tilde{Z}\|_{C^0(V_{5R})}\frac{\|\nabla \bm{X}_n\|_{C^0(V_{5R})}}{|\bm{X}|_{\circ,n}^3}\!\int_{\partial V_{5R}}\!\!\!\!\!\!\!\! \frac{|E^{\etab}\bm{X}_n(\thetab)|\!+\!\|A_n(\cdot)\!-\!A\|_{C^0(V_{5R})}}{|\thetab\!-\!\etab|}dl(\etab),
    \end{aligned}
\end{equation*}
hence
\begin{equation*}
    \begin{aligned}
    |O^2|&\leq C(|\bm{X}|*,\|\nabla_{\mathbb{S}^2}\bm{X}\|_{C^0(\mathbb{S}^2)})\|\tilde{Z}\|_{C^0(V_{5R})}(R^{\frac{\gamma}{2}}\|\nabla_{\mathbb{S}^2}\bm{X}\|_{C^\frac{\gamma}{2}(B_{\hx_n,5R}\cap\mathbb{S}^2)}+\varepsilon(R)),
    \end{aligned}
\end{equation*}
and
\begin{equation*}
    \begin{aligned}
    |O^2|&\leq C(|\bm{X}|*,\|\nabla_{\mathbb{S}^2}\bm{X}\|_{C^0(\mathbb{S}^2)})R^\gamma\|\nabla_{\mathbb{S}^2}\bm{X}\|_{C^\gamma(B_{\hx_n,5R}\cap\mathbb{S}^2)}\|\tilde{Z}\|_{C^0(V_{5R})}.
    \end{aligned}
\end{equation*}
Together with \eqref{O1bound} back in \eqref{Jn81}, we conclude that
\begin{equation}\label{Jn81C0bound}
    \begin{aligned}
    |I_n^1|&\leq   C(|\bm{X}|_*,\|\nabla_{\mathbb{S}^2}\bm{X}\|_{C^0(\mathbb{S}^2)})R^\gamma\big(\|\tilde{Z}\|_{C^0(V_{5R})}\|\nabla_{\mathbb{S}^2}\bm{X}\|_{C^\gamma(B_{\hx_n,5R}\cap\mathbb{S}^2)}\\
    &\quad+\varepsilon(R)\|\tilde{Z}\|_{C^\gamma(V_{5R})}\big),
    \end{aligned}
\end{equation}
and also
\begin{equation*}
    \begin{aligned}
    |I_n^1|&\leq  C(|\bm{X}|_*,\|\nabla_{\mathbb{S}^2}\bm{X}\|_{C^0(\mathbb{S}^2)})\big(R^{\frac{\gamma}{2}}\|\nabla_{\mathbb{S}^2}\bm{X}\|_{C^{\frac{\gamma}2}(B_{\hx_n,5R}\cap\mathbb{S}^2)}\|\tilde{Z}\|_{C^{\frac{\gamma}2}(V_{5R})}\\
    &\quad+\varepsilon(R)\|\tilde{Z}\|_{C^\gamma(V_{5R})}\big).
    \end{aligned}
\end{equation*}
We proceed to estimate the H\"older seminorm. Take $\thetab, \thetab+\hb\in V_{4R}$. We use the splitting \eqref{Jn81}, and start with the estimate for $O^1$:
\begin{equation}\label{OHoldersplit}
    \begin{aligned}
|O^1(\thetab&+\hb)-O^1(\thetab)|&=|Q^1+Q^2+Q^3+Q^4+Q^5|,
    \end{aligned}
\end{equation}
where, using the notation \eqref{delta_eta},
\begin{equation*}
    Q^1=-\wh{\rho}_n(\thetab+\hb)\int_{\{|\thetab-\etab|\leq 2|\hb|\}\cap V_{5R}}\hspace{-0.3cm} P^1_{m,k,l}(\thetab+\hb,\etab)\delta_{\etab}\tilde{Z}_{l,m}(\thetab+\hb)d\eta_1d\eta_2,
\end{equation*}
\begin{equation*}
 Q^2=-\wh{\rho}_n(\thetab+\hb)\int_{\{|\thetab-\etab|\leq 2|\hb|\}\cap V_{5R}}\hspace{-0.3cm} P^1_{m,k,l}(\thetab,\etab)\delta_{\etab}\tilde{Z}_{l,m}(\thetab)d\eta_1d\eta_2,
\end{equation*}
\begin{equation*}
    Q^3=\wh{\rho}_n(\thetab+\hb)\delta_{\thetab}\tilde{Z}_{l,m}(\thetab+\hb)\int_{\{|\thetab-\etab|\geq 2|\hb|\}\cap V_{5R}}\hspace{-0.3cm}P^1_{m,k,l}(\thetab, \etab)d\eta_1d\eta_2,
\end{equation*}
\begin{equation*}
    Q^4=\wh{\rho}_n(\thetab+\hb)\!\!\int_{\{|\thetab-\etab|\geq 2|\hb|\}\cap V_{5R}} \hspace{-2.1cm}(P^1_{m,k,l}(\thetab,\etab)-P^1_{m,k,l}(\thetab+\hb,\etab))\delta_{\etab}\tilde{Z}_{l,m}(\thetab+\hb)d\eta_1d\eta_2,
\end{equation*}
\begin{equation*}
    Q^5=(\wh{\rho}_n(\thetab)-\wh{\rho}_n(\thetab+\hb))\int_{V_{5R}} P^1_{m,k,l}(\thetab,\etab)\delta_{\etab}\tilde{Z}_{l,m}(\thetab)d\eta_1d\eta_2.
\end{equation*}
Recalling \eqref{mcP} and the bounds for $\mc{P}^1_{m,k,l}$ \eqref{mcPbound} and $\mc{K}^1_{m,k,l}$ \eqref{kRbounds1}, we obtain
\begin{equation*}
    \begin{aligned}
    |Q^1|+|Q^2|&\leq C\Big(\frac{\|B_n(\cdot)-B\|_{C^0(V_{5R})}}{|\bm{X}|_{\circ,n}^3}+\frac{\|A_n(\cdot)-A\|_{C^0(V_{5R})}\|\nabla\bm{X}_n\|_{C^0(V_{5R})}^7}{|\bm{X}|_{\circ,n}^9}\Big)\\
    &\hspace{1cm}\times\int_{\{|\thetab-\etab|\leq 2|\hb|\}\cap V_{5R}}\hspace{-0.3cm}\frac{[\tilde{Z}]_{C^\gamma(V_{5R})}}{|\thetab-\etab|^{2-\gamma}}d\eta_1d\eta_2\\
    &\hspace{-0.4cm}+C\frac{\|\nabla \bm{X}_n\|_{C^0(V_{5R})}}{|\bm{X}|_{\circ,n}^3}\int_{\{|\thetab-\etab|\leq 2|\hb|\}\cap V_{5R}}\hspace{-1.1cm}\frac{[\nabla \bm{X}_n]_{C^\gamma(V_{5R})}\|\tilde{Z}\|_{C^0(V_{5R})}}{|\thetab-\etab|^{2-\gamma}}d\eta_1d\eta_2,
    \end{aligned}
\end{equation*}
thus
\begin{equation}\label{Q1Q2bound}
    \begin{aligned}
    |Q^1|\!+\!|Q^2|&\leq  C(|\bm{X}|_*,\|\nabla_{\mathbb{S}^2}\bm{X}\|_{C^0(\mathbb{S}^2)})\big(\|\nabla_{\mathbb{S}^2}\bm{X}\|_{C^\gamma(B_{\hx_n,5R}\cap\mathbb{S}^2)}\|\tilde{Z}\|_{C^0(V_{5R})}\\
    &\quad+\varepsilon(R)\|\tilde{Z}\|_{C^\gamma(V_{5R})}\big)|\hb|^\gamma.
    \end{aligned}
\end{equation}
It is clear that we could also obtain the estimate
\begin{equation*}
    \begin{aligned}
    |Q^1|\!+\!|Q^2|&\leq  C(|\bm{X}|_*,\|\nabla_{\mathbb{S}^2}\bm{X}\|_{C^0(\mathbb{S}^2)})|\hb|^\gamma\big(\|\nabla_{\mathbb{S}^2}\bm{X}\|_{C^\frac{\gamma}2(B_{\hx_n,5R}\cap\mathbb{S}^2)}\|\tilde{Z}\|_{C^\frac{\gamma}{2}(V_{5R})}\\
    &\qquad+\varepsilon(R)\|\tilde{Z}\|_{C^\gamma(V_{5R})}\big).
    \end{aligned}
\end{equation*}
We see that the difference between those two type of estimate comes only from the kernel $\mc{K}$, where we can distribute half a derivative. The same idea propagates along the lines below, hence we only show the first estimate \eqref{Jn8bound}.
Then, integration by parts in $Q^3$ gives
\begin{equation*}
    \begin{aligned}
        |Q^3|&\leq C|\wh{\rho}_n(\thetab+\hb)||\delta_{\thetab}\tilde{Z}(\thetab+\hb)|\\
        &\hspace{1cm}\times\int_{\{|\thetab-\etab|=2|\hb|\}\cup\partial V_{5R}} \Big|\frac{1}{|A(\thetab-\etab)|}-\frac{1}{|\bm{X}_n(\thetab)-\bm{X}_n(\etab)|}\Big|dl(\etab).
    \end{aligned}
\end{equation*}
Using \eqref{boundary_terms}, $\abs{\hb}\leq 8R$ and that $\thetab\in V_{4R}$, we obtain that 
\begin{equation*}
    \begin{aligned}
        &\int_{\{|\thetab-\etab|=2|\hb|\}\cup\partial V_{5R}} \Big|\frac{1}{|A(\thetab-\etab)|}-\frac{1}{|\bm{X}_n(\thetab)-\bm{X}_n(\etab)|}\Big|dl(\etab)\\
        &\leq C\frac{\|\nabla\bm{X}\|_{C^0(V_{5R})}}{|\bm{X}|_{\circ,n}^3}\int_{\{|\thetab-\etab|=2|\hb|\}\cup\partial V_{5R}}\!\!\!\!\!\!\!\! \frac{|E^{\etab}\bm{X}_n(\thetab)|}{|\thetab\!-\!\etab|}+\frac{\|A_n(\cdot)\!-\!A\|_{C^0(V_{5R})}}{|\thetab\!-\!\etab|}dl(\etab)\\
        &\leq C\frac{\|\nabla\bm{X}\|_{C^0(V_{5R})}}{|\bm{X}|_{\circ,n}^3}\Big(\|\nabla \bm{X}\|_{C^0(V_{5R})}(\varepsilon(|\hb|)+\varepsilon(R))+\|\nabla \bm{X}\|_{C^0(V_{5R})}\varepsilon(R)\Big)\\
        &\leq C\frac{\|\nabla\bm{X}\|_{C^0(V_{5R})}^2}{|\bm{X}|_{\circ,n}^3}\varepsilon(R),
    \end{aligned}
\end{equation*}
hence
\begin{equation}\label{Q3bound}
    \begin{aligned}
        |Q^3|  &\leq \varepsilon(R)C(|\bm{X}|_*,\|\nabla_{\mathbb{S}^2}\bm{X}\|_{C^0(\mathbb{S}^2)})\|\tilde{Z}\|_{C^\gamma(V_{5R})}|\hb|^\gamma.
    \end{aligned}
\end{equation}
The term $Q^4$ in \eqref{OHoldersplit} is estimated by applying the mean-value theorem. As in the previous terms, we need to consider separately the two  kernels in $P^1_{m,k,l}$ \eqref{mcP}. We take a derivative on $\mc{P}^1_{m,k,l}$,
\begin{equation*}
    \begin{aligned}
    \PD{}{\theta_j}\mc{P}^1_{m,k,l}(\thetab,\etab)&=\frac{\delta_{k,l}}{8\pi}\Big(-\frac{B_{m,j}}{|A(\thetab-\etab)|^3}+\frac{(B_n(\etab))_{m,j}}{|A_n(\etab)(\thetab-\etab)|^3}\Big)\\
    &\hspace{-1.5cm}+\frac{\delta_{k,l}}{8\pi}\Big(3\frac{(B(\thetab-\etab))_m(B(\thetab-\etab))_j}{|A(\thetab-\etab)|^5}-3\frac{(B_n(\etab)(\thetab-\etab))_m(B_n(\etab)(\thetab-\etab))_j}{|A_n(\etab)(\thetab-\etab)|^5}\Big),
    \end{aligned}
\end{equation*}
and thus
\begin{equation*}
    \begin{aligned}
    |\PD{}{\theta_j}\mc{P}^1_{m,k,l}(\thetab,\etab)|&\leq C(\|\nabla\bm{X}_n\|_{C^0(V_{4R})},|\bm{X}|_{\circ,n})\frac{\varepsilon(R)}{|\thetab-\etab|^3},
    \end{aligned}
\end{equation*}
while the derivative of $\mc{K}^1_{m,k,l}$ is computed and bounded below
\begin{align*}
    \PD{}{\theta_j}\mathcal{K}_{m,k,l}^{1,1}(\thetab,\etab)&=-\delta_{k,l}\PD{\bm{X}_n(\etab)}{\eta_m} \cdot\\
    &\times\paren{3\frac{\Delta_{\etab}\bm{X}_n(\thetab)\cdot \PD{\bm{X}_n(\thetab)}{\theta_j}}{|\Delta_{\etab}\bm{X}_n(\thetab)|^5}\frac{ E^{\etab}\bm{X}_n(\thetab)}{|\thetab-\etab|^3}+\frac{\delta_{\etab}\PD{\bm{X}_n(\thetab)}{\theta_j}}{|\Delta_{\etab}\bm{X}_n(\thetab)|^3|\thetab-\etab|^3}},
\end{align*}
\begin{equation*}
    \begin{aligned}
     |\PD{}{\theta_j}\mathcal{K}^{1,1}(\thetab,\etab)|&\leq \frac{\|\nabla \bm{X}_n\|_{C^0(V_{4R})}}{|\bm{X}|_{\circ,n}^3}\frac{[\nabla \bm{X}_n]_{C^\gamma(V_{4R})}}{|\thetab-\etab|^{3-\gamma}}\paren{1\!+\! 3\frac{\|\nabla \bm{X}_n\|_{C^0(V_{4R})}}{|\bm{X}|_{\circ,n}}}\\
      &\leq C(|\bm{X}|_*,\|\nabla_{\mathbb{S}^2}\bm{X}\|_{C^0(\mathbb{S}^2)})\frac{\|\nabla_{\mathbb{S}^2}\bm{X}\|_{C^\gamma(B_{\hx_n,2R}\cap\mathbb{S}^2)}}{|\thetab-\etab|^{3-\gamma}}.
    \end{aligned}
\end{equation*}
Therefore, we obtain
\begin{equation}\label{Q4bound}
    \begin{aligned}
    |Q_4|&\leq C(|\bm{X}|_*,\|\nabla_{\mathbb{S}^2}\bm{X}\|_{C^0(\mathbb{S}^2)})\big(\|\nabla_{\mathbb{S}^2}\bm{X}\|_{C^\gamma(B_{\hx_n,5R}\cap\mathbb{S}^2)}\|\tilde{Z}\|_{C^0(V_{5R})}\\
    &\quad+\varepsilon(R)\|\tilde{Z}\|_{C^\gamma(V_{5R})}\big)|\hb|^\gamma.
    \end{aligned}
\end{equation}
Finally, the estimate for $Q^5$ \eqref{OHoldersplit} follows from the bounds \eqref{mcPbound} and \eqref{kRbounds1},
\begin{equation}\label{Q5bound}
    \begin{aligned}
    |Q^5|&\leq C(\|\nabla\bm{X}_n\|_{C^0(V_{5R})},|\bm{X}|_{\circ,n})\|\wh{\rho}_n\|_{C^\gamma(V_{5R})}|\hb|^\gamma\\
    &\qquad\times\Big(\varepsilon(R)\|\tilde{Z}\|_{C^\gamma(V_{5R})}\!+\!\|\nabla\bm{X}_n\|_{C^\gamma(V_{5R})}\|\tilde{Z}\|_{C^0(V_{5R})}\Big)\\
    &\leq C(|\bm{X}|_*,\|\nabla_{\mathbb{S}^2}\bm{X}\|_{C^0(\mathbb{S}^2)})\big(\|\nabla_{\mathbb{S}^2}\bm{X}\|_{C^\gamma(B_{\hx_n,5R}\cap\mathbb{S}^2)}\|\tilde{Z}\|_{C^0(V_{5R})}\\
    &\quad+\varepsilon(R)\|\tilde{Z}\|_{C^\gamma(V_{5R})}\big)|\hb|^\gamma.
    \end{aligned}
\end{equation}
We combine the bounds \eqref{Q1Q2bound}, \eqref{Q3bound}, \eqref{Q4bound}, and \eqref{Q5bound} into \eqref{OHoldersplit} to conclude that
\begin{equation}\label{O1Holderbound}
    \begin{aligned}
    |O^1(\thetab+\hb)-O^1(\thetab)|&\leq C(|\bm{X}|_*,\|\nabla_{\mathbb{S}^2}\bm{X}\|_{C^0(\mathbb{S}^2)})\big(\|\nabla_{\mathbb{S}^2}\bm{X}\|_{C^\gamma(B_{\hx_n,5R}\cap\mathbb{S}^2)}\|\tilde{Z}\|_{C^0(V_{5R})}\\
    &\quad+\varepsilon(R)\|\tilde{Z}\|_{C^\gamma(V_{5R})}\big)|\hb|^\gamma.
    \end{aligned}
\end{equation}
We continue with the H\"older seminorm for $O^2$ in \eqref{Jn81}. Integrating by parts
\begin{equation*}
    \begin{aligned}
    O^2&(\thetab+\hb)-O^2(\thetab)\\
    &=\frac{\delta_{k,l}}{8\pi}\wh{\rho}_n(\thetab+\hb)\tilde{Z}_{l,m}(\thetab+\hb)\\
    &\hspace{1.5cm}\times\int_{\partial V_{5R}}\Big(\frac{1}{|A(\thetab+\hb-\etab)|}-\frac{1}{|\bm{X}(\thetab+\hb)-\bm{X}_n(\etab)|}\Big)n_m(\etab) dl(\etab)\\
    &\quad-\frac{\delta_{k,l}}{8\pi}\wh{\rho}_n(\thetab)\tilde{Z}_{l,m}(\thetab)\int_{\partial V_{5R}}\!\! \Big(\frac{1}{|A(\thetab\!-\!\etab)|}\!-\!\frac{1}{|\bm{X}_n(\thetab)\!-\!\bm{X}_n(\etab)|}\Big)n_m(\etab)dl(\etab).
    \end{aligned}
\end{equation*}
Therefore, 
\begin{equation}\label{O2split}
    \begin{aligned}
    |O^2(\thetab+\hb)-O^2(\thetab)|&=|Q^6+Q^7|,
    \end{aligned}
\end{equation}
    with
\begin{equation*}
    \begin{aligned}
    Q^6&=  \frac{\delta_{k,l}}{8\pi}\Big(\wh{\rho}_n(\thetab+\hb)\tilde{Z}_{l,m}(\thetab+\hb)-\wh{\rho}_n(\thetab)\tilde{Z}_{l,m}(\thetab)\Big)\\
    &\qquad\times\int_{\partial V_{5R}}|\frac{1}{|A(\thetab+\hb-\etab)|}-\frac{1}{|\bm{X}(\thetab+\hb)-\bm{X}_n(\etab)|}|n_m(\etab)dl(\etab),
    \end{aligned}
\end{equation*}
\begin{equation*}
    \begin{aligned}
    Q^7&=\frac{\delta_{k,l}}{8\pi}\wh{\rho}_n(\thetab)\tilde{Z}_{l,m}(\thetab)\\
    &\qquad\times\bigg(
    \int_{\partial V_{5R}} \Big(\frac{1}{|A(\thetab+\hb-\etab)|}-\frac{1}{|\bm{X}_n(\thetab+\hb)-\bm{X}_n(\etab)|}\Big)n_m(\etab)dl(\etab)
    \\
    &\hspace{1.5cm}-\int_{\partial V_{5R}} \Big(\frac{1}{|A(\thetab-\etab)|}-\frac{1}{|\bm{X}_n(\thetab)-\bm{X}_n(\etab)|}\Big)n_m(\etab)dl(\etab)\bigg).
    \end{aligned}
\end{equation*}
Using \eqref{boundary_terms} and that $\thetab\in V_{4R}$, we obtain
\begin{equation*}
    \begin{aligned}
    |Q^6|&\leq C\|\tilde{Z}\|_{C^\gamma(V_{5R})}|\hb|^\gamma \frac{\|\nabla\bm{X}_n\|_{C^0(V_{5R})}}{|\bm{X}|_{\circ,n}^3}\\
    &\hspace{1cm}\times\int_{\partial V_{5R}} \frac{|E^{\etab}\bm{X}_n(\thetab)|\!+\!\|A_n(\cdot)\!-\!A\|_{C^0(V_{5R})}}{|\thetab\!-\!\etab|}dl(\etab)\\
    &\leq C\|\tilde{Z}\|_{C^\gamma(V_{5R})}|\hb|^\gamma \frac{\|\nabla\bm{X}_n\|_{C^0(V_{5R})}}{|\bm{X}|_{\circ,n}^3}\|\nabla\bm{X}_n\|_{C^0(V_{5R})}\varepsilon(R).
    \end{aligned}
\end{equation*}
Finally, we estimate $Q^7$, 
\begin{equation*}
    \begin{aligned}
    |&Q^7|\leq C\|\tilde{Z}\|_{C^0(V_{5R})}\int_{\partial V_{5R}}|\frac{1}{|A(\thetab+\hb-\etab)|}-\frac{1}{|A(\thetab-\etab)|}|dl(\etab)\\
    &+C\|\tilde{Z}\|_{C^0(V_{5R})}\int_{\partial V_{5R}}|\frac{1}{|\bm{X}_n(\thetab+\hb)-\bm{X}_n(\etab)|}-\frac{1}{|\bm{X}_n(\thetab)-\bm{X}_n(\etab)|}|dl(\etab).
    \end{aligned}
\end{equation*}
Performing the differences we obtain that
\begin{equation*}
    \begin{aligned}
    |Q^7|&\leq C(|\bm{X}|_{\circ,n},\|\nabla\bm{X}_n\|_{C^0(V_{5R})})\frac{(|\hb|+|\hb|^{1-\gamma}R^\gamma)}{R}\|\tilde{Z}\|_{C^0(V_{5R})}|\hb|^\gamma\\
    &\leq  C(|\bm{X}|_{\circ,n},\|\nabla\bm{X}_n\|_{C^0(V_{5R})})\|\tilde{Z}\|_{C^0(V_{5R})}|\hb|^\gamma.
    \end{aligned}
\end{equation*}
Thus, going back to \eqref{O2split}, we conclude that
\begin{equation}\label{O2Holderbound}
    \begin{aligned}
    |&O^2(\thetab\!+\!\hb)\!-\!O^2(\thetab)|\!\leq\! C(|\bm{X}|_*,\|\nabla_{\mathbb{S}^2}\bm{X}\|_{C^0(\mathbb{S}^2)})\big(\|\tilde{Z}\|_{C^0(V_{5R})}\\
    &\quad+\varepsilon(R)\|\tilde{Z}\|_{C^\gamma(V_{5R})}\big)|\hb|^\gamma,
    \end{aligned}
\end{equation}
and together with \eqref{O1Holderbound} and \eqref{Jn81C0bound} back into \eqref{Jn81} we obtain the H\"older norm estimate for $I_n^1$. Since the kernel in $I_n^2$ \eqref{Jn8split} satisfy the same estimates, we conclude that
\begin{equation*}
    \begin{aligned}
    \|I_n\|_{C^\gamma(\mathbb{R}^2)}&\leq C(|\bm{X}|_*,\|\nabla_{\mathbb{S}^2}\bm{X}\|_{C^0(\mathbb{S}^2)})\big((1+\|\nabla_{\mathbb{S}^2}\bm{X}\|_{C^\gamma(B_{\hx_n,5R}\cap\mathbb{S}^2)})\|\tilde{Z}\|_{C^0(V_{5R})}\\
    &\quad+\varepsilon(R)\|\tilde{Z}\|_{C^\gamma(V_{5R})}\big)|\hb|^\gamma.
    \end{aligned}
\end{equation*}

\end{proof}

\section{Frozen-coefficient Semigroup}\label{sec:frozen}

We will need later in the proof (see Lemma \ref{Fsemigroup}) to deal with the following kernels, with $0\leq \alpha \leq 1$:
\begin{equation*}
\begin{split}
G_{\alpha}(A\bm{\theta})&=G^1(A\thetab)+\alpha G^2(A\thetab),
\end{split}
\end{equation*}
where $G^1$, $G^2$ are given in \eqref{GG1G2}.
From  Section \ref{sec:symbol}, we have that
\begin{equation*}
\begin{split}
\paren{\mc{F}_\theta G_{\alpha,A}}(\bm{\xi})&=\paren{\mc{F}_\theta G^1(A\bm{\theta})}(\bm{\xi})+\alpha\paren{\mc{F}_\theta G^2(A\bm{\theta})}(\bm{\xi}),\\
\paren{\mc{F}_\theta G^1(A\bm{\theta})}(\bm{\xi})&=\frac{1}{4\det(B)\abs{U\bm{\xi}}},\\
\paren{\mc{F}_\theta G^2(A\bm{\theta})}(\bm{\xi})&=\frac{1}{4\det(B)\abs{U\bm{\xi}}}\paren{ P- \frac{U\bm{\xi}}{\abs{U\bm{\xi}}}\otimes \frac{U\bm{\xi}}{\abs{U\bm{\xi}}}},
\end{split}
\end{equation*}
where $\lambda=\norm{A}_F, B=\sqrt{A^{\rm T}A}, \; P=A(A^{\rm T}A)^{-1}A^{\rm T}, \; U=A(A^{\rm T}A)^{-1}$.
Let us consider the operator defined in \eqref{defM}.
In preparation for this, we consider the operator with $\nabla\bm{X}$ given by a constant matrix and with $\wh{g}=I_2$, as defined in \eqref{defnL2}:

\begin{align*}
\begin{split}
     \mc{L}_{\alpha, A}\bm{Y}
    :=&\tilde{\mc{M}}_\alpha\paren{A}\paren{T_F(A)\nabla \bm{Y}}
    =\paren{\tilde{\mc{M}}^1(A)+\alpha \tilde{\mc{M}}^2(A)}\paren{T_F(A)\nabla \bm{Y}}
\end{split}
\end{align*}
where $T_F(A)$ is defined in \eqref{tensionD}. 
The parameter $\alpha$ is useful in Section \ref{sec: local well-posed}.
We will prove $\mc{L}_{\alpha, A}$ is a sectorial operator first.
That is to say, we have to estimate $\paren{z+\mc{L}_{\alpha, A}}^{-1}$ where $z$ is in a set with some $\omega\in \mathbb{R}, 0<\delta<\frac{\pi}{2}$ in the complex plane:
\begin{equation*}
\mc{S}_{\omega,\delta}=\lbrace z\in \mathbb{C}:\abs{{\rm arg}(z-\omega)}\leq \pi-\delta \rbrace.
\end{equation*}
Since $\tilde{\mc{M}}^j(A)$ is a singular integral operator, it is difficult to compute its inverse operator.
However, $\tilde{\mc{M}}^j(A)$ is a convolution with kernel $G^j\paren{A\thetab}$, so we may use its Fourier multiplier to obtain 
\begin{align*}
    \paren{z+\mc{L}_{\alpha, A}}^{-1}\bm{Y}=&\mc{F}^{-1}(\paren{z-L_{\alpha,A}(\bm{\xi})}^{-1}\mc{F}\bm{Y}(\xib))(\thetab)
\end{align*}
where $L_{\alpha,A}$ is defined in \eqref{multiplierNon} (in this section, we will write $L_A$ instead of $L_{\alpha, A}$).
Then, we will use the Fourier multiplier to estimate the original operator.
In harmonic analysis, the Fourier multiplier theorem in $L^p$ norms is well-known \cite{Loukas:classical-fourier-analysis}.
We will use a Fourier multiplier theorem in semi-norms $\jump{\cdot}_{C^{\gamma}\paren{\mbr}}$. 

\begin{thm}\label{fouriermuliplierholdernorm}
If $T$ is a Fourier multiplier operator with multiplier \begin{equation*}
m\paren{\bm{\xi}}\in C^s \paren{\mathbb{R}^n\setminus \{\bm{0}\}}\cap L^\infty\paren{\mathbb{R}^n},\qquad s>\frac{n}{2},
\end{equation*}
and 
\begin{align*}
    \norm{\partial_{\bm{\xi}}^\alpha m\paren{\bm{\xi}}}\leq C_\alpha \abs{\bm{\xi}}^{-\abs{\alpha}}
\end{align*}
for all $\abs{\alpha}\leq s$, then for all $u\in C^\gamma\paren{\mathbb{R}^n}\cap L^2\paren{\mathbb{R}^n}$ where $0<\gamma<1$,
\begin{align*}
    \jump{T u}_{C^{\gamma}\paren{\mathbb{R}^n}}\leq C_{\gamma, s, n} D_m\jump{u}_{C^{\gamma}\paren{\mathbb{R}^n}},
\end{align*}
where $D_m= \max_{\abs{\alpha}\leq s}C_\alpha$.
\end{thm}
\begin{remark}
The proof of Theorem \ref{fouriermuliplierholdernorm} may be split into two parts.
The first part is the well-known equivalence between  the homogeneous Besov norm $\norm{\cdot}_{\Dot{B}_{\infty, \infty}^\gamma\paren{\mathbb{R}^n}}$  and H\"older seminorm  $\jump{\cdot}_{C^{\gamma}\paren{\mathbb{R}^n}}$ \cite{Loukas:modern-fourier-analysis}.
The second part is proving the Fourier multiplier theorem in homogeneous Besov norms $\norm{\cdot}_{\Dot{B}_{\infty, \infty}^\gamma\paren{\mathbb{R}^n}}$. Although these results are classical, we include the proof of the version we need in Appendix \ref{sec:appA} for the covenience of the reader. %(taken from the Japanese book \cite{Yoshihiro:theory-of-the-lebesque-integral}).
\end{remark}

We will compute $\paren{z-L_{\alpha,A}(\bm{\xi})}^{-1}$ in Section \ref{sec:fourier_fund_est}.
Next, for the norm $\norm{\cdot}_{C^{0}\paren{\mbr}}$, we may expand the result in \cite[Section 3.1]{Rodenberg:peskin-thesis} in $\mathbb{R}^2$.
\begin{lemma}[{\cite[Proposition 3.1.2]{Rodenberg:peskin-thesis}}]\label{c0normseminorm}
If $\jump{u}_{C^\gamma\paren{\mbr}}<\infty$ for some $\gamma\in\paren{0,1}$ and $\mc{F}u \paren{\bm{\xi}}=0$ in a neighborhood of $\bm{\xi}=\bm{0}$, then $\norm{u}_{C^0\paren{\mbr}}\leq C \jump{u}_{C^\gamma\paren{\mbr}}$, where $C$ depends on the neighborhood.
\end{lemma}
Therefore, we may choose a suitable cutting function $\varphi\paren{\xib}$ with $\varphi\paren{\xib}=1$ in a neighborhood of $\bm{\xi}=\bm{0}$.
Then, the rest of the work for $\norm{\paren{z+\mc{L}_{\alpha, A}}^{-1}\bm{Y}}_{C^{0}\paren{\mbr}}$ is only $\norm{\mc{F}^{-1}(\paren{z-L_{\alpha,A}(\bm{\xi})}^{-1}\varphi\paren{\xib}\mc{F}\bm{Y}(\xib))(\thetab)}_{C^{0}\paren{\mbr}}$.

\subsection{Fundamental estimates}\label{sec:fourier_fund_est}
We need some elementary estimates on operators $\mc{L}_A$ and $L_A$.
To achieve it, we first compute the estimate of $T_F(A)$.
\begin{lemma}\label{TF_est}
Given  matrice $A_1,A_2$ in $\domainA$, we have $\sigma_2\leq\norm{A_1}_F,\norm{A_2}_F\leq\sqrt{2}\sigma_1$.
Then, we have the following estimates for $T_F(A)$:
\begin{align*}
    \abs{T_F(A_1)_{ijkl}}\leq&  z_M^{(0)},\\
    \abs{T_F(A_1)_{ijkl}-T_F(A_2)_{ijkl}}\leq& C\paren{z_M^{(0)}\frac{\sigma_1}{\sigma_2^2}+z_M^{(1)}}\norm{A_1-A_2},
\end{align*}
where
\begin{align*}
    z_M^{(0)}=\max_{\sigma_2\leq\lambda\leq\sqrt{2}\sigma_1}  \abs{f_1\paren{\lambda}}+\abs{f_2\paren{\lambda}},&\quad 
    z_M^{(1)}=\max_{\sigma_2\leq\lambda\leq\sqrt{2}\sigma_1}  \abs{\D{f_1}{\lambda}}+\abs{\D{f_2}{\lambda}},\\
    f_1\paren{\lambda}=\frac{\mc{T}}{\lambda},&\quad
    f_2\paren{\lambda}=\frac{\mc{T}}{\lambda}-\D{\mc{T}}{\lambda}.
\end{align*}
More specific, given $Z\in C^\gamma\paren{\mbr}\bigcap L^2\paren{\mbr}$ with the size of $A_1$,
\begin{align}
    \norm{T_F(A_1)Z}_{C^\gamma\paren{\mbr}\bigcap L^2\paren{\mbr}}\leq& C z_M^{(0)} \norm{Z}_{C^\gamma\paren{\mbr}\bigcap L^2\paren{\mbr}}\label{TF_est00}\\
    \norm{\paren{T_F(A_1)-T_F(A_2)}Z}_{C^\gamma\paren{\mbr}\bigcap L^2\paren{\mbr}}\leq& C \paren{z_M^{(0)}\frac{\sigma_1}{\sigma_2^2}+z_M^{(1)}}\norm{A_1-A_2} \norm{Z}_{C^\gamma\paren{\mbr}\bigcap L^2\paren{\mbr}}\label{TF_est01}
\end{align}
\end{lemma}
\begin{proof}
Set $\lambda_i=\norm{A_i}_F$.
Since
\begin{align*}
    T_F(A_1)_{ijkl}=f_1\paren{\lambda_1}\delta_{ik}\delta_{jl}-f_2\paren{\lambda_1}\frac{\paren{A_1}_{ij}\paren{A_1}_{kl}}{\lambda_1^2},
\end{align*}
It is obvious to obtain the result of $T_F(A)_{ijkl}$.

Next,
\begin{align*}
\begin{split}
    &T_F(A_1)_{ijkl}-T_F(A_2)_{ijkl}
    =\quad\paren{f_1\paren{\lambda_1}-f_1\paren{\lambda_2}}\delta_{ik}\delta_{jl}\\
    &-\paren{f_2\paren{\lambda_1}-f_2\paren{\lambda_2}}\frac{\paren{A_1}_{ij}\paren{A_1}_{kl}}{\lambda_1^2}
    -f_2\paren{\lambda_2}\paren{\frac{\paren{A_1}_{ij}\paren{A_1}_{kl}}{\lambda_1^2}-\frac{\paren{A_2}_{ij}\paren{A_2}_{kl}}{\lambda_2^2}}
\end{split}
\end{align*}
Since
\begin{align*}
    \abs{\lambda_1-\lambda_2}\leq \norm{A_1-A_2}_F\leq C \norm{A_1-A_2},
\end{align*}
we obtain
\begin{align*}
\begin{split}
    \abs{f_i\paren{\lambda_1}-f_i\paren{\lambda_2}}\leq z_M^{(1)} \abs{\lambda_1-\lambda_2}\leq C z_M^{(1)} \norm{A_1-A_2}.
\end{split}
\end{align*}
\begin{align*}
\begin{split}
    &\frac{\paren{A_1}_{ij}\paren{A_1}_{kl}}{\lambda_1^2}-\frac{\paren{A_2}_{ij}\paren{A_2}_{kl}}{\lambda_2^2}\\
    =&\frac{\paren{A_1-A_2}_{ij}\paren{A_1}_{kl}+\paren{A_2}_{ij}\paren{A_1-A_2}_{kl}}{\lambda_1^2}+\frac{\paren{A_2}_{ij}\paren{A_2}_{kl}\paren{\lambda_2+\lambda_1}\paren{\lambda_2-\lambda_1}}{\lambda_1^2\lambda_2^2},
\end{split}
\end{align*}
so
\begin{align*}
    \abs{\frac{\paren{A_1}_{ij}\paren{A_1}_{kl}}{\lambda_1^2}-\frac{\paren{A_2}_{ij}\paren{A_2}_{kl}}{\lambda_2^2}}\leq C \frac{\sigma_1}{\sigma_2^2}\norm{A_1-A_2}
\end{align*}
Therefore,
\begin{align*}
    \abs{T_F(A_1)_{ijkl}-T_F(A_2)_{ijkl}}\leq C\paren{z_M^{(0)}\frac{\sigma_1}{\sigma_2^2}+z_M^{(1)}}\norm{A_1-A_2}
\end{align*}
Since $T_F(A_1),T_F(A_2)$ are linear operators, we may obtain the estimates in $C^\gamma\paren{\mbr}\cap L^2\paren{\mbr}$  of
$T_F(A_1)Z$ and $\paren{T_F(A_1)-T_F(A_2)}Z$.

\end{proof}
Then, because we have estimated $\tilde{\mc{M}}_\alpha$ in Thoerem \ref{Mtilde_bound}, we can obtain the bounds of $\mc{L}_{\alpha, A}$ and its difference $\mc{L}_{\alpha, A_1}-\mc{L}_{\alpha, A_2}$.

\begin{thm}\label{l:regul of LA}
Given a matrix $A_1, A_2\in \mc{DA}_{\sigma_1,\sigma_2}$, then for all $\bm{Y}\in C^{1,\gamma}\paren{\mbr}$ compactly supported,
$\mc{L}_{\alpha, A}\bm{Y}\in  C^\gamma\paren{\mathbb{R}^2}\cap L^2\paren{\mathbb{R}^2}$, and
\begin{align*}
\begin{split}
    \norm{\mc{L}_{\alpha, A}\bm{Y}}_{C^\gamma\paren{\mathbb{R}}}
    \leq \frac{z_M^{(0)}}{\sigma_2}\Big(1+\Big(\frac{\sigma_1}{\sigma_2}\Big)^2\Big)\norm{\nabla \bm{Y}}_{C^\gamma\paren{\mbr}\bigcap L^2\paren{\mbr}},
\end{split}
\end{align*}
\begin{align*}
\begin{split}
    &\norm{\mc{L}_{\alpha, A_1}\bm{Y}-\mc{L}_{\alpha, A_2}\bm{Y}}_{C^\gamma\paren{\mathbb{R}}}\\
    \leq&\frac{C}{\sigma_2}\Big(\frac{z_M^{(0)}}{\sigma_2}\Big(1+\frac{\sigma_1}{\sigma_2}\Big)+z_M^{(1)}\Big)\Big(1+\Big(\frac{\sigma_1}{\sigma_2}\Big)^5\Big)\norm{\nabla \bm{Y}}_{C^\gamma\paren{\mbr}\bigcap L^2\paren{\mbr}}\norm{A_1-A_2}
\end{split}    
\end{align*}
\end{thm}
\begin{proof}
Given $\bm{Y}\in C^{1,\gamma}\paren{\mbr}$ compactly supported, $\nabla\bm{Y}$ is also in $L^2\paren{\mathbb{R}^2}$, so $\paren{T_F(A)\nabla \bm{Y}}\in  C^\gamma\paren{\mathbb{R}^2}\cap L^2\paren{\mathbb{R}^2}$ by Lemma \ref{TF_est}.
By Theorem \ref{Mtilde_bound} and Lemma \ref{TF_est}, $\mc{L}_{\alpha, A}\bm{Y}\in  C^\gamma\paren{\mathbb{R}^2}\cap L^2\paren{\mathbb{R}^2}$ and
\begin{align*}
\begin{split}
    \norm{\mc{L}_{\alpha, A}\bm{Y}}_{C^\gamma\paren{\mathbb{R}}}
    \leq& \frac{C}{\sigma_2}\Big(1+\Big(\frac{\sigma_1}{\sigma_2}\Big)^2\Big)\norm{T_F(A)\nabla \bm{Y}}_{C^\gamma\paren{\mbr}\bigcap L^2\paren{\mbr}}\\
    \leq& \frac{z_M^{(0)}}{\sigma_2}\Big(1+\Big(\frac{\sigma_1}{\sigma_2}\Big)^2\Big)\norm{\nabla \bm{Y}}_{C^\gamma\paren{\mbr}\bigcap L^2\paren{\mbr}}.
\end{split}
\end{align*}
Next, since
\begin{align*}
\begin{split}
    \mc{L}_{\alpha, A_1}\bm{Y}-\mc{L}_{\alpha, A_2}\bm{Y}
    =&\quad\paren{\tilde{\mc{M}}_\alpha\paren{A_1}-\tilde{\mc{M}}_\alpha\paren{A_2}}\paren{T_F(A_1)\nabla \bm{Y}}\\
    &-\tilde{\mc{M}}_\alpha\paren{A_2}\paren{\paren{T_F(A_1)-T_F(A_2)}\nabla \bm{Y}},
\end{split}
\end{align*}
by Theorem \ref{Mtilde_bound} and Lemma \ref{TF_est},
\begin{align*}
\begin{split}
    &\norm{\mc{L}_{\alpha, A_1}\bm{Y}-\mc{L}_{\alpha, A_2}\bm{Y}}_{C^\gamma\paren{\mathbb{R}}}\\
    \leq&\quad C\frac{z_M^{(0)}}{\sigma_2^2}\Big(1+\Big(\frac{\sigma_1}{\sigma_2}\Big)^5\Big)\norm{\nabla \bm{Y}}_{C^\gamma\paren{\mbr}\bigcap L^2\paren{\mbr}}\norm{A_1-A_2}\\
    &+\frac{C}{\sigma_2}\paren{z_M^{(0)}\frac{\sigma_1}{\sigma_2^2}+z_M^{(1)}}\Big(1+\Big(\frac{\sigma_1}{\sigma_2}\Big)^2\Big)\norm{\nabla \bm{Y}}_{C^\gamma\paren{\mbr}\bigcap L^2\paren{\mbr}}\norm{A_1-A_2}
\end{split}
\end{align*}

\end{proof}

Next, we  compute some elementary estimates on the symbol $L_A$.
We denote $G_{\alpha,A}\paren{\thetab}:=G^1\paren{A\thetab}+\alpha G^2\paren{A\thetab}$.
The pairs of eigenvalues and respective eigenvectors of $\paren{\mc{F}_\theta G_{\alpha,A}}(\bm{\xi})$ are
\begin{align}
\paren{\paren{1+\alpha}\frac{\mu}{4}, \bm{v}\paren{\xib}}, \paren{\frac{\mu}{4},\bm{v}^\perp\paren{\xib}}, \paren{\frac{\mu}{4}, \bm{v}_0}, \label{eigen_G}   
\end{align}
and the pairs of $M_A(\bm{\xi})$ \eqref{tesionD_fm} are
\begin{align}
\paren{\frac{\mc{T}}{\lambda}\paren{\abs{\bm{\xi}}^2 -\frac{\abs{A\bm{\xi}}^2}{\lambda^2}}+\D{\mc{T}}{\lambda}\frac{\abs{A\bm{\xi}}^2}{\lambda^2}, A\bm{\xi}}, \paren{\frac{\mc{T}}{\lambda}\abs{\bm{\xi}}^2,U\mc{R}_{\pi/2}\xib}, \paren{\frac{\mc{T}}{\lambda}\abs{\bm{\xi}}^2, \bm{v}_0}. \label{eigen_Z}   
\end{align}

Since $\paren{\mc{F}_\theta G_{\alpha,A}}(\bm{\xi})$ and $M_A(\bm{\xi})$ are symmetric positive definite (s.p.d.), $L_A\paren{\bm{\xi}}$ is diagonalizable and p.d.
Then, we have some estimates of $L_A$ and its derivatives.
\begin{lemma}\label{l:LA_est}
Given $A\in \mc{DA}_{\sigma_1,\sigma_2}$, $L_A$ and its derivatives satisfy 

(i)
\begin{align}
\frac{\sigma_2}{\sigma_1^2}\abs{\bm{\xi}}^{-1}&\leq\mu\leq\frac{\sigma_1}{\sigma_2^2}\abs{\bm{\xi}}^{-1}, \label{mubnd}\\
z_{m}\abs{\bm{\xi}}^2 &\leq \norm{M_A(\bm{\xi})}\leq z_M^{(0)}\abs{\bm{\xi}}^2,\label{Zbnd}\\
\frac{\sigma_2 z_m}{4\sigma_1^2}\abs{\bm{\xi}}&\leq \norm{L_A(\bm{\xi})}\leq \frac{\sigma_1 z_M^{(0)}}{2\sigma_2^2}\abs{\bm{\xi}},\label{LAbnd}
\end{align}
where $z_m=\min_{\sigma_1\leq\lambda\leq\sqrt{2}\sigma_1} \paren{\frac{\mc{T}}{\lambda}, \paren{\frac{\mc{T}}{\lambda}+\D{\mc{T}}{\lambda}}\frac{\sigma_2^2}{\lambda^2}}, z_M^{(0)}=\max_{\sigma_1\leq\lambda\leq\sqrt{2}\sigma_1}  \paren{\frac{\mc{T}}{\lambda}+\D{\mc{T}}{\lambda}}$.

(ii)
\begin{align}
    \norm{\PD{L_{A}}{\xi_k}}&\leq C_1\frac{\sigma_1^2}{\sigma_2^3}z_M^{(0)},\label{LADbnd}\\
    \norm{\frac{\partial^2 L_{A}}{\partial \xi_k\partial \xi_l}}&\leq C_1\frac{\sigma_1^3}{\sigma_2^4}z_M^{(0)}\abs{\xib}^{-1}, \label{LADbnd02}
\end{align}
where $C_1$ is a constant that does not depend on $\alpha$ or $A$.

(iii) $\PD{}{\xi_j}\Delta_{\xib} L_A(\xib), \paren{\Delta_{\xib}}^2 L_A(\xib)$ and $\PD{}{\xi_j}\paren{\Delta_{\xib}}^2 L_A(\xib)$ may be written as
\begin{align}
    \PD{}{\xi_j}\Delta_{\xib} L_A(\xib)=&\frac{1}{|\xib|^2}\Phi^{(3)}_{A,j}(\hat{\xib})\label{LADfactor03},\\
    \paren{\Delta_{\xib}}^2 L_A(\xib)=&\frac{1}{|\xib|^3}\Phi^{(4)}_{A}(\hat{\xib})\label{LADfactor04},\\
    \PD{}{\xi_j}\paren{\Delta_{\xib}}^2 L_A(\xib)=&\frac{1}{|\xib|^4}\Phi^{(5)}_{A,j}(\hat{\xib}),\label{LADfactor05}
\end{align}
where $\Phi^{(3)}_{A,j}, \Phi^{(4)}_{A}, \Phi^{(5)}_{A,j}$ are bounded on $\abs{\hat{\xib}}=1$.
\end{lemma}

\begin{proof}
(i) 
We first note the following inequalities:
\begin{equation}\label{BUbnd}
\sigma_2 \leq \norm{B}\leq \sigma_1, \; \frac{1}{\sigma_1}\leq \norm{B^{-1}}=\norm{U} \leq \frac{1}{\sigma_2}.
\end{equation}
We thus have:
\begin{equation}\label{detBbnd}
\sigma_2^2 \leq \det(B)=\norm{B}\norm{B^{-1}}^{-1}\leq \sigma_1^2,
\end{equation}
where we used the fact that $B$ is a $2\times 2$ symmetric positive definite matrix. From \eqref{BUbnd}, we immediately have:
\begin{equation}
\abs{U\bm{\xi}}=\abs{B^{-1}\bm{\xi}}\geq \frac{1}{\sigma_1}\abs{\bm{\xi}}.\label{Bxibnd}
\end{equation}
Therefore,
\begin{align*}
    \frac{\sigma_2}{\sigma_1^2}\abs{\bm{\xi}}^{-1}\leq\mu=\frac{1}{{\rm det}(B)\abs{B^{-1}\bm{\xi}}}\leq\frac{\sigma_1}{\sigma_2^2}\abs{\bm{\xi}}^{-1}.
\end{align*}
Next, through \eqref{eigen_Z}, one of the eigenvalues of $M_A$ is bounded by
\begin{align*}
     \paren{\frac{\mc{T}}{\lambda}+\D{\mc{T}}{\lambda}}\frac{\sigma_2^2}{\lambda^2}\abs{\bm{\xi}}^2\leq \frac{\mc{T}}{\lambda}\paren{\abs{\bm{\xi}}^2 -\frac{\abs{A\bm{\xi}}^2}{\lambda^2}}+\D{\mc{T}}{\lambda}\frac{\abs{A\bm{\xi}}^2}{\lambda^2} \leq  \paren{\frac{\mc{T}}{\lambda}+\D{\mc{T}}{\lambda}}\abs{\bm{\xi}}^2,
\end{align*}
so we may obtain \eqref{Zbnd}.
Finally, since $L_A$ is diagonalizable and p.d. with $L_A=\mc{F}_\theta G_{\alpha,A} M_A$, the eigenvalues of $L_A$ are between $\frac{\mu}{4} z_m\abs{\bm{\xi}}^2$ and $\frac{\mu}{2} z_M^{(0)}\abs{\bm{\xi}}^2$.
Hence, we get the bound \eqref{LAbnd}.

(ii) 
We now turn to \eqref{LADbnd}. Note that:
\begin{align}
\label{dUinv}
\PD{}{\xi_k}\paren{\frac{1}{\abs{U\bm{\xi}}}}&=-\frac{\paren{U^{\rm T}U\bm{\xi}}_k}{\abs{U\bm{\xi}}^3},\\
\label{d2Uinv}
\PD{}{\xi_k}\PD{}{\xi_l}\paren{\frac{1}{\abs{U\bm{\xi}}}}&=-\frac{\paren{U^{\rm T}U}_{k,l}}{\abs{U\bm{\xi}}^3}+3\frac{\paren{U^{\rm T}U\bm{\xi}}_k\paren{U^{\rm T}U\bm{\xi}}_l}{\abs{U\bm{\xi}}^5}.
\end{align}
Likewise, we have:
\begin{align}
\label{dUunit}
\PD{}{\xi_k}\paren{\frac{(U\bm{\xi})_j}{\abs{U\bm{\xi}}}}&=\frac{U_{jk}}{\abs{U\bm{\xi}}}-\frac{(U\bm{\xi})_j\paren{U^{\rm T}U\bm{\xi}}_k}{\abs{U\bm{\xi}}^3},\\
\nonumber \PD{}{\xi_k}\PD{}{\xi_l}\paren{\frac{(U\bm{\xi})_j}{\abs{U\bm{\xi}}}}&=-\frac{U_{jk}\paren{U^{\rm T}U\bm{\xi}}_l}{\abs{U\bm{\xi}}^3}-\frac{U_{jl}\paren{U^{\rm T}U\bm{\xi}}_k}{\abs{U\bm{\xi}}^3}\\
\label{d2Uunit}
&-\frac{(U\bm{\xi})_j\paren{U^{\rm T}U}_{k,l}}{\abs{U\bm{\xi}}^3}+3\frac{(U\bm{\xi})_j\paren{U^{\rm T}U\bm{\xi}}_k\paren{U^{\rm T}U\bm{\xi}}_l}{\abs{U\bm{\xi}}^5}.
\end{align}
The above relations, together with \eqref{BUbnd}, show that:
\begin{align*}
\abs{\PD{}{\xi_k}\paren{\frac{1}{\abs{U\bm{\xi}}}}}&\leq \frac{\sigma_1^2}{\sigma_2}\frac{1}{\abs{\bm{\xi}}^2},\quad
\abs{\PD{}{\xi_k}\PD{}{\xi_l}\paren{\frac{1}{\abs{U\bm{\xi}}}}}\leq \frac{4\sigma_1^3}{\sigma_2^2}\frac{1}{\abs{\bm{\xi}}^3},\\
\abs{\PD{}{\xi_k}\paren{\frac{(U\bm{\xi})_j}{\abs{U\bm{\xi}}}}}&\leq \frac{2\sigma_1}{\sigma_2}\frac{1}{\abs{\bm{\xi}}}, \quad
\abs{\PD{}{\xi_k}\PD{}{\xi_l}\paren{\frac{(U\bm{\xi})_j}{\abs{U\bm{\xi}}}}}\leq \frac{6\sigma_1^2}{\sigma_2^2}\frac{1}{\abs{\bm{\xi}}^2}.
\end{align*}
Thus, we obtain
\begin{align*}
    \norm{\PD{\mc{F}_\theta G_{\alpha,A}}{\xi_k}}&\leq C\frac{\sigma_1^2}{\sigma_2^3}\abs{\bm{\xi}}^{-2},\\
    \norm{\frac{\partial^2 \mc{F}_\theta G_{\alpha,A}}{\partial \xi_k\partial \xi_l}}&\leq C\frac{\sigma_1^3}{\sigma_2^4}\abs{\bm{\xi}}^{-3}, 
\end{align*}
Next, set $A=[A_1 A_2]$, since
\begin{align*}
    \abs{\PD{}{\xi_k} A\xib}=\abs{A_k}\leq \lambda,
\end{align*}
we have
\begin{align*}
    \norm{\PD{}{\xi_k} \paren{ A\xib \otimes A\xib}}\leq C\sigma_1\lambda \abs{\xib},\\
    \norm{\PD{}{\xi_k}\PD{}{\xi_l}\paren{ A\xib \otimes A\xib}}\leq C\lambda^2.
\end{align*}
Therefore,
\begin{align*}
    \norm{\PD{}{\xi_k} M_A}\leq C \paren{\frac{\mc{T}}{\lambda}+\D{\mc{T}}{\lambda}} \abs{\xib},\\
    \norm{\PD{}{\xi_k}\PD{}{\xi_l} M_A}\leq C \paren{\frac{\mc{T}}{\lambda}+\D{\mc{T}}{\lambda}}.
\end{align*}
The desired bound \eqref{LADbnd} now follows easily by combining the above estimates and $1\leq \frac{\sigma_1}{\sigma_2}$.

(iii)
By lemma \ref{LA_specform}, we may obtain
\begin{align*}
    \begin{split}
        \PD{}{\xi_j}\Delta_{\xib} L_A(\xib)
        =\sum_{i=1,2}\PD{}{\xi_{jii}}L_A(\xib)
        =\sum_{i=1,2}\frac{1}{\abs{\xib}^{2}}\frac{P_{jii}\paren{\hat{\xi}_1,\hat{\xi}_2}}{\abs{U\hat{\xib}}^{9}}.
    \end{split}
\end{align*}
Since $\sigma_2\leq\abs{U\hat{\xib}}\leq \sigma_1$ and $P_{jii}\paren{\hat{\xi}_1,\hat{\xi}_2}$ is a matrix of polynomials on the domain $\abs{\hat{\xib}}=1$,
$\frac{P_{jii}\paren{\hat{\xi}_1,\hat{\xi}_2}}{\abs{U\hat{\xib}}^{9}}$ is bounded.
Therefore, we have
\begin{align*}
    \Phi^{(3)}_{A,j}(\hat{\xib})=\sum_{i=1,2}\frac{P_{jii}\paren{\hat{\xi}_1,\hat{\xi}_2}}{\abs{U\hat{\xib}}^{9}},
\end{align*}
\begin{align*}
    \PD{}{\xi_j}\Delta_{\xib} L_A(\xib)=&\frac{1}{|\xib|^2}\Phi^{(3)}_{A,j}(\hat{\xib}),\\
    \paren{\Delta_{\xib}}^2 L_A(\xib)=&\frac{1}{|\xib|^3}\Phi^{(4)}_{A}(\hat{\xib}),\\
    \PD{}{\xi_j}\paren{\Delta_{\xib}}^2 L_A(\xib)=&\frac{1}{|\xib|^4}\Phi^{(5)}_{A,j}(\hat{\xib}).
\end{align*}
Similarly,
\begin{align*}
    \Phi^{(4)}_{A}(\hat{\xib})=&\sum_{i,k=1,2}\frac{P_{kkii}\paren{\hat{\xi}_1,\hat{\xi}_2}}{\abs{U\hat{\xib}}^{11}},\\
    \Phi^{(5)}_{A,j}(\hat{\xib})=&\sum_{i,k=1,2}\frac{P_{jkkii}\paren{\hat{\xi}_1,\hat{\xi}_2}}{\abs{U\hat{\xib}}^{13}}.
\end{align*}

\end{proof}

\subsection{Some estimates for $z+\mc{L}_{\alpha, A}$ and $\paren{z+\mc{L}_{\alpha, A}}^{-1}$}

Since $L_A\paren{\bm{\xi}}$ is p.d. and diagonalizable,  $P^{-1}\paren{\bm{\xi}} L_A\paren{\bm{\xi}}P\paren{\bm{\xi}} =D\paren{\bm{\xi}}$ where $D$ is a positive diagonal matrix.
Then,
\begin{align*}
\begin{split}
    P^{-1}\paren{\bm{\xi}} \paren{z+ L_A\paren{\bm{\xi}}}P\paren{\bm{\xi}}& =z+D\paren{\bm{\xi}},\\
    P^{-1}\paren{\bm{\xi}} \paren{z+ L_A\paren{\bm{\xi}}}^{-1}P\paren{\bm{\xi}}&= \paren{ z+D\paren{\bm{\xi}}}^{-1},
\end{split}
\end{align*}
and the eigenvalues of $\paren{z+ L_A\paren{\bm{\xi}}}^{-1}$ are on a curve
\begin{align*}
    \left\{\frac{1}{z+a}\Big|\,a\in \Big[\frac{\sigma_1 z_M^{(0)}}{2\sigma_2^2}\abs{\bm{\xi}},\frac{\sigma_2 z_m}{4\sigma_1^2}\abs{\bm{\xi}}\Big]\right\}.
\end{align*}
\begin{remark}
Let $\lambda>0$ and $z\in \mc{S}_{\omega,\delta}$ with $\omega>0$, $\beta:=\pi-$\text{arg}$\paren{z+\lambda-\omega}> \delta$.
Then, we obtain the following  inequality
\begin{align*}
\begin{split}
    \abs{z+\lambda}^2
    =   &\lambda^2+\abs{z}^2-2\abs{z}\lambda\cos \beta\\
    \geq&\lambda^2+\abs{z}^2-2\abs{z}\lambda\cos \delta\\
    =   &\cos\delta \paren{\lambda-\abs{z}}^2+\paren{1-\cos \delta}\paren{\lambda^2+\abs{z}^2}\\
    \geq&\paren{1-\cos \delta}\paren{\lambda^2+\abs{z}^2}.
\end{split}
\end{align*}
\end{remark}
Now, we estimate $\norm{\partial_{\bm{\xi}}^\alpha \paren{z+ L_A\paren{\bm{\xi}}}^{-1}}$ with $\abs{\alpha}\leq 2$.

\begin{lemma}
Given $\mc{S}_{\omega,\delta}$ with $\omega>0$ and $A\in \domainA$, we have the following estimates for all $z\in \mc{S}_{\omega,\delta}$,
\begin{align}
\frac{1}{\abs{z}+\frac{\sigma_1 z_M^{(0)}}{2\sigma_2^2}\abs{\bm{\xi}}}\leq\norm{\paren{z+ L_A\paren{\bm{\xi}}}^{-1}}\leq \frac{2}{\sqrt{\paren{1-\cos \delta}\big((\frac{\sigma_2 z_m}{4\sigma_1^2}\abs{\bm{\xi}})^2+\abs{z}^2}\big)},\label{IzLAbnd}
\end{align}
\begin{align}
\norm{\PD{}{\xi_k}\paren{z+ L_A\paren{\bm{\xi}}}^{-1}}\leq C_1\frac{\sigma_1^2}{\sigma_2^3}z_M^{(0)}\frac{4}{\paren{1-\cos \delta}\big((\frac{\sigma_2 z_m}{4\sigma_1^2}\abs{\bm{\xi}})^2+\abs{z}^2\big)},\label{DIzLAbnd}
\end{align}
and
\begin{align}
    \begin{split}
        \norm{\frac{\partial^2}{\partial \xi_l\partial \xi_k}\paren{z+ L_A\paren{\bm{\xi}}}^{-1}}&\leq C_1^2\frac{\sigma_1^4}{\sigma_2^6}{z_M^{(0)}}^2\frac{16}{\paren{\paren{1-\cos \delta}\paren{\paren{\frac{\sigma_2 z_m}{4\sigma_1^2}\abs{\bm{\xi}}}^2+\abs{z}^2}}^{\frac{3}{2}}}\\
        &\quad+C_1\frac{\sigma_1^3}{\sigma_2^4}z_M^{(0)}\frac{4}{\paren{1-\cos \delta}\abs{\bm{\xi}}\paren{\paren{\frac{\sigma_2 z_m}{4\sigma_1^2}\abs{\bm{\xi}}}^2+\abs{z}^2}}.
    \end{split}\label{D2IzLAbnd}
\end{align}
Moreover, there exists a constant $C_{\delta, \sigma_1, \sigma_2, \mc{T}}$ depending on $\delta, \sigma_1, \sigma_2 $ and $\mc{T}$ s.t. for all $\abs{\alpha}\leq 2$,
\begin{align}
    \norm{\partial_{\bm{\xi}}^\alpha \paren{z+ L_A\paren{\bm{\xi}}}^{-1}}\leq \frac{C_{\delta, \sigma_1, \sigma_2, \mc{T}}}{\abs{z}}\abs{\bm{\xi}}^{-\abs{\alpha}},\label{DDIzLAbnd}
\end{align}
and
\begin{align}
    \norm{\partial_{\bm{\xi}}^\alpha \paren{z+ L_A\paren{\bm{\xi}}}^{-1}}\leq \frac{C_{\delta, \sigma_1, \sigma_2, \mc{T}}}{\abs{z}^2}\abs{\bm{\xi}}^{1-\abs{\alpha}}.\label{DDIzLAbnd02}
\end{align}
\end{lemma}
\begin{proof}
Since the eigenvalues of $\paren{z+ L_A\paren{\bm{\xi}}}^{-1}$ are between
$\Big(z+\frac{\sigma_1 z_M^{(0)}}{2\sigma_2^2}\abs{\bm{\xi}}\Big)^{-1}$ and $\Big(z+\frac{\sigma_2 z_m}{4\sigma_1^2}\abs{\bm{\xi}}\Big)^{-1}$, it follows that
\begin{align*}
     \frac{1}{\abs{z}+\frac{\sigma_1 z_M^{(0)}}{2\sigma_2^2}\abs{\bm{\xi}}}\leq\norm{\paren{z+ L_A\paren{\bm{\xi}}}^{-1}}\leq \frac{2}{\sqrt{\paren{1-\cos \delta}\paren{\paren{\frac{\sigma_2 z_m}{4\sigma_1^2}\abs{\bm{\xi}}}^2+\abs{z}^2}}}.
\end{align*}
Next,
\begin{align*}
    \begin{split}
        \PD{}{\xi_k}\paren{z+ L_A\paren{\bm{\xi}}}^{-1}=-\paren{z+ L_A\paren{\bm{\xi}}}^{-1}\PD{}{\xi_k}L_A\paren{\bm{\xi}}\paren{z+ L_A\paren{\bm{\xi}}}^{-1},
    \end{split}
\end{align*}
so by \eqref{LADbnd},
\begin{align*}
    \norm{\PD{}{\xi_k}\paren{z+ L_A\paren{\bm{\xi}}}^{-1}}\leq C_1\frac{\sigma_1^2}{\sigma_2^3}z_M^{(0)}\frac{4 }{\paren{1-\cos \delta}\paren{\paren{\frac{\sigma_2 z_m}{4\sigma_1^2}\abs{\bm{\xi}}}^2+\abs{z}^2}}.
\end{align*}
Finally,
\begin{align*}
    \begin{split}
         &\quad\frac{\partial^2}{\partial \xi_l\partial \xi_k}\paren{z+ L_A\paren{\bm{\xi}}}^{-1}\\
        =&\quad\paren{z+ L_A\paren{\bm{\xi}}}^{-1}\PD{}{\xi_l}L_A\paren{\bm{\xi}}\paren{z+ L_A\paren{\bm{\xi}}}^{-1}\PD{}{\xi_k}L_A\paren{\bm{\xi}}\paren{z+ L_A\paren{\bm{\xi}}}^{-1}\\
         &+\paren{z+ L_A\paren{\bm{\xi}}}^{-1}\PD{}{\xi_l}L_A\paren{\bm{\xi}}\paren{z+ L_A\paren{\bm{\xi}}}^{-1}\PD{}{\xi_k}L_A\paren{\bm{\xi}}\paren{z+ L_A\paren{\bm{\xi}}}^{-1}\\
         &-\paren{z+ L_A\paren{\bm{\xi}}}^{-1}\frac{\partial^2}{\partial \xi_l\partial \xi_k} L_A\paren{\bm{\xi}}\paren{z+ L_A\paren{\bm{\xi}}}^{-1}.
    \end{split}
\end{align*}
Therefore,
\begin{align*}
    \begin{split}
        \norm{\frac{\partial^2}{\partial \xi_l\partial \xi_k}\paren{z+ L_A\paren{\bm{\xi}}}^{-1}}&\leq C_1^2\frac{\sigma_1^4}{\sigma_2^6}{z_M^{(0)}}^2\frac{16 }{\paren{\paren{1-\cos \delta}\paren{\paren{\frac{\sigma_2 z_m}{4\sigma_1^2}\abs{\bm{\xi}}}^2+\abs{z}^2}}^{\frac{3}{2}}}\\
        &\quad+C_1\frac{\sigma_1^3}{\sigma_2^4}{z_M^{(0)}}\frac{4 }{\paren{1-\cos \delta}\abs{\bm{\xi}}\paren{\paren{\frac{\sigma_2 z_m}{4\sigma_1^2}\abs{\bm{\xi}}}^2+\abs{z}^2}}.
    \end{split}
\end{align*}
From the inequalities
\begin{align*}
    \frac{1}{\sqrt{\paren{\paren{\frac{\sigma_2 z_m}{4\sigma_1^2}\abs{\bm{\xi}}}^2+\abs{z}^2}}}\leq \frac{1}{\abs{z}} \mbox{ and } \frac{1}{\sqrt{\paren{\paren{\frac{\sigma_2 z_m}{4\sigma_1^2}\abs{\bm{\xi}}}^2+\abs{z}^2}}}\leq \frac{1}{\frac{\sigma_2 z_m}{4\sigma_1^2}\abs{\bm{\xi}}},
\end{align*}
we obtain 
\begin{align*}
    \norm{\partial_{\bm{\xi}}^\alpha \paren{z+ L_A\paren{\bm{\xi}}}^{-1}}\leq \frac{C_{\delta, \sigma_1, \sigma_2}}{\abs{z}}\abs{\bm{\xi}}^{-\abs{\alpha}},\\
    \norm{\partial_{\bm{\xi}}^\alpha \paren{z+ L_A\paren{\bm{\xi}}}^{-1}}\leq \frac{C_{\delta, \sigma_1, \sigma_2}}{\abs{z}^2}\abs{\bm{\xi}}^{1-\abs{\alpha}}
\end{align*}
for all $\abs{\alpha}\leq 2$, where the constant $C_{\delta, \sigma_1, \sigma_2, \mc{T}}$ only depends on $\delta, \sigma_1, \sigma_2$ and $\mc{T}$.
\end{proof}
Next, let us consider $\norm{\partial_{\bm{\xi}}^\alpha \abs{\bm{\xi}}\paren{z+ L_A\paren{\bm{\xi}}}^{-1}}$ with $\abs{\alpha}\leq 2$.
\begin{lemma}
Given $\mc{S}_{\omega,\delta}$ with $\omega>0$ and $A\in \domainA$, the following estimates hold for all $z\in \mc{S}_{\omega,\delta}$
\begin{align}
\norm{\abs{\bm{\xi}}\paren{z+ L_A\paren{\bm{\xi}}}^{-1}}\leq \frac{2\abs{\bm{\xi}}}{\sqrt{\paren{1-\cos \delta}\paren{\paren{\frac{\sigma_2 z_m}{4\sigma_1^2}\abs{\bm{\xi}}}^2+\abs{z}^2}}},\label{IxizLAbnd}
\end{align}
\begin{align}
\begin{split}
         \quad\norm{\PD{}{\xi_k}\paren{\abs{\bm{\xi}}\paren{z+ L_A\paren{\bm{\xi}}}^{-1}}}
    &\leq \frac{2}{\sqrt{\paren{1-\cos \delta}\paren{\paren{\frac{\sigma_2 z_m}{4\sigma_1^2}\abs{\bm{\xi}}}^2+\abs{z}^2}}}\\
    &\quad+ C_1\frac{\sigma_1^2}{\sigma_2^3}{z_M^{(0)}}\frac{4\abs{\bm{\xi}}}{\paren{1-\cos \delta}\paren{\paren{\frac{\sigma_2 z_m}{4\sigma_1^2}\abs{\bm{\xi}}}^2+\abs{z}^2}},
\end{split}\label{DIxizLAbnd}
\end{align}
\begin{align}
    \begin{split}
        \norm{\frac{\partial^2}{\partial \xi_l\partial \xi_k}\paren{\abs{\bm{\xi}}\paren{z+ L_A\paren{\bm{\xi}}}^{-1}}}&\leq \frac{4}{\abs{\bm{\xi}}\sqrt{\paren{1-\cos \delta}\paren{\paren{\frac{\sigma_2 z_m}{4\sigma_1^2}\abs{\bm{\xi}}}^2+\abs{z}^2}}}\\
        &\quad+C_1\frac{\sigma_1^3}{\sigma_2^4} {z_M^{(0)}}\frac{12}{\paren{1-\cos \delta}\paren{\paren{\frac{\sigma_2 z_m}{4\sigma_1^2}\abs{\bm{\xi}}}^2+\abs{z}^2}}\\
        &\quad+C_1^2\frac{\sigma_1^4}{\sigma_2^6} {z_M^{(0)}}^2\frac{16\abs{\bm{\xi}}}{\paren{\paren{1-\cos \delta}\paren{\paren{\frac{\sigma_2 z_m}{4\sigma_1^2}\abs{\bm{\xi}}}^2+\abs{z}^2}}^{\frac{3}{2}}}.
    \end{split}\label{D2IxizLAbnd}
\end{align}
Moreover, there exists a constant $C_{\delta, \sigma_1, \sigma_2, \mc{T}}$ depending on $\delta, \sigma_1, \sigma_2$ and $\mc{T}$ s.t. for all $\abs{\alpha}\leq 2$,
\begin{align}
    \norm{\partial_{\bm{\xi}}^\alpha \abs{\bm{\xi}}\paren{z+ L_A\paren{\bm{\xi}}}^{-1}}\leq C_{\delta, \sigma_1, \sigma_2, \mc{T}}\abs{\bm{\xi}}^{-\abs{\alpha}}.\label{DDIxizLAbnd}
\end{align}
\end{lemma}
\begin{proof}
By \eqref{IzLAbnd}, it is easy to obtain
\begin{align*}
\norm{\abs{\bm{\xi}}\paren{z+ L_A\paren{\bm{\xi}}}^{-1}}\leq \frac{2\abs{\bm{\xi}}}{\sqrt{\paren{1-\cos \delta}\paren{\paren{\frac{\sigma_2 z_m}{4\sigma_1^2}\abs{\bm{\xi}}}^2+\abs{z}^2}}}.
\end{align*}
Next,
\begin{align*}
    \PD{}{\xi_k}\paren{\abs{\bm{\xi}}\paren{z+ L_A\paren{\bm{\xi}}}^{-1}}=&\PD{\abs{\bm{\xi}}}{\xi_k}\paren{z+ L_A\paren{\bm{\xi}}}^{-1}+\abs{\bm{\xi}}\PD{}{\xi_k}\paren{z+ L_A\paren{\bm{\xi}}}^{-1}.
\end{align*}
Therefore, with \eqref{IzLAbnd} and \eqref{DIzLAbnd},
\begin{align*}
\begin{split}
     \quad\norm{\PD{}{\xi_k}\paren{\abs{\bm{\xi}}\paren{z+ L_A\paren{\bm{\xi}}}^{-1}}}
    \leq &\quad \frac{2}{\sqrt{\paren{1-\cos \delta}\paren{\paren{\frac{\sigma_2 z_m}{4\sigma_1^2}\abs{\bm{\xi}}}^2+\abs{z}^2}}}\\
    &+ C_1\frac{\sigma_1^2}{\sigma_2^3} {z_M^{(0)}}\frac{4\abs{\bm{\xi}}}{\paren{1-\cos \delta}\paren{\paren{\frac{\sigma_2 z_m}{4\sigma_1^2}\abs{\bm{\xi}}}^2+\abs{z}^2}}.
\end{split}
\end{align*}
Finally,
\begin{align*}
    \begin{split}
        &\quad\frac{\partial^2}{\partial \xi_l\partial \xi_k}\paren{\abs{\bm{\xi}}\paren{z+ L_A\paren{\bm{\xi}}}^{-1}}\\
        =&\quad \frac{\partial^2}{\partial \xi_l\partial \xi_k}\abs{\bm{\xi}}\paren{z+ L_A\paren{\bm{\xi}}}^{-1}+\PD{}{\xi_k}\abs{\bm{\xi}}\PD{}{\xi_l}\paren{z+ L_A\paren{\bm{\xi}}}^{-1}\\
        &+\PD{}{\xi_l}\abs{\bm{\xi}}\PD{}{\xi_k}\paren{z+ L_A\paren{\bm{\xi}}}^{-1}+\abs{\bm{\xi}}\frac{\partial^2}{\partial \xi_l\partial \xi_k}\paren{z+ L_A\paren{\bm{\xi}}}^{-1}.
    \end{split}
\end{align*}
Hence,
\begin{align*}
    \begin{split}
        \norm{\frac{\partial^2}{\partial \xi_l\partial \xi_k}\paren{\abs{\bm{\xi}}\paren{z+ L_A\paren{\bm{\xi}}}^{-1}}}&\leq \frac{4}{\abs{\bm{\xi}}\sqrt{\paren{1-\cos \delta}\paren{\paren{\frac{\sigma_2 z_m}{4\sigma_1^2}\abs{\bm{\xi}}}^2+\abs{z}^2}}}\\
        &\quad+C_1\frac{\sigma_1^3}{\sigma_2^4} {z_M^{(0)}}\frac{12}{\paren{1-\cos \delta}\paren{\paren{\frac{\sigma_2 z_m}{4\sigma_1^2}\abs{\bm{\xi}}}^2+\abs{z}^2}}\\
        &\quad+C_1^2\frac{\sigma_1^4}{\sigma_2^6} {z_M^{(0)}}^2\frac{16\abs{\bm{\xi}}}{\paren{\paren{1-\cos \delta}\paren{\paren{\frac{\sigma_2 z_m}{4\sigma_1^2}\abs{\bm{\xi}}}^2+\abs{z}^2}}^{\frac{3}{2}}}.
    \end{split}
\end{align*}
With $\frac{1}{\sqrt{\paren{\paren{\frac{\sigma_2 z_m}{4\sigma_1^2}\abs{\bm{\xi}}}^2+\abs{z}^2}}}\leq \frac{1}{\frac{\sigma_2 z_m}{4\sigma_1^2}\abs{\bm{\xi}}}$,
we obtain 
\begin{align*}
    \norm{\partial_{\bm{\xi}}^\alpha \paren{z+ L_A\paren{\bm{\xi}}}^{-1}}\leq C_{\delta, \sigma_1, \sigma_2, \mc{T}}\abs{\bm{\xi}}^{-\abs{\alpha}}
\end{align*}
for all $\abs{\alpha}\leq 2$, where the constant $C_{\delta, \sigma_1, \sigma_2, \mc{T}}$ only depends on $\delta, \sigma_1, \sigma_2$ and $\mc{T}$.
\end{proof}
Now, we may prove $\mc{L}_{\alpha, A}$ is a sectorial operator and  obtain the estimate of $\paren{z-\mc{L}_{\alpha, A}}^{-1}$.
\begin{thm}\label{inverse_bounds}
Given a matrix $A$ satisfying the condition \eqref{s1s2} and a constant $K>0$ , there exists $\mc{S}_{\omega,\delta}$ with $\omega,\delta>0$ s.t. for all $z\in \mc{S}_{\omega,\delta}$
\begin{align*}
    \norm{\paren{z-\mc{L}_{\alpha, A}}^{-1}\bm{Y}}_{C^\gamma\paren{\mbr}}\leq \frac{C_{\omega,\delta, \sigma_1, \sigma_2,\mc{T}}}{\abs{z}}\norm{\bm{Y}}_{C^\gamma\paren{\mbr}},
\end{align*}
and
\begin{align*}
    \norm{\paren{z-\mc{L}_{\alpha, A}}^{-1}\bm{Y}}_{C^{1,\gamma}\paren{\mbr}}\leq C_{\omega,\delta, \sigma_1, \sigma_2,\mc{T}}\norm{\bm{Y}}_{C^\gamma\paren{\mbr}},
\end{align*}
for all $\bm{Y}\in C^\gamma(\mbr)\cap L^p(\mbr)$ with $1\leq p\leq 2$.
\end{thm}
\begin{proof}
Since
\begin{align*}
    \mc{L}_{\alpha, A}\bm{Y}(\bm{\theta})=-\mc{F}^{-1}L_{A}(\bm{\xi})\mc{F}\bm{Y},
\end{align*}
the Fourier multiplier of $\paren{z-\mc{L}_{\alpha, A}}^{-1}$ is $\paren{z+L_{A}}^{-1}$ and of $\PD{}{\theta_i}\paren{z-\mc{L}_{\alpha, A}}^{-1}$ is $\xi_i\paren{z+L_{A}}^{-1}$.
With \eqref{DDIzLAbnd} and \eqref{DDIxizLAbnd}, we obtain there exists $C_{\omega,\delta, \sigma_1, \sigma_2,\mc{T}}$ s.t.
\begin{align}
    \jump{\paren{z-\mc{L}_{\alpha, A}}^{-1}\bm{Y}}_{C^\gamma\paren{\mbr}}\leq& \frac{C_{\omega,\delta, \sigma_1, \sigma_2,\mc{T}}}{\abs{z}}\jump{\bm{Y}}_{C^\gamma\paren{\mbr}},\label{HolderzLAest00}\\
    \jump{\paren{z-\mc{L}_{\alpha, A}}^{-1}\bm{Y}}_{C^{1,\gamma}\paren{\mbr}}\leq& C_{\omega,\delta, \sigma_1, \sigma_2,\mc{T}}\jump{\bm{Y}}_{C^\gamma\paren{\mbr}}.\label{HolderzLAest10}
\end{align}
Next, for $\norm{\paren{z-\mc{L}_{\alpha, A}}^{-1}\bm{Y}}_{C^0\paren{\mbr}}$, set $\varphi(\bm{\xi})$ to be a smooth and radial cutting function with a compact support in $\mcbr{1}$ and $\varphi(\bm{\xi})=1$ in a neighborhood of $\bm{\xi}=\bm{0}$.
Then,
\begin{align*}
    \begin{split}
        \norm{\paren{z-\mc{L}_{\alpha, A}}^{-1}\bm{Y}}_{C^0\paren{\mbr}}&=\norm{\mc{F}^{-1}\paren{z+L_{A}}^{-1}(\bm{\xi})\mc{F}\bm{Y}}_{C^0\paren{\mbr}}\\
        &\leq  \norm{\mc{F}^{-1}\paren{z+L_{A}}^{-1}(\bm{\xi})\paren{1-\varphi(\bm{\xi})}\mc{F}\bm{Y}}_{C^0\paren{\mbr}}\\
        &\quad+\norm{\mc{F}^{-1}\paren{z+L_{A}}^{-1}(\bm{\xi})\varphi(\bm{\xi})\mc{F}\bm{Y}}_{C^0\paren{\mbr}}.
    \end{split}
\end{align*}
For the first term, since $1-\varphi(\bm{\xi})=0$ in a neighborhood of $\bm{\xi}=\bm{0}$ and $\abs{1-\varphi(\bm{\xi})}\leq 1$, by Lemma \ref{c0normseminorm}, we obtain
\begin{align*}
\begin{split}
        &\norm{\mc{F}^{-1}\paren{z+L_{A}}^{-1}(\bm{\xi})\paren{1-\varphi(\bm{\xi})}\mc{F}\bm{Y}}_{C^0\paren{\mbr}}\\
    \leq& \minspace C\jump{\mc{F}^{-1}\paren{z+L_{A}}^{-1}(\bm{\xi})\paren{1-\varphi(\bm{\xi})}\mc{F}\bm{Y}}_{C^\gamma\paren{\mbr}}
    \leq \frac{C_{\omega,\delta, \sigma_1, \sigma_2,\mc{T}}}{\abs{z}}\jump{\bm{Y}}_{C^\gamma\paren{\mbr}}.
\end{split}
\end{align*}
For the second term, define the kernel
\begin{align*}
    K_0\paren{\bm{\theta}}:= \mc{F}^{-1}\paren{z+L_{A}}^{-1}(\bm{\xi})\varphi(\bm{\xi}),
\end{align*}
so
\begin{align*}
\begin{split}
        &\norm{\mc{F}^{-1}\paren{z+L_{A}}^{-1}(\bm{\xi})\varphi(\bm{\xi})\mc{F}\bm{Y}}_{C^0\paren{\mbr}}\\
        =&\norm{K_0 * \bm{Y}}_{C^0\paren{\mbr}}\
    \leq \norm{K_0 }_{L^1\paren{\mbr}}\norm{\bm{Y}}_{C^0\paren{\mbr}}.
\end{split}
\end{align*}

Then, we estimate $\norm{K_0 }_{L^1}$, by Lemma \ref{LA_kernelest}, we have
\begin{align*}
    \norm{K_0\paren{\bm{\theta}}}\leq& \frac{C_{\omega,\delta,\sigma_1, \sigma_2,\mc{T}}}{\abs{z}}\frac{1}{1+\abs{\thetab}^3},
\end{align*}
so
\begin{align*}
    \begin{split}
        \norm{K_0 }_{L^1}\leq\frac{C_{\omega,\delta,\sigma_1, \sigma_2,\mc{T}}}{\abs{z}}\int_{\mathbb{R}^2}\frac{1}{1+\abs{\thetab}^3}d\thetab\leq \frac{C_{\omega,\delta,\sigma_1, \sigma_2,\mc{T}}}{\abs{z}}.
    \end{split}
\end{align*}
Therefore,
\begin{align*}
    \norm{\mc{F}^{-1}\paren{z+L_{A}}^{-1}(\bm{\xi})\varphi(\bm{\xi})\mc{F}\bm{Y}}_{C^0\paren{\mbr}}\leq\frac{C_{\omega,\delta, \sigma_1, \sigma_2,\mc{T}}}{\abs{z}}\norm{\bm{Y}}_{C^0\paren{\mbr}},
\end{align*}
so
\begin{align*}
    \norm{\paren{z-\mc{L}_{\alpha, A}}^{-1}\bm{Y}}_{C^\gamma\paren{\mbr}}\leq \frac{C_{\omega,\delta, \sigma_1, \sigma_2,\mc{T}}}{\abs{z}}\norm{\bm{Y}}_{C^\gamma\paren{\mbr}}.
\end{align*}
Similarly, for $\norm{\PD{}{\theta_i}\paren{z-\mc{L}_{\alpha, A}}^{-1}\bm{Y}}_{C^0\paren{\mbr}}$, we may use the above technique with the kernel $K_{1,i}$
\begin{align*}
    K_{1,i}\paren{\bm{\theta}}:=& \mc{F}^{-1}\xi_i\paren{z+L_{A}}^{-1}(\bm{\xi})\varphi(\bm{\xi}),\\
    \norm{K_{1,j}\paren{\bm{\theta}}}\leq& C_{\omega,\delta,\sigma_1, \sigma_2,\mc{T}}\frac{1}{1+\abs{\thetab}^4}
\end{align*}
to obtain
\begin{align*}
    \norm{\PD{}{\theta_i}\paren{z-\mc{L}_{\alpha, A}}^{-1}\bm{Y}}_{C^0\paren{\mbr}}\leq C_{\omega,\delta, \sigma_1, \sigma_2,\mc{T}}\norm{\bm{Y}}_{C^0\paren{\mbr}}.
\end{align*}

Thus,
\begin{align*}
    \norm{\paren{z-\mc{L}_{\alpha, A}}^{-1}\bm{Y}}_{C^{1,\gamma}\paren{\mbr}}\leq C_{\omega,\delta, \sigma_1, \sigma_2,\mc{T}}\norm{\bm{Y}}_{C^\gamma\paren{\mbr}}.
\end{align*}

\end{proof}

\begin{thm}\label{t:sectorial_est}
Given a matrix $A$ in $\mc{DA}_{\sigma_1,\sigma_2}$, there exists $\mc{S}_{\omega,\delta}$ with $\omega,\delta>0$ s.t. for all $z\in \mc{S}_{\omega,\delta}$
\begin{align*}
    \norm{\paren{z-\mc{L}_{\alpha, A}}\bm{Y}}_{C^\gamma\paren{\mbr}}\geq C_{\omega,\delta, \sigma_1, \sigma_2,\mc{T}}\abs{z}\norm{\bm{Y}}_{C^\gamma\paren{\mbr}},
\end{align*}
and
\begin{align*}
    \norm{\paren{z-\mc{L}_{\alpha, A}}\bm{Y}}_{C^{\gamma}\paren{\mbr}}\geq C_{\omega,\delta, \sigma_1, \sigma_2,\mc{T}}\norm{\bm{Y}}_{C^{1,\gamma}\paren{\mbr}},
\end{align*}
for all compactly supported $\bm{Y}\in C^{1,\gamma}\paren{\mbr}$.
\end{thm}
\begin{proof}
Given $\bm{Y}\in C^{1,\gamma}$ with a compact support, by Theorem \ref{l:regul of LA}, $\bm{W}=\paren{ z-\mc{L}_{\alpha, A}}\bm{Y}\in  C^\gamma\paren{\mbr}\cap L^2\paren{\mbr}$.
Through Theorem \ref{inverse_bounds},
\begin{align}
    \norm{\bm{W}}_{C^\gamma\paren{\mbr}}\geq C_{\omega,\delta, \sigma_1, \sigma_2,\mc{T}}^{(1)}\abs{z}\norm{\paren{z-\mc{L}_{\alpha, A}}^{-1}\bm{W}}_{C^\gamma\paren{\mbr}}= C_{\omega,\delta, \sigma_1, \sigma_2,\mc{T}}^{(1)}\abs{z}\norm{\bm{Y}}_{C^\gamma\paren{\mbr}}\label{ineq_semi0}
\end{align}
and
\begin{align}
    \norm{\bm{W}}_{C^\gamma\paren{\mbr}}\geq C_{\omega,\delta, \sigma_1, \sigma_2,\mc{T}}^{(2)}\jump{\paren{z-\mc{L}_{\alpha, A}}^{-1}\bm{W}}_{C^{1,\gamma}\paren{\mbr}}= C_{\omega,\delta, \sigma_1, \sigma_2,\mc{T}}^{(2)}\norm{\bm{Y}}_{C^{1,\gamma}\paren{\mbr}}\label{ineq_semi1}
\end{align}
\end{proof}
In each chart $\hX_n\paren{\thetab}$ and $\bm{Y}_n\paren{\thetab}$, $A$ is $\nabla \bm{X}_n\paren{0}$ and $\rho_n\bm{Y}_n$ is supported in $V_{4R}$ \eqref{def_VR}.
$\mc{L}_{\alpha,A}$ will be used in Proposition \ref{frozen_coef_prop} and \ref{Fsemigroup}.
\begin{remark}\label{arc_chord}
Since 
\begin{align*}
\begin{split}
    A\hxi=\lim_{h\rightarrow0}\frac{\bm{X}_n(h\hxi)-\bm{X}_n(0)}{h},
\end{split}
\end{align*}
it is clear that
\begin{align*}
    \starnorm{\bm{X}}\leq\abs{\bm{X}}_{\circ,n}\leq \liminf_{\xib\rightarrow 0}\frac{\abs{\bm{X}_n\paren{\xib}-\bm{X}_n\paren{0}}}{\abs{\xib}}&\leq \abs{A\hxi},\\
     C\norm{\bm{X}}_{C^{1,\gamma}\paren{\mbs}}\geq\norm{\bm{X}_n}_{C^1\paren{V_{4R}}}\geq \limsup_{\xib\rightarrow 0}\frac{\abs{\bm{X}_n\paren{\xib}-\bm{X}_n\paren{0}}}{\abs{\xib}}&\geq \abs{A\hxi}.
\end{align*}
Thus, we may set $\sigma_1=C \norm{\bm{X}}_{C^{1,\gamma}\paren{\mbs}}, \sigma_2=\starnorm{\bm{X}}$ .
\end{remark}
Next, set $A_0=\nabla \hX_n\paren{0}$, so
\begin{align*}
    A_0=
    \begin{bmatrix}
    2&0\\
    0&2\\
    0&0
    \end{bmatrix},
\end{align*}
and
\begin{align*}
    \mc{L}_{0,A_0}\bm{Y}\paren{\thetab}=-\int_\mbr \PD{}{\eta_k} \frac{1}{4\pi \abs{\thetab}}\bm{W}_k d\etab, 
    L_{0, A_0}\paren{\xib}=\frac{\abs{\xib}}{8}I 
\end{align*}
Set $\mc{L}_A^{(\sigma)}=\paren{1-\sigma}\mc{L}_{0,A_0}+\sigma \mc{L}_{0,A}$, which will be used in Proposition \ref{Fsemigroup} , then
\begin{align*}
    \mc{L}_A^{(\sigma)}\bm{Y}\paren{\thetab}=&-\frac{1}{8\pi}\int_\mbr \PD{}{\eta_k}\paren{ \frac{2\paren{1-\sigma}}{ \abs{\thetab}}+\frac{\sigma}{\abs{A\bm{\theta}}}}\bm{W}_k d\etab,\\
    L_A^{(\sigma)}\paren{\xib}=&\frac{\abs{\xib}}{4}\paren{\frac{1-\sigma}{2}+\sigma\frac{\abs{\xib}}{\det\paren{B}\abs{U\xib}}}
\end{align*}

We just have to adapt the above theorems and their proofs. 
Finally, we obtain
\begin{prop}\label{frozen_coef_prop}
Given $\bm{X}\in C^{1,\gamma}(\mbs)$, and our Stereoghraphic projection charts and the partition functions $\{\hX_n, \rho_n\}$ with the radius $R$, set $\sigma_1=C \norm{\bm{X}}_{C^{1,\gamma}\paren{\mbs}}$,  $\sigma_2=\starnorm{\bm{X}}$.
There exists $\mc{S}_{\omega,\delta}$ with $\omega,\delta>0$ s.t. in each chart $\hX_n$, for $\bm{Y}\in C^{1,\gamma}(\mbs)$, we have the following inequalities:
\begin{equation}\label{continuity_est}
\begin{aligned}
    \norm{\paren{z-\mc{L}_{\alpha, A}}\rho_n\bm{Y}_n}_{C^\gamma\paren{\mbr}}&\geq C_{\omega,\delta, \sigma_1, \sigma_2, \mc{T}}^{(1)}\abs{z}\norm{\rho_n\bm{Y}_n}_{C^\gamma\paren{\mbr}},\\
    \norm{\paren{z-\mc{L}_{\alpha, A}}\rho_n\bm{Y}_n}_{C^{\gamma}\paren{\mbr}}&\geq C_{\omega,\delta, \sigma_1, \sigma_2, \mc{T}}^{(2)}\norm{\rho_n\bm{Y}_n}_{C^{1,\gamma}\paren{\mbr}},\\
    \norm{\paren{z-\mc{L}_A^{(\sigma)}}\rho_n\bm{Y}_n}_{C^{\gamma}\paren{\mbr}}&\geq C_{\omega,\delta, \sigma_1, \sigma_2, \mc{T}}^{(3)}\abs{z}\norm{\rho_n\bm{Y}_n}_{C^{\gamma}\paren{\mbr}},\\
    \norm{\paren{z-\mc{L}_A^{(\sigma)}}\rho_n\bm{Y}_n}_{C^{\gamma}\paren{\mbr}}&\geq C_{\omega,\delta, \sigma_1, \sigma_2, \mc{T}}^{(4)}\norm{\rho_n\bm{Y}_n}_{C^{1,\gamma}\paren{\mbr}},
    \end{aligned}
\end{equation}
where $A=\nabla \bm{X}_n\paren{0}$, and $\sigma, \alpha\in(0,1)$.
\end{prop}

\subsection{Some estimates for $e^{ t\mc{L}_A}$ }
We have two ways of representing the semigroup $e^{ -t\mc{L}_A}$. One is by the Dunford integral
\begin{align}
    e^{ t\mc{L}_A}=\frac{1}{2\pi i}\int_{\omega+\gamma_{r,\eta}}e^{tz}\paren{z+\mc{L}_A}^{-1}dz,\label{D_int}
\end{align}
where $r>0, \delta<\eta<\frac{\pi}{2}$ and the curve $\gamma_{r,\eta}=\{z\in\mathbb{C} : \abs{\rm{arg} z}=\pi-\eta,\abs{z}\geq r  \}\cap \{z\in\mathbb{C} : \abs{\rm{arg} z}\leq\pi-\eta,\abs{z}= r  \}$.
The other is through Fourier transform,
\begin{align*}
    e^{ t\mc{L}_A}\bm{f}\paren{\thetab}:=\mc{F}^{-1}\left[ e^{ -tL_A\paren{\xib}}\mc{F}[\bm{f}]\paren{\xib}\right]\paren{\thetab}=\int_\mbr K_A\paren{t,\bm{\thetab-\etab}}\bm{f}\paren{\etab}d\etab,
\end{align*}
where $K_A\paren{t,\thetab}=\mc{F}^{-1}\left[ e^{ -tL_A\paren{\xib}}\right]\paren{\thetab}$.
\begin{prop}\label{lin_heat_semi}
Given $0\leq t_0\leq t\leq T$ and $0\leq\beta-\alpha<\frac{1}{2}$, we have the following estimates:
\begin{align*}
    \|e^{(t-t_0) \mc{L}_A}\bm{f}(t_0)\|_{C^\beta(\mathbb{R}^2)}&\leq \frac{C}{(t-t_0)^{\beta-\alpha}}\|\bm{f}(t_0)\|_{C^\alpha(\mathbb{R}^2)},\\
    \|e^{(t-t_0) \mc{L}_A}\bm{f}(t_0)\|_{L^2(t_0,T;C^{\beta}(\mathbb{R}^2))}&\leq C\|\bm{f}(t_0)\|_{C^{\alpha}(\mathbb{R}^2)},
\end{align*}
\end{prop}

\begin{proof}
By \eqref{D_int}, Theorem \ref{inverse_bounds} and $0\leq t-t_0\leq T$, we have 
\begin{align*}
\begin{split}
    &\norm{e^{(t-t_0) \mc{L}_A}\bm{f}(t_0)}_{C^\alpha(\mathbb{R}^2)}=\norm{\frac{1}{2\pi i}\int_{\omega+\gamma_{r,\eta}}e^{(t-t_0)z}\paren{z+\mc{L}_A}^{-1}\bm{f}(t_0)dz}_{C^\alpha(\mathbb{R}^2)}\\
    &\leq C_{\omega,\delta, \sigma_1, \sigma_2, \mc{T}}\norm{\bm{f}(t_0)}_{C^\alpha(\mathbb{R}^2)}\int_{\omega+\gamma_{r,\eta}}\abs{e^{(t-t_0)z}} \frac{1}{\abs{z}}  d\abs{z}\\
    &\leq C_{\omega,\delta, \sigma_1, \sigma_2, \mc{T}}\norm{\bm{f}(t_0)}_{C^\alpha(\mathbb{R}^2)}\int_{(t-t_0)\paren{\omega+\gamma_{r,\eta}}}\abs{e^{z}} \frac{1}{\abs{z}}  d\abs{z}\\
    &\leq C_{\omega,\delta, \sigma_1, \sigma_2, \mc{T},T}\norm{\bm{f}(t_0)}_{C^\alpha(\mathbb{R}^2)},
\end{split}
\end{align*}
and
\begin{align*}
\begin{split}
    &\norm{e^{(t-t_0) \mc{L}_A}\bm{f}(t_0)}_{C^{1,\alpha}(\mathbb{R}^2)}
    \leq C_{\omega,\delta, \sigma_1, \sigma_2, \mc{T}}\norm{\bm{f}(t_0)}_{C^\alpha(\mathbb{R}^2)}\int_{\omega+\gamma_{r,\eta}}\abs{e^{(t-t_0)z}}   d\abs{z}\\
    &\leq \frac{C_{\omega,\delta, \sigma_1, \sigma_2, \mc{T}}}{t-t_0}\norm{\bm{f}(t_0)}_{C^\alpha(\mathbb{R}^2)}\int_{(t-t_0)\paren{\omega+\gamma_{r,\eta}}}\abs{e^{z}}   d\abs{z}\\
    &\leq \frac{C_{\omega,\delta, \sigma_1, \sigma_2, \mc{T},T}}{t-t_0}\norm{\bm{f}(t_0)}_{C^\alpha(\mathbb{R}^2)}.
\end{split}
\end{align*}
Then, by interpolation theorem, we obtain
\begin{align*}
    \norm{e^{(t-t_0) \mc{L}_A}\bm{f}(t_0)}_{C^{\beta}(\mathbb{R}^2)}\leq\frac{C_{\omega,\delta, \sigma_1, \sigma_2, \mc{T},T, \beta-\alpha}}{\paren{t-t_0}^{\beta-\alpha}}\norm{\bm{f}(t_0)}_{C^\alpha(\mathbb{R}^2)},
\end{align*}
so
\begin{align*}
    \|e^{(t-t_0) \mc{L}_A}\bm{f}(t_0)\|_{L^2(t_0,T;C^{\beta}(\mathbb{R}^2))}&\leq C_{\omega,\delta, \sigma_1, \sigma_2, \mc{T},T, \beta-\alpha}\|\bm{f}(t_0)\|_{C^{\alpha}(\mathbb{R}^2)}.
\end{align*}
\end{proof}

\begin{prop}\label{lin_heat_semi2}
Given $0\leq t_0\leq t\leq T$ and $0\leq\alpha<1$, we have the following estimates:
\begin{align*}
    \norm{\int_{t_0}^t e^{(t-s)\mc{L}_A}\bm{f}(s)ds}_{C^{1,\alpha}(\mathbb{R}^2)}&\leq C\sup_{t_0\leq s\leq t}\|\bm{f}(s)\|_{C^\alpha(\mathbb{R}^2)},\\
    \norm{\int_{t_0}^t e^{(t-s)\mc{L}_A}\bm{f}(s)ds}_{L^2(t_0,T;C^{1,\alpha}(\mathbb{R}^2))}&\leq C\|\bm{f}\|_{L^2(t_0,T;C^{\alpha}(\mathbb{R}^2))}.
\end{align*}
\end{prop}

\begin{proof}
First, define $\bm{u}$ as
\begin{align*}
    \bm{u}(t,t_0,\thetab):=\bm{u}[\bm{f}]\paren{\thetab} \!=\!\int_{t_0}^t\!\! e^{(t-s)\mc{L}_A} \bm{f}(s,\thetab)ds=\!\int_{t_0}^t \int_\mbr \!\!K_A(t\!-\!s,\thetab\!-\!\etab)\bm{f}(s,\etab)d \etab ds.
\end{align*}
Since $t L_A\paren{\xib}=L_A\paren{t\xib}$, 
\begin{align*}
    K_A(t-s,\bm{x}-\bm{y})=\frac{1}{\paren{t-s}^2}K_A\paren{1,\frac{\thetab-\etab}{t-s}}.
\end{align*}
By \eqref{eLA_est00} in Lemma \ref{eLA_est},
\begin{align*}
\begin{split}
    \abs{\bm{u}(t,t_0,\thetab)}
    =&\abs{\int_{t_0}^t \int_\mbr \frac{1}{\paren{t-s}^2}K_A\paren{1,\frac{\thetab-\etab}{t-s}} \bm{f}(s,\etab)d \etab ds}\\
    \leq& \int_{t_0}^t \int_\mbr \frac{1}{\paren{t-s}^2} \frac{C}{1+\abs{\frac{\thetab-\etab}{t-s}}^3}\abs{\bm{f}(s,\etab)}d \etab ds\\
    \leq& C\norm{\bm{f}}_{C^0}\int_{t_0}^t \int_\mbr  \frac{1}{1+\abs{\thetab-\etab}^3}d \etab ds\leq C\norm{\bm{f}}_{C^0}T.
\end{split}
\end{align*}
Next, we assume $\bm{f}\paren{s,\thetab}=0$ when $s< t_0$, so
\begin{align*}
\begin{split}
    \PD{\bm{u}}{\theta_i}
    =&\int_{t_0}^t \int_\mbr \frac{1}{\paren{t-s}^3}\PD{K_A}{\theta_i}\paren{1,\frac{\thetab-\etab}{t-s}} \bm{f}(s,\etab)d \etab ds\\
    =&\int_{-\infty}^t \int_\mbr \frac{1}{\paren{t-s}^3}\PD{K_A}{\theta_i}\paren{1,\frac{\thetab-\etab}{t-s}} \bm{f}(s,\etab)d \etab ds.
\end{split}
\end{align*}
Through \eqref{eLA_est01},
\begin{multline*}
        \int_\mbr \frac{1}{\paren{t-s}^3}\PD{K_A}{\theta_i}\paren{1,\frac{\thetab-\etab}{t-s}} \bm{f}(s,\etab)d \etab 
        \\
        =
        \int_\mbr \frac{1}{\paren{t-s}^3}\PD{K_A}{\theta_i}\paren{1,\frac{\thetab-\etab}{t-s}} \paren{ \bm{f}(s,\etab)-\bm{f}\paren{s,\thetab}}d \etab,
\end{multline*}
and
\begin{align*}
\begin{split}
    &\abs{\int_\mbr \frac{1}{\paren{t-s}^3}\PD{K_A}{\theta_i}\paren{1,\frac{\thetab-\etab}{t-s}} \paren{ \bm{f}(s,\etab)-\bm{f}\paren{s,\thetab}}d \etab}\\
    &\leq\frac{1}{\paren{t-s}^3}\int_\mbr \frac{\abs{\bm{f}(s,\etab)-\bm{f}\paren{s,\thetab}}}{1+\abs{\frac{\thetab-\etab}{t-s}}^4} d \etab\\
    &\leq C\jump{\bm{f}\paren{s,\cdot}}_{C^\alpha}\frac{1}{\paren{t-s}^3}\int_\mbr \frac{\abs{\thetab-\etab}^\alpha}{1+\abs{\frac{\thetab-\etab}{t-s}}^4} d \etab\\
    &=C\jump{\bm{f}\paren{s,\cdot}}_{C^\alpha}\frac{1}{\paren{t-s}^{1-\alpha}}\int_\mbr \frac{\abs{\thetab-\etab}^\alpha}{1+\abs{\thetab-\etab}^4} d \etab\\
    &=C\jump{\bm{f}\paren{s,\cdot}}_{C^\alpha}\frac{1}{\paren{t-s}^{1-\alpha}}.
\end{split}
\end{align*}
By Lemma \ref{t:integral_est02}, for all $ m\leq t$, we obtain
\begin{align*}
\begin{split}
    &\abs{\int_{m}^t \int_\mbr \frac{1}{\paren{t-s}^3}\PD{K_A}{\theta_i}\paren{1,\frac{\thetab-\etab}{t-s}} \bm{f}(s,\etab)d \etab ds}\\
    \leq &C\int_{m}^t \jump{\bm{f}\paren{s,\cdot}}_{C^\alpha}\frac{1}{\paren{t-s}^{1-\alpha}} ds
    \leq C  \paren{t-m}^\alpha M_l\left[\jump{\bm{f}\paren{s,\cdot}}_{C^\alpha}\right]\paren{t},
\end{split}
\end{align*}
so set $m=0$,
\begin{align*}
    \abs{\PD{\bm{u}}{\theta_i}(t,t_0,\thetab)}\leq C T^\alpha M_l\left[\jump{\bm{f}\paren{s,\cdot}}_{C^\alpha}\right]\paren{t}.
\end{align*}
For $\jump{\PD{\bm{u}}{\theta_i}}_{C^\alpha}$, we first compute 
\begin{align*}
\begin{split}
    \frac{\partial^2}{\partial \theta_i \partial \theta_j}\int_{-\infty}^M \int_\mbr& K_A(t-s,\thetab-\etab) \bm{f}(s,\etab)d \etab ds\\
    &=\int_{\infty}^M \int_\mbr \frac{1}{\paren{t-s}^4}\frac{\partial^2 K_A}{\partial \theta_i \partial_j}\paren{1,\frac{\thetab-\etab}{t-s}} \bm{f}(s,\etab)d \etab ds,
\end{split}
\end{align*}
where $ M< t$.
Then, by \eqref{eLA_est02} and Lemma \ref{t:integral_est02}, we have
\begin{align*}
\begin{split}
   \abs{\int_\mbr \!\!\frac{1}{\paren{t\!-\!s}^4}\frac{\partial^2 K_A}{\partial \theta_i \partial_j}\Big(1,\frac{\thetab-\etab}{t\!-\!s}\Big) \bm{f}(s,\etab)d \etab} &\leq  C\jump{\bm{f}\paren{s,\cdot}}_{C^\alpha}\frac{1}{\paren{t\!-\!s}^{2-\alpha}}\int_\mbr \!\frac{\abs{\thetab\!-\!\etab}^\alpha}{1\!+\!\abs{\thetab\!-\!\etab}^4} d \etab\\
    &\leq C\jump{\bm{f}\paren{s,\cdot}}_{C^\alpha}\frac{1}{\paren{t-s}^{2-\alpha}}.
\end{split}
\end{align*}
and
\begin{align*}
\begin{split}
    \abs{\frac{\partial^2}{\partial \theta_i \partial \theta_j}\int_{-\infty}^M \int_\mbr K_A(t-s,\thetab-\etab) \bm{f}(s,\etab)d \etab ds}\leq C \paren{t-M}^{\alpha-1} M_l\left[\jump{\bm{f}\paren{s,\cdot}}_{C^\alpha}\right]\paren{t}.
\end{split}
\end{align*}
When $\abs{\thetab-\etab}=1$, we set a cutting function $\phi(s)$ s.t. $\phi(s)=1$ on $[-\infty,-2]$ and  $\phi(s)=0$ on $[-1, \infty]$, and define $\phi^{(t)}(s)=\phi(s-t)$.
We obtain
\begin{align*}
\begin{split}
    &\abs{\PD{\bm{u}}{\theta_i}(t,t_0,\thetab)-\PD{\bm{u}}{\eta_i}(t,t_0,\etab)}
    =\abs{\PD{\bm{u}}{\theta_i}[\bm{f}](\thetab)-\PD{\bm{u}}{\eta_i}[\bm{f}](\etab)}\\
    &\leq \abs{\PD{\bm{u}}{\theta_i}[\phi^{(t)}\bm{f}](\thetab)-\PD{\bm{u}}{\eta_i}[\phi^{(t)}\bm{f}](\etab)}\!+\!\abs{\PD{\bm{u}}{\theta_i}[(1\!-\!\phi^{(t)})\bm{f}](\thetab)}\!+\!\abs{\PD{\bm{u}}{\eta_i}[(1\!-\!\phi^{(t)})\bm{f}](\etab)}\\
    &\leq  C (t\!-\!M)^{\alpha-1} M_l\big[\jump{\phi^{(t)}\bm{f}(s,\cdot)}_{C^\alpha}\big](t)\!\!+C (t\!-\!m)^{\alpha} M_l\big[\jump{(1\!-\!\phi^{(t)})\bm{f}(s,\cdot)}_{C^\alpha}\big](t),
\end{split}
\end{align*}
where $M=t-1$ and $m=t-2$.
Therefore, since $0\leq\phi\leq 1$ and $\jump{\bm{f}\paren{s,\cdot}}_{C^\alpha}\geq 0$, 
\begin{align*}
\begin{split}
    \abs{\PD{\bm{u}}{\theta_i}(t,t_0,\thetab)-\PD{\bm{u}}{\eta_i}(t,t_0,\etab)}
    \leq C M_l\left[\jump{\bm{f}\paren{s,\cdot}}_{C^\alpha}\right]\paren{t}.
\end{split}
\end{align*}
When $\rho=\abs{\thetab-\etab}\neq 1$, we define $\bm{u}_\rho\paren{t,t_0, \theta}=\frac{1}{\rho}\bm{u}\paren{\rho t,\rho t_0, \rho\theta}, \bm{f}_\rho\paren{t,\theta}=\bm{f}\paren{\rho t,\rho\theta}$ and $\bar{\thetab}=\frac{\thetab}{\rho},\bar{\etab}=\frac{\etab}{\rho}, \bar{t}=\frac{t}{\rho},\bar{t_0}=\frac{t_0}{\rho}$.
We have
\begin{align*}
    \PD{\bm{u}_\rho}{t}\paren{t,\theta}=&\mc{L}_A \bm{u}_\rho\paren{t,\theta}+ \bm{f}_\rho \paren{t,\theta},\\
    \PD{\bm{u}_\rho}{\theta_i}\paren{t,\theta}=&\PD{\bm{u}}{\theta_i}\paren{\rho t,\rho\theta},
\end{align*}
so
\begin{align*}
\begin{split}
    \Big|\PD{\bm{u}}{\theta_i}(t,t_0,\thetab)\!-\!\PD{\bm{u}}{\eta_i}(t,t_0,\etab)\Big|
    \!=\!\Big|\PD{\bm{u}_\rho}{\bar{\theta}_i}(\bar{t},\bar{t_0},\bar{\thetab})\!-\!\PD{\bm{u}_\rho}{\bar{\eta}_i}(\bar{t},\bar{t_0},\bar{\etab})\Big|
    \!\leq\! C M_l[\jump{\bm{f}_\rho\paren{\bar{s},\cdot}}_{C^\alpha}](t).
\end{split}
\end{align*}
Since $\jump{\bm{f}_\rho}_{C^\alpha}\paren{\bar{s}}=\rho^\alpha \jump{\bm{f}}_{C^\alpha}\paren{\rho \bar{s}}$,
\begin{align*}
\begin{split}
    &M_l\left[\jump{\bm{f}_\rho\paren{\bar{s},\cdot}}_{C^\alpha}\right]\paren{\bar{t}}
    =\sup_{\bar{r}>0}\frac{1}{\bar{r}}\int_{\bar{t}-\bar{r}}^{\bar{t}}\jump{\bm{f}_\rho}_{C^\alpha}\paren{\bar{s}}d\bar{s}
    =\rho^\alpha\sup_{\bar{r}>0}\frac{1}{\bar{r}}\int_{\bar{t}-\bar{r}}^{\bar{t}}\jump{\bm{f}}_{C^\alpha}\paren{\rho \bar{s}}d\bar{s}\\
    =&\rho^\alpha\sup_{r>0}\frac{1}{r}\int_{t-r}^{t}\jump{\bm{f}}_{C^\alpha}\paren{s}ds
    =\rho^\alpha M_l\left[\jump{\bm{f}\paren{s,\cdot}}_{C^\alpha}\right]\paren{t}.
\end{split}   
\end{align*}
Hence,
\begin{align*}
    \abs{\PD{\bm{u}}{\theta_i}(t,t_0,\thetab)-\PD{\bm{u}}{\eta_i}(t,t_0,\etab)}
    \leq C M_l\left[\jump{\bm{f}\paren{s,\cdot}}_{C^\alpha}\right]\paren{t}\abs{\thetab-\etab}^\alpha,
\end{align*}
and by the Hardy-Littlewood maximal function theorem, for all $1\leq p\leq \infty$
\begin{align*}
\begin{split}
    \norm{\norm{\bm{u}}_{C^\gamma\paren{\mbr}}}_{L^p\paren{t_0,T}}
    \leq& \norm{C\norm{\bm{f}}_{C^0\paren{\mbr}} +C M_l\left[\jump{\bm{f}\paren{s,\cdot}}_{C^\alpha\paren{\mbr}}\right]\paren{t}}_{L^p\paren{t_0,T}}\\
    \leq& C\norm{\norm{\bm{f}}_{C^0\paren{\mbr}}}_{L^p\paren{t_0,T}} +C\norm{ M_l\left[\jump{\bm{f}\paren{s,\cdot}}_{C^\alpha\paren{\mbr}}\right]\paren{t}}_{L^p\paren{t_0,T}}\\
    \leq& C\norm{\norm{\bm{f}}_{C^0\paren{\mbr}}}_{L^p\paren{t_0,T}} +C\norm{ \jump{\bm{f}\paren{t,\cdot}}_{C^\alpha\paren{\mbr}}}_{L^p\paren{t_0,T}}\\
    \leq& C\norm{\norm{\bm{f}}_{C^\gamma\paren{\mbr}}}_{L^p\paren{t_0,T}}.
\end{split}
\end{align*}

\end{proof}

\section{Local well-posedness}\label{sec: local well-posed}

We write the Peskin problem as an evolution equation 
\begin{equation}\label{Peskin_3D_F}
    \PD{\bm{X}}{t}=F(\bm{X}),\quad t>0, \quad \bm{X}(0)=\bm{X}_0,
\end{equation}
where $F(\bm{X})$ is given in \eqref{Xeq_F}.
We will make use of Theorem 8.4.1 in \cite{Lunardi:analytic-semigroups-optimal-parabolic}:
\begin{thm}\label{Lunardi_thm}
Let $E_1\subset E_0\subset E$ be Banach spaces and let $0<\sigma<1$. Given $T>0$, open set $\mathcal{O}_1\subset E_1$ and a function
\begin{equation*}
    F:[0,T]\times \mathcal{O}_1\mapsto E_0, \quad (t,u)\mapsto F(t,u)
\end{equation*}
such that $F$ and $F_u$ are continuous in $[0,T]\times \mathcal{O}_1$. If for every $(\bar{t},\bar{u})\in [0,T]\times \mathcal{O}_1$ we have $F_u(\bar{t},\bar{u}):E_1\mapsto E_0$ is the part of a sectorial operator $S:D(S)\subset E\mapsto E$ with $D_S(\sigma)\simeq E_0$ and $D_S(\sigma+1)\simeq E_1$, then for every $\bar{t}\in[0,t]$ and $\bar{u}\in\mathcal{O}_1$ there are $\delta>0$, $r>0$ such that if $t_0\in[0,T)$, $|t_0-\bar{t}|\leq \delta$, and $\|u_0-\bar{u}\|\leq r$ then the problem
\begin{equation*}
v'(t)=F(t,v(t)), \quad t_0\leq t\leq t_0+\delta, \quad v(t_0)=u_0,
\end{equation*}
has a unique solution $v\in C([t_0,t_0+\delta];E_1)\cap C^1([t_0;t_0+\delta];E_0).$
\end{thm}
Then, our main result is the following Theorem:
\begin{thm}\label{MainTh}
Consider the 3D Peskin problem \eqref{Peskin_3D_F} with initial data  satisfying $\bm{X}_0\in h^{1,\gamma}(\mathbb{S}^2)$, $|\bm{X}_0|_*>0$, and $\mc{T}\in C^3$ such that $\mc{T}>0$, $d\mc{T}/d\lambda\geq0$. Then, there exists some time $T>0$ such that \eqref{Peskin_3D_F} has a unique solution $\bm{X}$,
\begin{equation*}
    \bm{X}\in C([0,T];h^{1,\gamma}(\mathbb{S}^2))\cap C^1([0,T];h^\gamma(\mathbb{S}^2)).
\end{equation*}
\end{thm}
\begin{proof}
Let $\mathcal{O}_m=\{\bm{Y}\in h^{1,\gamma}(\mathbb{S}^2):|\bm{Y}|_*\geq m>0\}$, $E_1=h^{1,\gamma}(\mathbb{S}^2)$, $E_0=h^{\gamma}(\mathbb{S}^2)$, and $E=h^{\alpha}(\mathbb{S}^2)$, with $0<\alpha<\gamma$. Define the operator $\mc{S}$ as the linearization of $F$ \eqref{F_expr} around $\bm{X}_0$:
\begin{equation*}
    \mc{S}(\bm{X}_0)\bm{Y}:=\partial_{\bm{X}}F(\bm{X}_0)\bm{Y}=\frac{d}{d\varepsilon}F(\bm{X}_0+\varepsilon\bm{Y})|_{\varepsilon=0}.
\end{equation*}
Since $\bm{X}_0\in\mc{O}_m$ is arbitrary, we can study the G\^ateaux derivative of $F$ at any $\bm{X}\in \mathcal{O}_m$, which is given by
\begin{equation}\label{Soperator}
    \begin{aligned}
    \mc{S}(\bm{X})\bm{Y}&=\mc{S}^1(\bm{X})\bm{Y}+\mc{S}^2(\bm{X})\bm{Y},
    \end{aligned}
\end{equation}
with
\begin{equation*}
    \begin{aligned}
    \mc{S}^1(\bm{X})\bm{Y}&=-\int_{\mathbb{S}^2} \nabla_{\mathbb{S}^2}G(\bm{X}(\hx)-\bm{X}(\hy))\cdot\\
    &\qquad\times\frac{d}{d\varepsilon}\Big(T(|\nabla_{\mathbb{S}^2}(\bm{X}(\hy)+\varepsilon\bm{Y}(\hy))|)(\nabla_{\mathbb{S}^2}(\bm{X}+\varepsilon\bm{Y}))(\hy)\Big)|_{\varepsilon=0}d\hy,
    \end{aligned}
\end{equation*}
\begin{equation*}
    \begin{aligned}
    \mc{S}^2(\bm{X})\bm{Y}&=-\int_{\mathbb{S}^2}\!\! \nabla_{\mathbb{S}^2}\frac{d}{d\varepsilon}\Big(G(\bm{X}(\hx)\!-\!\bm{X}(\hy)\!+\!\varepsilon(\bm{Y}(\hx)-\bm{Y}(\hy)))\Big)\Big|_{\varepsilon=0}\!\!\cdot\\
    &\qquad\times T(|\nabla_{\mathbb{S}^2}\bm{X}(\bm{\hy})|)\nabla_{\mathbb{S}^2}\bm{X}(\bm{\hy})d\hy.
    \end{aligned}
\end{equation*}
It remains to check that the hypothesis of Theorem \ref{Lunardi_thm} are satisfied, which follow from Propositions 
\ref{Festimate}, \ref{FGateauxEstimate}, and \ref{Fsemigroup} below.

\end{proof}

\begin{prop}\label{Festimate} If $m>0$ and $\gamma\in(0,1)$, $\mc{T}\in C^2$, then $F$ \eqref{F_expr} is a continuous map from $\mathcal{O}_m\subset h^{1,\gamma}(\mathbb{S}^2)$ to $h^\gamma(\mathbb{S}^2)$. 
\end{prop}
\begin{proof}
Given that $F(\bm{X})=\mc{N}(\bm{X})(T(|\nabla_{\mathbb{S}^2}\bm{X}|)\nabla_{\mathbb{S}^2}\bm{X})$ \eqref{Ndef}, we apply Proposition \ref{Nbound} to obtain that
\begin{equation*}
    \begin{aligned}
    \|F(\bm{X})\|_{C^\gamma(\mathbb{S}^2)}&\leq C\frac{1}{|\bm{X}|_*}\Big(1+\Big(\frac{\|\nabla_{\mathbb{S}^2}\bm{X}\|_{C^0(\mathbb{S}^2)}}{|\bm{X}_*|}\Big)^2\Big)\|T(|\nabla_{\mathbb{S}^2}\bm{X})|)\nabla_{\mathbb{S}^2}\bm{X})\|_{C^\gamma(\mathbb{S}^2)},
    \end{aligned}
\end{equation*}
hence recalling the expression for $T$ \eqref{Taulaw}, the bound above yields that
\begin{equation}\label{FHolderbound}
    \begin{aligned}
    \|F(\bm{X})\|_{C^\gamma(\mathbb{S}^2)}&\leq C(|\bm{X}|_*,\|\nabla_{\mathbb{S}^2}\bm{X}\|_{C^0(\mathbb{S}^2)},\|\mc{T}\|_{C^1})\|\nabla_{\mathbb{S}^2}\bm{X}\|_{C^\gamma(\mathbb{S}^2)}.
    \end{aligned}
\end{equation}

We have thus proved that $F$ maps $C^{1,\gamma}(\mathbb{S}^2)$ to $C^\gamma(\mathbb{S}^2)$. We need to show that it also maps $h^{1,\gamma}(\mathbb{S}^2)$ to $h^\gamma(\mathbb{S}^2)$. Having the estimate \eqref{FHolderbound}, it suffices to show that if $\bm{X}\in h^{1,\gamma}(\mathbb{S}^2)$, then $F(\bm{X})\in h^\gamma(\mathbb{S}^2)$. Since $h^{k,\gamma}(\mathbb{S}^2)$ is the completion of $C^{k,\gamma}(\mathbb{S}^2)$ in any $C^{k,\alpha}(\mathbb{S}^2)$ with $0<\gamma<\alpha<1$, $k\geq0$, let $\bm{X}\in h^{1,\gamma}(\mathbb{S}^2)$, and $\{\bm{X}^m\}_m$ a sequence $\bm{X}^m \in C^{1,\alpha}(\mathbb{S}^2)$, $\alpha>\gamma$, such that $\bm{X}^m\rightarrow \bm{X}$ in $C^{1,\gamma}(\mathbb{S}^2)$. It is clear that the previous estimate \eqref{FHolderbound} also holds replacing $\gamma$ by $\alpha$, thus $F(\bm{X}^m)\in C^\alpha$. We conclude that $F(\bm{X})\in h^{\gamma}(\mathbb{S}^2)$ by showing that
\begin{equation}\label{continuity_X}
    \begin{aligned}
     \|F(\bm{X}^m)-F(\bm{X})\|_{C^\gamma(\mathbb{S}^2)}\leq C \|\bm{X}^m-\bm{X}\|_{C^{1,\gamma}(\mathbb{S}^2)}.
    \end{aligned}
\end{equation}
The estimate will follow from the previous ones by writing $F(\bm{X}^m)-F(\bm{X})$ as follows:
\begin{equation}\label{continuity_X2}
    \begin{aligned}
    F(\bm{X}^m)(\hx)-F(\bm{X})(\hx)&=\Delta_1(\hx)+\Delta_2(\hx),
    \end{aligned}
\end{equation}
with
\begin{equation*}
    \begin{aligned}
    \Delta_1(\hx)&=-\int_{\mathbb{S}^2}\!\! \nabla_{\mathbb{S}^2}G(\bm{X}^m(\hx)\!-\!\bm{X}^m(\hy))\cdot\\
    &\qquad\times\big(T(\nabla_{\mathbb{S}^2}\bm{X}^m(\bm{\hy}))\nabla_{\mathbb{S}^2}\bm{X}^m(\bm{\hy})\!-\!T(\nabla_{\mathbb{S}^2}\bm{X}(\bm{\hy}))\nabla_{\mathbb{S}^2}\bm{X}(\bm{\hy})\big)d\hy,\\
    \Delta_2(\hx)&=\int_{\mathbb{S}^2}\!\! \nabla_{\mathbb{S}^2}\Big(G(\bm{X}^m(\hx)\!-\!\bm{X}^m(\hy))\!-\!G(\bm{X}(\hx)\!-\!\bm{X}(\hy))\Big)\cdot\\
    &\qquad\times T(\nabla_{\mathbb{S}^2}\bm{X}(\bm{\hy}))\nabla_{\mathbb{S}^2}\bm{X}(\bm{\hy})d\hy,
\end{aligned}
\end{equation*}
where both terms have kernels given by a derivative.
The first term has thus already been treated,
\begin{equation}\label{Delta1_bound}
    \begin{aligned}
    &\|\Delta_1\|_{C^\gamma(\mathbb{S}^2)}\\
    &\quad\leq C(\|\nabla_{\mathbb{S}^2} \bm{X}^m\|_{C^0(\mathbb{S}^2)},|\bm{X}^m|_*)\|T(\nabla_{\mathbb{S}^2}\bm{X}^m)\nabla_{\mathbb{S}^2}\bm{X}^m\!-\!T(\nabla_{\mathbb{S}^2}\bm{X})\nabla_{\mathbb{S}^2}\bm{X}\|_{C^\gamma(\mathbb{S}^2)}\\
    &\quad\leq \!C(\|\nabla_{\mathbb{S}^2} \bm{X}^m\|_{C^0(\mathbb{S}^2)},|\bm{X}^m|_*,\|\nabla_{\mathbb{S}^2} \bm{X}\|_{C^0(\mathbb{S}^2)},|\bm{X}|_*,\|\mc{T}\|_{C^2})\!\Big(\!\|\nabla_{\mathbb{S}^2} (\bm{X}^m\!\!-\!\bm{X})\|_{C^\gamma(\mathbb{S}^2)}\\
    &\qquad+(\|\nabla_{\mathbb{S}^2} \bm{X}\|_{C^\gamma(\mathbb{S}^2)}+\|\nabla_{\mathbb{S}^2} \bm{X}^m\|_{C^\gamma(\mathbb{S}^2)})\|\nabla_{\mathbb{S}^2}(\bm{X}^m-\bm{X})\|_{C^0(\mathbb{S}^2)}\Big),
    \end{aligned}
\end{equation}
while the second one can be estimated in a similar manner by noticing that one can always extract $\bm{X}^m-\bm{X}$ from the difference of kernels. Consider for example the kernel $q^1_{k,l}$ \eqref{q_kernels}, 
\begin{equation*}
    \begin{aligned}
    q^1_{k,l}(\hx,\hy)=-\frac{1}{8\pi}\frac{\delta_{\hy}X_i(\hx)}{|\delta_{\hy}\bm{X}(\hx)|^3}\nabla_{\mathbb{S}^2}X_i(\hy)\delta_{k,l}.
    \end{aligned}
\end{equation*}
We can write
\begin{equation*}
    \begin{aligned}
    \frac{1}{8\pi}&\frac{\delta_{\hy}X_i^m(\hx)}{|\delta_{\hy}\bm{X}^m(\hx)|^3}\nabla_{\mathbb{S}^2}X_i^m(\hy)\delta_{k,l}-\frac{1}{8\pi}\frac{\delta_{\hy}X_i(\hx)}{|\delta_{\hy}\bm{X}(\hx)|^3}\nabla_{\mathbb{S}^2}X_i(\hy)\delta_{k,l}\\
    &=\frac{1}{8\pi}\frac{\delta_{\hy}(X_i^m-X_i)(\hx)\nabla_{\mathbb{S}^2}X_i^m(\hy)}{|\delta_{\hy}\bm{X}^m(\hx)|^3}+\frac{1}{8\pi}\frac{\delta_{\hy}X_i(\hx)\nabla_{\mathbb{S}^2}(X_i^m-X_i)(\hy)}{|\delta_{\hy}\bm{X}^m(\hx)|^3}\\
    &\quad+\frac{1}{8\pi}\delta_{\hy}X_i(\hx)\nabla_{\mathbb{S}^2}X_i(\hy)\Big(\frac{1}{|\delta_{\hy}\bm{X}^m(\hx)|^3}-\frac{1}{|\delta_{\hy}\bm{X}(\hx)|^3}\Big).
    \end{aligned}
\end{equation*}
Therefore, it holds that
\begin{equation*}
    \begin{aligned}
       \|\Delta_2\|_{C^\gamma(\mathbb{S}^2)}&\leq C(\|\nabla_{\mathbb{S}^2} \bm{X}\|_{C^0(\mathbb{S}^2)},|\bm{X}|_*,\|\nabla_{\mathbb{S}^2} \bm{X}^m\|_{C^0(\mathbb{S}^2)},|\bm{X}^m|_*,\|\mc{T}\|_{C^1})\\
       &\hspace{2.5cm}\times\|\nabla_{\mathbb{S}^2} \bm{X}\|_{C^\gamma(\mathbb{S}^2)}\|\nabla_{\mathbb{S}^2} (\bm{X}^m-\bm{X})\|_{C^0(\mathbb{S}^2)},
    \end{aligned}
\end{equation*}
which together with \eqref{Delta1_bound} proves \eqref{continuity_X}.

\end{proof}

\begin{prop}\label{FGateauxEstimate}
If $m>0$ and $\gamma\in(0,1)$, $\mc{T}\in C^3$, then the G\^ateaux derivative of $F$ at any $\bm{X}\in \mathcal{O}_m\subset h^{1,\gamma}(\mathbb{S}^2)$ \eqref{Soperator} is continuous and maps  $h^{1,\gamma}(\mathbb{S}^2)$ to $h^\gamma(\mathbb{S}^2)$.
\end{prop}
\begin{proof}

The first term $\mc{S}^1(\bm{X})\bm{Y}$ in the G\^ateaux derivative of $F$ \eqref{Soperator} is given in terms of the operator $\mc{N}(\bm{X})$ \eqref{Ndef}, 
\begin{equation}\label{S1N1}
\begin{aligned}
    \mc{S}^1(\bm{X})\bm{Y}(\hx)&=\mc{N}(\bm{X})(T_S(\nabla_{\mathbb{S}^2}\bm{X})\nabla_{\mathbb{S}^2}\bm{Y})(\hx),
\end{aligned}
\end{equation}
with $T_S$ given by
\begin{equation}\label{Ttilde_law}
    \begin{aligned}
    T_S(\nabla_{\mathbb{S}^2}\bm{X})&\!=\!\frac{\mc{T}(|\nabla_{\mathbb{S}^2}\bm{X}|)}{|\nabla_{\mathbb{S}^2}\bm{X}|}\!+\!\Big(\mc{T}'(|\nabla_{\mathbb{S}^2}\bm{X}|)\!-\!\frac{\mc{T}(|\nabla_{\mathbb{S}^2}\bm{X}|)}{|\nabla_{\mathbb{S}^2}\bm{X}|}\Big)\frac{\nabla_{\mathbb{S}^2}\bm{X}\otimes \nabla_{\mathbb{S}^2}\bm{X}}{|\nabla_{\mathbb{S}^2}\bm{X}|^2},
    \end{aligned}
\end{equation}
and, in index notation, 
\begin{equation*}
    \begin{aligned}
    (T_S(\nabla_{\mathbb{S}^2}\bm{X})\nabla_{\mathbb{S}^2}\bm{Y})_{l,i}&=\frac{\mc{T}(|\nabla_{\mathbb{S}^2}\bm{X}|)}{|\nabla_{\mathbb{S}^2}\bm{X}|}(\nabla_{\mathbb{S}^2}\bm{Y})_{l,i}\\
    &\hspace{-0.6cm}+\Big(\mc{T}'(|\nabla_{\mathbb{S}^2}\bm{X}|)\!-\!\frac{\mc{T}(|\nabla_{\mathbb{S}^2}\bm{X}|)}{|\nabla_{\mathbb{S}^2}\bm{X}|}\Big)\frac{(\nabla_{\mathbb{S}^2}\bm{X})_{l,i} (\nabla_{\mathbb{S}^2}\bm{X})_{q,m}}{|\nabla_{\mathbb{S}^2}\bm{X}|^2}(\nabla_{\mathbb{S}^2}\bm{Y})_{q,m}.
    \end{aligned}
\end{equation*}
Proposition \ref{Nbound} then gives that
\begin{equation}\label{S1bound}
\begin{aligned}
    \|\mc{S}^1(\bm{X})\bm{Y}&\|_{C^\gamma(\mathbb{S}^2)}\leq C(|\bm{X}|_*,\|\nabla_{\mathbb{S}^2}\bm{X}\|_{C^0(\mathbb{S}^2)})\|T_S(\nabla_{\mathbb{S}^2} \bm{X})\nabla_{\mathbb{S}^2} \bm{Y}\|_{C^\gamma(\mathbb{S})}\\
    &\leq C(|\bm{X}|_*,\|\nabla_{\mathbb{S}^2}\bm{X}\|_{C^0(\mathbb{S}^2)},\|\mc{T}\|_{C^1})\|\nabla_{\mathbb{S}^2} \bm{Y}\|_{C^\gamma(\mathbb{S})}\\
    &\quad+C(|\bm{X}|_*,\|\nabla_{\mathbb{S}^2}\bm{X}\|_{C^0(\mathbb{S}^2)},\|\mc{T}\|_{C^2})\|\nabla_{\mathbb{S}^2}\bm{X}\|_{C^\gamma(\mathbb{S}^2)}\|\nabla_{\mathbb{S}^2}\bm{Y}\|_{C^0(\mathbb{S})}.
\end{aligned}
\end{equation}
We proceed with $\mc{S}^2(\bm{X})\bm{Y}$,
\begin{equation*}
    \begin{aligned}
    \mc{S}^2(\bm{X})\bm{Y}&=\mc{S}^{2,1}(\bm{X})\bm{Y}+\mc{S}^{2,2}(\bm{X})\bm{Y},\\
    (\mc{S}^{2,j}(\bm{X})\bm{Y})_k(\hx)&=\int_{\mathbb{S}^2} Q^j_{k,l}(\hx,\hy)\cdot(T(\nabla_{\mathbb{S}^2}\bm{X})\nabla_{\mathbb{S}^2}X_l(\bm{\hy})-\bm{\mc{C}}_l)d\hy,
    \end{aligned}
\end{equation*}
where we define the kernels
\begin{equation*}
Q^j_{k,l}(\hx,\hy)=\nabla_{\mathbb{S}^2}\frac{d}{d\varepsilon}\Big(G^j(\bm{X}(\hx)-\bm{X}(\hy)+\varepsilon(\bm{Y}(\hx)-\bm{Y}(\hy)))\Big)\Big|_{\varepsilon=0}.
\end{equation*}
Taking the derivatives we see that
\begin{equation*}
\begin{aligned}
Q^j_{k,l}(\hx,\hy)&=\nabla_{\mathbb{S}^2}\Big(\PD{}{x_i}G^j(\bm{X}(\hx)-\bm{X}(\hy))(Y_i(\hx)-Y_i(\hy))\Big)\\
&=-\PD{}{x_l}\PD{}{x_i}G^j(\bm{X}(\hx)-\bm{X}(\hy))\nabla_{\mathbb{S}^2}X_l(\hy)(Y_i(\hx)-Y_i(\hy))\\
&\quad-\PD{}{x_i}G^j(\bm{X}(\hx)-\bm{X}(\hy))\nabla_{\mathbb{S}^2}Y_i(\hy),
\end{aligned}
\end{equation*}
hence we have the following bound, similarly as in \eqref{qkernel_bounds}
\begin{equation*}
    \begin{aligned}
    |Q_{k,l}^j(\hx,\hy)|&\leq C\Big(\frac{|\nabla_{\mathbb{S}^2}\bm{Y}(\hy)|}{|\Delta_{\hy}\bm{X}(\hx)|^2}+\frac{|\nabla_{\mathbb{S}^2}\bm{X}(\hy)||\Delta_{\hy}\bm{Y}(\hx)|}{|\Delta_{\hy}\bm{X}(\hx)|^3}\Big)\\
    &\leq C\frac{\|\nabla_{\mathbb{S}^2}\bm{Y}\|_{C^0(\mathbb{S}^2)}}{|\bm{X}|_*^2}\Big(1+\frac{\|\nabla_{\mathbb{S}^2}\bm{X}\|_{C^0(\mathbb{S}^2)}}{|\bm{X}|_*}\Big)\frac{1}{|\hx-\hy|^2}.
    \end{aligned}
\end{equation*}
Therefore, 
\begin{equation*}
    \begin{aligned}
    |\mc{S}^{2,j}(\bm{X})\bm{Y}(\hx)|&\!\leq\! C\frac{\|T(\nabla_{\mathbb{S}^2}\bm{X})\nabla_{\mathbb{S}^2}\bm{X}\|_{C^\gamma(\mathbb{S}^2)}}{|\bm{X}|_*^2}\Big(1\!+\!\frac{\|\nabla_{\mathbb{S}^2}\bm{X}\|_{C^0(\mathbb{S}^2)}}{|\bm{X}|_*}\Big)\|\nabla_{\mathbb{S}^2}\bm{Y}\|_{C^0(\mathbb{S}^2)},
    \end{aligned}
\end{equation*}
hence
\begin{equation*}
    \begin{aligned}
    |\mc{S}^{2,j}(\bm{X})\bm{Y}(\hx)|&\leq  C(|\bm{X}|_*,\|\nabla_{\mathbb{S}^2}\bm{X}\|_{C^0(\mathbb{S}^2)},\|\mc{T}\|_{C^1})\|\nabla_{\mathbb{S}^2}\bm{X}\|_{C^\gamma(\mathbb{S}^2)}\|\nabla_{\mathbb{S}^2}\bm{Y}\|_{C^0(\mathbb{S}^2)}.
    \end{aligned}
\end{equation*}
Given that the kernels $Q_{k,l}^j$ are also a derivative, the estimate of the H\"older seminorm follows the same steps as in Proposition \ref{Nbound}. In fact, performing the splitting as in \eqref{Nholder_split}, we find that
\begin{equation*}
    \begin{aligned}
    [\mc{S}^{2,j}&(\bm{X})\bm{Y}]_{C^\gamma(\mathbb{S}^2)}\\
    &\leq C\frac{\|T(\nabla_{\mathbb{S}^2}\bm{X})\nabla_{\mathbb{S}^2}\bm{X}\|_{C^\gamma(\mathbb{S}^2)}}{|\bm{X}|_*^2}\Big(1+\Big(\frac{\|\nabla_{\mathbb{S}^2}\bm{X}\|_{C^0(\mathbb{S}^2)}}{|\bm{X}_*|}\Big)^2\Big)\|\nabla_{\mathbb{S}^2}\bm{Y}\|_{C^0(\mathbb{S}^2)},
    \end{aligned}
\end{equation*}
thus
\begin{equation}\label{S2bound}
    \begin{aligned}
    \|\mc{S}^2(\bm{X})\bm{Y}\|_{C^\gamma(\mathbb{S}^2)}&\leq C(|\bm{X}|_*,\|\nabla_{\mathbb{S}^2}\bm{X}\|_{C^0(\mathbb{S}^2)},\|\mc{T}\|_{C^1})\|\nabla_{\mathbb{S}^2}\bm{X}\|_{C^\gamma(\mathbb{S}^2)}\|\nabla_{\mathbb{S}^2}\bm{Y}\|_{C^0(\mathbb{S}^2)}.
    \end{aligned}
\end{equation}
Together with \eqref{S1bound}, this shows  that $S(\bm{X})$ maps $C^{1,\gamma}(\mathbb{S}^2)$ to $C^\gamma(\mathbb{S}^2)$,
\begin{align}\label{Sbound}
    \begin{split}
        \|\mc{S}(\bm{X})\bm{Y}\|_{C^\gamma(\mathbb{S}^2)}&\leq C(|\bm{X}|_*,\|\nabla_{\mathbb{S}^2}\bm{X}\|_{C^0(\mathbb{S}^2)},\|\mc{T}\|_{C^2})\|\nabla_{\mathbb{S}^2}\bm{X}\|_{C^{\gamma}(\mathbb{S}^2)} \|\nabla_{\mathbb{S}^2}\bm{Y}\|_{C^0(\mathbb{S}^2)}\\
        &\quad+C(|\bm{X}|_*,\|\nabla_{\mathbb{S}^2}\bm{X}\|_{C^0(\mathbb{S}^2)},\|\mc{T}\|_{C^1})\|\nabla_{\mathbb{S}^2} \bm{Y}\|_{C^\gamma(\mathbb{S})}\\
        &\leq C(\starnorm{\bm{X}},\|\nabla_{\mathbb{S}^2}\bm{X}\|_{C^{\gamma}(\mathbb{S}^2)},\|\mc{T}\|_{C^2})\norm{\nabla_{\mathbb{S}^2} \bm{Y}}_{C^\gamma(\mathbb{S}^2)}.
    \end{split}
\end{align}

We are left to show that $\mc{S}(\bm{X})$ also maps $h^{1,\gamma}(\mathbb{S}^2)$ to $h^{\gamma}(\mathbb{S}^2)$ and that it is continuous with respect to $\bm{X}$. We follow the lines below \eqref{FHolderbound}. 
It suffices to show that if $\bm{Y}\in h^{1,\gamma}(\mathbb{S}^2)$, then $\mc{S}(\bm{X})\bm{Y}\in h^\gamma(\mathbb{S}^2)$. Let $\bm{Y}\in h^{1,\gamma}(\mathbb{S}^2)$, and $\{\bm{Y}^m\}_m$ a sequence $\bm{Y}^m \in C^{1,\alpha}(\mathbb{S}^2)$, $\alpha>\gamma$, such that $\bm{Y}^m\rightarrow \bm{Y}$ in $C^{1,\gamma}(\mathbb{S}^2)$. Since $\mc{S}(\bm{X})\bm{Y}^m\in C^\alpha(\mathbb{S}^2)$, we conclude that $\mc{S}(\bm{X})\bm{Y}\in h^{\gamma}(\mathbb{S}^2)$ by showing that
\begin{equation*}
    \begin{aligned}
     \|\mc{S}(\bm{X})\bm{Y}^m-\mc{S}(\bm{X})\bm{Y}\|_{C^\gamma(\mathbb{S}^2)}\leq C \|\bm{Y}^m-\bm{Y}\|_{C^{1,\gamma}(\mathbb{S}^2)}.
    \end{aligned}
\end{equation*}
But since we are dealing with a linear operator, the estimate is trivially satisfied from \eqref{Sbound}.
That the G\^ateaux derivative is continuous in $\bm{X}$ follows along the lines below \eqref{continuity_X}. In fact,
\begin{equation*}
    \begin{aligned}
    \big(\mc{S}(\bm{X}^1)-\mc{S}(\bm{X}^2)\big)\bm{Y}=(\mc{S}^1(\bm{X}^1)-\mc{S}^1(\bm{X}^2))\bm{Y}+(\mc{S}^2(\bm{X}^1)-\mc{S}^2(\bm{X}^2))\bm{Y},
    \end{aligned}
\end{equation*}
and we decompose each $\mc{S}^j$ as in \eqref{continuity_X2}. Then, it is not hard to see that the following bound holds 
\begin{equation*}
    \begin{aligned}
    \|&(\mc{S}(\bm{X}^2)-\mc{S}(\bm{X}^1))\bm{Y}\|_{C^\gamma(\mathbb{S}^2)}\\
    &\leq C(\|\nabla_{\mathbb{S}^2} \bm{X}^1\|_{C^0(\mathbb{S}^2)},|\bm{X}^1|_*,\|\nabla_{\mathbb{S}^2} \bm{X}^2\|_{C^0(\mathbb{S}^2)},|\bm{X}^2|_*,\|\mc{T}\|_{C^3})\\
    &\quad\times\Big(\|\nabla_{\mathbb{S}^2}\bm{Y}\|_{C^\gamma(\mathbb{S}^2)}\|\nabla_{\mathbb{S}^2} (\bm{X}^1-\bm{X}^2)\|_{C^\gamma(\mathbb{S}^2)}\\
    &\quad+\|\nabla_{\mathbb{S}^2}\bm{Y}\|_{C^0(\mathbb{S}^2)}\|\nabla_{\mathbb{S}^2} (\bm{X}^1-\bm{X}^2)\|_{C^0(\mathbb{S}^2)}(\|\nabla_{\mathbb{S}^2} \bm{X}^1\|_{C^\gamma(\mathbb{S}^2)}+\|\nabla_{\mathbb{S}^2} \bm{X}^2\|_{C^\gamma(\mathbb{S}^2)})\Big).
    \end{aligned}
\end{equation*}

\end{proof}

\begin{prop}\label{semigroup_est}
Consider the linear operator $\mc{S}(\bm{X}):C^{1,\gamma}(\mathbb{S}^2)\rightarrow C^{\gamma}(\mathbb{S}^2)$ defined in \eqref{Soperator} with $\bm{X}\in C^{1,\gamma}(\mathbb{S}^2)$, $\mc{T}\in C^{2}$, $\mc{T}>0$, $d\mc{T}/d\lambda\geq0$. Then, there exists a sector such that for all $z$ in the sector
\begin{equation*}
    \|z-\mc{S}(\bm{X})\bm{Y}\|_{C^{\gamma}(\mathbb{S}^2)}\geq C(\|\bm{Y}\|_{C^{1,\gamma}(\mathbb{S}^2)}+|z|\|\bm{Y}\|_{C^{\gamma}(\mathbb{S}^2)}),
\end{equation*}
where the constant $C$ depends only on the sector, $\gamma$,  the norms $\|\bm{X}\|_{C^{1,\gamma}(\mathbb{S}^2)}$ and $\|\mc{T}\|_{C^2}$, and the arc-chord condition $|\bm{X}|_*$.
\end{prop}
\begin{proof}
From \eqref{Soperator}, we have
\begin{equation}\label{Semi_spliting}
    \begin{aligned}
    \|(z-\mc{S}(\bm{X}))\bm{Y}\|_{C^{\gamma}(\mathbb{S}^2)}  \geq     \|(z-\mc{S}^1(\bm{X}))\bm{Y}\|_{C^{\gamma}(\mathbb{S}^2)}   -\|\mc{S}^2(\bm{X})\bm{Y}\|_{C^{\gamma}(\mathbb{S}^2)},
    \end{aligned}
\end{equation}
and using \eqref{S2bound} we obtain that
\begin{equation}\label{semi_bound01}
    \begin{aligned}
    \|(z-\mc{S}(\bm{X}))\bm{Y}\|_{C^{\gamma}(\mathbb{S}^2)}  \geq     \|(z-\mc{S}^1(\bm{X}))\bm{Y}\|_{C^{\gamma}(\mathbb{S}^2)}-C\norm{\nabla_{\mathbb{S}^2} \bm{Y}}_{C^0(\mathbb{S}^2)}.
    \end{aligned}
\end{equation}
We use the notation \eqref{Ttilde_law}. Then, we can write $\mc{S}^1(\bm{X})$ \eqref{S1N1} as
\begin{equation}\label{S1op}
    \begin{aligned}
    \mc{S}^1(\bm{X})\bm{Y}(\hx)&=-\int_{\mathbb{S}^2} \nabla_{\mathbb{S}^2}G(\bm{X}(\hx)-\bm{X}(\hy))\cdot( T_S(\nabla_{\mathbb{S}^2}\bm{X}(\hy))\nabla_{\mathbb{S}^2}\bm{Y}(\hy)-\mc{C})d\hy.
    \end{aligned}
\end{equation}
Next, we introduce the partition of unity $\rho_n$ (see Section \ref{covering_ste}) to write
\begin{equation*}
    \begin{aligned}
    \mc{S}^1&(\bm{X})\bm{Y}(\hx)=\sum_n \mc{S}^1(\bm{X})\paren{\rho_n\bm{Y}}(\hx)\\
    &=\!-\!\sum_n\! \int_{\mathbb{S}^2}\!\! \nabla_{\mathbb{S}^2}G(\bm{X}(\hx)\!-\!\bm{X}(\hy))\!\cdot\!( T_S(\nabla_{\mathbb{S}^2}\bm{X}(\hy))\nabla_{\mathbb{S}^2}(\rho_n\bm{Y})(\hy)\!-\!\mc{C}_n)d\hy,
    \end{aligned}
\end{equation*}
where now we will choose $\mc{C}_n=0$ or $\mc{C}_n=T_S(\nabla_{\mathbb{S}^2}\bm{X}(\hx))\nabla_{\mathbb{S}^2}(\rho_n\bm{Y})(\hx)$. We will extensively use that
\begin{equation}\label{supprhon}
    \text{supp }\rho_n \subset B_{\hx_n,2R}\cap \mathbb{S}^2.
\end{equation}
 We notice that
\begin{equation*}
    \begin{aligned}
    \rho_n(\hx)z\bm{Y}(\hx)-\mc{S}^1(\bm{X})(\rho_n\bm{Y})(\hx)&=\rho_n(\hx)\Big(z\bm{Y}(\hx)-\mc{S}^1(\bm{X})\bm{Y}(\hx)\Big)\\
    &\quad+\Big(\rho_n(\hx)\mc{S}^1(\bm{X})\bm{Y}-\mc{S}^1(\bm{X})(\rho_n\bm{Y})(\hx)\Big),
    \end{aligned}
\end{equation*}
hence
\begin{equation*}
    \begin{aligned}
     2\|\rho_n\|_{C^{\gamma}(\mathbb{S}^2)}\|(z-\mc{S}^1(\bm{X}))\bm{Y}\|_{C^{\gamma}(\mathbb{S}^2)}\geq 
     \|\rho_n\Big(z\bm{Y}-\mc{S}^1(\bm{X})\bm{Y}\Big)\|_{C^{\gamma}(\mathbb{S}^2)}\\
     \geq \|z\rho_n\bm{Y}-\mc{S}^1(\bm{X})(\rho_n\bm{Y})\|_{C^{\gamma}(\mathbb{S}^2)}-\|\rho_n\mc{S}^1(\bm{X})\bm{Y}-\mc{S}^1(\bm{X})(\rho_n\bm{Y})\|_{C^{\gamma}(\mathbb{S}^2)},
    \end{aligned}
\end{equation*}
and summing in $n$ we obtain
\begin{equation}\label{I1I2split}
    \begin{aligned}
     \|(z-\mc{S}^1(\bm{X}))\bm{Y}\|_{C^{\gamma}(\mathbb{S}^2)} &\geq C\sum_n (\|I_n^1\|_{C^{\gamma}(\mathbb{S}^2)}-\|I_n^2\|_{C^{\gamma}(\mathbb{S}^2)}),
     \end{aligned}
 \end{equation}
 where
 \begin{equation}\label{In1term}
     \begin{aligned}
    I_n^{1}= z\rho_n\bm{Y}-\mc{S}^1(\bm{X})(\rho_n\bm{Y}),
    \end{aligned}
\end{equation}
     \begin{equation*}
         \begin{aligned}
    I_n^2=\rho_n\mc{S}^1(\bm{X})\bm{Y}-\mc{S}^1(\bm{X})(\rho_n\bm{Y}).
    \end{aligned}
\end{equation*}
Recalling that $\mc{S}^1(\bm{X})$ is given in terms of $\mc{N}(\bm{X})$ \eqref{S1N1}, we split $I_n^2$ further:
\begin{equation*}
    \begin{aligned}
    I_n^2&=[\rho_n,\mc{N}(\bm{X})](T_S(\nabla_{\mathbb{S}^2}\bm{X})\nabla_{\mathbb{S}^2}\bm{Y})+\mc{N}(\bm{X})\big(T_s(\nabla_{\mathbb{S}^2}\bm{X})\bm{Y}\otimes\nabla_{\mathbb{S}^2}\rho_n\big),
    \end{aligned}
\end{equation*}
and by Proposition \ref{Nbound} and Lemma \ref{lem_commutator}, 
\begin{equation*}
    \begin{aligned}
    \|I_n^2\|_{C^\gamma(\mathbb{S}^2)}&\leq C(|\bm{X}|_*,\|\nabla_{\mathbb{S}^2}\bm{X}\|_{C^0(\mathbb{S}^2)})\Big(\|\nabla_{\mathbb{S}^2}\rho_n\|_{C^0(\mathbb{S}^2)}\|T_S(\nabla_{\mathbb{S}^2}\bm{X})\nabla_{\mathbb{S}^2}\bm{Y}\|_{C^0(\mathbb{S}^2)}\\
    &\quad+\|T_s(\nabla_{\mathbb{S}^2}\bm{X})\bm{Y}\otimes\nabla_{\mathbb{S}^2}\rho_n\|_{C^\gamma(\mathbb{S}^2)}\Big).
    \end{aligned}
\end{equation*}
Therefore,
\begin{equation*}
    \begin{aligned}
     \|I_n^{2}\|_{C^\gamma(\mathbb{S}^2)}&\leq C(|\bm{X}|_*,\|\nabla_{\mathbb{S}^2}\bm{X}\|_{C^0(\mathbb{S}^2)},\|\mc{T}\|_{C^2})\|\nabla_{\mathbb{S}^2}\rho_n\|_{C^\gamma(\mathbb{S}^2)}\\
    &\qquad \times(\|\nabla_{\mathbb{S}^2}\bm{X}\|_{C^\gamma(\mathbb{S}^2)}\|\bm{Y}\|_{C^\gamma(\mathbb{S}^2)}\!+\!\|\nabla_{\mathbb{S}^2}\bm{Y}\|_{C^0(\mathbb{S}^2)}).
    \end{aligned}
\end{equation*}
We proceed to deal with the term $I_n^1$ \eqref{In1term}.
We introduce the cutoff \eqref{cutoff} so that
\begin{equation}\label{In1split}
\begin{aligned}
 \|I_n^1\|_{C^{\gamma}(\mathbb{S}^2)}&\geq  \|z\rho_n\bm{Y}\!-\!\wh{\rho}_n\mc{S}^1(\bm{X})(\rho_n\bm{Y})\|_{C^{\gamma}(\mathbb{S}^2)}\!-\|(\!1-\wh{\rho}_n)\mc{S}^1(\bm{X})(\rho_n\bm{Y})\|_{C^{\gamma}(\mathbb{S}^2)}\\
 &= \|I_n^{1,1}\|_{C^{\gamma}(\mathbb{S}^2)}-\|I_n^{1,2}\|_{C^{\gamma}(\mathbb{S}^2)}.
\end{aligned}
\end{equation}
The last term will be smoother because the integral is not singular. 
In fact, recalling again the expression of $\mc{S}^1(\bm{X})$ in terms of $\mc{N}(\bm{X})$ \eqref{S1N1}, we use Proposition \ref{Nout_bound} to obtain 
\begin{equation*}
    \begin{aligned}
    \|I_n^{1,2}\|_{C^\gamma(\mathbb{S}^2)}&\leq C(R,|\bm{X}|_*,\|\nabla_{\mathbb{S}^2}\bm{X}\|_{C^0(\mathbb{S}^2)})\|T_S(\nabla_{\mathbb{S}^2}\bm{X})\nabla_{\mathbb{S}^2}(\rho_n\bm{Y})\|_{C^0(\mathbb{S}^2)}\\
    &\leq C(R,|\bm{X}|_*,\|\nabla_{\mathbb{S}^2}\bm{X}\|_{C^0(\mathbb{S}^2)},\|\mc{T}\|_{C^1})\|\nabla_{\mathbb{S}^2}(\rho_n\bm{Y})\|_{C^0(\mathbb{S}^2)}.
    \end{aligned}
\end{equation*}
Although the constant in the bound above \eqref{In21bound} becomes large for $R$ small, it will suffice since it is lower order in terms of regularity for $\bm{Y}$.

Next, we proceed to estimate $I_n^{1,1}$ \eqref{In1split}.
 We can decompose further by introducing the frozen-coefficient linear operator.  We denote by $\hX_n$ the stereographic projection centered at $\hx_n$, i.e. $\hx_n=\hX_n(\bm{0})$, and $\bm{X}_n(\thetab)=\bm{X}(\hX_n(\thetab))$ (see Section \ref{covering_ste}).
 Recalling \eqref{S1N1} and \eqref{Ntheta}, \eqref{nonlinear_split}, we have
\begin{equation}\label{In11bound_split}
    \begin{aligned}
\|I_n^{1,1}\|_{C^{\gamma}(\mathbb{S}^2)}\geq C\|z\rho_n\bm{Y}_n-\wh{\rho}_n\mc{N}(\bm{X})(T_S(\nabla_{\mathbb{S}^2}\bm{X})\nabla_{\mathbb{S}^2}(\rho_n\bm{Y}))_n\|_{C^{\gamma}(\mathbb{R}^2)},
    \end{aligned}
\end{equation}
and
\begin{equation*}
    \begin{aligned}
    I_n^{1,1}(\thetab)&=z\rho_n(\thetab)\bm{Y}_n(\thetab)-\wh{\rho}_n(\thetab)\mc{N}(\bm{X})(T_S(\nabla_{\mathbb{S}^2}\bm{X})\nabla_{\mathbb{S}^2}(\rho_n\bm{Y}))_n(\thetab)\\
    &=J^3+J^4+J^5+J^6,
    \end{aligned}
\end{equation*}
with
\begin{equation*}
    \begin{aligned}
    J^3=z\rho_n(\thetab)\bm{Y}_n(\thetab)-\mc{L}_A(\rho_n\bm{Y}_n)(\thetab),
    \end{aligned}
\end{equation*}   
\begin{equation*}
    \begin{aligned}
    J^4& =\wh{\rho}_n(\thetab)[\mc{M}(A)-\mc{N}(\bm{X})](T_S(\nabla_{\mathbb{S}^2}\bm{X})\nabla_{\mathbb{S}^2}(\rho_n\bm{Y}))_n(\thetab)\\
    &=   \wh{\rho}_n(\thetab)[\mc{M}(A)-\mc{M}(\nabla\bm{X}_n)-\mc{R}_n(\bm{X}_n)](T_S(\nabla_{\mathbb{S}^2}\bm{X})\nabla_{\mathbb{S}^2}(\rho_n\bm{Y}))_n(\thetab),
    \end{aligned}
\end{equation*}    
\begin{equation*}
    \begin{aligned}
    J^5=\wh{\rho}_n(\thetab)\Big(\mc{L}_A(\rho_n\bm{Y}_n)(\thetab)-\mc{M}(A)(T_S(\nabla_{\mathbb{S}^2}\bm{X}))\nabla_{\mathbb{S}^2}(\rho_n\bm{Y})\big)_n(\thetab)\Big),
    \end{aligned}
\end{equation*}    
\begin{equation*}
    \begin{aligned}
    J^6=(1-\wh{\rho}_n(\thetab))\mc{L}_A(\rho_n\bm{Y}_n)(\thetab),
    \end{aligned}
\end{equation*}   
where we denote $A$ the constant matrix $A=\nabla \bm{X}_n(\bm{0})$, $\mc{M}(\nabla\bm{X}_n)$ is defined in \eqref{defM},  $\mc{R}(\bm{X}_n)$ in \eqref{defR}, $\mc{L}_A$ in \eqref{defnL2}, $\hx_n=\hX_n(\bm{0})$. 
The bound for $J^6$ follows from Lemma \ref{MAout_bound} together with Remark \ref{arc_chord},
\begin{align*}
    \|J^6\|_{C^1(\mathbb{R}^2)}&\leq C(R,|\bm{X}|_*,\|\nabla_{\mathbb{S}^2}\bm{X}\|_{C^0(\mathbb{S}^2)})\|T_F(A)\nabla(\rho_n\bm{Y}_n)\|_{C^0(\mathbb{R}^2)},
\end{align*}
thus,
\begin{equation}\label{Jn10bound}
    \begin{aligned}
        \|J^{6}\|_{C^1(\mathbb{R}^2)}
    \leq C(R,|\bm{X}|_*,\|\nabla_{\mathbb{S}^2}\bm{X}\|_{C^0(\mathbb{S}^2)},\|\mc{T}\|_{C^1})\|\nabla_{\mathbb{S}^2}(\rho_n\bm{Y})\|_{C^0(\mathbb{S}^2)}.
    \end{aligned}
\end{equation}
Next, Lemma \ref{lem_dif} with $Z=T_S(\nabla_{\mathbb{S}^2}\bm{X})\nabla_{\mathbb{S}^2}(\rho_n\bm{Y})$ provides the following bound for $J^4$:
\begin{equation*}
    \begin{aligned}
    \|J^4&\|_{C^\gamma(\mathbb{R}^2)}\\
    &\leq C(|\bm{X}|_*,\|\nabla_{\mathbb{S}^2}\bm{X}\|_{C^0(\mathbb{S}^2)})\big((1+\|\nabla_{\mathbb{S}^2}\bm{X}\|_{C^\gamma(\mathbb{S}^2)})\|T_S(\nabla_{\mathbb{S}^2}\bm{X})\nabla_{\mathbb{S}^2}(\rho_n\bm{Y})\|_{C^0(\mathbb{S}^2)}\\
    &\quad+\varepsilon(R)\|T_S(\nabla_{\mathbb{S}^2}\bm{X})\nabla_{\mathbb{S}^2}(\rho_n\bm{Y})\|_{C^\gamma(\mathbb{S}^2)}\big),
    \end{aligned}
\end{equation*}
where $\varepsilon(R)\to 0$ as $R\to 0$.
Thus,
\begin{equation*}
    \begin{aligned}
    \|J^4\|_{C^\gamma(\mathbb{R}^2)}&\leq C(|\bm{X}|_*,\|\nabla_{\mathbb{S}^2}\bm{X}\|_{C^0(\mathbb{S}^2)},\|\mc{T}\|_{C^2})\Big(\varepsilon(R)\|\nabla_{\mathbb{S}^2}(\rho_n\bm{Y})\|_{C^\gamma(\mathbb{S}^2)}\\
    &\quad+(1+\|\nabla_{\mathbb{S}^2}\bm{X}\|_{C^\gamma(\mathbb{S}^2)})\|\nabla_{\mathbb{S}^2}(\rho_n\bm{Y})\|_{C^0(\mathbb{S}^2)}\Big).
    \end{aligned}
\end{equation*}
We proceed with $J^5$. Recalling the expression for $T_S$ \eqref{Ttilde_law}, we have
\begin{equation*}
    \begin{aligned}
    (T_S(&\nabla_{\mathbb{S}^2}\bm{X})\nabla_{\mathbb{S}^2}(\rho_n\bm{Y}))_{n,li}(\etab)=(T_S(\nabla_{\mathbb{S}^2}\bm{X})\nabla_{\mathbb{S}^2}(\rho_n\bm{Y})\circ\hX_n)_{l,i}(\etab)\\
    &=(\text{det}(\wh{g}(\etab)))^{-\frac12}\PD{}{\eta_r}(\rho_n Y_{n,q})(\etab)\PD{\wh{X}_m}{\eta_r}(\etab)\Big(\frac{\mc{T}(\lambda_n(\etab))}{\lambda_n(\etab)}\delta_{lq}\delta_{im}\\
    &\quad+\big(\mc{T}'(\lambda_n(\etab))-\frac{\mc{T}(\lambda_n(\etab)}{\lambda_n(\etab)}\big)\frac{\PD{X_{n,l}}{\eta_j}(\etab)\PD{\wh{X}_i}{\eta_j}(\etab)\PD{X_{n,q}}{\eta_p}(\etab)\PD{\wh{X}_m}{\eta_p}(\etab)}{(\lambda_n(\etab))^2\text{det}(\wh{g}(\etab))}\Big),
    \end{aligned}
\end{equation*}
with $\lambda_n(\etab)$ given in \eqref{lambda_n}. Substituting into \eqref{defM},
\begin{equation}\label{Ntheta_aux}
\begin{aligned}
    (\mc{M}&(A)(T_S(\nabla_{\mathbb{S}^2}\bm{X})\nabla_{\mathbb{S}^2}(\rho_n\bm{Y}))_n)_k(\thetab)\\
    &=-\int_{\mathbb{R}^2}m_{m,k,l}(\thetab,\etab)\PD{\wh{X}_i}{\eta_m}(\etab)(T_S(\nabla_{\mathbb{S}^2}\bm{X}(\etab))\nabla_{\mathbb{S}^2}(\rho_n\bm{Y}_n)(\etab))_{l,i}d\eta_1d\eta_2\\
    &=-\int_{\mathbb{R}^2}m_{i,k,l}(\thetab,\etab)(\tilde{T}_S)_{ipql}(\etab)\PD{}{\eta_p}(\rho_n Y_{n,q})(\etab)d\eta_1d\eta_2\\
    &=\tilde{\mc{M}}(A)(\tilde{T}_S\nabla(\rho_n\bm{Y}_n))(\thetab),
\end{aligned}
\end{equation}
where we denote
\begin{equation}\label{Ttildelaw2}
    \begin{aligned}
    (\tilde{T}_S(\nabla\bm{X}))_{ipql}(\etab)\!=\!\frac{\mc{T}(\lambda_n(\etab))}{\lambda_n(\etab)}\delta_{pi}\delta_{ql}\!+\!\big(\mc{T}'(\lambda_n(\etab))\!-\!\frac{\mc{T}(\lambda_n(\etab)}{\lambda_n(\etab))}\big)\frac{\PD{X_{n,l}}{\eta_i}(\etab)\PD{X_{l,q}}{\eta_p}(\etab)}{(\lambda_n(\etab))^2\sqrt{\text{det}(\wh{g}(\etab))}}.
    \end{aligned}
\end{equation}
Thus we can write
\begin{equation*}
    \begin{aligned}
    J^5=\wh{\rho}_n(\thetab)\tilde{\mc{M}}(A)((T_F(A)-\tilde{T}_S(\nabla\bm{X}))\nabla(\rho_n\bm{Y}_n))(\thetab).
    \end{aligned}
\end{equation*}   
Proposition \ref{Mtilde_bound} with Lemma \ref{arc_chord} gives that
\begin{equation*}
    \begin{aligned}
    \|J^5\|_{C^\gamma(\mathbb{R}^2)}&\leq\! C(|\bm{X}_*,\|\nabla_{\mathbb{S}^2}\bm{X}\|_{C^0(\mathbb{S}^2)})\|(T_F(A)-\tilde{T}_S(\nabla\bm{X}))\nabla(\rho_n\bm{Y}_n)\|_{C^\gamma(\mathbb{R}^2)},
    \end{aligned}
\end{equation*}
and since $T_F(A)=\tilde{T}_S(\nabla\bm{X}(\bm{0}))(\bm{0})$, we obtain
\begin{equation*}
    \begin{aligned}
    \|J_5\|_{C^\gamma(\mathbb{R}^2)}&\leq C(|\bm{X}_*,\|\nabla_{\mathbb{S}^2}\bm{X}\|_{C^0(\mathbb{S}^2)},\|\mc{T}\|_{C^2})\Big(\varepsilon(R)\|\nabla_{\mathbb{S}^2}(\rho_n\bm{Y})\|_{C^\gamma(\mathbb{S}^2)}\\
    &\quad+\|\nabla_{\mathbb{S}^2}\bm{X}\|_{C^\gamma(\mathbb{S}^2)}\|\nabla_{\mathbb{S}^2}(\rho_n\bm{Y})\|_{C^0(\mathbb{S}^2)}\Big).
    \end{aligned}
\end{equation*}
Then, we continue from \eqref{In11bound_split},
\begin{equation*}
    \begin{aligned}
    \|I_n^{1,1}\|_{C^\gamma(\mathbb{S}^2)}&\geq \|J_n^3\|_{C^\gamma(\mathbb{R}^2)}-\|J_n^4\|_{C^\gamma(\mathbb{R}^2)}-\|J_n^5\|_{C^\gamma(\mathbb{R}^2)}-\|J_n^6\|_{C^\gamma(\mathbb{R}^2)},
    \end{aligned}
\end{equation*}
so inserting back the bounds for $J_n^4$, $J_n^5$, and $J_n^6$, we have that
\begin{equation*}
    \begin{aligned}
    \|I_n^{1,1}\|_{C^\gamma(\mathbb{S}^2)}&\geq
    \|J_n^3\|_{C^\gamma(\mathbb{R}^2)}\\
    &\quad-C(|\bm{X}|_*,\|\nabla_{\mathbb{S}^2}\bm{X}\|_{C^0(\mathbb{S}^2)})\varepsilon(R)\|\nabla_{\mathbb{S}^2}(\rho_n\bm{Y})\|_{C^\gamma(\mathbb{S}^2)}\\
    &\quad-C(|\bm{X}|_*,\|\nabla_{\mathbb{S}^2}\bm{X}\|_{C^0({\mathbb{S}^2})})\|\nabla_{\mathbb{S}^2}\bm{X}\|_{C^\gamma({\mathbb{S}^2})}\|\nabla_{\mathbb{S}^2}(\rho_n\bm{Y})\|_{C^0(\mathbb{S}^2)}\\
    &\quad-C(R,|\bm{X}|_*,\|\nabla_{\mathbb{S}^2}\bm{X}\|_{C^0({\mathbb{S}^2})})\|\nabla_{\mathbb{S}^2}(\rho_n\bm{Y})\|_{C^0(\mathbb{S}^2)}.
    \end{aligned}
\end{equation*}
Then, we use the frozen-coefficient estimate in Proposition \ref{frozen_coef_prop} for $J_n^3$ \eqref{In11bound_split}. We first interpolate the inequalities in Theorem \ref{t:sectorial_est} to
control the lower-order terms,
\begin{equation*}
    \begin{aligned}
    J_n^3&=(z-\mc{L}_{A})(\rho_n\bm{Y}_n)(\thetab),
    \end{aligned}
\end{equation*}    
\begin{equation}\label{Jn7bound}
    \begin{aligned}
    \|J_n^3\|_{C^\gamma(\mathbb{R}^2)}\geq C|z|\|\rho_n\bm{Y}_n\|_{C^\gamma(\mathbb{R}^2)}+C\|\rho_n\bm{Y}_n\|_{C^{1,\gamma}(\mathbb{R}^2)}+ C|z|^{1-\sigma}\|\rho_n\bm{Y}_n\|_{C^{\gamma+\sigma}(\mathbb{R}^2)},
    \end{aligned}
\end{equation}
where $\sigma\in[0,1]$ is chosen so that $1<\gamma+\sigma<1+\gamma$.
Therefore, we have
\begin{equation*}
    \begin{aligned}
    \|I_n^{1,1}\|_{C^\gamma(\mathbb{R}^2)}&\geq
    C|z|\|\rho_n\bm{Y}\|_{C^\gamma(\mathbb{S}^2)}+C\|\rho_n\bm{Y}\|_{C^{1,\gamma}(\mathbb{S}^2)}+ C|z|^{1-\sigma}\|\rho_n\bm{Y}\|_{C^{\gamma+\sigma}(\mathbb{S}^2)}\\
     &\quad-C(|\bm{X}|_*,\|\nabla_{\mathbb{S}^2}\bm{X}\|_{C^0(\mathbb{S}^2)})\varepsilon(R)\|\nabla_{\mathbb{S}^2}(\rho_n\bm{Y})\|_{C^\gamma(\mathbb{S}^2)}\\
    &\quad-C(|\bm{X}|_*,\|\nabla_{\mathbb{S}^2}\bm{X}\|_{C^0({\mathbb{S}^2})})\|\nabla_{\mathbb{S}^2}\bm{X}\|_{C^\gamma({\mathbb{S}^2})}\|\nabla_{\mathbb{S}^2}(\rho_n\bm{Y})\|_{C^0(\mathbb{S}^2)}\\
    &\quad-C(R,|\bm{X}|_*,\|\nabla_{\mathbb{S}^2}\bm{X}\|_{C^0({\mathbb{S}^2})})\|\nabla_{\mathbb{S}^2}(\rho_n\bm{Y})\|_{C^0(\mathbb{S}^2)}.
    \end{aligned}
\end{equation*}
and taking $R$ small enough, 
\begin{equation}\label{In1bound}
    \begin{aligned}
    \|I_n^{1,1}\|_{C^\gamma(\mathbb{R}^2)}&\geq
    C|z|\|\rho_n\bm{Y}\|_{C^\gamma(\mathbb{S}^2)}+C\|\rho_n\bm{Y}\|_{C^{1,\gamma}(\mathbb{S}^2)}+ C|z|^{1-\sigma}\|\rho_n\bm{Y}\|_{C^{\gamma+\sigma}(\mathbb{S}^2)}\\
    &\quad-C(R,|\bm{X}|_*,\|\nabla_{\mathbb{S}^2}\bm{X}\|_{C^\gamma({\mathbb{S}^2})})\|\nabla_{\mathbb{S}^2}(\rho_n\bm{Y})\|_{C^0(\mathbb{S}^2)}.
    \end{aligned}
\end{equation}
Next, we go back to \eqref{I1I2split} and substitute the above bound together with \eqref{In2bound} and \eqref{In1bound},
\begin{equation*}
    \begin{aligned}
     \|(z&-\mc{S}^1(\bm{X}))\bm{Y}\|_{C^{\gamma}(\mathbb{S}^2)} \\
     &\geq \sum_n \Big(  C|z|\|\rho_n\bm{Y}\|_{C^\gamma(\mathbb{S}^2)}+C\|\rho_n\bm{Y}\|_{C^{1,\gamma}(\mathbb{S}^2)}+ C|z|^{1-\sigma}\|\rho_n\bm{Y}\|_{C^{\gamma+\sigma}(\mathbb{S}^2)}\Big)\\    &\quad-C(R,|\bm{X}|_*,\|\nabla_{\mathbb{S}^2}\bm{X}\|_{C^\gamma({\mathbb{S}^2})})\|\nabla_{\mathbb{S}^2}\bm{Y}\|_{C^0(\mathbb{S}^2)}\\
    &\quad- C(|\bm{X}|_*,\|\nabla_{\mathbb{S}^2}\bm{X}\|_{C^\gamma(\mathbb{S}^2)})\|\nabla_{\mathbb{S}^2}\rho_n\|_{C^{\gamma}(\mathbb{S}^2)}(\|\bm{Y}\|_{C^\gamma(\mathbb{S}^2)}+\|\nabla_{\mathbb{S}^2}\bm{Y}\|_{C^0(\mathbb{S}^2)}).
    \end{aligned}
\end{equation*}
Plugging this inequality in \eqref{semi_bound01},
and then using the triangle inequality and the fact that $C^\alpha\hookrightarrow C^\beta$ for $\alpha\geq\beta$, we obtain
\begin{equation*}
    \begin{aligned}
     \|(z-\mc{S}(\bm{X}))\bm{Y}\|_{C^{\gamma}(\mathbb{S}^2)}
     &\geq   C|z|\|\bm{Y}\|_{C^\gamma(\mathbb{S}^2)}+C\|\bm{Y}\|_{C^{1,\gamma}(\mathbb{S}^2)}+ C|z|^{1-\sigma}\|\bm{Y}\|_{C^{\gamma+\sigma}(\mathbb{S}^2)}\\
     &\quad-C(R,|\bm{X}|_*,\|\nabla_{\mathbb{S}^2}\bm{X}\|_{C^\gamma({\mathbb{S}^2})})\|\bm{Y}\|_{C^{\gamma+\sigma}(\mathbb{S}^2)}.
    \end{aligned}
\end{equation*}
Finally, by moving the sector if necessary to make $|z|$ big, we conclude the result
\begin{equation*}
    \begin{aligned}
     \|(z-\mc{S}(\bm{X}))&\bm{Y}\|_{C^{\gamma}(\mathbb{S}^2)} \geq   C|z|\|\bm{Y}\|_{C^\gamma(\mathbb{S}^2)}+C\|\bm{Y}\|_{C^{1,\gamma}(\mathbb{S}^2)}.
    \end{aligned}
\end{equation*}

\end{proof}

\begin{prop}\label{Fsemigroup}
The G\^ateaux derivative of $F$ at any $\bm{X}\in\mathcal{O}_m$, $\mc{S}(\bm{X})$ \eqref{Soperator}, generates an analytic semigroup on the space $h^{0,\gamma}(\mathbb{S}^2)$.
\end{prop}
\begin{proof}
We need to prove that the operator $\mc{S}(\bm{X})$ is sectorial, i.e., that there exists a sector such that for any $z$ in the sector
\begin{equation*}
    \begin{aligned}
    \|(z-\mc{S}(\bm{X}))^{-1}\bm{Y}\|_{h^\gamma(\mathbb{S}^2)}\leq \frac{C}{|z|}\|\bm{Y}\|_{h^\gamma(\mathbb{S}^2)}.
    \end{aligned}
\end{equation*}
Since the norm on little H\"older spaces $h^{\gamma}(\mathbb{S}^2)$ is the same as in the usual H\"older spaces $C^{\gamma}(\mathbb{S}^2)$, from the previous Proposition \ref{semigroup_est} we are left to prove that the operator $(z-\mc{S}(\bm{X}))$ is invertible from $h^\gamma(\mathbb{S}^2)$ to $h^{1,\gamma}(\mathbb{S}^2)$ for any $z$ in the sector. Similarly as we did in Section \ref{sec:frozen}, define the following family of operators $\mc{S}_\alpha(\bm{X})$, $\alpha\in[0,1]$,
\begin{equation*}
    \begin{aligned}
\mc{S}_\alpha&(\bm{X})\bm{Y}(\hx)\\
&=-\alpha\!\int_{\mathbb{S}^2}\!\!\! \nabla_{\mathbb{S}^2}\Big(G_1(\bm{X}(\hx)\!-\!\bm{X}(\hy))\!+\!\alpha G_2(\bm{X}(\hx)\!-\!\bm{X}(\hy))\Big)\!\cdot\! (T_S(\nabla_{\mathbb{S}^2}\bm{X})\nabla_{\mathbb{S}^2}\bm{Y}(\bm{\hy}))d\hy\\
&\quad-\!(1-\alpha)\int_{\mathbb{S}^2}\!\!\! \nabla_{\mathbb{S}^2}\Big(G_1(\bm{X}(\hx)\!-\!\bm{X}(\hy))\!+\!\alpha G_2(\bm{X}(\hx)\!-\!\bm{X}(\hy))\Big)\cdot\nabla_{\mathbb{S}^2}\bm{Y}(\bm{\hy})d\hy\\
&\quad+\alpha\minspace\mc{S}^2(\bm{X})\bm{Y}(\hx),
    \end{aligned}
\end{equation*}
with
\begin{equation*}
\begin{aligned}
G_\alpha(\bm{x})&=\frac{1}{8\pi}\paren{ G_1(\bm{x})+\alpha G_2(\bm{x})}, \; \bm{x}=(x_1,x_2,x_3),\\
(G_1)_{i,j}(\bm{x})&=\frac{\delta_{ij}}{\abs{\bm{x}}},\qquad (G_2)_{i,j}(\bm{x})=\frac{x_ix_j}{\abs{\bm{x}}^3},
\end{aligned}
\end{equation*}
and $T_S$ given in \eqref{Ttilde_law}.
In particular, $\mc{S}(\bm{X})=\mc{S}_1(\bm{X})$. Propositions \ref{FGateauxEstimate} and \ref{semigroup_est} hold analogously for $\mc{S}_\alpha(\bm{X})$, as all the remainder estimates were always done independently for each part of the kernel $G$ and Proposition \ref{frozen_coef_prop} already included the parameter $\alpha$. In particular, for all $\alpha\in[0,1]$, it holds that
\begin{equation*}
   \frac{1}{C}\|\bm{Y}\|_{C^{1,\gamma}(\mathbb{S}^2)}\geq \|(z-\mc{S}_\alpha(\bm{X}))\bm{Y}\|_{C^{\gamma}(\mathbb{S}^2)}\geq C\|\bm{Y}\|_{C^{1,\gamma}(\mathbb{S}^2)}.
\end{equation*}
Then, by the method of continuity, it suffices to show that the inverse of $(z-\mc{S}_0(\bm{X}))$ exists. Additionally, define a new family of operators $\mc{S}_{0,\sigma}(\bm{X})$, $\sigma\in[0,1]$, as follows
\begin{equation*}
    \begin{aligned}
&\mc{S}_{0,\sigma}(\bm{X})\bm{Y}(\hx)\\
&\quad=-\!\int_{\mathbb{S}^2}\!\!\! \nabla_{\mathbb{S}^2}\Big((1\!-\!\sigma) G_1(\hx\!-\!\hy)\!+\!\sigma G_1(\bm{X}(\hx)\!-\!\bm{X}(\hy))\Big)\cdot\nabla_{\mathbb{S}^2}\bm{Y}(\bm{\hy})d\hy,
    \end{aligned}
\end{equation*}
so that $\mc{S}_{0,1}(\bm{X})=\mc{S}_0(\bm{X})$. 
Then, taking into account \eqref{continuity_est}, it is clear that the following bound holds for all $\sigma\in[0,1]$,
\begin{equation*}
   \frac{1}{C}\|\bm{Y}\|_{C^{1,\gamma}(\mathbb{S}^2)}\geq \|(z-\mc{S}_{0,\sigma}(\bm{X}))\bm{Y}\|_{C^{\gamma}(\mathbb{S}^2)}\geq C\|\bm{Y}\|_{C^{1,\gamma}(\mathbb{S}^2)}.
\end{equation*}
Hence, by the method of continuity again we just need to show that $(z-\mc{S}_{0,0}(\bm{X}))$ is invertible. Since the range is closed, it suffices to show that it is also dense. The operator $\mc{S}_{0,0}(\bm{X})$ is linear and explicit, so we can compute its eigenspace.
Since
\begin{align*}
    \mc{S}_{0,0}(\bm{X})\bm{Y}=\frac{1}{8\pi}\int_{\mathbb{S}^2} \frac{1}{\abs{\hx-\hy}}\Delta_{\mathbb{S}^2}\bm{Y}(\bm{\hy})d\hy,
\end{align*}
we only have to check a component.
From \cite{Folland:introduction-pdes}, a single layer potential $u\paren{\bm{x}}$ of $g\paren{\hy}$ with
\begin{align*}
    u\paren{\bm{x}}=\frac{1}{4\pi}\int_{\mathbb{S}^2} \frac{1}{\abs{\xb-\hy}}g\paren{\hy}d\hy
\end{align*}
can be transformed into a harmonic problem with
\begin{align}
\begin{array}{ll}
    \Delta u=0 &\mbox{ in } \mathbb{R}^3\setminus \mbs\\
    \jump{u}=0,~~ \jump{\nabla u \cdot \bm{n}}=g &\mbox{ on }  \mbs
\end{array}\label{laplace_eq}
\end{align}
If we denote the standard spherical coordinate system $(r,\theta,\vph)$, where $r$ is the radial coordinate, $\theta$ is the polar angle, and $\vph$ is the azimuthal angle, then for the harmonic equation on $\mathbb{R}^3\setminus \mbs$, by separation of variables \cite{Courant-Hilbert:mathematical-physics-vol1-book}, we obtain some solutions $u_{\ell m}\paren{r,\theta, \vph}$ with $l\geq0 $ and $ \abs{m}\leq \ell$ :
\begin{align*}
\begin{split}
    u_{\ell m}\paren{r,\theta, \vph}=\left\{
    \begin{array}{rc}
        Ar^\ell Y_{\ell m},    &   \abs{r}<1\\
        Br^{-(\ell+1)}Y_{\ell m},&   \abs{r}>1
    \end{array}\right.
\end{split}
\end{align*}
where $Y_{l,m}\paren{\theta, \vph}$ is the usual spherical harmonic function of degree $l$ and order $m$, which satisfies the following equation:
\begin{equation}\label{Ylmlap}
\Delta_{\mathbb{S}^2}Y_{\ell m}=\frac{1}{\sin\theta}\PD{}{\theta}\paren{\sin\theta\PD{Y_{\ell m}}{\theta}}+\frac{1}{\sin^2\theta}\PDD{2}{Y_{\ell m}}{\vph}=-\ell(\ell+1)Y_{\ell m}.
\end{equation}
By plugging $u_{\ell m}$ into \eqref{laplace_eq}, we obtain
\begin{align*}
    u_{\ell m}=\frac{1}{2\ell+1}Y_{\ell m}.
\end{align*}
Therefore, combining \eqref{Ylmlap},
\begin{equation*}
    \mc{S}_{0,0}(\bm{X})Y_{\ell,m}=-\frac{\ell(\ell+1)}{2(2\ell+1)}Y_{\ell,m},\qquad \ell\geq0,\quad  \abs{m}\leq \ell.
\end{equation*}
Finally, since finite linear combinations of $Y_{l,m}$ are dense in $C^\infty(\mathbb{S}^2)$, we conclude the existence of the inverse $(z-\mc{S}_{0,0}(\bm{X}))^{-1}:h^{\gamma}(\mathbb{S}^2)\rightarrow h^{1,\gamma}(\mathbb{S}^2)$. 

\end{proof}

\begin{comment}
(The following two propositions are standard and follow exactly as in Rodenberg's thesis) \begin{prop}\label{IntermediateSpacesHolder}
Let $A:D(\mc{S})=C^{1,\alpha}(\mathbb{S}^2)\rightarrow C^\alpha(\mathbb{S}^2)$ be $A=\mc{S}(\bm{X})$ as in \eqref{Soperator}, with $\bm{X}\in\mc{O}=\{\bm{X}\in C^{1,\gamma}(\mathbb{S}^2):|\bm{X}|_*\geq m>0\}$ and $\alpha<\gamma<1$.
Then, there is some $\sigma\in(0,1)$ such that $D_A(\sigma,\infty)\simeq C^\gamma(\mathbb{S}^2)$ and $D_A(\sigma+1,\infty)\simeq C^{1,\gamma}(\mathbb{S}^2)$.

\textcolor{blue}{(to complete)}
\end{prop}

\begin{prop}\label{IntermediateSpacesLittleHolder}
Let $A:D(\mc{S})=h^{1,\alpha}(\mathbb{S}^2)\rightarrow h^\alpha(\mathbb{S}^2)$ be $A=\mc{S}(\bm{X})$ as in \eqref{Soperator}, with $\bm{X}\in\mc{O}=\{\bm{X}\in h^{1,\gamma}(\mathbb{S}^2):|\bm{X}|_*\geq m>0\}$ and $\alpha<\gamma<1$.
Then, there is some $\sigma\in(0,1)$ such that $D_A(\sigma)\simeq h^\gamma(\mathbb{S}^2)$ and $D_A(\sigma+1)\simeq h^{1,\gamma}(\mathbb{S}^2)$.

\textcolor{blue}{(to complete)}
\end{prop}
\end{comment}

\section{Higher Regularity}\label{sec:highreg}

Following the notation in Section \ref{sec:leading:N}, we recall that
\begin{equation*}
    \begin{aligned}
    &\PD{\bm{X}}{t}(\hx)=\mc{N}(\bm{X})(\bm{T}(\nabla_{\mathbb{S}^2}\bm{X}))(\hx),
    \end{aligned}
\end{equation*}
where we denote, with $T$ given in \eqref{Taulaw},
\begin{equation*}
    \bm{T}(\nabla_{\mathbb{S}^2}\bm{X})=T(|\nabla_{\mathbb{S}^2}\bm{X}|)\nabla_{\mathbb{S}^2}\bm{X}.
\end{equation*}
We localize using the partition $\{\rho_n\}$ (see Section \ref{covering_ste}), and
linearize $\bm{T}$ at $\bm{X}(\hx_n)$,
\begin{equation*}
    \begin{aligned}
    \PD{}{t}(\rho_n\bm{X})(\hx)&=\rho_n(\hx)\mc{N}(\bm{X})(T_S(\nabla_{\mathbb{S}^2}\bm{X}(\hx_n))\nabla_{\mathbb{S}^2}\bm{X}))(\hx)\\
    &+\rho_n(\hx)\mc{N}(\bm{X})\Big(\bm{T}(\nabla_{\mathbb{S}^2}\bm{X})\!-\!T_S(\nabla_{\mathbb{S}^2}\bm{X}(\hx_n))\nabla_{\mathbb{S}^2}\bm{X}\Big)(\hx),
    \end{aligned}
\end{equation*}
where we recall that $T_S(\nabla_{\mathbb{S}^2}\bm{X})\nabla_{\mathbb{S}^2}\bm{Y}=\frac{d}{ds}\bm{T}(\nabla_{\mathbb{S}^2}(\bm{X}+s\bm{Y}))|_{s=0}$ was given in \eqref{Ttilde_law}.
 Next, we introduce the commutators,
\begin{equation}\label{eqrhoN}
    \begin{aligned}
    \PD{}{t}(\rho_n&\bm{X})(\hx)=\mc{N}(\bm{X})(T_S(\nabla_{\mathbb{S}^2}\bm{X}(\hx_n))\nabla_{\mathbb{S}^2}(\rho_n\bm{X}))(\hx)\\
    &+\mc{N}(\bm{X})\Big(\rho_n\big(\bm{T}(\nabla_{\mathbb{S}^2}\bm{X})-T_S(\nabla_{\mathbb{S}^2}\bm{X}(\hx_n))\nabla_{\mathbb{S}^2}\bm{X}\big)\Big)(\hx)\\
    &+[\rho_n,\mc{N}(\bm{X})](T_S(\nabla_{\mathbb{S}^2}\bm{X}(\hx_n))\nabla_{\mathbb{S}^2}\bm{X})(\hx)\\
    &-\mc{N}(\bm{X})(T_S(\nabla_{\mathbb{S}^2}\bm{X}(\hx_n))\bm{X}\nabla_{\mathbb{S}^2}\rho_n)(\hx)\\
    &+[\rho_n,\mc{N}(\bm{X})]\Big(\bm{T}(\nabla_{\mathbb{S}^2}\bm{X})-T_S(\nabla_{\mathbb{S}^2}\bm{X}(\hx_n))\nabla_{\mathbb{S}^2}\bm{X}\Big)(\hx),
    \end{aligned}
\end{equation}
and we move to stereographic coordinates to introduce the frozen-coefficient (at $t=0, \hx=\hx_n$) operator and the cutoff $\wh{\rho}_n$ \eqref{cutoff},
\begin{equation}\label{reg_splitting}
    \begin{aligned}
    \PD{}{t}(\rho_n\bm{X}_n)(\thetab)&=\mc{L}_{A_0}(\rho_n\bm{X}_n)(\thetab)+\sum_{j=1}^7\bm{f}^j(\bm{X})(\thetab),
    \end{aligned}
\end{equation}
with
\begin{equation*}
    \begin{aligned}
    \bm{f}^1(\bm{X})(\thetab)&=\mc{N}(\bm{X})\big(\rho_n\big(\bm{T}(\nabla_{\mathbb{S}^2}
    \bm{X})-T_S(\nabla_{\mathbb{S}^2}\bm{X}(\hx_n))\nabla_{\mathbb{S}^2}\bm{X}\big)\big)_n(\thetab),
    \end{aligned}
\end{equation*}
\begin{equation*}
    \begin{aligned}
    \bm{f}^2(\bm{X})&(\thetab)=\wh{\rho}_n(\thetab)\big[\mc{N}(\bm{X})-\mc{M}(A)\big](T_S(\nabla_{\mathbb{S}^2}\bm{X}(\hx_n))\nabla_{\mathbb{S}^2}(\rho_n\bm{X}))_n(\thetab),
    \end{aligned}
\end{equation*}
\begin{equation*}
    \begin{aligned}
    \bm{f}^3(\bm{X})&(\thetab)=\wh{\rho}_n(\thetab)\big(\mc{M}(A)(T_S(\nabla_{\mathbb{S}^2}\bm{X}(\hx_n))\nabla_{\mathbb{S}^2}(\rho_n\bm{X}))_n(\thetab)\!-\!\mc{L}_{A}(\rho_n\bm{X}_n)(\thetab)\big),
    \end{aligned}
\end{equation*}
\begin{equation*}
    \begin{aligned}
    \bm{f}^4(\bm{X})&(\thetab)=(1-\wh{\rho}_n(\thetab))\mc{N}(\bm{X})(T_S(\nabla_{\mathbb{S}^2}\bm{X}(\hx_n))\nabla_{\mathbb{S}^2}(\rho_n\bm{X}))_n(\thetab),
    \end{aligned}
\end{equation*}
\begin{equation*}
    \begin{aligned}
    \bm{f}^5(\bm{X})(\thetab)&=-(1-\wh{\rho}_n(\thetab))\mc{L}_{A}(\rho_n\bm{X}_n)(\thetab),
    \end{aligned}
\end{equation*}
\begin{equation*}
    \begin{aligned}
    \bm{f}^6(\bm{X})(\thetab)&=[\mc{L}_{A}-\mc{L}_{A_0}](\rho_n\bm{X}_n)(\thetab),
    \end{aligned}
\end{equation*}
and
\begin{equation*}
    \begin{aligned}
    \bm{f}^7(\bm{X})(\thetab)&=-\mc{N}(\bm{X})(T_S(\nabla_{\mathbb{S}^2}\bm{X}(\hx_n))\bm{X}\nabla_{\mathbb{S}^2}\rho_n)(\hX_n(\thetab))\\
    &\quad+[\rho_n,\mc{N}(\bm{X})]\bm{T}(\nabla_{\mathbb{S}^2}\bm{X})(\hX_n(\thetab)),
    \end{aligned}
\end{equation*}
where $A_0=\nabla\bm{X}_{0,n}(\bm{0})$ and $A=\nabla\bm{X}_n(\bm{0},t)$. 

\begin{prop}
Let $\bm{X}$ be the solution to the Peskin problem with initial data $\bm{X}_0\in h^{1,\gamma}(\mathbb{S}^2)$ constructed in Theorem \ref{MainTh}. Then,  for any $\alpha\in(0,1)$, it holds that $\bm{X}\in C^1((0,T]; C^{3,\alpha}(\mathbb{S}^2))$. Moreover, for any $3\leq n\in\mathbb{N}$ and $\alpha\in(0,1)$, assuming that $\mc{T}\in C^{n,\alpha}$, it holds that $\bm{X}\in C^1((0,T]; C^{n+1,\beta}(\mathbb{S}^2))$, for any $\beta<\alpha$.
\end{prop}
\begin{proof}

The main difficulty is to show the smoothing in space. In fact, assume we have the higher regularity information $\bm{X}\in L^\infty(0,T; C^{n+1,\alpha}(\mathbb{S}^2))$. Then, Theorem \ref{MainTh} states that $\partial_t \bm{X}\in C^0([0,T]; C^\gamma(\mathbb{S}^2))$, and using the equation together with $\bm{X}\in L^\infty(0,T; C^{n+1,\alpha}(\mathbb{S}^2))$, it is straightforward to see that $\partial_t\bm{X}\in L^\infty(0,T; C^{n,\alpha}(\mathbb{S}^2))$. Finally, to get the continuity in time for the higher regularity, it suffices to interpolate taking into account the higher regularity bounds and the continuity in the lower norm. 

We proceed to show the smoothing in space.
We will consider the following mollified version of the system \eqref{reg_splitting},
\begin{equation}\label{reg_splitting_mol}
    \begin{aligned}
    \PD{}{t}(\rho_n\bm{X}_n^\delta)(\thetab)&=\mc{L}_{A_0}(\rho_n\bm{X}_n^\delta)(\thetab)+\sum_{j=1}^7\mc{J}_\delta\bm{f}^j(\bm{X}^\delta)(\thetab),
    \end{aligned}
\end{equation}
 with mollified initial data $\bm{X}_{0,n}^\delta(\thetab)=\mc{J}_\delta \bm{X}_{0,n}(\thetab)$, where $\mc{J}_\delta$ is the standard mollifier by convolution with a Gaussian.

Our main goal is to obtain uniform in $\delta$ bounds for $\bm{X}^\delta$ in $L^\infty(0,T;C^{n,\alpha}(\mathbb{S}^2))$. In fact, by construction, $\bm{X}^\delta$ is smooth, and it is not hard to show that the limit of $\{\bm{X}^\delta\}$ in  $L^\infty(0,T; C^{1,\gamma}(\mathbb{S}^2))$ is given by the solution $\bm{X}$ in Theorem \ref{MainTh}. Hence, by interpolation and using the uniform bounds, we would conclude that $\bm{X}\in L^\infty(0,T;C^{n,\beta}(\mathbb{S}^2))$ for any $\beta<\alpha$. We thus proceed to obtain the uniform bounds first, and show the convergence  $\bm{X}^\delta\to\bm{X}$ at the end.

We use the semigroup $e^{t\mc{L}_{A_0}}$ to write $\rho_n\bm{X}_n^\delta(t)$ in Duhamel form:
\begin{equation}\label{duhamel}
    \begin{aligned}
    \rho_n\bm{X}_n^\delta(t)&=e^{(t-t_0) \mc{L}_{A_0}}(\rho_n\bm{X}_n^\delta(t_0))+\sum_{j=1}^7\int_{t_0}^t e^{(t-\tau)\mc{L}_{A_0}}\mc{J}_\delta\bm{f}^j(\bm{X}^\delta)(\tau)d\tau.
    \end{aligned}
\end{equation}
In the following, we will repeatedly use the estimates in Propositions \ref{lin_heat_semi}-\ref{lin_heat_semi2}.
For simplicity of notation, we will drop the index $\delta$ and the mollifier $\mc{J}_\delta$.

\vspace{0.2cm}

\noindent\underline{Improving regularity to $C^{1,\alpha}(\mathbb{S}^2)$:}
We proceed to obtain bounds in $C^\alpha$, $\alpha\in(0,1)$, for the terms $\bm{f}^j$.
We will be denoting  $C=C(|\bm{X}|_*,\|\bm{X}\|_{C^1(\mathbb{S}^2)},\|\mc{T}\|_{C^2})$, $C(R)=C(|\bm{X}|_*,\|\bm{X}\|_{C^1(\mathbb{S}^2)},\|\mc{T}\|_{C^2},R)$ in the bounds that follow. Lemma \ref{Nbound_lem} gives that
\begin{equation*}
    \begin{aligned}
    \|\bm{f}^1\|_{C^\alpha(\mathbb{R}^2)}&\leq\! C\|\rho_n(\bm{T}(\nabla_{\mathbb{S}^2}\bm{X})\!-\!T_S(\nabla_{\mathbb{S}^2}\bm{X}(\hx_n))\nabla_{\mathbb{S}^2}\bm{X})\|_{C^\alpha(\mathbb{S}^2)}.
    \end{aligned}
\end{equation*}
We note that
\begin{equation*}
    \begin{aligned}
    I&:=\bm{T}(\nabla_{\mathbb{S}^2}\bm{X}(\hx_1))\!-\!T_S(\nabla_{\mathbb{S}^2}\bm{X}(\hx_n))\nabla_{\mathbb{S}^2}\bm{X}(\hx_1)\\
    &\hspace{2cm}-\bm{T}(\nabla_{\mathbb{S}^2}\bm{X}(\hx_2))\!+\!T_S(\nabla_{\mathbb{S}^2}\bm{X}(\hx_n))\nabla_{\mathbb{S}^2}\bm{X}(\hx_2)\\
    &=\bm{T}(\nabla_{\mathbb{S}^2}\bm{X}(\hx_1))\!-\!\bm{T}(\nabla_{\mathbb{S}^2}\bm{X}(\hx_2))\!-\!T_S(\nabla_{\mathbb{S}^2}\bm{X}(\hx_n))(\nabla_{\mathbb{S}^2}\bm{X}(\hx_1)\!-\!\nabla_{\mathbb{S}^2}\bm{X}(\hx_2)),
    \end{aligned}
    \end{equation*}
    thus
    \begin{equation*}
        \begin{aligned}
    |I|&=|\int_0^1 \big(T_S(s\nabla_{\mathbb{S}^2}\bm{X}(\hx_1)+(1-s)\nabla_{\mathbb{S}^2}\bm{X}(\hx_2))-T_S(\nabla_{\mathbb{S}^2}\bm{X}(\hx_n))\big)ds\\\
    &\hspace{2cm}\times(\nabla_{\mathbb{S}^2}\bm{X}(\hx_1)\!-\!\nabla_{\mathbb{S}^2}\bm{X}(\hx_2))|\\
    &\leq C\max\{|\nabla_{\mathbb{S}^2}\bm{X}(\hx_1)\!-\!\nabla_{\mathbb{S}^2}\bm{X}(\hx_n)|,|\nabla_{\mathbb{S}^2}\bm{X}(\hx_2)\!-\!\nabla_{\mathbb{S}^2}\bm{X}(\hx_n)\}\\
    &\hspace{1cm}\times|\nabla_{\mathbb{S}^2}\bm{X}(\hx_1)\!-\!\nabla_{\mathbb{S}^2}\bm{X}(\hx_2)|.
    \end{aligned}
\end{equation*}
Hence, thanks to the presence of $\rho_n$, we obtain
\begin{equation}\label{f1bound}
    \begin{aligned}
    \|\bm{f}^1\|_{C^\alpha(\mathbb{R}^2)}&\leq\! C\varepsilon(R)\|\rho_n\|_{C^\alpha(\mathbb{S}^2)}\|\nabla_{\mathbb{S}^2}\bm{X}\|_{C^\alpha(B_{\hx_n,2R}\cap \mathbb{S}^2)}.
    \end{aligned}
\end{equation}
Using Lemma \ref{lem_dif},
\begin{equation}\label{f2_bound}
    \begin{aligned}
    \|\bm{f}^2\|_{C^\alpha(\mathbb{S}^2)}&\leq C\Big(\varepsilon(R)\|\nabla_{\mathbb{S}^2}(\rho_n\bm{X})\|_{C^\alpha(\mathbb{S}^2)}\\
    &\quad+\|\nabla_{\mathbb{S}^2}\bm{X}\|_{C^{\frac{\alpha}{2}}(B_{5R}(\hx_n)\cap\mathbb{S}^2)}\|\nabla_{\mathbb{S}^2}(\rho_n\bm{X})\|_{C^{\frac{\alpha}{2}}(\mathbb{S}^2)}\Big),
    \end{aligned}
\end{equation}
while $\bm{f}^3$ is identically zero (see \eqref{Ntheta_aux} and \eqref{defnL2}). Lemma \ref{Nout_bound} provides the estimate for $\bm{f}^4$,
\begin{equation}\label{f4_bound}
\begin{aligned}
    \|\bm{f}^4\|_{C^\alpha(\mathbb{S}^2)}&\leq C(R)\|\nabla_{\mathbb{S}^2}(\rho_n\bm{X})\|_{C^0(\mathbb{S}^2)}.
    \end{aligned}
\end{equation}
Then, by Lemmas \ref{Nbound_lem} and \ref{lem_commutator},
\begin{equation}\label{f7}
\begin{aligned}
    \|\bm{f}^7\|_{C^\alpha(\mathbb{S}^2)}&\leq C\|\nabla_{\mathbb{S}^2}\rho_n\|_{C^{\alpha}(\mathbb{S}^2)},
    \end{aligned}
\end{equation}
and Lemma \ref{MAout_bound},
\begin{equation}\label{f5}
\begin{aligned}
    \|\bm{f}^5\|_{C^\alpha(\mathbb{S}^2)}&\leq C\|\nabla_{\mathbb{S}^2}(\rho_n\bm{X})\|_{C^0(\mathbb{S}^2)}.
    \end{aligned}
\end{equation}
Finally, by writing
\begin{equation*}
    \begin{aligned}
    [\mc{L}_A-\mc{L}_{A_0}](\rho_n\bm{X}_n)(\thetab)&=[\tilde{M}(A)-\tilde{M}(A_0)](T_F(A)\nabla(\rho_n\bm{X}_n))(\thetab)\\
    &\quad+\tilde{M}(A_0)((T_F(A)-T_F(A_0))\nabla(\rho_n\bm{X}_n))(\thetab),
    \end{aligned}
\end{equation*}
Lemma \ref{MA_bound} yields that
\begin{equation}\label{f6}
\begin{aligned}
    \|\bm{f}^6\|_{C^\alpha(\mathbb{S}^2)}&\leq C\|A-A_0\|\|\nabla_{\mathbb{S}^2}(\rho_n\bm{X})\|_{C^\alpha(\mathbb{S}^2)}.
    \end{aligned}
\end{equation}
We thus see from \eqref{duhamel} and Propositions \ref{lin_heat_semi}-\ref{lin_heat_semi2} that we can bootstrap to get that $\bm{X}\in L^\infty(0,T;C^{1,\alpha}(\mathbb{S}^2))$ for all $\alpha\in(0,1)$. In fact, consider the case $\gamma<\frac12$ and take $\alpha$ such that $\gamma<\alpha+\frac12\leq2\gamma$.
Then, with $0<\epsilon\leq \gamma-\alpha$ arbitrarily small, and substituting the bounds for $\bm{f}^j$, we obtain 
\begin{equation*}
    \begin{aligned}
     \|\rho_n\bm{X}&\|_{L^2(0,T;C^{\frac32+\alpha}(\mathbb{S}^2))}\leq C \|\rho_n(0)\bm{X}_0\|_{C^{1+\alpha+\epsilon}(\mathbb{S}^2))}+C\sum_{j=1}^7\|\bm{f}^j\|_{L^2(0,T;C^{\frac12+\alpha}(\mathbb{R}^2))}\\
     &\leq C \|\rho_n(0)\bm{X}_0\|_{C^{1+\gamma}(\mathbb{S}^2)}+C(R,T)+\|\bm{X}\|_{L^4(0,T;C^{1+\frac14+\frac{\alpha}2}(\mathbb{S}^2))}^2\\
     &\quad+C\varepsilon(R,T)\big(\|\bm{X}\|_{L^2(0,T;C^{\frac32+\alpha}(B_{5R}(\hx_n)\cap\mathbb{S}^2))}+\|\rho_n\bm{X}\|_{L^2(0,T;C^{\frac32+\alpha}(\mathbb{S}^2)}\big),
    \end{aligned}
\end{equation*}
and so
\begin{equation}\label{aux_reg}
    \begin{aligned}
     \|\rho_n\bm{X}\|_{L^2(0,T;C^{\frac32+\alpha}(\mathbb{S}^2))}&\leq C \|\rho_n(0)\bm{X}_0\|_{C^{1+\gamma}(\mathbb{S}^2)}+C(R,T)\\
     &\quad+C\varepsilon(R,T)\|\bm{X}\|_{L^2(0,T;C^{\frac32+\alpha}(B_{5R}(\hx_n)\cap\mathbb{S}^2))}.
    \end{aligned}
\end{equation}
Now, we can write
\begin{equation*}
    \begin{aligned}
    \|\bm{X}\|_{L^2(0,T;C^{\frac32+\alpha}(B_{5R}(\hx_n)\cap\mathbb{S}^2))}\leq \|\!\!\!\sum_{m\in M_n}\rho_{m}\bm{X}\|_{L^2(0,T;C^{\frac32+\alpha}(B_{5R}(\hx_n)\cap\mathbb{S}^2))},
    \end{aligned}
\end{equation*}
where the cardinal number $|M_n|$ can be picked independent of $R$ and $n$, since the radius of the support of $\rho_n$ and $B_{5R}(\hx_n)\cap\mathbb{S}^2$ are comparable. Therefore, adding in $n$ in \eqref{aux_reg} we obtain
\begin{equation*}
    \begin{aligned}
     \sum_n\|\rho_n\bm{X}\|_{L^2(0,T;C^{\frac32+\alpha}(\mathbb{S}^2))}&\leq C(R,T)+C\varepsilon(R,T)|M_n|\sum_n\|\rho_n\bm{X}\|_{L^2(0,T;C^{\frac32+\alpha}(\mathbb{S}^2))},
    \end{aligned}
\end{equation*}
hence we conclude that
\begin{equation*}
    \begin{aligned}
     \|\bm{X}\|_{L^2(0,T;C^{\frac32+\alpha}(\mathbb{S}^2))}&\leq \sum_n\|\rho_n\bm{X}\|_{L^2(0,T;C^{\frac32+\alpha}(\mathbb{S}^2))}\leq C(R,T).
    \end{aligned}
\end{equation*}
In particular, choosing $\alpha=2\gamma-\frac12$, this uniform bound allows us to conclude that $\bm{X}\in L^2(0,T;C^{1+2\gamma}(\mathbb{S}^2))$, and thus $\bm{X}(t) \in C^{1+2\gamma}(\mathbb{S}^2)$ for a.e. $t\in(0,T)$. Now, pick $t_0\in(0,T)$ arbitrarily close to $0$ and such that $\bm{X}(t_0)\in C^{1+2\gamma}(\mathbb{S}^2)$. It is clear that we can repeat the process to find $t_1>t_0$ such that $\bm{X}(t_1)\in C^{1,\alpha}(\mathbb{S}^2)$ for any $\alpha\in(0,1)$ (the case $\gamma>\frac12$ follows in one step). Starting at $t_1$, we find that
\begin{equation*}
    \begin{aligned}
    \|\rho_n\bm{X}(t)\|_{C^{1,\alpha}(\mathbb{S}^2)}&\leq \|\rho_n\bm{X}(t_1)\|_{C^{1,\alpha}(\mathbb{S}^2)}+C\sup_{t_1\leq\tau\leq t}\sum_{m=1}^7\|\bm{f}^m\|_{C^{\alpha}(\mathbb{S}^2)}.
    \end{aligned}
\end{equation*}
We can thus take the supremum in $t\in(t_1,T)$ and use the previous estimates on $\bm{f}^m$ to conclude that $\bm{X} \in L^\infty(t_1,T;C^{1,\alpha}(\mathbb{S}^2))$ for any $t_1>0$ and any $\alpha\in(0,1)$.
\vspace{0.2cm}

\noindent\underline{Higher regularity:}
To study further smoothing, we first show that we can move derivatives in $\hx$ to derivatives in $\hy$. In fact,  
denoting $\Delta\bm{X}=\bm{X}(\hx)-\bm{X}(\hy)$,
\begin{equation*}
    \begin{aligned}
    &\nabla_{\mathbb{S}^2}\mc{N}(\bm{X})\bm{Y}(\hx)=\\
    &=-\int_{\mathbb{S}^2}\nabla_{\mathbb{S}^2,\hx}\nabla_{\mathbb{S}^2,\hy}G(\Delta\bm{X})\cdot\Delta\nabla_{\mathbb{S}^2}\bm{Y}d\hy\\
    &=-\!\int_{\mathbb{S}^2}\!\!\Big(-\nabla_{\mathbb{S}^2,\hy}\nabla_{\mathbb{S}^2,\hy}G(\Delta\bm{X})\!+\!\nabla_{\mathbb{S}^2,\hy}\big(\nabla_{\mathbb{S}^2,\hx}G(\Delta\bm{X})\!+\!\nabla_{\mathbb{S}^2,\hy}G(\Delta\bm{X})\big)\Big)\!\cdot\!\Delta\nabla_{\mathbb{S}^2}\bm{Y}d\hy,
    \end{aligned}
\end{equation*}
so further integration by parts gives that
\begin{equation}\label{move_der}
    \begin{aligned}
    \nabla_{\mathbb{S}^2}\mc{N}&(\bm{X})\bm{Y}(\hx)=\mc{N}(\bm{X})\nabla_{\mathbb{S}^2}\bm{Y}(\hx)\\
    &\quad-\int_{\mathbb{S}^2}\nabla_{\mathbb{S}^2,\hy}\big(\nabla_{\mathbb{S}^2,\hx}G(\Delta\bm{X})\!+\!\nabla_{\mathbb{S}^2,\hy}G(\Delta\bm{X})\big)\cdot \Delta\nabla_{\mathbb{S}^2}\bm{Y}(\hy)d\hy\\
    &=\mc{N}(\bm{X})\nabla_{\mathbb{S}^2}\bm{Y}(\hx)\\
    &\quad-\int_{\mathbb{S}^2}\nabla_{\mathbb{S}^2,\hy}\Big( \PD{}{x_i}G(\Delta\bm{X})\big(\nabla_{\mathbb{S}^2}X_i(\hx)-\nabla_{\mathbb{S}^2}X_i(\hy)\big)\Big)\cdot \Delta\nabla_{\mathbb{S}^2}\bm{Y}(\hy)d\hy.
    \end{aligned}
\end{equation}
Therefore, we take a derivative in \eqref{eqrhoN} to get
\begin{equation*}
    \begin{aligned}
    \PD{}{t}\nabla_{\mathbb{S}^2}&(\rho_n\bm{X})(\hx)=\mc{N}(\bm{X})\big(\nabla_{\mathbb{S}^2}\big(T_S(\nabla_{\mathbb{S}^2}\bm{X}(\hx_n))\nabla_{\mathbb{S}^2}(\rho_n\bm{X})\big)\big)(\hx)\\
    &\quad+\mc{N}(\bm{X})\Big(\nabla_{\mathbb{S}^2}\big(\rho_n\big(\bm{T}(\nabla_{\mathbb{S}^2}\bm{X})-T_S(\nabla_{\mathbb{S}^2}\bm{X}(\hx_n))\nabla_{\mathbb{S}^2}\bm{X}\big)\big)\Big)(\hx)\\
    &-\int_{\mathbb{S}^2}\nabla_{\mathbb{S}^2,\hy}\Big( \PD{}{x_i}G(\Delta\bm{X})\big(\nabla_{\mathbb{S}^2}X_i(\hx)-\nabla_{\mathbb{S}^2}X_i(\hy)\big)\Big)\cdot\\
    &\hspace{1cm}\times\Delta \big(T_S(\nabla_{\mathbb{S}^2}\bm{X}(\hx_n))\nabla_{\mathbb{S}^2}(\rho_n\bm{X})\big)(\hy)d\hy\\
    &-\int_{\mathbb{S}^2}\nabla_{\mathbb{S}^2,\hy}\Big( \PD{}{x_i}G(\Delta\bm{X})\big(\nabla_{\mathbb{S}^2}X_i(\hx)-\nabla_{\mathbb{S}^2}X_i(\hy)\big)\Big)\cdot\\
    &\hspace{1cm}\times\Delta \big(\rho_n\big(\bm{T}(\nabla_{\mathbb{S}^2}\bm{X})-T_S(\nabla_{\mathbb{S}^2}\bm{X}(\hx_n))\nabla_{\mathbb{S}^2}\bm{X}\big)(\hy)d\hy\\
    &+\nabla_{\mathbb{S}^2}[\rho_n,\mc{N}(\bm{X})]\bm{T}(\nabla_{\mathbb{S}^2}\bm{X})(\hx)-\nabla_{\mathbb{S}^2}\mc{N}(\bm{X})(T_S(\nabla_{\mathbb{S}^2}\bm{X}(\hx_n))\bm{X}\nabla_{\mathbb{S}^2}\rho_n)(\hx)).
    \end{aligned}
\end{equation*}
We introduce the frozen-coefficient operator and the cutoff $\wh{\rho}_n$, 
\begin{equation*}
    \begin{aligned}
    \PD{}{t}\nabla_{\mathbb{S}^2}(\rho_n\bm{X}_n)(\thetab)&=\mc{L}_{A_1}(\nabla_{\mathbb{S}^2}(\rho_n\bm{X})\big)_n(\thetab)+\sum_{j=1}^8 f^j(\thetab),
    \end{aligned}
\end{equation*}
\begin{equation*}
    \begin{aligned}
    f^1(\thetab)&=\mc{N}(\bm{X})\Big(\nabla_{\mathbb{S}^2}\big(\rho_n\big(\bm{T}(\nabla_{\mathbb{S}^2}\bm{X})-T_S(\nabla_{\mathbb{S}^2}\bm{X}(\hx_n))\nabla_{\mathbb{S}^2}\bm{X}\big)\big)\Big)_n(\thetab),
    \end{aligned}
\end{equation*}
\begin{equation*}
    \begin{aligned}
    f^2(\thetab)&=    \wh{\rho}_n(\thetab)[\mc{N}(\bm{X})-\mc{M}(A)]\big(\big(T_S(\nabla_{\mathbb{S}^2}\bm{X}(\hx_n))\nabla_{\mathbb{S}^2}\nabla_{\mathbb{S}^2}(\rho_n\bm{X})\big)\big)_n(\thetab),
    \end{aligned}
\end{equation*}
\begin{equation*}
    \begin{aligned}
    f^3&(\thetab)\!=\!\wh{\rho}_n(\thetab)\Big(\mc{M}(A)(T_S(\nabla_{\mathbb{S}^2}\bm{X}(\hx_n))\nabla_{\mathbb{S}^2}\nabla_{\mathbb{S}^2}(\rho_n\bm{X}))_n(\thetab)\!-\!\mc{L}_A(\nabla_{\mathbb{S}^2}(\rho_n\bm{X}))_n(\thetab)\Big),
    \end{aligned}
\end{equation*}
\begin{equation*}
    \begin{aligned}
    f^4(\thetab)&=(1\!-\!\wh{\rho}_n(\thetab))\mc{N}(\bm{X})\big(T_S(\nabla_{\mathbb{S}^2}\bm{X}(\hx_n))\nabla_{\mathbb{S}^2}\nabla_{\mathbb{S}^2}(\rho_n\bm{X})\big)_n(\thetab),
    \end{aligned}
\end{equation*}
\begin{equation*}
    \begin{aligned}
    f^5(\thetab)&=(1\!-\!\wh{\rho}_n(\thetab))\mc{L}_A(\nabla_{\mathbb{S}^2}(\rho_n\bm{X}))_n(\thetab),
    \end{aligned}
\end{equation*}
\begin{equation*}
    \begin{aligned}
    f^6(\thetab)=[\mc{L}_A-\mc{L}_{A_1}]\big(\nabla_{\mathbb{S}^2}(\rho_n\bm{X})\big)_n(\thetab),
    \end{aligned}
\end{equation*}
\begin{equation*}
    \begin{aligned}
    f^7(\thetab)&=\nabla_{\mathbb{S}^2}[\rho_n,\mc{N}(\bm{X})]\bm{T}(\nabla_{\mathbb{S}^2}\bm{X})(\hX_n(\thetab))\\
    &\quad-\nabla_{\mathbb{S}^2}\mc{N}(\bm{X})(T_S(\nabla_{\mathbb{S}^2}\bm{X}(\hx_n))\bm{X}\nabla_{\mathbb{S}^2}\rho_n)(\hX_n(\thetab)),
    \end{aligned}
\end{equation*}
and
\begin{equation*}
    \begin{aligned}
    f^8(\thetab)&=f^{8,1}(\thetab)+f^{8,2}(\thetab),
    \end{aligned}
\end{equation*}
with
\begin{equation*}
    \begin{aligned}
    f^{8,1}(\thetab)&=-\int_{\mathbb{S}^2}\nabla_{\mathbb{S}^2,\hy}\Big( \PD{}{x_i}G(\bm{X}(\hX(\thetab)-\bm{X}(\hy))\big(\nabla_{\mathbb{S}^2}X_i(\hX(\thetab))-\nabla_{\mathbb{S}^2}X_i(\hy)\big)\Big)\cdot\\
    &\hspace{1cm}\times\Delta \big(T_S(\nabla_{\mathbb{S}^2}\bm{X}(\hx_n))\nabla_{\mathbb{S}^2}(\rho_n\bm{X})\big)(\hy)d\hy,
    \end{aligned}
    \end{equation*}
\begin{equation*}
    \begin{aligned}
    f^{8,2}(\thetab)&=-\int_{\mathbb{S}^2}\nabla_{\mathbb{S}^2,\hy}\Big( \PD{}{x_i}G(\bm{X}(\hX(\thetab))-\bm{X}(\hy))\big(\nabla_{\mathbb{S}^2}X_i(\hX(\thetab))-\nabla_{\mathbb{S}^2}X_i(\hy)\big)\Big)\cdot\\
    &\hspace{1cm}\times\Delta \big(\rho_n\big(\bm{T}(\nabla_{\mathbb{S}^2}\bm{X})-T_S(\nabla_{\mathbb{S}^2}\bm{X}(\hx_n))\nabla_{\mathbb{S}^2}\bm{X}\big)(\hy)d\hy,
    \end{aligned}
\end{equation*}
and $A_1=\nabla\bm{X}_n(\bm{0},t_1)$, $A=\nabla\bm{X}_n(\bm{0},t)$.
Thus, we proceed as we previously did in \eqref{duhamel},
\begin{equation}\label{nablaX_eq_duham}
    \begin{aligned}
    \nabla_{\mathbb{S}^2}(\rho_n\bm{X}_n)(t)&=e^{(t-t_0) \mc{L}_{A_1}}(\nabla_{\mathbb{S}^2}(\rho_n\bm{X}_n)(t_0))+\sum_{j=1}^8\int_{t_0}^t e^{(t-\tau)\mc{L}_{A_1}}f^j(\tau)d\tau.
    \end{aligned}
\end{equation}
Therefore, to bootstrap and get $C^{2,\alpha}$ regularity we need to use Propositions \ref{lin_heat_semi}-\ref{lin_heat_semi2} and obtain $C^{\alpha}$ estimates for the forced terms above.
The estimate \eqref{Jn8bound} in Lemma \ref{lem_dif} gives that
\begin{equation*}
    \begin{aligned}
    \|f^2\|_{C^\alpha(\mathbb{S}^2)} &\leq C\Big(\varepsilon(R)\|\nabla_{\mathbb{S}^2}^2(\rho_n\bm{X})\|_{C^\alpha(\mathbb{S}^2)}\\
    &\quad+\|\nabla_{\mathbb{S}^2}\bm{X}\|_{C^\alpha(\mathbb{S}^2)}\|\nabla_{\mathbb{S}^2}^2(\rho_n\bm{X})\|_{C^0(\mathbb{S}^2)}+\|\nabla_{\mathbb{S}^2}^2(\rho_n\bm{X})\|_{C^0(\mathbb{S}^2)}\Big),
    \end{aligned}
\end{equation*}
while $f^3\equiv0$, and Lemmas \ref{Nout_bound} and \ref{MAout_bound} provide that
\begin{equation*}
\begin{aligned}
    \|f^4\|_{C^\alpha(\mathbb{S}^2)}+\|f^5\|_{C^\alpha(\mathbb{S}^2)}&\leq C(R)\|\nabla_{\mathbb{S}^2}^2\bm{X}\|_{C^0(\mathbb{S}^2)}.
    \end{aligned}
\end{equation*}
As done before in \eqref{f6}, we have that
\begin{equation*}
\begin{aligned}
    \|f^6\|_{C^\alpha(\mathbb{S}^2)}&\leq C\|A-A_1\|\|\nabla_{\mathbb{S}^2}^2(\rho_n\bm{X})\|_{C^\alpha(\mathbb{S}^2)}.
    \end{aligned}
\end{equation*}
We interpolate the $C^2(\mathbb{S}^2)$ norm followed by Young's inequality to get a small coefficient for the higher regularity part:
\begin{equation*}
    \begin{aligned}
    \|\nabla_{\mathbb{S}^2}^2(\rho_n\bm{X})\|_{C^0(\mathbb{S}^2)}\leq C(\varepsilon)\|\nabla_{\mathbb{S}^2}(\rho_n\bm{X})\|_{C^{1-\alpha}(\mathbb{S}^2)}+\varepsilon\|\nabla_{\mathbb{S}^2}^2(\rho_n\bm{X})\|_{C^\alpha(\mathbb{S}^2)},
    \end{aligned}
\end{equation*}
so that
\begin{equation*}
    \begin{aligned}
    \sum_{j=2}^6\|f^j\|_{C^\alpha(\mathbb{S}^2)}\leq C\varepsilon(R,\Delta t)\|\nabla_{\mathbb{S}^2}^2(\rho_n\bm{X})\|_{C^\alpha(\mathbb{S}^2)}\!+\!C(R)\|\nabla_{\mathbb{S}^2}(\rho_n\bm{X})\|_{C^{1-\alpha}(\mathbb{S}^2)},
    \end{aligned}
\end{equation*}
where from now on the constants $C$ and $C(R,\Delta t)$, $\Delta t=t-t_1,$ also depend on the controlled norm $\|\bm{X}\|_{L^\infty(0,T;C^{1,\max\{\alpha,1-\alpha\}}(\mathbb{S}^2))}$.
Next, 
\begin{equation*}
    \begin{aligned}
    \|f^1\|_{C^\alpha(\mathbb{S}^2)}&\leq C\varepsilon \|\nabla_{\mathbb{S}^2}\bm{X}\|_{C^\alpha(\mathbb{S}^2)}\\
    &\quad+C\|\rho_n\big(T_S(\nabla_{\mathbb{S}^2}\bm{X})\nabla_{\mathbb{S}^2}^2\bm{X}-T_S(\nabla_{\mathbb{S}^2}(\hx_n))\nabla_{\mathbb{S}^2}^2\bm{X}\big)\|_{C^\alpha(\mathbb{S}^2)}\\
    &\leq C+C\|\nabla_{\mathbb{S}^2}\bm{X}\|_{C^\alpha(\mathbb{S}^2)}\|\nabla_{\mathbb{S}^2}^2\bm{X}\|_{C^0(B_{\hx_n,2R}\cap\mathbb{S}^2)}\\
    &\quad+C\varepsilon(R) \|\nabla_{\mathbb{S}^2}^2\bm{X}\|_{C^\alpha(B_{\hx_n,2R}\cap\mathbb{S}^2)},
    \end{aligned}
\end{equation*}
so by interpolation again
\begin{equation*}
    \begin{aligned}
    \|f^1\|_{C^\alpha(\mathbb{S}^2)}&\leq C(R)+C\varepsilon(R) \|\nabla_{\mathbb{S}^2}^2\bm{X}\|_{C^\alpha(B_{\hx_n,2R}\cap\mathbb{S}^2)}.
    \end{aligned}
\end{equation*}
The term $f^8$ is lower order, and thus we can control it using interpolation once more. In fact, taking the derivative in the kernel, we have for $f^{8,1}$
\begin{equation*}
    \begin{aligned}
    f^{8,1}(\hx)&=\int_{\mathbb{S}^2}\PD{}{x_j} \PD{}{x_i}G(\Delta\bm{X})\nabla_{\mathbb{S}^2}X_j(\hy)\Delta\nabla_{\mathbb{S}^2}X_i\cdot \Delta\big(T_S(\nabla_{\mathbb{S}^2}\bm{X}(\hx_n))\nabla_{\mathbb{S}^2}(\rho_n\bm{X})\big)d\hy\\
    &\quad-\int_{\mathbb{S}^2}\PD{}{x_i}G(\Delta\bm{X})\nabla_{\mathbb{S}^2}^2X_i(\hy)\cdot \Delta\big(T_S(\nabla_{\mathbb{S}^2}\bm{X}(\hx_n))\nabla_{\mathbb{S}^2}(\rho_n\bm{X})\big)d\hy,
    \end{aligned}
\end{equation*}
and therefore, proceeding as in Lemma \ref{lem_commutator}, we obtain
\begin{equation*}
    \begin{aligned}
    \|f^{8,1}\|_{C^\alpha(\mathbb{S}^2)}&\leq C+C\|\nabla_{\mathbb{S}^2}^2\bm{X}\|_{C^0(B_{\hx_n,2R}\cap\mathbb{S}^2)}\\
    &\leq C(R)+C\varepsilon(R)\|\nabla_{\mathbb{S}^2}^2\bm{X}\|_{C^\alpha(B_{2R}(\hx)\cap\mathbb{S}^2)}).
    \end{aligned}
\end{equation*}
The estimate for $f^{8,2}$ follows in the same manner.
Next, we estimate the commutator terms, $f^7$. Using \eqref{move_der}, we write
\begin{equation*}
    \begin{aligned}
    f^7(\hx)&=f^{7,1}(\hx)+f^{7,2}(\hx)+f^{7,3}(\hx),
    \end{aligned}
\end{equation*}
with
\begin{equation*}
    \begin{aligned}
    f^{7,1}(\hx)&=-[\mc{N}(\bm{X})\nabla_{\mathbb{S}^2},\rho_n]\bm{T}(\nabla_{\mathbb{S}^2}\bm{X})(\hx),
    \end{aligned}
\end{equation*}    
\begin{equation*}
    \begin{aligned}
    &f^{7,2}(\hx)=\int_{\mathbb{S}^2}\nabla_{\mathbb{S}^2,\hy}\Big( \PD{}{x_i}G(\Delta\bm{X})\big(\nabla_{\mathbb{S}^2}X_i(\hx)-\nabla_{\mathbb{S}^2}X_i(\hy)\big)\Big)\cdot \Delta \big(\rho_n\bm{T}(\nabla_{\mathbb{S}^2}\bm{X})(\hy)\big)d\hy\\
    &-\rho_n(\hx)\int_{\mathbb{S}^2}\nabla_{\mathbb{S}^2,\hy}\Big( \PD{}{x_i}G(\Delta\bm{X})\big(\nabla_{\mathbb{S}^2}X_i(\hx)-\nabla_{\mathbb{S}^2}X_i(\hy)\big)\Big)\cdot\Delta \big(\bm{T}(\nabla_{\mathbb{S}^2}\bm{X})(\hy)\big)d\hy\\
    &+\!\!\int_{\mathbb{S}^2}\!\!\!\!\nabla_{\mathbb{S}^2,\hy}\Big( \PD{}{x_i}G(\Delta\bm{X})\big(\nabla_{\mathbb{S}^2}X_i(\hx)\!-\!\nabla_{\mathbb{S}^2}X_i(\hy)\big)\!\Big)\!\cdot\! \Delta \big(T_S(\nabla_{\mathbb{S}^2}\bm{X}(\hx_n))\bm{X}\nabla_{\mathbb{S}^2}\rho_n(\hy)\big)d\hy,
    \end{aligned}
\end{equation*}
and
\begin{equation*}
    \begin{aligned}
    f^{7,3}(\hx)&=-\mc{N}(\bm{X})\big(T_S(\nabla_{\mathbb{S}^2}\bm{X}(\hx_n))\nabla_{\mathbb{S}^2}(\bm{X}\nabla_{\mathbb{S}^2}\rho_n)\big)(\hx)\\
    &\quad+\nabla_{\mathbb{S}^2}\rho_n\mc{N}(\bm{X})(\bm{T}(\nabla_{\mathbb{S}^2}\bm{X}))(\hx).
    \end{aligned}
\end{equation*}    
The term $f^{7,3}$ is lower order and it only requires $C^{1,\alpha}(\mathbb{S}^2)$ regularity for $\bm{X}$, while the estimate for $f^{7,2}$ follows taking the derivative of the kernel, as done for $f^8$. We get that 
\begin{equation*}
    \begin{aligned}
    \|f^{7,2}\|_{C^\alpha(\mathbb{S}^2)}&\leq C+C\|\nabla_{\mathbb{S}^2}^2\bm{X}\|_{C^0(B_{\hx_n,2R}\cap\mathbb{S}^2)}+C\|\nabla_{\mathbb{S}^2}^2\bm{X}\|_{C^0(\mathbb{S}^2)}.
    \end{aligned}
\end{equation*}
By interpolation,
\begin{equation*}
    \begin{aligned}
    \|f^{7,2}\|_{C^\alpha(\mathbb{S}^2)}&+\|f^{7,3}\|_{C^\alpha(\mathbb{S}^2)}\leq C(\tilde{\varepsilon})+C\tilde{\varepsilon}\minspace \|\nabla_{\mathbb{S}^2}^2\bm{X}\|_{C^\alpha(\mathbb{S}^2)},
    \end{aligned}
\end{equation*}
with $\tilde{\varepsilon}>0$ to be chosen.
The term $f^{7,1}$ is written as follows:
\begin{equation*}
    \begin{aligned}
    f^{7,1}(\hx)&= -\int_{\mathbb{S}^2} \nabla_{\mathbb{S}^2,\hy} G(\bm{X}(\hx)\!-\!\bm{X}(\bm{\hy}))\\
    &\qquad\cdot\big(\rho_n(\hx)\nabla_{\mathbb{S}^2} \bm{T}(\nabla_{\mathbb{S}^2}\bm{ X}(\hy))-\nabla_{\mathbb{S}^2}\big(\rho_n(\hy)\bm{T} (\nabla_{\mathbb{S}^2}\bm{X}(\hy))\big)\big)d\hy\\
    &=[\rho_n,\mc{N}(\bm{X})]\nabla_{\mathbb{S}^2}\bm{T}(\nabla_{\mathbb{S}^2}\bm{X})(\hx)\\
     &\quad+\int_{\mathbb{S}^2} \nabla_{\mathbb{S}^2,\hy} G(\bm{X}(\hx)\!-\!\bm{X}(\bm{\hy}))\cdot \nabla_{\mathbb{S}^2}\rho_n(\hy)\bm{T}(\nabla_{\mathbb{S}^2}\bm{X}(\hy))d\hy,
    \end{aligned}
\end{equation*}
hence, by Lemmas \ref{lem_commutator} and \ref{Nbound_lem}, we have that
\begin{equation*}
    \begin{aligned}
    \|f^{7,1}&\|_{C^\alpha(\mathbb{S}^2)}\leq C\|\nabla_{\mathbb{S}^2}\rho_n\|_{C^{1,\alpha}(\mathbb{S}^2)}+C\|\nabla_{\mathbb{S}^2}\rho_n\|_{C^0(\mathbb{S}^2)}\|\nabla_{\mathbb{S}^2}^2\bm{X}\|_{C^0(\mathbb{S}^2)},
    \end{aligned}
\end{equation*}
and, by interpolation,
\begin{equation*}
    \begin{aligned}
    \|f^{7,1}\|_{C^\alpha(\mathbb{S}^2)}&\leq C(\tilde{\varepsilon})+C\tilde{\varepsilon} \|\nabla_{\mathbb{S}^2}^2\bm{X}\|_{C^\alpha(\mathbb{S}^2)}.
    \end{aligned}
\end{equation*}
Then, we have that for any $t_1>0$ and $\alpha\in(0,\frac12)$,
\begin{equation*}
    \begin{aligned}
    \|\nabla_{\mathbb{S}^2}(\rho_n\bm{X})(t)\|_{L^2(t_1,T;C^{1,\alpha}(\mathbb{S}^2))}&\leq C\|\rho_n(t_1)\bm{X}(t_1)\|_{C^{\frac32+\alpha-\epsilon}}\\
    &\quad+\sum_{m=1}^7\!\|f^m(\tau)\|_{L^2(t_1,T;C^\alpha(\mathbb{S}^2))}.
    \end{aligned}
\end{equation*}
Hence, introducing the estimates above for $f^j$ and summing in $n$, we take the partition so that $\varepsilon(R)$ is small enough and then choose $\tilde{\varepsilon}$ small enough (depending on the partition $\rho_n$), to obtain that $\bm{X}\in L^2(t_1,T;C^{2,\alpha}(\mathbb{S}^2))$. 
Finally, we can take $t_2>t_1$ so that $\bm{X}(t_2)\in C^{2,\alpha}(\mathbb{S}^2)$ and use \eqref{duhamel} to conclude that $\bm{X}\in L^\infty(t_2,T; C^{2,\alpha}(\mathbb{S}^2))$ for any $t_2>0$, $\alpha\in(0,\frac12)$. Now, starting with the upgraded regularity and repeating the same steps with no changes, we conclude that $\bm{X}\in L^\infty(t_2,T; C^{2,\alpha}(\mathbb{S}^2))$ for any $t_2>0$, $\alpha\in(0,1)$.

 It is not difficult to show by induction that an analogous formula to \eqref{move_der} holds for higher derivatives. Then, by repeating the steps above one can continue the bootstrapping argument, concluding that for any $n\in\mathbb{N}$, $\bm{X}\in L^\infty(0,T; C^{n,\alpha}(\mathbb{S}^2))$. 
\vspace{0.2cm}

\noindent\underline{$\bm{X}^\delta\to\bm{X}$ in $L^\infty(0,T;C^{1,\gamma}(\mathbb{S}^2))$:}
We write the difference $\Delta^\delta\bm{X}:=\bm{X}^\delta-\bm{X}$ as follows:
\begin{equation*}
    \begin{aligned}
    \rho_n\Delta^\delta\bm{X}_n(t)&=e^{t \mc{L}_{A_0}}\big(\rho_n\Delta^\delta\bm{X}_{0,n}\big)\\
    &\hspace{-1cm}+\sum_{j=1}^7\int_{0}^t e^{(t-\tau)\mc{L}_{A_0}}\Big((\mc{J}_\delta-1)\bm{f}^j(\bm{X}^\delta)(\tau)+\bm{f}^j(\bm{X}^\delta)(\tau)-\bm{f}^j(\bm{X})(\tau)\Big)d\tau,
    \end{aligned}
\end{equation*}
with $\bm{f}^j(\bm{X})$ given in \eqref{reg_splitting}.
Thus,
\begin{equation}\label{convergence_split}
    \begin{aligned}
    \|\rho_n\Delta^\delta\bm{X}_n&\|_{L^\infty(0,T;C^{1,\gamma}(\mathbb{R}^2))}\leq C\|\rho_n\Delta^\delta\bm{X}_{0,n}\|_{C^{1,\gamma}(\mathbb{R}^2)}\\
    &+C\sum_{j=1}^7\|(\mc{J}_\delta-1)\bm{f}^j(\bm{X}^\delta)\|_{L^\infty(0,T;C^\gamma(\mathbb{R}^2))}\\
    &+C\sup_{t\in[0,T]}\sum_{j=1}^7\|\int_{0}^t e^{(t-\tau)\mc{L}_{A_0}}(\bm{f}^j(\bm{X}^\delta)-\bm{f}^j(\bm{X}))(\tau)\|_{C^\gamma(\mathbb{R}^2)}.
    \end{aligned}
\end{equation}
Since $\bm{X}_0\in h^{1,\gamma}(\mathbb{S}^2)$ and  $\bm{f}^j(\bm{X}^\delta)\in h^{\gamma}(\mathbb{S}^2)$, the first two terms converge to zero as $\delta\to 0$. For the third term, we need to show that it can be absorbed by the left-hand side. As in the previous arguments, we will show that for the quasilinear terms we can find a small coefficient, while for the lower order ones we will take advantage of the extra regularity via \eqref{lin_heat_semi} to get a small coefficient for $T$ small enough.

From previous estimates we immediately get that for $j=4,5,7$ and $0<\epsilon<1-\gamma$,
\begin{equation*}
    \begin{aligned}
    \|\bm{f}^j(\bm{X}^\delta)-\bm{f}^j(\bm{X})\|_{C^{\gamma+\epsilon}(\mathbb{S}^2)}\leq C\|\Delta^\delta\bm{X}\|_{C^1(\mathbb{S}^2)},
    \end{aligned}
\end{equation*}
hence \eqref{lin_heat_semi} gives that
\begin{equation*}
    \begin{aligned}
    \sup_{t\in[0,T]}\sum_{j=4,5,7}\|\int_{0}^t e^{(t-\tau)\mc{L}_{A_0}}(\bm{f}^j(\bm{X}^\delta)-\bm{f}^j(\bm{X}))(\tau)\|_{C^\gamma(\mathbb{R}^2)}&\leq C\minspace T^\epsilon \|\Delta^\delta\bm{X}\|_{C^1(\mathbb{S}^2)}.
    \end{aligned}
\end{equation*}
Also, for $\bm{f}^6$,
\begin{equation*}
    \begin{aligned}
    \|\bm{f}^6(\bm{X}^\delta)-\bm{f}^6(\bm{X})\|_{C^{\gamma}(\mathbb{S}^2)}&\leq C\|A-A_0\|\|\nabla_{\mathbb{S}^2}(\rho_n\Delta^\delta\bm{X})\|_{C^{\gamma}(\mathbb{S}^2)}\\
    &\quad+C\|\Delta^\delta \bm{X}\|_{C^1(\mathbb{S}^2)}\|\rho_n\bm{X}^\delta\|_{C^{1,\gamma}(\mathbb{S}^2)},
    \end{aligned}
\end{equation*}
so that \eqref{lin_heat_semi}-\eqref{lin_heat_semi2} give
\begin{equation*}
    \begin{aligned}
    \|\int_{0}^t e^{(t-\tau)\mc{L}_{A_0}}(\bm{f}^6(\bm{X}^\delta)-\bm{f}^6(\bm{X}))(\tau)\|_{L^\infty(0,T;C^\gamma(\mathbb{R}^2))}&\leq C\varepsilon(T) \|\rho_n\Delta^\delta\bm{X}\|_{C^{1,\gamma}(\mathbb{S}^2)}\\
    &\quad+C T^\varepsilon \|\Delta^\delta \bm{X}\|_{C^1(\mathbb{S}^2)}.
    \end{aligned}
\end{equation*}
Next, we proceed with the term $\bm{f}^1$:
\begin{equation*}
    \begin{aligned}
    (\bm{f}^1(\bm{X}^\delta)-\bm{f}^1(\bm{X}))(\thetab)=I_1+I_2,
    \end{aligned}
\end{equation*}
with
\begin{equation*}
\begin{aligned}
    I_1(\thetab)&=(\mc{N}(\bm{X}^\delta)-\mc{N}(\bm{X}))\big(\rho_n\big(\bm{T}(\nabla_{\mathbb{S}^2}
    \bm{X}^\delta)-T_S(\nabla_{\mathbb{S}^2}\bm{X}^\delta(\hx_n))\nabla_{\mathbb{S}^2}\bm{X}^\delta\big)\big)_n(\thetab)\\
    I_2(\thetab)&=\mc{N}(\bm{X})\Big(\rho_n\Big(\bm{T}(\nabla_{\mathbb{S}^2}
    \bm{X}^\delta)-T_S(\nabla_{\mathbb{S}^2}\bm{X}^\delta(\hx_n))\nabla_{\mathbb{S}^2}\bm{X}^\delta\\
    &\hspace{2.5cm}-\bm{T}(\nabla_{\mathbb{S}^2}
    \bm{X})+T_S(\nabla_{\mathbb{S}^2}\bm{X}(\hx_n))\nabla_{\mathbb{S}^2}\bm{X} \Big)\Big)_n(\thetab).
\end{aligned}
\end{equation*}
The first term is estimated easily from Proposition \ref{Nbound_lem} by noticing that one can always extract $\Delta^\delta\bm{X}=\bm{X}^\delta-\bm{X}$ from the difference of the kernels (similarly as done in the proof of Proposition \ref{Festimate}),
\begin{equation}\label{I1_bound}
    \begin{aligned}
    \|I_1\|_{C^\gamma(\mathbb{R}^2)}&\!\leq\! C\|\nabla_{\mathbb{S}^2}\Delta^\delta\bm{X}\|_{C^0(\mathbb{S}^2)}\|\rho_n(\bm{T}(\nabla_{\mathbb{S}^2}\bm{X}^\delta)\!-\!T_S(\nabla_{\mathbb{S}^2}\bm{X}^\delta(\hx_n))\nabla_{\mathbb{S}^2}\bm{X}^\delta)\|_{C^\gamma(\mathbb{S}^2)}\\
    &\leq C\varepsilon(R)\|\nabla_{\mathbb{S}^2}\Delta^\delta\bm{X}\|_{C^0(\mathbb{S}^2)}\|\nabla_{\mathbb{S}^2}\bm{X}\|_{C^\gamma(B_{\hx_n,2R}\cap \mathbb{S}^2)},
    \end{aligned}
\end{equation}
where we have used \eqref{f1bound} in the second step. Proposition \ref{Nbound_lem} gives that
\begin{equation}\label{I2_reg}
\begin{aligned}
    \|I_2\|_{C^\gamma(\mathbb{R}^2)}&=C\|\rho_n\Big(\bm{T}(\nabla_{\mathbb{S}^2}
    \bm{X}^\delta)-T_S(\nabla_{\mathbb{S}^2}\bm{X}^\delta(\hx_n))\nabla_{\mathbb{S}^2}\bm{X}^\delta\\
    &\hspace{2.5cm}-\bm{T}(\nabla_{\mathbb{S}^2}
    \bm{X})+T_S(\nabla_{\mathbb{S}^2}\bm{X}(\hx_n))\nabla_{\mathbb{S}^2}\bm{X} \Big)\|_{C^\gamma(\mathbb{R}^2)}.
\end{aligned}
\end{equation}
Denote 
\begin{equation*}
    \begin{aligned}
    \bm{J}(\bm{Y})(\hx)&=\bm{T}(\nabla_{\mathbb{S}^2}
    \bm{Y}(\hx))-T_S(\nabla_{\mathbb{S}^2}\bm{Y}(\hx_n))\nabla_{\mathbb{S}^2}\bm{Y}(\hx),
    \end{aligned}
\end{equation*}
so that we can write
\begin{equation*}
    \begin{aligned}
    \bm{J}(\bm{X}^\delta&(\hx_1))-\bm{J}(\bm{X}^\delta(\hx_2))=(\nabla_{\mathbb{S}^2}\bm{X}^\delta(\hx_1)-\nabla_{\mathbb{S}^2}\bm{X}^\delta(\hx_2))\\
    &\times\int_0^1 \big(T_S(s_1\nabla_{\mathbb{S}^2}\bm{X}^\delta(\hx_1)+(1-s_1)\nabla_{\mathbb{S}^2}\bm{X}^\delta(\hx_2))-T_S(\nabla_{\mathbb{S}^2}\bm{X}^\delta(\hx_n))\big)ds_1.
    \end{aligned}
\end{equation*}
Then,
\begin{equation*}
    \begin{aligned}
    \bm{J}(\bm{X}^\delta(\hx_1))-\bm{J}(\bm{X}^\delta(\hx_2))-\bm{J}(\bm{X}(\hx_1))+\bm{J}(\bm{X}(\hx_2))=\bm{J}^1+\bm{J}^2,
    \end{aligned}
    \end{equation*}
    with
\begin{equation*}
   \begin{aligned}
        \bm{J}^1&=(\nabla_{\mathbb{S}^2}\Delta^\delta\bm{X}(\hx_1)-\nabla_{\mathbb{S}^2}\Delta^\delta\bm{X}(\hx_2))\\
        &\quad\times\int_0^1 \big(T_S(s_1\nabla_{\mathbb{S}^2}\bm{X}^\delta(\hx_1)+(1-s_1)\nabla_{\mathbb{S}^2}\bm{X}^\delta(\hx_2))-T_S(\nabla_{\mathbb{S}^2}\bm{X}^\delta(\hx_n))\big)ds_1,
   \end{aligned}
\end{equation*}
and
\begin{equation*}
    \begin{aligned}
    \bm{J}^2&=(\nabla_{\mathbb{S}^2}\bm{X}(\hx_1)-\nabla_{\mathbb{S}^2}\bm{X}(\hx_2))\\
    &\quad\times\Big(\int_0^1 \big(T_S(s_1\nabla_{\mathbb{S}^2}\bm{X}^\delta(\hx_1)+(1-s_1)\nabla_{\mathbb{S}^2}\bm{X}^\delta(\hx_2))-T_S(\nabla_{\mathbb{S}^2}\bm{X}^\delta(\hx_n))\big)ds_1\\
    &\quad\quad-\int_0^1 \big(T_S(s_1\nabla_{\mathbb{S}^2}\bm{X}(\hx_1)+(1-s_1)\nabla_{\mathbb{S}^2}\bm{X}(\hx_2))-T_S(\nabla_{\mathbb{S}^2}\bm{X}(\hx_n))\big)ds_1\Big).
    \end{aligned}
\end{equation*}
It follows that
\begin{equation*}
    \begin{aligned}
    |\bm{J}^1|&\leq C\max\{|\nabla_{\mathbb{S}^2}\bm{X}^\delta(\hx_1)\!-\!\nabla_{\mathbb{S}^2}\bm{X}^\delta(\hx_n)|,|\nabla_{\mathbb{S}^2}\bm{X}^\delta(\hx_2)\!-\!\nabla_{\mathbb{S}^2}\bm{X}^\delta(\hx_n)\}\\
    &\hspace{1cm}\times|\nabla_{\mathbb{S}^2}\Delta^\delta\bm{X}(\hx_1)\!-\!\nabla_{\mathbb{S}^2}\Delta^\delta\bm{X}(\hx_2)|.
    \end{aligned}
\end{equation*}
We apply the mean-value theorem again in $\bm{J}^2$:
\begin{equation*}
    \begin{aligned}
    &\int_0^1 \big(T_S(s_1\nabla_{\mathbb{S}^2}\bm{X}^\delta(\hx_1)+(1-s_1)\nabla_{\mathbb{S}^2}\bm{X}^\delta(\hx_2))-T_S(\nabla_{\mathbb{S}^2}\bm{X}^\delta(\hx_n))\big)ds_1\\
    &=\!\int_0^1\!\int_0^1\! DT_S\big(s_2(s_1\nabla_{\mathbb{S}^2}\bm{X}^\delta(\hx_1)\!+\!(1\!-\!s_1)\nabla_{\mathbb{S}^2}\bm{X}^\delta(\hx_2))\!+\!(1\!-\!s_2)\nabla_{\mathbb{S}^2}\bm{X}^\delta (\hx_n)\big)ds_1ds_2\\
    &\qquad\times \big(s_1\nabla_{\mathbb{S}^2}\bm{X}^\delta(\hx_1)+(1-s_1)\nabla_{\mathbb{S}^2}\bm{X}^\delta(\hx_2)-\nabla_{\mathbb{S}^2}\bm{X}^\delta (\hx_n)\big),
    \end{aligned}
\end{equation*}
hence adding and subtracting we obtain
\begin{equation*}
    \begin{aligned}
    |\bm{J}^2|&\leq|\bm{J}^{2,1}|+|\bm{J}^{2,2}|,
    \end{aligned}
    \end{equation*}
    with
    \begin{equation*}
    \begin{aligned}
    |\bm{J}^{2,1}|&\leq|\nabla_{\mathbb{S}^2}\bm{X}(\hx_1)-\nabla_{\mathbb{S}^2}\bm{X}(\hx_2)|\\
    &\quad\times\max\{|\nabla_{\mathbb{S}^2}\bm{X}^\delta(\hx_1)\!-\!\nabla_{\mathbb{S}^2}\bm{X}^\delta(\hx_n)|,|\nabla_{\mathbb{S}^2}\bm{X}^\delta(\hx_2)\!-\!\nabla_{\mathbb{S}^2}\bm{X}^\delta(\hx_n)|\}\\
    &\quad\times\Big|DT_S(s_2(s_1\nabla_{\mathbb{S}^2}\bm{X}^\delta(\hx_1)+(1-s_1)\nabla_{\mathbb{S}^2}\bm{X}^\delta(\hx_2))+(1-s_2)\nabla_{\mathbb{S}^2}\bm{X}^\delta(\hx_n))\\
    &\hspace{1cm}-DT_S(s_2(s_1\nabla_{\mathbb{S}^2}\bm{X}(\hx_1)+(1-s_1)\nabla_{\mathbb{S}^2}\bm{X}(\hx_2))+(1-s_2)\nabla_{\mathbb{S}^2}\bm{X}(\hx_n))\Big|,
    \end{aligned}
    \end{equation*}
    \begin{equation*}
    \begin{aligned}
    |\bm{J}^{2,2}|&\leq|\nabla_{\mathbb{S}^2}\bm{X}(\hx_1)-\nabla_{\mathbb{S}^2}\bm{X}(\hx_2)|\\
    &\quad\times|DT_S(s_2(s_1\nabla_{\mathbb{S}^2}\bm{X}(\hx_1)+(1-s_1)\nabla_{\mathbb{S}^2}\bm{X}(\hx_2))+(1-s_2)\nabla_{\mathbb{S}^2}\bm{X}(\hx_n))|\\
    &\quad\times \max\{|\nabla_{\mathbb{S}^2}\Delta^\delta\bm{X}(\hx_1)\!-\!\nabla_{\mathbb{S}^2}\Delta^\delta\bm{X}(\hx_n)|,|\nabla_{\mathbb{S}^2}\Delta^\delta\bm{X}(\hx_2)\!-\!\nabla_{\mathbb{S}^2}\Delta^\delta\bm{X}(\hx_n)\}.
    \end{aligned}
\end{equation*}
Going back to \eqref{I2_reg}, we thus conclude that
\begin{equation*}
    \begin{aligned}
    \|I_2\|_{C^\gamma(\mathbb{R}^2)}&\leq C\varepsilon(R)
    \|\nabla_{\mathbb{S}^2}\Delta^\delta\bm{X}\|_{C^\gamma(B_{\hx_n,2R}\cap \mathbb{S}^2)}.
    \end{aligned}
\end{equation*}
Together with the bound for $I_1$ \eqref{I1_bound}, we obtain the following estimate for $\bm{f}^1$:
\begin{equation*}
    \begin{aligned}
    \|\int_{0}^t &e^{(t-\tau)\mc{L}_{A_0}}(\bm{f}^1(\bm{X}^\delta)-\bm{f}^1(\bm{X}))(\tau)\|_{L^\infty(0,T;C^\gamma(\mathbb{R}^2))}\leq C\minspace T^\epsilon \|\Delta^\delta\bm{X}\|_{C^1(\mathbb{S}^2)}\\
    &\quad+C\varepsilon(R)\|\Delta^\delta\bm{X}\|_{C^{1}(\mathbb{S}^2)}\|\bm{X}\|_{C^{1,\gamma}(B_{\hx_n,2R}\cap\mathbb{S}^2)}+C\varepsilon(R)\|\Delta^\delta\bm{X}\|_{C^{1,\gamma}(B_{\hx_n,2R}\cap\mathbb{S}^2)}.
    \end{aligned}
\end{equation*}
The estimate for $\bm{f}^2$ follows in the same way than those for $\bm{f}^1$ and $\bm{f}^6$, from which we conclude that
\begin{equation*}
    \begin{aligned}
    \|\rho_n\Delta^\delta\bm{X}_n&\|_{L^\infty(0,T;C^{1,\gamma}(\mathbb{R}^2))}\leq C\|\rho_n\Delta^\delta\bm{X}_{0,n}\|_{C^{1,\gamma}(\mathbb{R}^2)}\\
    &+C\sum_{j=1}^7\|(\mc{J}_\delta-1)\bm{f}^j(\bm{X}^\delta)\|_{L^\infty(0,T;C^\gamma(\mathbb{R}^2))}\\
    &+CT^\varepsilon\|\Delta^\delta\bm{X}\|_{C^1(\mathbb{S}^2)}+C\varepsilon(R)\|\Delta^\delta\bm{X}\|_{C^{1}(\mathbb{S}^2)}\|\bm{X}\|_{C^{1,\gamma}(B_{\hx_n,2R}\cap\mathbb{S}^2)}
    \\
    &+C(\varepsilon(T)+\varepsilon(R))\|\Delta^\delta \bm{X}\|_{C^{1,\gamma}(B_{\hx,2R}\cap\mathbb{S}^2)}.
    \end{aligned}
\end{equation*}
Taking $R$ and $T$ small enough, the last term is absorbed by the left-hand side. Then, adding in $n$, we conclude that, for $T$ and $R$ small enough, the desired estimate holds
\begin{equation*}
    \begin{aligned}
    \|\Delta^\delta\bm{X}&\|_{L^\infty(0,T;C^{1,\gamma}(\mathbb{S}^2))}\!\leq\! C\|\Delta^\delta\bm{X}_{0}\|_{C^{1,\gamma}(\mathbb{S}^2)}\!+\!C\!\sum_{j=1}^7\!\|(\mc{J}_\delta\!-\!1)\bm{f}^j(\bm{X}^\delta)\|_{L^\infty(0,T;C^\gamma(\mathbb{R}^2))}.
    \end{aligned}
\end{equation*}

\end{proof}

\appendix 

\section{Besov Spaces and Fourier Multiplier Theorems}\label{sec:appA}
In this section, we will proof Theorem \ref{fouriermuliplierholdernorm}.
First, define Besov Spaces $B^\gamma_{p,q}$ by a dyadic decomposition.
Set a function $\psi\paren{\xib}\in C^\infty\paren{\mathbb{R}^n}$ s.t.
\begin{align*}
    \psi\paren{\xib}=\left\{
    \begin{array}{cc}
         1& \abs{\xib}\leq 1\\
         0& \abs{\xib}\geq 2
    \end{array}\right.
\end{align*}
and define $\phi\paren{\xib}:=\psi\paren{\xib}-\psi\paren{2\xib}$.
Hence, $\phi\paren{\xib}\in C^\infty\paren{\mathbb{R}^n}$ and
\begin{align*}
    \phi\paren{\xib}=0,& \qquad \abs{\xib}\leq \frac{1}{2}, \abs{\xib}\geq 2,\\
    \sum_{j=-\infty}^\infty \phi\paren{2^{-j}\xib}=1,& \qquad  \abs{\xib}\neq 0.
\end{align*}
Next, the homogeneous dyadic blocks $\Dot{\Delta}_j$ are defined by 
\begin{align}
    \Dot{\Delta}_j f\paren{\thetab}:= \mc{F}^{-1}\paren{\phi\paren{2^{-j}\xib}\mc{F}f\paren{\xib}}\paren{\thetab}=\mc{K}_j \ast f
\end{align}
where $\mc{K}\paren{\thetab}:=\mc{F}^{-1}\paren{\phi\paren{\xib}}\paren{\thetab}$ and $\mc{K}_j\paren{\thetab}:=\mc{F}^{-1}\paren{\phi\paren{2^{-j}\xib}}\paren{\thetab}=2^{jn}\mc{K}\paren{2^j\thetab}$.
Now, given $\gamma$ a real number and $p,q\geq 1$, we may define homogeneous Besov spaces $\Dot{B}_{p,q}^\gamma\paren{\mathbb{R}^n}$ with its seminorm $\norm{\cdot}_{\Dot{B}_{p,q}^\gamma\paren{\mathbb{R}^n}}$ by
\begin{align}
    \norm{f}_{\Dot{B}_{p,q}^\gamma\paren{\mathbb{R}^n}}:=&\paren{\sum_{j=-\infty}^\infty\paren{2^{j\gamma}\norm{\Dot{\Delta}_j f}_{L^p\paren{\mathbb{R}^n}}}^q}^{\frac{1}{q}},\\
    \norm{f}_{\Dot{B}_{p,\infty}^\gamma\paren{\mathbb{R}^n}}:=&\sup_{j\in\mathbb{Z}}\paren{2^{j\gamma}\norm{\Dot{\Delta}_j f}_{L^p\paren{\mathbb{R}^n}}}.
\end{align}
According to \cite[Remark 2.2.2]{Loukas:modern-fourier-analysis} and \cite[Lemma 8.4.2]{Yoshihiro:theory-of-the-lebesque-integral}, we know for all $0<\gamma<1$, $\norm{\cdot}_{\Dot{B}_{p,q}^\gamma\paren{\mathbb{R}^n}}$ and $\jump{\cdot}_{C^\gamma\paren{\mathbb{R}^n}}$ are equivalent, so we only need to prove the Fourier multiplier theorem on $\Dot{B}_{p,q}^\gamma\paren{\mathbb{R}^n}$. The proof is from \cite[Theorem 8.4.3]{Yoshihiro:theory-of-the-lebesque-integral}.

Given $T$ a Fourier multiplier operator with multiplier $m\paren{\bm{\xi}}\in C^s \paren{\mathbb{R}^n\setminus \{\bm{0}\}}\cap L^\infty\paren{\mathbb{R}^n}$, for $s>\frac{n}{2}$ and for all $\abs{\alphab}\leq s$, such that
\begin{align}
    \norm{\partial_{\bm{\xi}}^\alphab m\paren{\bm{\xi}}}\leq C_\alphab \abs{\bm{\xi}}^{-\abs{\alphab}},\label{multi_coef_prop}
\end{align}
we first define a related kernel with $\lambda>0$ by $K_{\lambda}\paren{\thetab}:=\mc{F}^{-1}\paren{\phi\paren{\xib}m\paren{\lambda \xib}}\paren{\thetab}$.
\begin{lemma}
Given $m\paren{\xib}$ satisfying \eqref{multi_coef_prop}, then $K_{\lambda}\paren{\thetab}$ is bounded by
\begin{align}
    \int_{\mathbb{R}^n}\abs{K_\lambda\paren{\thetab}} d\thetab\leq C_{s,n,\phi}D_{m},
\end{align}
where $D_{m}=\max_{\abs{\bm{\alpha}}\leq s}C_{\bm{\alpha}}$
\end{lemma}
\begin{proof}
Since there exist $C_s$ s.t. for all $\thetab\in \mathbb{R}^n$
\begin{align}
    \paren{1+\abs{\thetab}^2}^s \leq C_s \sum_{\abs{\bm{\alpha}}\leq s} \abs{\thetab^\alphab}^2,
\end{align}
we obtain
\begin{align}
\begin{split}
        &\int_{\mathbb{R}^n}\abs{K_\lambda\paren{\thetab}}^2\paren{1+\abs{\thetab}^2}^s d\thetab\\
    \leq& C_s \sum_{\abs{\bm{\alpha}}\leq s}\int_{\mathbb{R}^n}\abs{\thetab^\alphab K_\lambda\paren{\thetab}}^2 d\thetab\\
    =   & C_{n,s}\sum_{\abs{\bm{\alpha}}\leq s} \int_{\mathbb{R}^n} \abs{\partial_{\bm{\xi}}^\alphab\paren{\phi\paren{\xib} m\paren{\lambda\bm{\xi}}}}^2 d\xib\\
    =   & C_{n,s}\sum_{\abs{\bm{\alpha}}\leq s} \int_{\mathbb{R}^n} \abs{\sum_{\betab\leq\alphab}\combin{\alphab}{\betab}\paren{\partial_{\bm{\xi}}^\betab\phi}\paren{\xib}\lambda^{\abs{\alphab-\betab}}\paren{\partial_{\bm{\xi}}^{\alphab-\betab} m}\paren{\lambda\bm{\xi}}}^2 d\xib\\
    % =   & C_{n,s}\sum_{\abs{\bm{\alpha}}\leq s}\lambda^{-n} \int_{\mathbb{R}^n} \abs{\sum_{\betab\leq\alphab}\combin{\alphab}{\betab}\lambda^{\abs{\alphab-\betab}}\paren{\partial_{\bm{\xi}}^\betab\phi}\paren{\frac{\xib}{\lambda}}\paren{\partial_{\bm{\xi}}^{\alphab-\betab} m}\paren{\bm{\xi}}}^2 d\xib
    \leq& C_{n,s}D_{m}^2\sum_{\abs{\bm{\alpha}}\leq s} \int_{\mathbb{R}^n} \abs{\sum_{\betab\leq\alphab}\combin{\alphab}{\betab}\paren{\partial_{\bm{\xi}}^\betab\phi}\paren{\xib}\lambda^{\abs{\alphab-\betab}}\abs{\lambda\bm{\xi}}^{-\abs{\alphab-\betab}}}^2 d\xib.
\end{split}
\end{align}
$supp\paren{\phi}\subset\left\{\xib | \frac{1}{2}\leq \abs{\xib}\leq 2\right\}$, so
\begin{align}
\begin{split}
    &\int_{\mathbb{R}^n} \abs{\sum_{\betab\leq\alphab}\combin{\alphab}{\betab}\paren{\partial_{\bm{\xi}}^\betab\phi}\paren{\xib}\lambda^{\abs{\alphab-\betab}}\abs{\lambda\bm{\xi}}^{-\abs{\alphab-\betab}}}^2 d\xib\\
    =&\int_{\frac{1}{2}\leq \abs{\xib}\leq 2} \abs{\sum_{\betab\leq\alphab}\combin{\alphab}{\betab}\paren{\partial_{\bm{\xi}}^\betab\phi}\paren{\xib}\abs{\bm{\xi}}^{-\abs{\alphab-\betab}}}^2 d\xib\leq C_{\phi,\alphab}.
\end{split}
\end{align}
Thus,
\begin{align}
    \int_{\mathbb{R}^n}\abs{K_\lambda\paren{\thetab}}^2\paren{1+\abs{\thetab}^2}^s d\thetab\leq C_{n,s}D_{m}^2\sum_{\abs{\bm{\alpha}}\leq s}C_{\phi,\alphab}\leq  C_{s,n,\phi}D_{m}^2,
\end{align}
and by Holder inequality,
\begin{align}
\begin{split}
        \int_{\mathbb{R}^n}\abs{K_\lambda\paren{\thetab}} d\thetab
    \leq&\paren{\int_{\mathbb{R}^n}\abs{K_\lambda\paren{\thetab}}^2\paren{1+\abs{\thetab}^2}^s d\thetab}^\frac{1}{2}\paren{\int_{\mathbb{R}^n}\paren{1+\abs{\thetab}^2}^{-s} d\thetab}^\frac{1}{2}\\
    \leq&\sqrt{C_{s,n,\phi}}D_{m}\paren{\int_{\mathbb{R}^n}\paren{1+\abs{\thetab}^2}^{-s} d\thetab}^\frac{1}{2}\leq C_{s,n,\phi}D_{m}.
\end{split}
\end{align}
\end{proof}
Next, we may use $K_{\lambda}\paren{\thetab}$ and homogeneous Besov semi norm $\norm{\cdot}_{\Dot{B}_{p,q}^\gamma\paren{\mathbb{R}^n}}$ to prove the Fourier multiplier theorem.
\begin{proof}[Proof of Theorem \ref{fouriermuliplierholdernorm}]
First, set 
\begin{align}
    T_m^jf\paren{\thetab}:=\Dot{\Delta}_j T_m f\paren{\thetab}=\mc{F}^{-1}\paren{\phi\paren{2^{-j}\xib}m\paren{\xib}\mc{F}f\paren{\xib}}\paren{\thetab}.
\end{align}
Since 
\begin{align}
    \mc{F}^{-1}\paren{\phi\paren{2^{-j}\xib}m\paren{ \xib}}\paren{\thetab}=2^{nj} K_{2^j}\paren{2^j\thetab},
\end{align}
\begin{align}
\begin{split}
    \abs{T_m^jf\paren{\thetab}}
    =   &\abs{\mc{F}^{-1}\paren{\phi\paren{2^{-j}\xib}m\paren{ \xib}}\ast f \paren{\thetab}}\\
    \leq&\norm{\mc{F}^{-1}\paren{\phi\paren{2^{-j}\xib}m\paren{ \xib}}\paren{\thetab}}_{L^1\paren{\mbrn{n}}} \norm{f}_{L^\infty\paren{\mbrn{n}}}\\
    \leq&\norm{2^{nj} K_{2^j}\paren{2^j\thetab}}_{L^1\paren{\mbrn{n}}} \norm{f}_{L^\infty\paren{\mbrn{n}}}\\
    \leq&\norm{K_{2^j}\paren{\thetab}}_{L^1\paren{\mbrn{n}}} \norm{f}_{L^\infty\paren{\mbrn{n}}}\\
    \leq&C_{s,n,\phi}D_{m} \norm{f}_{L^\infty\paren{\mbrn{n}}}.
\end{split}
\end{align}
Next, $supp\paren{\phi\paren{2^{-j}\xib}}\subset\left\{\xib \left| 2^{j-1}\leq \abs{\xib}\leq 2^{j+1}\right.\right\}$, so for all $j+1\leq k-1$ or $j-1\geq k+1$
\begin{align}
    \phi\paren{2^{-j}\xib}\phi\paren{2^{-k}\xib}=0.
\end{align}
Therefore,
\begin{align}
    \phi\paren{2^{-j}\xib}\mc{F}f\paren{\xib}= \phi\paren{2^{-j}\xib}\paren{\mc{F}\paren{\Dot{\Delta}_{j-1}f}\paren{\xib}+\mc{F}\paren{\Dot{\Delta}_j f}\paren{\xib}+\mc{F}\paren{\Dot{\Delta}_{j+1}f}\paren{\xib}},
\end{align}
so we obtain
\begin{align}
    T_m^jf\paren{\thetab}=T_m^j\paren{\Dot{\Delta}_{j-1}f}\paren{\thetab} +T_m^j\paren{\Dot{\Delta}_{j}f}\paren{\thetab} +T_m^j\paren{\Dot{\Delta}_{j+1}f}\paren{\thetab}.
\end{align}
Finally, 
\begin{align}
\begin{split}
    &\norm{T_m f}_{\Dot{B}_{\infty,\infty}^\gamma\paren{\mathbb{R}^n}}
    =   \sup_{j\in\mathbb{Z}}2^{j\gamma}\norm{T_m^j f}_{L^\infty\paren{\mathbb{R}^n}}\\
    \leq&\sup_{j\in\mathbb{Z}}2^{j\gamma}\norm{T_m^j \paren{\Dot{\Delta}_{j-1}f}}_{L^\infty\paren{\mathbb{R}^n}}+\sup_{j\in\mathbb{Z}}2^{j\gamma}\norm{T_m^j \paren{\Dot{\Delta}_{j}f}}_{L^\infty\paren{\mathbb{R}^n}}+\sup_{j\in\mathbb{Z}}2^{j\gamma}\norm{T_m^j \paren{\Dot{\Delta}_{j+1}f}}_{L^\infty\paren{\mathbb{R}^n}}\\
    \leq&C_{s,n}D_{m}\paren{\sup_{j\in\mathbb{Z}}2^{j\gamma}\norm{\Dot{\Delta}_{j-1}f}_{L^\infty\paren{\mathbb{R}^n}}+\sup_{j\in\mathbb{Z}}2^{j\gamma}\norm{\Dot{\Delta}_{j}f}_{L^\infty\paren{\mathbb{R}^n}}+\sup_{j\in\mathbb{Z}}2^{j\gamma}\norm{\Dot{\Delta}_{j+1}f}_{L^\infty\paren{\mathbb{R}^n}}}\\
    =   &C_{s,n}D_{m}\paren{2^\gamma+1+2^{-\gamma}}\norm{ f}_{\Dot{B}_{\infty,\infty}^\gamma\paren{\mathbb{R}^n}}.
\end{split}
\end{align}
Since $\norm{\cdot}_{\Dot{B}_{p,q}^\gamma\paren{\mathbb{R}^n}}$ and $\jump{\cdot}_{C^\gamma\paren{\mathbb{R}^n}}$ are equivalent,
\begin{align}
    \jump{T_m u}_{C^{\gamma}\paren{\mathbb{R}^n}}\leq C_{\gamma, s, n} D_m\jump{u}_{C^{\gamma}\paren{\mathbb{R}^n}}.
\end{align}
\end{proof}

\section{Estimates for the semigroup  $e^{-t L_A\paren{\bm{\xi}}}$ }\label{sec:appB}

\begin{lemma}\label{LA_specform}
For all $\beta=\beta_1 \beta_2 \cdots \beta_k$, there exists a matrix $P_{\beta}\paren{\hat{\xi}_1,\hat{\xi}_2}$ of polynomials with degree  $deg\paren{P_{\beta}}\leq  3\abs{\beta}+4$ s.t.
\begin{align}
    \partial_{\beta}L_A\paren{\bm{\xi}}=\frac{1}{\abs{\xib}^{\abs{\beta}-1}}\frac{P_{\beta}\paren{\hat{\xi}_1,\hat{\xi}_2}}{\abs{U\hat{\xib}}^{2\abs{\beta}+3}}.\label{LA_Dform}
\end{align}
More specifically, $P_{\beta}\paren{\hat{\xi}_1,\hat{\xi}_2}$ can be written as
\begin{align}
    \paren{P_{\beta}}_{i_1 i_2}\paren{\hat{\xi}_1,\hat{\xi}_2}=\sum_{j_1,j_2\geq 0,j_1+j_2\leq 3\abs{\beta}+4} c^{\paren{\beta,i_1, i_2}}_{j_1,j_2}\paren{A, U, P,\frac{1}{\det(B)}, \mc{T}, \frac{d \mc{T}}{d\lambda}, \frac{1}{\lambda^2} }\hat{\xi}_1^{j_1}\hat{\xi}_2^{j_2},\label{LA_DPform}
\end{align}
where
\begin{align}
\begin{split}
    &c^{\paren{\beta,i_1, i_2}}_{j_1,j_2}\paren{A, U, P,\frac{1}{\det(B)}, \mc{T}, \frac{d \mc{T}}{d\lambda}, \frac{1}{\lambda^2} }\\
    =&c^{\paren{\beta,i_1, i_2}}_{j_1,j_2}\paren{A_{11},\cdots, A_{32}, U_{11},\cdots, U_{32}, P_{11},\cdots,P_{33}, \frac{1}{\det(B)}, \frac{\mc{T}}{\lambda}, \frac{d \mc{T}}{d\lambda}, \frac{1}{\lambda^2}}
\end{split}
\end{align}
is a polynomial function.

Moreover, $\norm{P_{\beta}\paren{\hat{\xi}_1,\hat{\xi}_2}}_{C^1\paren{\mc{DA}_{\sigma1,\sigma_2}}}$ and $\norm{U\hat{\xib}}_{C^1\paren{\mc{DA}_{\sigma1,\sigma_2}}}$ are uniformly bounded, i.e. there exists $C^{(\beta)}_{\sigma_1, \sigma_2,\mc{T}}$ and $C^{(\beta)}_{\sigma_1, \sigma_2}$ s.t. for all $\hat{\xib}\in\mathbb{S}^1$,
\begin{align}
    \norm{P_{\beta}\paren{\hat{\xi}_1,\hat{\xi}_2}}_{C^1\paren{\mc{DA}_{\sigma1,\sigma_2}}}&\leq C^{(\beta)}_{\sigma_1, \sigma_2,\mc{T}},\\
    \norm{U\hat{\xib}}_{C^1\paren{\mc{DA}_{\sigma1,\sigma_2}}}&\leq C^{(\beta)}_{\sigma_1, \sigma_2}.
\end{align}
\end{lemma}
\begin{proof}
Since
\begin{align}
     &\paren{\mc{F}_\theta G_{\alpha,A}}(\bm{\xi})=\frac{1}{\abs{\xib}}\paren{\mc{F}_\theta G_{\alpha,A}}(\hat{\bm{\xi}})\\
    =&\frac{1}{\abs{\xib}}\frac{\paren{I+\alpha P}\abs{U\hat{\xib}}^2-\alpha U\hat{\xib}\otimes U\hat{\xib}}{4\det(B)\abs{U\hat{\xib}}^3},
\end{align}
and $Z(\bm{\xi})= \abs{\xib}^2 Z(\hat{\xib})$ where $Z(\hat{\xib})$ is a matrix of polynomials with degree $2$,
\begin{align}
    L_A\paren{\bm{\xi}}=\abs{\xib} \frac{P_{0}\paren{\hat{\xi}_1,\hat{\xi}_2}}{\abs{U\hat{\xib}}^{3}},
\end{align}
where
\begin{align}
    P_{0}=\frac{\paren{I+\alpha P}\abs{U\hat{\xib}}^2-\alpha U\hat{\xib}\otimes U\hat{\xib}}{4\det(B)}\paren{\frac{\mc{T}}{\lambda}\paren{I-\frac{A\hat{\xib}\otimes A\hat{\xib}}{\lambda^2}}+\D{\mc{T}}{\lambda}\frac{A\hat{\xib}\otimes A\hat{\xib}}{\lambda^2}},
\end{align}
where the degree of $P_0$ is $4$.
Obviously, $P_0$ can be written as \eqref{LA_DPform}.

When $\abs{\beta}=1$,
\begin{align}
\begin{split}
        &\PD{L_A}{\xi_i} \paren{\bm{\xi}}
        =\frac{\xi_i}{\abs{\xib}}\frac{P_{0}\paren{\hat{\xi}_1,\hat{\xi}_2}}{\abs{U\hat{\xib}}^{3}}+\abs{\xib} \sum_{j=1,2}\PD{}{\hat{\xi}_j}\frac{P_{0}\paren{\hat{\xi}_1,\hat{\xi}_2}}{\abs{U\hat{\xib}}^{3}}\PD{}{\xi_i}\frac{\xi_j}{\abs{\xib}}\\
        =&\frac{\xi_i}{\abs{\xib}}\frac{P_{0}\paren{\hat{\xi}_1,\hat{\xi}_2}}{\abs{U\hat{\xib}}^{3}}+\abs{\xib} \sum_{j=1,2}\frac{\PD{P_{0}}{\hat{\xi}_j} \abs{U\hat{\xib}}^{2}-3 P_{0}\paren{U^T U\hat{\xib}}_j}{\abs{ U\hat{\xib}}^{5}}\frac{\delta_{ij}-\hat{\xi}_i\hat{\xi}_j}{\abs{\xib}}\\
        =&\frac{P_i\paren{\hat{\xi}_1,\hat{\xi}_2}}{\abs{ U\hat{\xib}}^{5}},
\end{split}
\end{align}
where
\begin{align}
    P_i\paren{\hat{\xi}_1,\hat{\xi}_2}=\hat{\xi}_iP_{0}\abs{U\hat{\xib}}^{2}+ \sum_{j=1,2}\paren{\PD{P_{0}}{\hat{\xi}_j} \abs{U\hat{\xib}}^{2}-3 P_{0}\paren{U^T U\hat{\xib}}_j}\paren{\delta_{ij}-\hat{\xi}_i\hat{\xi}_j}.
\end{align}
The degrees of all terms are at most $1+4+2=3+2+2=4+1+2=7$.
Since
\begin{align}
    \abs{U\hat{\xib}}^{2}=& \sum_{j_1,j_2=1}^2\sum_{k=1}^3 U_{kj_1}U_{kj_2} \hat{\xi}_{j_1}\hat{\xi}_{j_2},\\
    \paren{U^TU\hat{\xib}}_j=&
    \begin{bmatrix}
       \sum_{k=1}^3 U_{kj}U_{k1} \hat{\xi}_{j}\\
       \sum_{k=1}^3 U_{kj}U_{k2} \hat{\xi}_{j}
    \end{bmatrix},
\end{align}
and $\PD{P_{0}}{\hat{\xi}_j}$ can be written as the form of \eqref{LA_DPform},
$P_i\paren{\hat{\xi}_1,\hat{\xi}_2}$ can be written as \eqref{LA_DPform}.
Thus, the case $\abs{\beta}=1$ holds.

Suppose $\abs{\beta}\leq k-1$ holds, for $\abs{\bar{\beta}}=k$, we may rewrite $\bar{\beta}$ as $\beta \beta_{k}$ where $\abs{\beta}=k-1$.
Then, 
\begin{align}
\begin{split}
     &\partial_{\bar{\beta}}L_A\paren{\bm{\xi}}
    =\partial_{\beta_k}\partial_{\beta}L_A\paren{\bm{\xi}}
    =\PD{}{\xi_{\beta_k}}\paren{\frac{1}{\abs{\xib}^{\abs{\beta}-1}}\frac{P_{\beta}\paren{\hat{\xi}_1,\hat{\xi}_2}}{\abs{U\hat{\xib}}^{2\abs{\beta}+3}}}\\
    =&-\paren{\abs{\beta}-1}\frac{1}{\abs{\xib}^{\abs{\beta}}}\hat{\xi}_{\beta_k}\frac{P_{\beta}}{\abs{U\hat{\xib}}^{2\abs{\beta}+3}}\\
     &+\frac{1}{\abs{\xib}^{\abs{\beta}-1}}\sum_{j=1,2}\frac{\PD{P_{\beta}}{\hat{\xi}_j} \abs{U\hat{\xib}}^{2}-\paren{2\abs{\beta}+3} P_{\beta}\paren{U^T U\hat{\xib}}_j}{\abs{ U\hat{\xib}}^{2\abs{\beta}+5}}\frac{\delta_{\beta_k j}-\hat{\xi}_{\beta_k}\hat{\xi}_j}{\abs{\xib}}\\
    =&\frac{1}{\abs{\xib}^{\abs{\bar{\beta}}}}\frac{P_{\bar{\beta}}\paren{\hat{\xi}_1,\hat{\xi}_2}}{\abs{ U\hat{\xib}}^{\bar{\beta}+3}}.
\end{split}
\end{align}
where 
\begin{align}
    P_{\hat{\beta}}=-\paren{\abs{\beta}-1}\hat{\xi}_iP_{\beta}\abs{U\hat{\xib}}^{2}+ \sum_{j=1,2}\paren{\PD{P_{\beta}}{\hat{\xi}_j} \abs{U\hat{\xib}}^{2}-\paren{2\abs{\bar{\beta}}+1} P_{\beta}\paren{U^T U\hat{\xib}}_j}\paren{\delta_{\beta_k j}-\hat{\xi}_{\beta_k}\hat{\xi}_j}.
\end{align}
The degrees of all terms are at most $1+\paren{3\abs{\beta}+4}+2=\paren{3\abs{\beta}+3}+2+2=\paren{3\abs{\beta}+4}+1+2=3\abs{\bar{\beta}}+4$.
Again, $\PD{P_{\beta}}{\hat{\xi}_j}$ is still able to written as the form of \eqref{LA_DPform},  so the case $\abs{\bar{\beta}}=k$ holds.
By Induction, for all $\beta$, the formulas \eqref{LA_Dform} and \eqref{LA_DPform} hold.

Next, in $P_{\beta}$, since for all square matrix $M$, $\norm{M}\leq \sum \abs{M_{ij}}$, we just need to estimate each element $\paren{P_{\beta}}_{i_1 i_2}$,
\begin{align}
    \abs{\paren{P_{\beta}}_{i_1 i_2}\paren{\hat{\xi}_1,\hat{\xi}_2}}\leq\sum \abs{ c^{\paren{\beta,i_1, i_2}}_{j_1,j_2}\paren{A, U, P,\frac{1}{\det(B)}, \frac{ \mc{T}}{\lambda}, \frac{d \mc{T}}{d\lambda}, \frac{1}{\lambda^2} }}
\end{align}
and $c^{\paren{\beta,i_1, i_2}}_{j_1,j_2}\paren{A, U, P,\frac{1}{\det(B)}, \frac{ \mc{T}}{\lambda}, \frac{d \mc{T}}{d\lambda}, \frac{1}{\lambda^2} }$ 
is a form of a polynomial, so we may only check each variable. Given $A\in \mc{DA}$, $\norm{\paren{ A^TA}^{-1}},\frac{1}{\det(B)}\leq \frac{1}{\sigma_2^2}$ and $\sigma_1\leq\lambda\leq\sqrt{2}\sigma_1$, so all variables are bounded by $\sigma_1$ and $\sigma_2$.
\begin{align}
    \norm{\PD{A^T A}{A_{ij}}}&\leq 2 \lambda\leq2\sqrt{2}\sigma_1,\\
    \norm{\PD{\paren{ A^TA}^{-1}}{A_{ij}}}&=\norm{\paren{ A^TA}^{-1}\PD{A^T A}{A_{ij}}\paren{ A^TA}^{-1}}\leq 2\sqrt{2}\frac{\sigma_1}{\sigma_2^4},
\end{align}
so on $\mc{DA}$, $U, P$ are  $C^1$ functions and their derivatives are bounded by $\sigma_1$ and $\sigma_2$.
Absolutely,
\begin{align}
    \norm{U\hat{\xib}}_{C^1\paren{\mc{DA}_{\sigma1,\sigma_2}}}\leq C^{(\beta)}_{\sigma_1, \sigma_2}.
\end{align}
Since
\begin{align}
    \abs{\PD{\det\paren{A^T A}}{A_{ij}}}&\leq \lambda^2 \leq 8\sigma_1^2,\\
    \abs{\PD{}{A_{ij}}\frac{1}{\det(B)}}&=\abs{\frac{1}{2\det(B)^3}\PD{\det\paren{A^T A}}{A_{ij}}}\leq 8\frac{\sigma_1^2}{\sigma_2^8},
\end{align}
and
\begin{align}
    \abs{\PD{\lambda}{A_{ij}}}=\abs{\frac{A_{ij}}{\lambda}}\leq 1
\end{align}
on $\mc{DA}$, $\frac{1}{\det(B)}, \frac{ \mc{T}}{\lambda}, \frac{d \mc{T}}{d\lambda}, \frac{1}{\lambda^2}$ are also  $C^1$ functions and their derivatives are bounded by $\sigma_1$ and $\sigma_2$.
Therefore, we obtain
\begin{align}
    \norm{P_{\beta}\paren{\hat{\xi}_1,\hat{\xi}_2}}_{C^1\paren{\mc{DA}_{\sigma1,\sigma_2}}}\leq C^{(\beta)}_{\sigma_1, \sigma_2,\mc{T}}.
\end{align}
\end{proof}

\begin{lemma}\label{LA_kernelest}
Given $A\in \mc{DA}_{\sigma_1,\sigma_2}$ and $\varphi\paren{\xib}=\varphi\paren{\abs{\xib}}$, a cutting and decreasing  respect $\abs{\xib}$ and supported in $\mc{B}\paren{1}$, and set 
\begin{align}
    K_0\paren{\bm{\theta}}&:= \mc{F}^{-1}\left[\paren{z+L_{A}}^{-1}(\bm{\xi})\varphi(\bm{\xi})\right],\\
    K_{1,j}\paren{\bm{\theta}}&:= \mc{F}^{-1}\left[ \xi_j\paren{z+L_{A}}^{-1}(\bm{\xi})\varphi(\bm{\xi})\right].
\end{align}
Then, for all $z\in \mc{S}_{\omega,\delta}$, we have the following estimates
\begin{align}
    \norm{K_0\paren{\bm{\theta}}}\leq& \frac{C_{\omega,\delta,\sigma_1, \sigma_2,\mc{T}}}{\abs{z}}\frac{1}{1+\abs{\thetab}^3},\\
    \norm{K_{1,j}\paren{\bm{\theta}}}\leq& C_{\omega,\delta,\sigma_1, \sigma_2,\mc{T}}\frac{1}{1+\abs{\thetab}^4}.
\end{align}

\end{lemma}
\begin{proof}
For convenience, we define $H(\bm{\xi}):=\paren{z+L_{A}}^{-1}(\bm{\xi})$
First, since
\begin{align}
\begin{split}
        &(1+|\thetab|^3)\int_{\mathbb{R}^2}e^{i\thetab\cdot\xib}\paren{z+L_{A}}^{-1}(\bm{\xi})\varphi(\bm{\xi})d\xib\\
    =&\int_{\mathbb{R}^2}(1-i\frac{\thetab}{|\thetab|}\cdot i\thetab|\thetab|^2)e^{i\thetab\cdot\xib}\paren{z+L_{A}}^{-1}(\bm{\xi})\varphi(\bm{\xi})d\xib\\
    =&\int_{\mathbb{R}^2}\left[(1-i\frac{\thetab}{|\thetab|}\cdot\nabla_{\xib}|\nabla_{\xib}|^2)e^{i\thetab\cdot\xib}\right]\paren{z+L_{A}}^{-1}(\bm{\xi})\varphi(\bm{\xi})d\xib\\
    =&\int_{\mathbb{R}^2}e^{i\thetab\cdot\xib}\paren{z+L_{A}}^{-1}(\bm{\xi})\varphi(\bm{\xi})+i e^{i\thetab\cdot\xib}\frac{\thetab}{|\thetab|}\cdot\nabla_{\xib}\Delta_{\xib}\left[\paren{z+L_{A}}^{-1}(\bm{\xi})\varphi(\bm{\xi}) \right] d\xib\\
    =&\int_{\mc{B}\paren{1}}e^{i\thetab\cdot\xib}\paren{z+L_{A}}^{-1}(\bm{\xi})\varphi(\bm{\xi})+i e^{i\thetab\cdot\xib}\frac{\thetab}{|\thetab|}\cdot\nabla_{\xib}\Delta_{\xib}\left[\paren{z+L_{A}}^{-1}(\bm{\xi})\varphi(\bm{\xi}) \right] d\xib,
\end{split}
\end{align}
By \eqref{IzLAbnd},
\begin{align}
    \norm{\int_{\mc{B}\paren{1}}e^{i\thetab\cdot\xib}\paren{z+L_{A}}^{-1}(\bm{\xi})\varphi(\bm{\xi})d\xib}\leq \frac{C_{\delta,\sigma_1, \sigma_2,\mc{T}}\norm{\varphi}_{C^0}}{\abs{z}}.
\end{align}
Next, we compute 
\begin{align}
\begin{split}
    &\PD{}{\xi_j}\Delta_{\xib}\left[\paren{z+L_{A}}^{-1}\varphi\right]\\
    =&\PD{}{\xi_j}\Delta_{\xib}\paren{z+L_{A}}^{-1}\varphi+\Delta_{\xib}\paren{z+L_{A}}^{-1}\PD{}{\xi_j}\varphi+2\PD{}{\xi_j}\nabla_{\xib}\paren{z+L_{A}}^{-1}\cdot\nabla_{\xib}\varphi\\
    +&2\nabla_{\xib}\paren{z+L_{A}}^{-1}\cdot\PD{}{\xi_j}\nabla_{\xib}\varphi+\PD{}{\xi_j}\paren{z+L_{A}}^{-1}\Delta_{\xib}\varphi+\paren{z+L_{A}}^{-1}\PD{}{\xi_j}\Delta_{\xib}\varphi.
\end{split}
\end{align}
Since $\varphi$ is smooth, we may estimate all of the terms except the first by \eqref{IzLAbnd} and \eqref{DDIzLAbnd02}. 
Obviously, for the last term,
\begin{align}
    \norm{i\frac{\theta_j}{|\thetab|}\int_{\mc{B}\paren{1}}e^{i\thetab\cdot\xib}\paren{z+L_{A}}^{-1}\paren{\PD{}{\xi_j}\Delta_{\xib}\varphi}  d\xib}\leq \frac{C_{\delta,\sigma_1, \sigma_2,\mc{T}}\norm{\varphi}_{C^3}}{\abs{z}}.
\end{align}
Then, for all $\abs{\alpha}=1,2$, $\abs{\alpha}+\abs{\beta}=3$
\begin{align}
\begin{split}
        &\norm{i\frac{\theta_j}{|\thetab|}\int_{\mc{B}\paren{1}}e^{i\thetab\cdot\xib}\partial_{\xib}^\alpha\paren{z+L_{A}}^{-1}(\bm{\xi})\partial_{\xib}^\beta\varphi(\bm{\xi})  d\xib}
    \leq \norm{\varphi}_{C^2}\int_{\mc{B}\paren{1}}\norm{\partial_{\xib}^\alpha\paren{z+L_{A}}^{-1}(\bm{\xi})}  d\xib\\
    \leq& \frac{C_{\delta, \sigma_1, \sigma_2, \mc{T}}}{\abs{z}^2}\norm{\varphi}_{C^2}\int_{\mc{B}\paren{1}}\abs{\bm{\xi}}^{1-\abs{\alpha}}  d\xib
    \leq\frac{C_{\delta, \sigma_1, \sigma_2, \mc{T}}}{\abs{z}^2}\norm{\varphi}_{C^2}.
\end{split}
\end{align}
Now, for the first term,
\begin{align}
\begin{split}
    \PD{}{\xi_j}\Delta_{\xib}H
    =&2\sum_{k=1,2}\left[-\paren{ E_{jkk}+E_{kjk}+E_{kkj}}+\paren{E_{k\paren{jk}}+E_{\paren{jk}k}}\right]\\
    &+\paren{ E_{j\paren{kk}}+E_{\paren{kk}j}}-H\PD{}{\xi_j}\Delta_{\xib}L_A H,
\end{split}\label{DDDIzLA}
\end{align}
where 
\begin{align*}
    E_{ijk}=&H\PD{L_A}{\xi_i}H\PD{L_A}{\xi_j}H\PD{L_A}{\xi_k}H,\\
    E_{k\paren{jk}}=&H\PD{L_A}{\xi_k}H\frac{\partial^2 L_A}{\partial\xi_j\partial\xi_k}H,\\
    E_{j\paren{kk}}=&H\PD{L_A}{\xi_j}H\Delta_{\xib}L_A H.
\end{align*}
Since $\norm{\PD{L_A}{\xi_j}}\lesssim 1$ and $\norm{\frac{\partial^2 L_A}{\partial\xi_j\partial\xi_k}}\lesssim \abs{\xib}^{-1}$, $\norm{E_{ijk}}\lesssim 1$ and $\norm{E_{k\paren{jk}}}, \norm{E_{j\paren{kk}}}\lesssim \abs{\xib}^{-1}$.
Therefore, we only have to check $\paren{z+L_{A}}^{-1}\PD{}{\xi_j}\Delta_{\xib}L_A\paren{z+L_{A}}^{-1}\varphi$.
$L_A$ is an even function, so the term is an odd function. 
By $L_A\paren{\xib}=\abs{\xib}L_A\paren{\hat{\xib}}$ and \eqref{LADfactor03},
\begin{align}
\begin{split}
    &\norm{i\frac{\theta_j}{|\thetab|}\int_{\mc{B}\paren{1}}e^{i\thetab\cdot\xib}\paren{z+L_{A}\paren{\xib}}^{-1}\PD{}{\xi_j}\Delta_{\xib}L_A\paren{\xib}\paren{z+L_{A}\paren{\xib}}^{-1}\varphi\paren{\xib}  d\xib}\\
    =&\norm{-\frac{\theta_j}{|\thetab|}\int_{\mc{B}\paren{1}}\sin \paren{\thetab\cdot\xib}\paren{z+L_{A}\paren{\xib}}^{-1}\PD{}{\xi_j}\Delta_{\xib}L_A\paren{\xib}\paren{z+L_{A}\paren{\xib}}^{-1}\varphi\paren{\xib}  d\xib}\\
    \leq&\norm{\int_{\mc{B}\paren{1}}\sin \paren{\thetab\cdot\hat{\xib} \abs{\xib}}\paren{z+\abs{\xib}L_A\paren{\hat{\xib}}}^{-1}\frac{\Phi^{(3)}_{A,j}\paren{\hat{\xib}}}{\abs{\xib}^2}\paren{z+\abs{\xib}L_A\paren{\hat{\xib}}}^{-1}\varphi\paren{\abs{\xib}}  d\xib}\\
    =&\norm{\int_{\mathbb{S}^1}\int_0^1 \paren{z+r L_A\paren{\hat{\xib}}}^{-1}\Phi^{(3)}_{A,j}\paren{\hat{\xib}}\paren{z+r L_A\paren{\hat{\xib}}}^{-1}\varphi\paren{r}\frac{\sin \paren{\thetab\cdot\hat{\xib} r}}{r} dr d\hat{\xib}}.
\end{split}
\end{align}
By Lemma \ref{t:integral_est01}, we obtain for all $\hat{\xib}\in \mathbb{S}^1$
\begin{align}
\begin{split}
    &\norm{\int_0^1 \paren{z+r L_A\paren{\hat{\xib}}}^{-1}\Phi^{(3)}_{A,j}\paren{\hat{\xib}}\paren{z+r L_A\paren{\hat{\xib}}}^{-1}\varphi\paren{r}\frac{\sin \paren{\thetab\cdot\hat{\xib} r}}{r} dr}\\
    \leq&2 \norm{\paren{z+r L_A\paren{\hat{\xib}}}^{-1}\Phi^{(3)}_{A,j}\paren{\hat{\xib}}\paren{z+r L_A\paren{\hat{\xib}}}^{-1}\varphi\paren{r}}_{C^1\paren{[0,1]}}\\
    \leq&4 \norm{\Phi^{(3)}_{A,j}\paren{\hat{\xib}}}\norm{\varphi\paren{r}}_{C^1\paren{[0,1]}}\norm{\paren{z+r L_A\paren{\hat{\xib}}}^{-1}}_{C^1\paren{[0,1]}}^2\\
    \leq&C_{\delta,\sigma_1,\sigma_2,\mc{T}}\norm{\varphi\paren{r}}_{C^1\paren{[0,1]}}\paren{\frac{1}{\abs{z}}+\frac{1}{\abs{z}^2}}^2.
\end{split}
\end{align}
Therefore,
\begin{align}
\begin{split}
    &\norm{\int_{\mathbb{S}^1}\int_0^1 \paren{z+r L_A\paren{\hat{\xib}}}^{-1}\Phi^{(3)}_{A,j}\paren{\hat{\xib}}\paren{z+r L_A\paren{\hat{\xib}}}^{-1}\varphi\paren{r}\frac{\sin \paren{\thetab\cdot\hat{\xib} r}}{r} dr d\hat{\xib}}\\
    \leq& C_{\delta,\sigma_1,\sigma_2,\mc{T}}\norm{\varphi\paren{r}}_{C^1}\paren{\frac{1}{\abs{z}}+\frac{1}{\abs{z}^2}}^2.
\end{split}
\end{align}
Since $\frac{1}{\abs{z}}\leq C_{\omega,\delta}$ if $z\in \mc{S}_{\omega,\delta}$,
\begin{align}
\begin{split}
    &\norm{\int_{\mc{B}\paren{1}}e^{i\thetab\cdot\xib}\paren{z+L_{A}}^{-1}(\bm{\xi})\varphi(\bm{\xi})+i e^{i\thetab\cdot\xib}\frac{\thetab}{|\thetab|}\cdot\nabla_{\xib}\Delta_{\xib}\left[\paren{z+L_{A}}^{-1}(\bm{\xi})\varphi(\bm{\xi}) \right] d\xib}\\
    \leq&\norm{\varphi\paren{r}}_{C^3}\sum_{k=1}^4 C^{(k)}_{\delta,\sigma_1,\sigma_2,\mc{T}}\frac{1}{\abs{z}^k}\\
    \leq &\frac{C_{\omega,\delta,\sigma_1, \sigma_2,\mc{T}}\norm{\varphi}_{C^3}}{\abs{z}},
\end{split}
\end{align}
and
\begin{align}
    \norm{\int_{\mathbb{R}^2}e^{i\thetab\cdot\xib}\paren{z+L_{A}}^{-1}(\bm{\xi})\varphi(\bm{\xi})d\xib}
    \leq& \frac{C_{\omega,\delta,\sigma_1, \sigma_2,\mc{T}}\norm{\varphi}_{C^3}}{\abs{z}} \frac{1}{1+\abs{\thetab}^3}.
\end{align}
Next, for $K_{1,j}$, we use the same technique.
$(1+|\thetab|^4)K_{1,j}\paren{\thetab}$ becomes
\begin{align}
\begin{split}
        &(1+|\thetab|^4)\int_{\mathbb{R}^2}e^{i\thetab\cdot\xib}\xi_j\paren{z+L_{A}}^{-1}(\bm{\xi})\varphi(\bm{\xi})d\xib\\
    =&\int_{\mathbb{R}^2}\left[(1+|\nabla_{\xib}|^4)e^{i\thetab\cdot\xib}\right]\xi_j\paren{z+L_{A}}^{-1}(\bm{\xi})\varphi(\bm{\xi})d\xib\\
    =&\int_{\mc{B}\paren{1}}e^{i\thetab\cdot\xib}\xi_j\paren{z+L_{A}}^{-1}(\bm{\xi})\varphi(\bm{\xi})+ e^{i\thetab\cdot\xib}\Delta_{\xib}^2\left[\xi_j\paren{z+L_{A}}^{-1}(\bm{\xi})\varphi(\bm{\xi}) \right] d\xib,
\end{split}
\end{align}
By \eqref{IzLAbnd},
\begin{align}
    \abs{\int_{\mc{B}\paren{1}}e^{i\thetab\cdot\xib}\xi_j\paren{z+L_{A}}^{-1}(\bm{\xi})\varphi(\bm{\xi})d\xib}\leq \frac{C_{\delta,\sigma_1, \sigma_2,\mc{T}}\norm{\varphi}_{C^0}}{\abs{z}}.
\end{align}
Next,
\begin{align}
    \Delta_{\xib}^2\left[\xi_j\paren{z+L_{A}}^{-1}(\bm{\xi})\varphi(\bm{\xi}) \right]=4 \PD{}{\xi_j}\Delta_{\xib}\left[\paren{z+L_{A}}^{-1}\varphi\right]+\xi_j\Delta_{\xib}^2\left[\paren{z+L_{A}}^{-1}(\bm{\xi})\varphi(\bm{\xi}) \right]
\end{align}
We have estimated the first term, so let us compute the second,
\begin{align}
\begin{split}
    &\Delta_{\xib}^2\left[\paren{z+L_{A}}^{-1}\varphi \right]\\
    =&\quad\Delta_{\xib}^2\paren{z+L_{A}}^{-1}\varphi+4\nabla_{\xib}\Delta_{\xib}\paren{z+L_{A}}^{-1}\cdot\nabla_{\xib} \varphi+4\left[\nabla_{\xib}^2\paren{z+L_{A}}^{-1}:\nabla_{\xib}^2\varphi\right]\\
     &+2\Delta_{\xib}\paren{z+L_{A}}^{-1}\Delta_{\xib}\varphi+4\nabla_{\xib}\paren{z+L_{A}}^{-1}\cdot\nabla_{\xib}\Delta_{\xib}\varphi+\paren{z+L_{A}}^{-1}\Delta_{\xib}^2\varphi.
\end{split}
\end{align}
For the last four terms, we may estimate them by \eqref{IzLAbnd} and \eqref{DDIzLAbnd02} again.
For the second term, since $\partial^\alpha_{\xib}L_A\lesssim \abs{\xib}^{1-\abs{\alpha}}$, by \eqref{DDDIzLA}, we may obtain
\begin{align}
\begin{split}
        &\norm{\int_{\mc{B}\paren{1}}e^{i\thetab\cdot\xib}\xi_j\nabla_{\xib}\Delta_{\xib}\paren{z+L_{A}}^{-1}\paren{\xib}\cdot\nabla_{\xib} \varphi\paren{\xib} d\xib}\\
    \leq&\norm{\varphi}_{C^1}\sum_{k=1}^4 C^{(k)}_{\delta,\sigma_1,\sigma_2,\mc{T}}\frac{1}{\abs{z}^k}.
\end{split}
\end{align}
In the first term, we have

\begin{align}
\begin{split}
    \Delta_{\xib}^2 H
    =&\sum_{j,k=1,2}\left[\quad 8\paren{E_{jjkk}+E_{jkjk}+E_{jkkj}}\right.\\
    &\qquad\quad\left.      -8\paren{E_{jk\paren{jk}}+E_{j\paren{jk}k}+E_{\paren{jk}jk}}+4H\frac{\partial^2 L_A}{\partial\xi_j\partial\xi_k}H\frac{\partial^2 L_A}{\partial\xi_j\partial\xi_k}H\right]\\
    &+\sum_{j=1,2}\left[-4\paren{E_{jj\paren{kk}}+E_{j\paren{kk}j}+E_{\paren{kk}jj}}+4\paren{E_{j\paren{jkk}}+E_{\paren{jkk}j}}\right]\\
    &+2H\Delta_{\xib}L_A H\Delta_{\xib}L_A H
    -H\Delta_{\xib}^2 L_A H,
\end{split}
\end{align}
where
\begin{align*}
    E_{jjkk}=&H\PD{L_A}{\xi_j}H\PD{L_A}{\xi_j}H\PD{L_A}{\xi_k}H\PD{L_A}{\xi_k}H,\\
    E_{jk\paren{jk}}=&H\PD{L_A}{\xi_j}H\PD{L_A}{\xi_k}H\frac{\partial^2 L_A}{\partial\xi_j\partial\xi_k}H,\\
    E_{jj\paren{kk}}=&H\PD{L_A}{\xi_j}H\PD{L_A}{\xi_j}H\Delta_{\xib}L_A H,\\
    E_{j\paren{jkk}}=&H\PD{L_A}{\xi_k}H\PD{\Delta_{\xib}L_A}{\xi_j}H,
\end{align*}
and we only have to compute the $-\paren{z+L_{A}}^{-1}\Delta_{\xib}^2 L_A\paren{z+L_{A}}^{-1}$ term.
Since $L_A$ is an even function, $-\xi_j\paren{z+L_{A}}^{-1}\Delta_{\xib}^2 L_A\paren{z+L_{A}}^{-1}\paren{\xib}\varphi(\bm{\xi})$ is odd, again, by Lemma \ref{l:LA_est} and \ref{t:integral_est01}
\begin{align}
\begin{split}
    &\norm{-\int_{\mc{B}\paren{1}}e^{i\thetab\cdot\xib}\xi_j\paren{z+L_{A}}^{-1}\Delta_{\xib}^2 L_A\paren{z+L_{A}}^{-1}\paren{\xib}\varphi(\bm{\xi})  d\xib}\\
    =&\norm{\int_{\mc{B}\paren{1}}\xi_j\sin \paren{\thetab\cdot\xib}\paren{z+L_{A}\paren{\xib}}^{-1}\Delta_{\xib}^2 L_A\paren{\xib}\paren{z+L_{A}\paren{\xib}}^{-1}\varphi\paren{\xib}  d\xib}\\
    \leq&\norm{\int_{\mc{B}\paren{1}}\sin \paren{\thetab\cdot\hat{\xib} \abs{\xib}}\paren{z+\abs{\xib}L_A\paren{\hat{\xib}}}^{-1}\frac{\xi_j\Phi^{(4)}_{A}\paren{\hat{\xib}}}{\abs{\xib}^3}\paren{z+\abs{\xib}L_A\paren{\hat{\xib}}}^{-1}\varphi\paren{\abs{\xib}}  d\xib}\\
    =&\norm{\int_{\mathbb{S}^1}\hat{\xi}_j\int_0^1 \paren{z+r L_A\paren{\hat{\xib}}}^{-1}\Phi^{(4)}_{A}\paren{\hat{\xib}}\paren{z+r L_A\paren{\hat{\xib}}}^{-1}\varphi\paren{r}\frac{\sin \paren{\thetab\cdot\hat{\xib} r}}{r} dr d\hat{\xib}}\\
    \leq& 4 \int_{\mathbb{S}^1}\abs{\hat{\xi}_j}\norm{\Phi^{(4)}_{A}\paren{\hat{\xib}}}\norm{\varphi\paren{r}}_{C^1\paren{[0,1]}}\norm{\paren{z+r L_A\paren{\hat{\xib}}}^{-1}}_{C^1\paren{[0,1]}}^2 d\hat{\xib}\\
    \leq&C_{\delta,\sigma_1,\sigma_2,\mc{T}}\norm{\varphi\paren{r}}_{C^1\paren{[0,1]}}\paren{\frac{1}{\abs{z}}+\frac{1}{\abs{z}^2}}^2.
\end{split}
\end{align}
Hence,
\begin{align}
\begin{split}
    &\norm{\int_{\mc{B}\paren{1}}e^{i\thetab\cdot\xib}\xi_j\paren{z+L_{A}}^{-1}(\bm{\xi})\varphi(\bm{\xi})+ e^{i\thetab\cdot\xib}\Delta_{\xib}^2\left[\xi_j\paren{z+L_{A}}^{-1}(\bm{\xi})\varphi(\bm{\xi}) \right] d\xib}\\
    \leq&\norm{\varphi\paren{r}}_{C^4}\sum_{k=1}^5 C^{(k)}_{\delta,\sigma_1,\sigma_2,\mc{T}}\frac{1}{\abs{z}^k}\\
    \leq &C_{\omega,\delta,\sigma_1, \sigma_2,\mc{T}}\norm{\varphi}_{C^4},
\end{split}
\end{align}
and
\begin{align}
    \norm{\int_{\mathbb{R}^2}e^{i\thetab\cdot\xib}\xi_j\paren{z+L_{A}}^{-1}(\bm{\xi})\varphi(\bm{\xi})d\xib}
    \leq& C_{\omega,\delta,\sigma_1, \sigma_2,\mc{T}}\norm{\varphi}_{C^4} \frac{1}{1+\abs{\thetab}^4}.
\end{align}

\end{proof}
% \begin{lemma}\label{diffLA_est}
% Given $A_1, A_2\in \mc{DA}_{\sigma_1,\sigma_2}$ and set 
% \begin{align}
%     K\paren{\bm{\theta}}&:= \mc{F}^{-1}\left[\paren{z+L_{A}}^{-1}(\bm{\xi})\varphi(\bm{\xi})\right].
% \end{align}
% Then, for all $z\in \mc{S}_{\omega,\delta}$, we have the following estimates
% \begin{align}
%     \abs{K\paren{\bm{\theta}}}\leq& C_{\omega,\delta,\sigma_1, \sigma_2,\mc{T}}\norm{A_1-A_2}\frac{1}{1+\abs{\thetab}^3}
% \end{align}
% \end{lemma}

\begin{lemma}\label{eLA_est}
Given $k(\bm{x})=\mc{F}^{-1}[e^{-L_A\paren{\xib}}]$, then we have the following estimates
\begin{align}
    \norm{k(\bm{x})}&\leq C\frac{1}{1+\abs{\bm{x}}^3}\label{eLA_est00},\\
    \norm{\PD{}{x_i} k(\bm{x})}&\leq C\frac{1}{1+\abs{\bm{x}}^4}\label{eLA_est01},\\
    \norm{\PD{}{x_i}\PD{}{x_j}k(\bm{x})}&\leq C\frac{1}{1+\abs{\bm{x}}^5}\label{eLA_est02}.
\end{align}
\end{lemma}
\begin{proof}
First,  
\begin{equation*}
    \begin{aligned}
    (1+|\bm{x}|^3)\int_{\mathbb{R}^2}e^{i\bm{x}\cdot\xib}e^{-L_A(\xib)}d\xib
    &=\int_{\mathbb{R}^2}(1-i\frac{\bm{x}}{|\bm{x}|}\cdot i\bm{x}|\bm{x}|^2)e^{i\bm{x}\cdot\xib}e^{-L_A(\xib)}d\xib\\
    &=\int_{\mathbb{R}^2}(1-i\frac{\bm{x}}{|\bm{x}|}\cdot\nabla_{\xib}|\nabla_{\xib}|^2)e^{i\bm{x}\cdot\xib}e^{-L_A(\xib)}d\xib\\
    &=\int_{\mathbb{R}^2}e^{i\bm{x}\cdot\xib}e^{-L_A(\xib)}+i\frac{\bm{x}}{|\bm{x}|}\cdot\nabla_{\xib}\Delta_{\xib}e^{-L_A(\xib)}d\xib.
    \end{aligned}
\end{equation*}
Since $e^{- L_A(\xib)}\lesssim e^{-\abs{\xib}}$, we obtain
\begin{equation}\label{aux_decay}
    \begin{aligned}
    (1+|\bm{x}|^3)\norm{k(\bm{x})}&\lesssim\Big(1+\norm{\int_{B_1(0)}e^{i\bm{x}\cdot\xib}\frac{\bm{x}}{|\bm{x}|}\cdot \nabla_{\xib}\Delta_{\xib} e^{-L_A(\xib)}d\xib}\Big).
    \end{aligned}
\end{equation}
Next, since
\begin{align*}
    \PD{}{\xi_i}e^{-L_A(\xib)}=-\int_0^1 e^{-\paren{1-t}L_A(\xib)}\PD{}{\xi_i}L_A(\xib)e^{-tL_A(\xib)}dt,
\end{align*}
\begin{align*}
\begin{split}
    &\partial_{\xib}^{jkk} e^{-L_A(\xib)}\\
    =&-\int_0^1 e^{-\paren{1-t_1}L_A(\xib)}\partial_{\xib}^{jkk}L_A(\xib)e^{-t_1 L_A(\xib)}dt_1\\
    &+H_{21}\paren{j,kk}+H_{21}\paren{k,jk}+H_{21}\paren{jk,k}+H_{22}\paren{kk,j}+H_{22}\paren{jk,k}+H_{22}\paren{k,jk}\\
    &-H_3\paren{\paren{1-t_1}\paren{1-t_2}\paren{1-t_3},j,\paren{1-t_1}\paren{1-t_2}t_3,k,\paren{1-t_1}t_2,k,t_1}\\
    &-H_3\paren{\paren{1-t_1}\paren{1-t_2},k,\paren{1-t_1}t_2\paren{1-t_3},j,\paren{1-t_1}t_2t_3,k,t_1}\\
    &-H_3\paren{\paren{1-t_1}\paren{1-t_2},k,\paren{1-t_1}t_2,k,t_1\paren{1-t_3},j,t_1t_3}\\
    &-H_3\paren{\paren{1-t_1}\paren{1-t_3},j,\paren{1-t_1}t_3,k,t_1\paren{1-t_2},k,t_1t_2}\\
    &-H_3\paren{\paren{1-t_1},k,t_1\paren{1-t_2}\paren{1-t_3},j,t_1\paren{1-t_2}t_3,k,t_1t_2}\\
    &-H_3\paren{\paren{1-t_1},k,t_1\paren{1-t_2}t_3,k,t_1t_2\paren{1-t_3},j,t_1t_2t_3},
\end{split}
\end{align*}
where
\begin{align*}
    H_{21}\paren{\alphab,\betab}
    =&\int_0^1\int_0^1 e^{-\paren{1-t_1}\paren{1-t_2}L_A(\xib)}\partial_{\xib}^{\alphab}L_A(\xib)e^{-\paren{1-t_1}t_2 L_A(\xib)}\partial_{\xib}^{\betab}L_A(\xib)e^{-t_1 L_A(\xib)}dt_1dt_2,\\
    H_{22}\paren{\alphab,\betab}
    =&\int_0^1\int_0^1 e^{-t_1L_A(\xib)}\partial_{\xib}^{\alphab}L_A(\xib)e^{-\paren{1-t_1}t_2 L_A(\xib)}\partial_{\xib}^{\betab}L_A(\xib)e^{- \paren{1-t_1}\paren{1-t_2}L_A(\xib)}dt_1dt_2.
\end{align*}

\begin{align*}
        &H_{3}\paren{s_1,\alphab,s_2,\betab,s_3, \gammab,s_4}\\
    =&\int_0^1\int_0^1\int_0^1 e^{-s_1L_A(\xib)}\partial_{\xib}^{\alphab}L_A(\xib)e^{-t_2 L_A(\xib)}\partial_{\xib}^{\betab}L_A(\xib)e^{-s_3 L_A(\xib)}\partial_{\xib}^{\gammab}L_A(\xib)e^{-s_4L_A(\xib)}dt_1dt_2dt_3.
\end{align*}
Since $L_A(\xib)$ is even and homogeneous of degree one, we have that the third derivatives of $L_A(\xib)$ are odd and homogeneous of degree minus two. The other terms are less singular and thus the corresponding integrals in \eqref{aux_decay} are bounded directly. 
That is to say, $\norm{\partial_{\xib}^{\alphab}L_A(\xib)}\lesssim \abs{\xib}^{1-\abs{\xib}}$ and $\norm{e^{-s L_A(\xib)}}\lesssim e^{-sC\abs{\xib}}$ for some $C>0$, so $\norm{H_{21}}, \norm{H_{22}}\lesssim \abs{\xib}^{-1} e^{-C\abs{\xib}}$,$\norm{ H_3}\lesssim  e^{-C\abs{\xib}}$, and all of them are integrable.
Therefore, we only have to estimate the first, i.e.
\begin{align*}
    \int_{B_1(0)}\int_0^1 e^{-\paren{1-t_1}L_A(\xib)}e^{i\bm{x}\cdot\xib}\frac{\bm{x}}{|\bm{x}|} \cdot\nabla_{\xib}\Delta_{\xib}L_A(\xib)e^{-t_1 L_A(\xib)}dt_1 d\xib.
\end{align*}
We can write 
\begin{equation*}
    \begin{aligned}
    e^{-\paren{1-t_1}L_A(\xib)}\PD{}{\xi_j}\Delta_{\xib} L_A(\xib)e^{-t_1 L_A(\xib)}=\frac{1}{|\xib|^2}e^{-\paren{1-t_1}L_A(\xib)}\Phi^{(3)}_{A,j}(\hat{\xib})e^{-t_1 L_A(\xib)},
    \end{aligned}
\end{equation*}
where $\hat{\xib}=\xib/|\xib|$ and $\Phi^{(3)}_{A,j}(\hat{\xib})$ is even and bounded from below and above, thanks to the arc-chord condition and the $C^1$ regularity (see \eqref{s1s2} and Remark \ref{arc_chord}).
We then have that
\begin{equation*}
    \begin{aligned}
    &\int_{B_1(0)}\int_0^1 e^{-\paren{1-t_1}L_A(\xib)}e^{i\bm{x}\cdot\xib}\frac{\bm{x}}{|\bm{x}|} \cdot\nabla_{\xib}\Delta_{\xib}L_A(\xib)e^{-t_1 L_A(\xib)}dt_1 d\xib\\
    =&\sum_{j=1,2}\frac{ix_j}{|\bm{x}|}\int_{\mathbb{S}^1}\int_0^1\int_0^1 \frac{\sin{(\bm{x}\cdot\hat{\xib}\minspace r})}{r}e^{-\paren{1-t_1}rL_A(\hat{\xib})} \Phi^{(3)}_{A,j}(\hat{\xib})e^{-t_1r L_A(\hat{\xib})} dr dt_1 d\hat{\xib}.
    \end{aligned}
\end{equation*}
By lemma \ref{t:integral_est01}, we obtain for all $\bm{x}\in \mbr,\hat{ \xib}\in \mathbb{S}^1$ and $t_1\in[0,1]$
\begin{align}
\begin{split}
       &\norm{ \int_0^1 \frac{\sin{(\bm{x}\cdot\hat{\xib}\minspace r})}{r} e^{-\paren{1-t_1}rL_A(\hat{\xib})} \Phi^{(3)}_{A,j}(\hat{\xib})e^{-t_1r L_A(\hat{\xib})}dr}\\ 
   \leq &2 \norm{e^{-\paren{1-t_1}rL_A(\hat{\xib})} \Phi^{(3)}_{A,j}(\hat{\xib})e^{-t_1r L_A(\hat{\xib})}}_{C^1\paren{[0,1];r}}.
\end{split}
\end{align}
$L_A(\hat{\xib})$ is positive definite and diagonalizable and $\Phi^{(3)}_{A,j}(\hat{\xib})$ is boundned, so for all $\hat{\xib}\in \mathbb{S}^1$ and $t_1\in[0,1]$,
\begin{align}
    \norm{e^{-\paren{1-t_1}rL_A(\hat{\xib})} \Phi^{(3)}_{A,j}(\hat{\xib})e^{-t_1r L_A(\hat{\xib})}}_{C^1\paren{[0,1];r}}\leq C\paren{ 1+\norm{L_A(\hat{\xib})}}.
\end{align}
Therefore, since $L_A(\hat{\xib})$ is bounded on $\mathbb{S}^1$,
\begin{align}
    \int_{B_1(0)}\int_0^1 e^{-\paren{1-t_1}L_A(\xib)}e^{i\bm{x}\cdot\xib}\frac{\bm{x}}{|\bm{x}|} \cdot\nabla_{\xib}\Delta_{\xib}L_A(\xib)e^{-t_1 L_A(\xib)}dt_1 d\xib
\end{align}
is bounded, and we may conclude that
\begin{equation*}
    \begin{aligned}
    \norm{k(\bm{x})}\lesssim (1+|\bm{x}|^3)^{-1}.
    \end{aligned}
\end{equation*}
Next, since
\begin{align*}
    \PD{}{x_i} k(\bm{x})=\int_{\mathbb{R}^2}i\xi_i e^{-L_A(\xib)}e^{i\bm{x}\cdot\xib}d\xib,
\end{align*}
we have
\begin{align*}
    \norm{\PD{}{x_i} k(\bm{x})}=\int_{\mathbb{R}^2}\abs{\xi_i} \norm{ e^{-L_A(\xib)}}d\xib\leq C,
\end{align*}
where $C$ only depends on $A$. Then,
\begin{align*}
    \abs{\bm{x}}^4\PD{}{x_i} k(\bm{x})=\int_{\mathbb{R}^2}\paren{\Delta_{\xib}}^2 \paren{i\xi_i e^{-L_A(\xib)}}e^{i\bm{x}\cdot\xib}d\xib,
\end{align*}
where $\paren{\Delta_{\xib}}^2\paren{\xi_i e^{-L_A(\xib)}}=4 \PD{}{\xi_i}\Delta_{\xib}e^{-L_A(\xib)}+\xi_i\paren{\Delta_{\xib}}^2 e^{-L_A(\xib)}$.
The first term is the $i$ component in $\nabla_{\xib}\Delta_{\xib} e^{-L_A(\xib)}$, so we may claim the term is integratable by the previous techniques.
% it is odd and 
% \begin{align*}
%     e^{-L_A(\xib)}\PD{}{\xi_i}\Delta_{\xib}L_A(\xib)= \frac{\xi_i}{|\xib|^3}\Phi_A(\hat{\xib})e^{-L_A(\xi)}.
% \end{align*}
For the second term, again,
\begin{align*}
        &\partial_{\xib}^{jjkk}L_A(\xib)\\
    =   &-\int_0^1 e^{-\paren{1-t_1}L_A(\xib)}\partial_{\xib}^{jjkk}L_A(\xib)e^{-t_1 L_A(\xib)}dt_1\\
        &+H_{2}\paren{\paren{1-t_1}\paren{1-t_2},jjk,\paren{1-t_1}t_2,k,t_1}+\cdots\\
        &-H_3\paren{\paren{1-t_1}\paren{1-t_2}\paren{1-t_3},jj,\paren{1-t_1}\paren{1-t_2}t_3,k,\paren{1-t_1}t_2,k,t_1}-\cdots\\
        &+H_4\left(\paren{1-t_1}\paren{1-t_2}\paren{1-t_3}\paren{1-t_4},\paren{1-t_1}\paren{1-t_2}\paren{1-t_3}t_4,\right.\\
        &\qquad\qquad\left. j,\paren{1-t_1}\paren{1-t_2}t_3,k,\paren{1-t_1}t_2,k,t_1\right)+\cdots
\end{align*}
where
\begin{align*}
        &H_{2}\paren{s_1,\alphab,s_2,\betab,s_3}\\
    =   &\int_0^1\int_0^1 e^{-s_1L_A(\xib)}\partial_{\xib}^{\alphab}L_A(\xib)e^{-s_2 L_A(\xib)}\partial_{\xib}^{\betab}L_A(\xib)e^{-s_3 L_A(\xib)}dt_1dt_2\\
        &H_{3}\paren{s_1,\alphab,s_2,\betab,s_3, \gammab,s_4}\\
    =   &\int_0^1\int_0^1\int_0^1 e^{-s_1L_A(\xib)}\partial_{\xib}^{\alphab}L_A(\xib)e^{-s_2 L_A(\xib)}\partial_{\xib}^{\betab}L_A(\xib)e^{-s_3 L_A(\xib)}\partial_{\xib}^{\gammab}L_A(\xib)e^{-s_4L_A(\xib)}\\
        &dt_1dt_2dt_3\\
        &H_{4}\paren{s_1,\alphab,s_2,\betab,s_3, \gammab,s_4,\deltab,s_5}\\
    =   &\int_0^1\int_0^1\int_0^1\int_0^1 e^{-s_1L_A(\xib)}\partial_{\xib}^{\alphab}L_A(\xib)e^{-s_2 L_A(\xib)}\partial_{\xib}^{\betab}L_A(\xib)e^{-s_3 L_A(\xib)}\partial_{\xib}^{\gammab}L_A(\xib)e^{-s_4L_A(\xib)}\\
        &\partial_{\xib}^{\deltab}L_A(\xib)e^{-s_5L_A(\xib)}dt_1dt_2dt_3dt_4.
\end{align*}
There are $14$ $H_2$-type terms, $36$ $H_3$-type terms, $24$ $H_4$-type terms in $\partial_{\xib}^{jjkk}L_A$.
Since $\norm{\partial_{\xib}^{\alphab}L_A(\xib)}\lesssim \abs{\xib}^{1-\abs{\xib}}$ and $\norm{e^{-s L_A(\xib)}}\lesssim e^{-sC\abs{\xib}}$ for some $C>0$, $\norm{\xi_i H_{2}}\lesssim \abs{\xib}^{-1} e^{-C\abs{\xib}}$,$\norm{\xi_i H_3}\lesssim  e^{-C\abs{\xib}}$ and $\norm{\xi_i H_4}\lesssim\abs{\xib}  e^{-C\abs{\xib}}$.
Hence, we only have to check
\begin{align*}
    \int_{\mathbb{R}^2}\int_0^1 e^{-\paren{1-t_1}L_A(\xib)}\xi_i\paren{\Delta_{\xib}}^2L_A(\xib)e^{-t_1 L_A(\xib)}dt_1d\xib.
\end{align*}
Since $\int_0^1 e^{-\paren{1-t_1}L_A(\xib)}\xi_i\paren{\Delta_{\xib}}^2L_A(\xib)e^{-t_1 L_A(\xib)}dt_1$ is odd, we may use the same technique in $k(\bm{x})$ term to obtain
\begin{align*}
    \abs{\bm{x}}^4\norm{ \PD{}{x_i} k(\bm{x})}\leq C,
\end{align*}
so
\begin{align*}
    \norm{ \PD{}{x_i} k(\bm{x})}\lesssim \frac{1}{1+\abs{\bm{x}}^4}.
\end{align*}
Finally, since
\begin{align*}
    \PD{}{x_i} k(\bm{x})=\int_{\mathbb{R}^2}\xi_i e^{-L_A(\xib)}e^{i\bm{x}\cdot\xib}d\xib,
\end{align*}
we have
\begin{align*}
    \norm{\PD{}{x_i} k(\bm{x})}=\int_{\mathbb{R}^2}\abs{\xi_i} \norm{ e^{-L_A(\xib)}}d\xib\leq C,
\end{align*}
where $C$ only depends on $A$.

\end{proof}

\begin{lemma}\label{t:integral_est01}
Given a vector function $M\paren{r}$ in $ C^1\paren{[0,1]}$, we have the following inequality: for all $A\geq 0$,
\begin{align}
    \norm{\int_0^1 M\paren{r}\frac{\sin{\paren{Ar}}}{r}dr}\leq 2\norm{M(0)}+\norm{\frac{d M\paren{r}}{dr}}_{C^0\paren{[0,1]}}\leq 2\norm{M}_{C^1\paren{[0,1]}}.
\end{align}
\end{lemma}

\begin{proof}
\begin{align}
\begin{split}
    \int_0^1 M\paren{r}\frac{\sin{\paren{Ar}}}{r}dr= \int_0^1 M\paren{0}\frac{\sin{\paren{Ar}}}{r}dr+\int_0^1 \frac{M\paren{r}-M\paren{0}}{r}\sin{\paren{Ar}}dr.
\end{split}
\end{align}
For the first term,
\begin{align}
\begin{split}
        &\norm{\int_0^1 M\paren{0}\frac{\sin{\paren{Ar}}}{r}dr}=\norm{M\paren{0}\int_0^1 \frac{\sin{\paren{Ar}}}{r}dr}=\norm{M(0)}\abs{\int_0^A \frac{\sin{\paren{r}}}{r}dr}\\
    \leq&\norm{M(0)}\int_0^\pi \frac{\sin{\paren{r}}}{r}dr\approx 1.852\norm{M(0)} \leq 2\norm{M(0)}.
\end{split}
\end{align}
For the second term, since
\begin{align}
\begin{split}
    \norm{\frac{M\paren{r}-M\paren{0}}{r}}=\frac{\norm{\int_0^r  \frac{d M}{ds}\paren{s}ds}}{r}\leq \norm{\frac{d M\paren{r}}{dr}}_{C^0\paren{[0,1]}},
\end{split}
\end{align}
\begin{align}
\begin{split}
    \norm{\int_0^1 \frac{M\paren{r}-M\paren{0}}{r}\sin{\paren{Ar}}dr}&\leq \int_0^1\norm{ \frac{M\paren{r}-M\paren{0}}{r}}\abs{\sin{\paren{Ar}}}dr\\
    &\leq \norm{\frac{d M\paren{r}}{dr}}_{C^0\paren{[0,1]}}.
\end{split}
\end{align}
Therefore, we obtain the result.

\end{proof}

\begin{lemma}\label{t:integral_est02}
Given $f(t)\geq 0$ is a locally integrable function on $\mathbb{R}$, we have the following estimates:

(i) If $\alpha> -1$, for all $m< t$,
\begin{align}
    \abs{\int_m^t \paren{t-s}^\alpha f ds}\leq C \paren{t-m}^{\alpha+1} M_l [f](t).
\end{align}

(ii) If $\alpha< -1$, for all $M<  t$,
\begin{align}
    \abs{\int_{-\infty}^M \paren{t-s}^\alpha f ds}\leq C \paren{t-M}^{\alpha+1} M_l [f](t).
\end{align}
\end{lemma}

\begin{proof}
(i) By integration by part theorem,
\begin{align*}
     &\int_m^t \paren{t-s}^\alpha fds= -\int_m^t \paren{t-s}^\alpha \paren{\PD{}{s} \int_s^t f\paren{r}dr} ds\\
    =& \left.\paren{t-s}^\alpha  \int_s^t f\paren{r}dr\right |_{s=t}^m-\alpha \int_m^t \paren{t-s}^\alpha \paren{\frac{1}{t-s} \int_s^t f\paren{r}dr} ds.
\end{align*}
Since
\begin{align*}
     &\limsup_{s\rightarrow t^-}\abs{\paren{t-s}^\alpha  \int_s^t f\paren{r}dr}=\limsup_{s\rightarrow t^-}\abs{\paren{t-s}^{\alpha+1}  \frac{1}{t-s}\int_s^t f\paren{r}dr}\\
    \leq&\limsup_{s\rightarrow t^-}\paren{t-s}^{\alpha+1} M_l [f](t)=0,
\end{align*}
\begin{align*}
    \abs{\left.\paren{t-s}^\alpha  \int_s^t f\paren{r}dr\right |_{s=t}^m}\leq \paren{t-m}^{\alpha+1} M_l [f](t).
\end{align*}
Next,
\begin{align*}
    &\abs{\int_m^t \paren{t-s}^\alpha \paren{\frac{1}{t-s} \int_s^t f\paren{r}dr} ds}\\
    \leq &M_l [f](t)\int_m^t \paren{t-s}^\alpha ds=\frac{1}{\alpha+1} \paren{t-m}^{\alpha+1}M_l [f](t).
\end{align*}
Therefore,
\begin{align*}
    \abs{\int_m^t \paren{t-s}^\alpha f ds}\leq C \paren{t-m}^{\alpha+1} M_l [f](t).
\end{align*}

(ii) The proof is basically the same. It will be
\begin{align*}
     &\abs{\int_{-\infty}^M \paren{t-s}^\alpha fds}\\
    =& \abs{\left.\paren{t-s}^\alpha  \int_s^t f\paren{r}dr\right |_{s=-\infty}^M-\alpha \int_{-\infty}^M \paren{t-s}^\alpha \paren{\frac{1}{t-s} \int_s^t f\paren{r}dr} ds}\\
    \leq&  \paren{t-M}^{\alpha+1} M_l [f](t)+\limsup_{s\rightarrow -\infty} \paren{t-s}^{\alpha+1} M_l [f](t)+M_l [f](t)\int_{-\infty}^M \paren{t-s}^\alpha ds\\
    =&\paren{t-M}^{\alpha+1} M_l [f](t)-\frac{1}{\alpha+1}\paren{t-M}^{\alpha+1} M_l [f](t).
\end{align*}
since $\alpha+1< 0$.
Therefore,
\begin{align*}
    \abs{\int_{-\infty}^M \paren{t-s}^\alpha fds}\leq C \paren{t-M}^{\alpha+1} M_l [f](t).
\end{align*}
\end{proof}

% \begin{align*}
%     \bm{u}(t,t_0,\bm{x})=\int_{t_0}^t e^{(t-s)\mc{L}_A} \bm{f}(s,\bm{x})ds=\int_{t_0}^t \int_\mbr K(t-s,\bm{x}-\bm{y})\bm{f}(s,\bm{y})d \bm{y} ds.
% \end{align*}

% \begin{thm}
% Given $0<t_0<t<1$,
% \begin{align}
%     \norm{\bm{u}(t,t_0,\cdot)}_{C^{1,\gamma}\paren{\mbr}}\leq M_l \left[\norm{\bm{f}(s,\cdot)}_{C^{\gamma}\paren{\mbr}}\right](t).
% \end{align}

% \end{thm}

% \bibliographystyle{plain}
\bibliographystyle{amsplain2link.bst}
\bibliography{references.bib}

% \bibliography{mylib}

%\bibliographystyle{amsplain2link.bst}
%\bibliography{bibliography.bib}

\end{document}